%% file: main.tex
\documentclass[10pt,a4paper,twoside,reqno]{amsart}

\input{Preamble}
\usepackage{mlmodern}

\begin{document}

\title[Complex spinorial forms, Brinkmann four-manifolds, and self-dual bundle gerbes]{Complex spinorial forms, Brinkmann four-manifolds, and self-dual bundle gerbes}

\author[Alejandro Gil-García]{Alejandro Gil-García \orcidlink{0000-0002-9370-241X}}
\address{Scuola Internazionale Superiore di Studi Avanzati, Trieste, Italy}
\email{agilgarc@sissa.it}

\author[C. S. Shahbazi]{C. S. Shahbazi \orcidlink{0000-0003-1185-9569}}
\address{Departamento de Matem\'aticas, Universidad UNED - Madrid, Reino de Espa\~na}
\email{cshahbazi@mat.uned.es}

\begin{abstract}
We develop the differential theory of complex spinorial forms associated with irreducible complex spinors across all dimensions and signatures. This framework enables the study of constrained parallelicity conditions for irreducible complex spinors by reformulating them as equivalent differential systems for exterior forms within a prescribed semi-algebraic body of the Kähler-Atiyah bundle. To illustrate this approach, we first apply it to the spin-c Killing spinor equation in low dimensions, refining existing results by relaxing standard assumptions of simply connectedness and completeness. Then, we proceed to apply our framework to supersymmetry conditions in supergravity, and we prove that every quasi-supersymmetric solution of Freedman’s gauged supergravity belongs to an explicit four-parameter family of geodesically complete, globally hyperbolic gyratonic Brinkmann waves with spherical wave fronts. Finally, we study the quasi-supersymmetric solutions of six-dimensional minimal supergravity, defined by a system that couples a self-dual curving on a bundle gerbe to a Lorentzian metric with an irreducible chiral spinor parallel under a metric connection with totally skew-symmetric torsion given by the curvature of the aforementioned curving. Along the way, we prove that a Lorentzian six-manifold admits a skew-torsion parallel spinor with an integrable screen bundle only if it admits a foliation whose leaves are locally conformally Kähler complex surfaces.\bigskip

\medskip


\end{abstract}
 
\maketitle

\setcounter{tocdepth}{1} 

\tableofcontents


\section{Introduction}


The main goal of this article is to develop a geometric framework that enables the study of any constrained parallelicity condition for an irreducible complex spinor in terms of an equivalent differential system for an algebraically constrained differential form, which is interpreted as the algebraic square of the original constrained parallel spinor \cite{AtiyahSeminar,Algebraic_Complex_Square_2025,Krasnov2024}. The development of this general framework, to which we refer as the \emph{theory of spinorial forms}, is inspired by the \emph{method of bilinears} \cite{Tod:1983pm,Gauntlett:2002nw,Gutowski:2003rg,Meessen:2006tu}, see also \cite{Ortin:2015hya,Tomasiello:2022dwe}, which emerged in the supergravity literature as a formalism for classifying and studying supersymmetric solutions. In line with its physical origin, we expect the theory of spinorial forms to be especially relevant in the following cases:

\begin{itemize}
    \item Generally, to study parallel spinors in pseudo-Riemannian signatures and in cases where the representation theory of the spin-c group is yet to be established.

    \item To study moduli spaces of solutions to parallelicity conditions for which the metric together with the spinor are both variables, which is the common situation.

    \item To study evolution problems involving spinorial equations for which the metric evolves non-trivially, and in particular to Cauchy problems for parallel spinors and supersymmetric configurations in Lorentzian signature, as illustrated in \cite{Murcia:2020zig,Murcia:2021dur,ShahbaziThesis} for real spinors of real type.
\end{itemize}

The reason why the theory of spinorial forms is potentially especially relevant for these cases is that it is a priori \emph{blind} to representation theory, and instead resorts to brute algebraic force. On the other hand, the theory of spinorial forms eliminates any explicit dependence on the choice of spinor bundle, and therefore removes at once any complication associated with dealing with bundles that formally depend on the choice of variable underlying metric \cite{BGM05}. This hints at the relevance of the theory of spinorial forms in the study of moduli spaces of spinorial parallelicity conditions.\medskip

A \emph{real version} of the framework of \emph{complex} spinorial forms presented here was developed for irreducible real spinors of real type in \cite{CLS21,Sha24}, which therefore restricts the signature $(p,q)$ of the underlying pseudo-Riemannian metric to $p-q\equiv_8 0, 1, 2$. Compared with the real case considered in \emph{opere citato}, the complex case presents new challenges; in particular, it involves two distinct notions of the square of a spinor, as already explained in \cite{Algebraic_Complex_Square_2025}. On the other hand, whereas the complex case considered in this article covers every dimension and signature, the real cases considered in \cite{CLS21,Sha24} are restricted to the aforementioned signatures to guarantee that the corresponding irreducible Clifford module is surjective. The real irreducible case of quaternionic type remains, therefore, open.


\subsection*{Outline and main results}


In Section \ref{sec:spinorialforms}, we review the algebraic theory of the square of an irreducible spinor as developed in \cite{Algebraic_Complex_Square_2025}, highlighting the existence of two natural squares associated with an irreducible complex spinor, which we call the Hermitian and complex-bilinear squares. They are obtained, as their name indicates, through the squaring construction associated with a choice of admissible Hermitian or complex-bilinear pairings, respectively.\medskip

In Section \ref{sec:constrained_parallel_spinors}, we introduce some geometric preliminaries and define the notion of constrained parallel spinor that we will use throughout the paper, which consists of a section $\eta$ of a bundle of irreducible complex spinors $S$ over a pseudo-Riemannian manifold $(M,g)$ satisfying $\cD\eta = 0$ and $\cQ(\eta) = 0$ for a connection $\cD$ on $S$ and an endomorphism $\cQ$ of $S$.\medskip

In Sections \ref{sec:Hermitian_parallel_forms} and \ref{sec:Complex_bilinear_parallel_forms}, we develop the differential theory for constrained parallel spinors through their Hermitian and complex-bilinear squares, respectively, hence establishing the theory of complex spinorial forms. Our main result characterizes constrained parallel spinors in terms of a necessary set of differential equations for their Hermitian squares, or an equivalent differential system for their complex-bilinear squares.\medskip

We summarize the results of these two sections in the following abbreviated theorems:

\begin{thm} 
\label{thm:hermitianeven}
Let $(M,g)$ be a strongly spin-c pseudo-Riemannian manifold of even dimension $d$ and let $\fra\in\Omega^1(M,\wedge T^*_\C M)$ and $\frq\in\Gamma(\wedge T^*_\C M\otimes W)$ be given, where $W$ is a vector bundle over $M$. Then $(M,g)$ admits a nowhere vanishing constrained parallel spinor relative to a pair $(\cA,\cQ)$ whose symbol is $(\fra,\frq)$ only if there exists a nowhere vanishing complex exterior form $\alpha\in\Omega_\C(M)$ satisfying the following algebraic and differential equations: 
\begin{gather*}
\alpha\diamond\alpha = 2^{\frac{d}{2}}\alpha^{(0)}\alpha,\qquad \alpha \diamond \beta \diamond \alpha=2^{\frac{d}{2}}(\alpha\diamond\beta)^{(0)}\alpha,\qquad(\pi^{\frac{1-s}{2}}\circ\tau)(\bar\kappa\alpha) = \kappa \bar{\alpha},\\
\nabla^g\alpha=\fra\diamond\alpha+\alpha\diamond(\pi^{\frac{1-s}{2}}\circ\tau)(\overline{\fra}),\qquad\frq\diamond\alpha=0
\end{gather*} 
for a complex exterior form $\beta\in\Omega_\C(M)$ satisfying $(\alpha\diamond\beta)^{(0)}\neq0$.
\end{thm}

\begin{thm}
\label{thm:hermitianodd}
Let $(M,g)$ be a strongly spin-c pseudo-Riemannian manifold of odd dimension $d$ and let $\fra\in\Omega^1(M,\wedge^< T^*_\C M)$ and $\frq\in\Gamma(\wedge^< T^*_\C M\otimes W)$ be given, where $W$ is a vector bundle over $M$. Then $(M,g)$ admits a nowhere vanishing constrained parallel spinor relative to a pair $(\cA,\cQ)$ whose symbol is $(\fra,\frq)$ only if there exists a nowhere vanishing complex exterior form $\alpha\in\Omega^<_\C(M)$ satisfying the following algebraic and differential equations: 
\begin{gather*}
\alpha\vee\alpha = 2^{\frac{d-1}{2}}\alpha^{(0)}\alpha,\qquad \alpha\vee\beta\vee\alpha=2^{\frac{d-1}{2}}(\alpha\vee\beta)^{(0)}\alpha,\qquad(\pi^{\frac{1-s}{2}}\circ\tau)(\bar\kappa\alpha) = \kappa \bar{\alpha},\\
\nabla^g\alpha=\fra\vee\alpha+\alpha\vee(\pi^{\frac{1-s}{2}}\circ\tau)(\overline{\fra}),\qquad\frq\vee\alpha=0
\end{gather*} 
for a complex exterior form $\beta\in\Omega^<_\C(M)$ satisfying $(\alpha\vee\beta)^{(0)}\neq0$.
\end{thm}

In the next two theorems, $\cL$ is a complex line bundle over $M$.

\begin{thm}
Let $(M,g)$ be a strongly spin-c pseudo-Riemannian manifold of even dimension $d$ and let $\fra\in\Omega^1(M,\wedge T^*_\C M)$ and $\frq\in\Gamma(\wedge T^*_\C M\otimes W)$ be given, where $W$ is a vector bundle over $M$. Then $(M,g)$ admits a nowhere vanishing constrained parallel spinor relative to a pair $(\cA,\cQ)$ whose symbol is $(\fra,\frq)$ if and only if there exists a nowhere vanishing $\cL$-valued complex exterior form $\alpha\in\Omega_\C(M,\cL)$ satisfying the following algebraic and differential equations: 
\begin{gather*}
    \alpha\diamond\alpha=2^{\frac{d}{2}}\alpha^{(0)}\otimes\alpha,\qquad \alpha\diamond\beta\diamond\alpha=2^{\frac{d}{2}}(\alpha\diamond\beta)^{(0)}\otimes\alpha,\qquad (\pi^{\frac{1-s}{2}}\circ\tau)(\alpha)=\sigma\alpha,\\
    \nabla^{g,A}\alpha=\fra\diamond\alpha+\alpha\diamond(\pi^{\frac{1-s}{2}}\circ\tau)(\fra),\qquad \frq\diamond\alpha=0
\end{gather*} 
for a $\cL$-valued complex exterior form $\beta\in\Omega_\C(M,\cL)$ satisfying $(\alpha\diamond\beta)^{(0)}\neq0$.
\end{thm}

\begin{thm} 
Let $(M,g)$ be a strongly spin-c pseudo-Riemannian manifold of odd dimension $d$ and let $\fra\in\Omega^1(M,\wedge^< T^*_\C M)$ and $\frq\in\Gamma(\wedge^< T^*_\C M\otimes W)$ be given, where $W$ is a vector bundle over $M$. Then $(M,g)$ admits a nowhere vanishing constrained parallel spinor relative to a pair $(\cA,\cQ)$ whose symbol is $(\fra,\frq)$ if and only if there exists a nowhere vanishing $\cL$-valued complex exterior form $\alpha\in\Omega^<_\C(M,\cL)$ satisfying the following algebraic and differential equations:
\begin{gather*}
    \alpha\vee\alpha=2^{\frac{d-1}{2}}\alpha^{(0)}\otimes\alpha,\qquad \alpha\vee\beta\vee\alpha=2^{\frac{d-1}{2}}(\alpha\vee\beta)^{(0)}\otimes\alpha,\qquad (\pi^{\frac{1-s}{2}}\circ\tau)(\alpha)=\sigma\alpha,\\
    \nabla^{g,A}\alpha=\fra\vee\alpha+\alpha\vee(\pi^{\frac{1-s}{2}}\circ\tau)(\fra),\qquad \frq\vee\alpha=0
\end{gather*} 
for a $\cL$-valued complex exterior form $\beta\in\Omega^<_\C(M,\cL)$ satisfying $(\alpha\vee\beta)^{(0)}\neq0$.
\end{thm}

We briefly explain the notation used in these theorems. The number $d\in\N$ denotes the real dimension of the manifold $M$, while the numbers $\sigma\in\Z_2$ and $s\in\Z_2$ denote the symmetry and adjoint type of the admissible bilinear pairing, respectively. Throughout the article, we will use $\Z_2=\{\pm1\}$. We denote by $(\wedge T^*_\C M,\diamond)$ the (complexified) Kähler-Atiyah bundle of $(M,g)$ when $d$ is even and by $(\wedge^<T^*_\C M,\vee)$ the (complexified) truncated Kähler-Atiyah bundle of $(M,g)$ when $d$ is odd. The symbol $\alpha^{(0)}\in C^{\infty}(M,\C)$ denotes the degree zero component of the complex exterior form $\alpha\in\Omega_\C(M)$. The map $\pi$ is the parity automorphism and $\tau$ is the reversion anti-automorphism. They act by $\pi(\alpha)=(-1)^k$ and $\tau(\alpha)=\smash{(-1)^{\binom{k}{2}}}\alpha$ for every $k$-form $\alpha$. Finally, $\kappa\in\U(1)$ is just a complex number of unit norm.\medskip

The \emph{strongly} spin-c condition is included in order to guarantee that every spin-c structure on $(M,g)$ reduces to its identity component. We refer the reader to Theorems \ref{thm:even_constrained_parallel_spinors}, \ref{thm:odd_constrained_parallel_spinors}, \ref{thm:even_CB_CPS}, and \ref{thm:odd_CB_CPS} for more details and for an expanded version of the previous theorems. It is important to recall that the formalism of spinorial forms laid out in the previous theorems establishes a correspondence between spinorial forms and pairs $(g,\eta)$, rather than simply a correspondence between spinorial forms and \emph{bare} spinors $\eta$, hinting at why this formalism is well-adapted to study evolution and moduli problems. The metric is naturally a variable of the system: explicit on one side and implicitly determined on the other.\medskip

The remaining sections of the article are devoted to applying the theory of complex spinorial forms established in the previous theorems to three different classes of parallelicity conditions.\medskip

In Section \ref{sec:lowKilling}, we apply the theory of spinorial forms to complex spin-c Killing spinors in low dimensions \cite{Moroianu1997}. The main goal is to illustrate the inner workings of the formalism and show how it readily recovers well-known results in the field. Among the several results recovered in this section, we highlight the following:

\begin{cor}
A three-dimensional Riemannian manifold $(M,g)$ admits a real spin-c Killing spinor with Killing number $\lambda\in\R\setminus\{0\}$ if and only if it is an $\alpha$-Sasakian manifold with $\alpha=2\lambda$.
\end{cor}

Hence, we recover the three-dimensional case of \cite[Thm.\ 4.1]{Moroianu1997} without assuming either simply connectedness or completeness. This illustrates the applicability of the formalism for studying the global topological and geometric structure of pseudo-Riemannian manifolds that carry parallel spinors. In particular, it is reasonable to expect that it can be used to obtain global geometric and topological results in other related types of Killing spinors studied recently in the literature \cite{Grosse_Nakad_2015,Lockman_2025,Carmona_2026}.\medskip

In Section \ref{sec:sugraspinors} we apply the framework of spinorial forms to the supersymmetry conditions of a particularly elegant supergravity theory, namely Freedman's minimal gauged four-dimensional supergravity. The bosonic sector of this theory reduces to the Einstein-Maxwell theory with a positive cosmological constant, and hence has as variables pairs of the form $(g,A)$, where $g$ is a Lorentzian metric on $M$ and $A$ is a connection on a principal $\U(1)$-bundle over the underlying four-manifold $M$. The supersymmetry conditions of the theory are given in Equation \eqref{eq:susytrans}. The first supersymmetry condition requires the existence of a chiral complex spinor $\eta$ parallel with respect to a $\Spin^c(3,1)$-connection constructed from $g$ and $A$. The second equation requires that the action of the curvature of $A$ on $\eta$ via Clifford multiplication is proportional to the identity, and therefore is a Lorentzian analog of an instanton condition for $A$. The main result of this section is given in Theorem \ref{thm:NoGoFreedman}, which we summarize here as follows:

\begin{thm}
Every standard stationary quasi-supersymmetric solution of Freedman's gauged supergravity is isometric to a Lorentzian four-manifold $(M,g)$ of the form:
\begin{equation*}
M = \mathbb{R}^2\times S^2 \, , \qquad g =  -(R^{-2}\psi^2+\frc)\dd\fru\otimes \dd\fru + \dd \fru \odot ( \dd\frv +  \ast_h\dd\psi) + h\, ,
\end{equation*}

\noindent
where $(\fru,\frv)$ are the Cartesian coordinates of $\mathbb{R}^2$, $h$ is the round metric on $S^2$ of radius $R$, $\frc$ is a real constant, and $\psi$ is a first eigenfunction of the Laplacian on $(S^2,h)$, that is, $\Delta_h\psi=2R^{-2}\psi$. Furthermore, every such standard stationary quasi-supersymmetric solution is geodesically complete and globally hyperbolic.
\end{thm} 

The notion of \emph{quasi-supersymmetric} solution of a supergravity theory is a novel addition of this article, and arises as a natural consequence of having two distinct squares of a given irreducible complex spinor, namely the Hermitian and complex-bilinear squares. Roughly speaking, a solution to a supergravity theory is said to be \emph{quasi-supersymmetric} if there exists an irreducible complex spinor whose Hermitian square satisfies the differential system induced by the supersymmetry conditions of the theory via Theorem \ref{thm:hermitianeven} or Theorem \ref{thm:hermitianodd}, depending on the dimension of the underlying manifold. This provides a first-order integrability condition for supergravity solutions that is a priori weaker than supersymmetry and may therefore be useful for exploring supergravity solutions beyond supersymmetry. It would be very interesting to compare the notion of quasi-supersymmetry proposed in this article with the supersymmetry-breaking deformation introduced in \cite{Legramandi:2019ulq} for pure spinors in type II supergravity.\medskip

In Section \ref{sec:skewtorsionparallel}, we apply the theory of spinorial forms to irreducible chiral complex spinors parallel under a metric connection with totally skew-symmetric torsion, to which we refer as skew-torsion parallel spinors. Our main result in this section is Theorem \ref{thm:transverse_LCK}, which proves that a spin Lorentzian six-manifold admits a skew-torsion parallel spinor with an integrable screen bundle only if it admits a foliation with locally conformally Kähler leaves. The notion of skew-torsion parallel spinor on a six-dimensional Lorentzian manifold is fundamental in ten-dimensional and six-dimensional supergravity and is one of the main motivations for their study.\medskip

In Section \ref{sec:sugra6d} we apply the theory of spinorial forms to minimal six-dimensional supergravity and its quasi-supersymmetric solutions. This supergravity theory can be formulated on a bundle gerbe with a connective structure, as we do in this article, or on an exact Courant algebroid, as proposed in \cite{Garcia-Fernandez:2015lsa,MarioStreets}. From the bundle gerbe perspective, the bosonic sector of the theory has as variables a Lorentzian metric and a curving. On the other hand, the supersymmetry condition of the theory requires the existence of a skew-torsion parallel spinor with respect to a metric connection with totally skew-symmetric torsion given by the curvature of the given curving. The general theory of its supersymmetric solutions was pioneered in \cite{Gutowski:2003rg}, see also \cite{Chamseddine:2003yy}, where the local isomorphism type of these solutions was determined in terms of a minimal set of differential equations in an adapted coordinate system. In particular, it was apparent that the class of supersymmetric solutions of six-dimensional minimal supergravity is vast and rich. In this article, we focus on the reduction of the general bosonic sector of the theory to the wave-front of a standard Kundt space-time, which in particular implies reducing the underlying bundle gerbe and its curvings. Our main result in this section is the complete decoupled reduced system, which we present as a purely four-dimensional differential system in Equation \eqref{eq:finalreducedsystem} for a Riemannian metric, a curving, and a function, whose solutions yield a large class of non-twisting solutions of six-dimensional minimal supergravity of infinite topological types. This reduced system, although different, belongs to the class of generalized Ricci soliton systems \cite{MarioStreets} in the Einstein frame \cite[Lemma 2.24]{Bunk:2025glt}, suggesting appropriate techniques for studying its solutions \cite{StreetsI,StreetsUstinovskiy,StreetsUstinovskiyII}. In Corollary \ref{cor:nonquasisusy}, we apply the results of Section \ref{sec:skewtorsionparallel} to show that these solutions are generically not quasi-supersymmetric, and hence non-supersymmetric. It would be very interesting to understand the relation between the Hermitian and complex-bilinear squares used in this article, with the results of \cite{Figueroa_Lorentzian_6d_2018}, which also involves the Hermitian square of the supersymmetry parameter.\medskip

Finally, we recall that in our conventions, the Laplacian on functions on a pseudo-Riemannian manifold $(M,g)$ is defined as follows: \begin{equation*}
\Delta_g f := \nabla^{g*}\dd f := -\Tr_g(\nabla^g\dd f) \, , \qquad f\in C^{\infty}(M).
\end{equation*}

\noindent
In Riemannian signature, this corresponds to the \emph{positive} Laplacian.

\begin{ack}
The work of AGG is supported by the Scuola Internazionale Superiore di Studi Avanzati (SISSA). The work of CSS was partially supported by the Leonardo grant LEO22-2-2155 of the BBVA Foundation and the research grant PID2023-152822NB-I00 of the Ministry of Science of the government of Spain. 
\end{ack}

\begin{center}
\begin{tikzpicture}
    \node[
        draw,                  
        minimum width=12cm,     
        minimum height=2cm,    
        align=center,          
        line width=1pt,        
        fill=white!10           
    ] 
    at (0,0) 
    {\quad \emph{Achtung}: we define Clifford algebras using the plus convention $v^2 = h(v,v)$!\quad}; 
\end{tikzpicture}
\end{center}


\section{Algebraic complex spinorial forms}
\label{sec:spinorialforms}


In this section we recall the algebraic characterization of the \emph{squares} of an irreducible complex spinor $\eta\in\Sigma$ for an irreducible complex Clifford module $(\Sigma,\gamma)$ associated to a real quadratic vector space $(V,h)$ of dimension $d$ and signature $(p,q)$, where $p$ is the number of \emph{pluses} and $q$ is the number of \emph{minuses}. We closely follow the exposition of \cite{Algebraic_Complex_Square_2025}, to which we refer for more details. Throughout the paper we will use $\Z_2=\{\pm1\}$.\medskip

Let $\Sigma$ be a complex vector space. A \emph{Hermitian pairing} $\scS$ on $\Sigma$ is a non-degenerate sesquilinear pairing $\scS\colon\Sigma\times\Sigma\to\C$ which is anti-linear in the second entry and satisfies $\scS(\eta_1,\eta_2)=\overline{\scS(\eta_2,\eta_1)}$ for all $\eta_1,\eta_2\in\Sigma$. We do not assume that $\scS$ is necessarily positive-definite. The \emph{Hermitian square maps} of such a $(\Sigma,\scS)$ are the quadratic maps $\hcE_\kappa\colon\Sigma\to\End(\Sigma)$ defined by: $$\hcE_\kappa(\eta):=\kappa\,\eta\otimes\scS(-,\eta),\quad \kappa\in\U(1).$$

Let $\scB\colon\Sigma\times\Sigma\to\C$ be a \emph{complex-bilinear pairing}. We assume $\scB$ to be either symmetric or skew-symmetric and we say that $\scB$ has \emph{symmetry type} $\sigma\in\Z_2$ if $\scB(\eta_1,\eta_2)=\sigma\scB(\eta_2,\eta_1)$. The \emph{complex-bilinear square map} of such a $(\Sigma,\scB)$ is the quadratic map $\cE\colon\Sigma\to\End(\Sigma)$ defined by: $$\cE(\eta):=\eta\otimes\scB(-,\eta).$$

Note that in contrast with the Hermitian case, any multiplicative factor included in the definition of $\cE$ can be reabsorbed by a complex homothety of $\eta$, and therefore is irrelevant for our purposes.


\subsection{Spinorial forms in even dimensions}


Let $(V,h)$ be a quadratic real vector space of even dimension $d$. We will use the \emph{Kähler-Atiyah model} $(\wedge V^*,\diamond)$ for the universal Clifford algebra $\mathrm{Cl}(V^*,h^*)$ of the quadratic vector space $(V^*,h^*)$ dual to $(V,h)$. We briefly recall this model in Appendix \ref{app:KA}. Let $\mathrm{Cl}(V^*,h^*)\to\End(\Sigma)$ be an irreducible complex representation of $\mathrm{Cl}(V^*,h^*)$ on a complex vector space $\Sigma$, whence $\dim_\C\Sigma=2^{\smash{\frac{d}{2}}}$. Consider the complexification $\mCl(V^*,h^*):=\mathrm{Cl}(V^*,h^*)\otimes_\R\C$ of the real Clifford algebra $\mathrm{Cl}(V^*,h^*)$. Since $d$ is even, the map: $$\gamma\colon\mCl(V^*,h^*)\to\End(\Sigma)$$ is an isomorphism of unital and associative algebras. Composing $\gamma$ with the Chevalley-Riesz isomorphism: $$\Psi\colon(\wedge V^*_\C,\diamond)\to\mCl(V^*,h^*)$$ gives an isomorphism of unital and associative algebras which we denote by: \begin{equation}\label{eq:Psi_gamma}
    \Psi_\gamma:=\gamma\circ\Psi\colon(\wedge V^*_\C,\diamond)\to(\End(\Sigma),\circ).
\end{equation}

To define the spinor square maps we combine the isomorphism $\Psi_\gamma\colon(\wedge V^*_\C,\diamond)\to\End(\Sigma)$ together with the squaring maps $\cE\colon\Sigma\to\End(\Sigma)$ and $\hcE_\kappa\colon\Sigma\to\End(\Sigma)$ introduced before as quadratic maps associated to a choice of complex-bilinear pairing $\scB$ and a choice of Hermitian pairing $\scS$ on $\Sigma$, respectively. For our purposes, we cannot use arbitrary non-degenerate complex-bilinear or Hermitian pairings on $\Sigma$. Instead, it is convenient to work with non-degenerate pairings on $\Sigma$ that are adapted to its structure as a complex Clifford module. This leads to the notion of \emph{admissible pairing}, introduced in \cite{AC97,ACDVP05}. In our case, we must distinguish between Hermitian and complex-bilinear admissible pairings.

\begin{definition}\label{def:admissible_pairing}
    Let $(\Sigma,\gamma)$ be an irreducible complex Clifford module. An \emph{admissible Hermitian pairing} on $(\Sigma,\gamma)$ is a non-degenerate Hermitian pairing $\scS$ on $\Sigma$ satisfying: $$\scS(\gamma(z)\eta_1,\eta_2)=\scS(\eta_1,\gamma(\overline{(\pi^{\frac{1-s}{2}}\circ\tau)(z)})\eta_2),\quad s\in\Z_2$$ for all $z\in\mCl(V^*,h^*)$ and $\eta_1,\eta_2\in\Sigma$. Similarly, an \emph{admissible complex-bilinear pairing} on $(\Sigma,\gamma)$ is a non-degenerate complex-bilinear pairing $\scB$ on $\Sigma$ satisfying: $$\scB(\gamma(z)\eta_1,\eta_2)=\scB(\eta_1,\gamma((\pi^{\frac{1-s}{2}}\circ\tau)(z))\eta_2),\quad s\in\mathbb{Z}_2$$ again for all $z\in\mCl(V^*,h^*)$ and $\eta_1,\eta_2\in\Sigma$. The sign factor $s$ is called the \emph{adjoint type} of the admissible pairing. We say that an admissible pairing is of \emph{positive adjoint type} if $s=+1$ and of \emph{negative adjoint type} if $s=-1$.
\end{definition}

Proof of the existence of admissible Hermitian and complex-bilinear pairings can be found in \cite{AC97,ACDVP05,HarveyBook,Mei13}, and is given explicitly in \cite[Prop.\ 3.9]{Algebraic_Complex_Square_2025} and \cite[Prop.\ 3.18]{Algebraic_Complex_Square_2025}, respectively. We will always assume that, given an irreducible complex Clifford module $(\Sigma,\gamma)$, the square maps $\cE\colon\Sigma\to\End(\Sigma)$ and $\hcE_\kappa\colon\Sigma\to\End(\Sigma)$ are constructed using an admissible pairing. We will refer to such triples $(\Sigma, \gamma ,\scS)$ and $(\Sigma, \gamma ,\scB)$ as \emph{Hermitian} and \emph{complex-bilinear paired Clifford modules}. We have the following commutative diagram:
\begin{equation*}
\begin{tikzcd}
{\mathbb{C}\mathrm{l}(V^*,h^*)} \arrow[r, "\gamma"]  & \mathrm{End}(\Sigma) & \Sigma \arrow[l, " \widehat{\mathcal{E}}_\kappa / \mathcal{E}"'] \\
(\wedge V^*_{\mathbb{C}},\diamond) \arrow[u, "\Psi"] \arrow[ru, "\Psi_\gamma"'] &                      &             
\end{tikzcd}
\end{equation*}

Consequently, we define the \emph{complex-bilinear} and \emph{Hermitian square spinor maps} respectively as follows: $$\cE_\gamma:=\Psi_\gamma^{-1}\circ\cE\colon\Sigma\to\wedge V^*_\C,\qquad \hcE_\gamma^\kappa:=\Psi_\gamma^{-1}\circ\hcE_\kappa\colon\Sigma\to\wedge V^*_\C.$$

We define the \emph{squares} of an irreducible complex spinor in even dimensions in terms of these square spinor maps.

\begin{definition}\label{def:squares}
    The \emph{complex-bilinear square} of an irreducible complex spinor $\eta \in \Sigma$ in even dimensions is the complex exterior form $\cE_\gamma(\eta)\in \wedge V^*_{\mathbb{C}}$. The \emph{Hermitian square} of an irreducible complex spinor $\eta \in \Sigma$ in even dimensions is the complex exterior form $\hcE_\gamma^\kappa(\eta)\in \wedge V^*_{\mathbb{C}}$ for $\kappa\in\U(1)$.
\end{definition}

An element in the image of $\hcE_\gamma^\kappa$ determines a complex spinor uniquely modulo a multiplicative unitary complex number. On the other hand, an element in the image of $\cE_\gamma$ determines a complex spinor uniquely modulo a sign. Hence, the complex-bilinear square of a complex spinor contains \emph{more} information than its Hermitian square and is equivalent to the spinor itself modulo a sign. This fact is key to developing the theory of \emph{complex spinorial forms} and its applications to the study of spinors parallel under a general connection on the spinor bundle.\medskip

We introduce the following terminology, which we borrow from \cite{Mei13}.

\begin{definition}
    Let $(\Sigma,\gamma)$ be an irreducible complex Clifford module. The \emph{dequantization} $\frq \in \wedge V^{\ast}_{\C}$ of an endomorphism $Q\in \End(\Sigma)$ is $\frq := \Psi_{\gamma}^{-1}(Q)$.
\end{definition}

\begin{lemma}[{\cite[Lemma 3.8]{Algebraic_Complex_Square_2025}}]\label{lemma:constrainedspinoreven}
    Let $Q\in \End(\Sigma)$. Then, the spinor $\eta\in \Sigma$ satisfies $Q(\eta) = 0$ if and only if $\frq\diamond \cE_{\gamma}(\eta) = 0$ if and only if $\frq\diamond \hcE^{\kappa}_{\gamma}(\eta) = 0$.
\end{lemma}

\begin{remark}\label{remark:equivalent_constrained_eq}
    Let $\alpha^t:=(\pi^{\frac{1-s}{2}}\circ\tau)(\alpha)$ be the $\scB$-transpose of $\alpha\in\wedge V^*_\C$, see \cite[Lemma 3.21]{Algebraic_Complex_Square_2025}. If $\alpha=\cE_\gamma(\eta)$, then the following equations are equivalent: $$\frq\diamond\alpha=0,\qquad \alpha\diamond(\pi^{\frac{1-s}{2}}\circ\tau)(\frq)=0.$$

    Similarly, let $\alpha^\dagger:=(\pi^{\frac{1-s}{2}}\circ\tau)(\overline{\alpha})$ be the $\scS$-transpose of $\alpha\in\wedge V^*_\C$, see \cite[Lemma 3.12]{Algebraic_Complex_Square_2025}. If $\alpha=\hcE^{\kappa}_{\gamma}(\eta)$ for some $\kappa\in\U(1)$, then the following equations are equivalent: $$\frq\diamond\alpha=0,\qquad\alpha\diamond(\pi^{\frac{1-s}{2}}\circ\tau)(\overline{\frq})=0.$$
\end{remark}

For further reference, we introduce the standard spin-c group as follows:
\begin{equation*}
\Spin^c(p,q):=\Spin(p,q)\U(1):=(\Spin(p,q)\times\U(1))/\mathbb{Z}_2 \subset\mCl(d)    
\end{equation*}
and we denote its identity component by $\Spin^c_o(p,q)$. Hence, we view elements in the group $\Spin^c(p,q)$ as equivalence classes $[u,z]$ with $u\in\Spin(p,q)$ and $z\in\U(1)$. There is a natural homomorphism $\Ad^c\colon\Spin^c(p,q)\to\SO(p,q)\times\U(1)$ defined by: 
\begin{equation*}
\Ad^c_{[u,z]}:=(\Ad_u,z^2) \, , \qquad [u,z] \in \Spin^c(p,q),
\end{equation*}
where $\Ad \colon \Spin(p,q) \to \SO(p,q)$ is the standard double cover of $\SO(p,q)$ defined via the adjoint representation of the spin group. Then, the map $\Ad^c$ is a double cover of $\SO(p,q)\times\U(1)$. \medskip

We now give the algebraic characterization of the complex spinorial forms associated to a Hermitian paired Clifford module $(\Sigma,\gamma,\scS)$. Recall that we denote by $\alpha^{(0)}\in\C$ the degree zero component of the exterior form $\alpha\in\wedge V^*_\C$, and that: $$\pi(\alpha)=(-1)^k\alpha,\qquad\tau(\alpha)=(-1)^{\binom{k}{2}}\alpha$$ for every $k$-form $\alpha\in\wedge^k V^*_\C$.

\begin{thm}[{\cite[Thm.\ 3.13 and Cor.\ 3.23]{Algebraic_Complex_Square_2025}}]\label{thm:even_Hermitian_forms}
    Let $(\Sigma,\gamma)$ be an irreducible complex Clifford module equipped with an admissible Hermitian pairing $\scS$ of adjoint type $s\in \mathbb{Z}_2$. Then the following statements are equivalent for a complex exterior form $\alpha\in\wedge V^*_\C$: \begin{enumerate}
        \item[\normalfont(a)] $\alpha$ is the Hermitian square of a spinor $\eta\in\Sigma$, that is, $\alpha = \hcE_\gamma^\kappa(\eta)$ for some $\kappa\in\U(1)$.
        \item[\normalfont(b)] $\alpha$ satisfies the following equations: $$\alpha\diamond\alpha = 2^{\frac{d}{2}}\alpha^{(0)}\alpha,\qquad \alpha \diamond \beta \diamond \alpha=2^{\frac{d}{2}}(\alpha\diamond\beta)^{(0)}\alpha,\qquad(\pi^{\frac{1-s}{2}}\circ\tau)(\bar\kappa\alpha) = \kappa \bar{\alpha}$$ for an exterior form $\beta\in\wedge V^*_\C$ satisfying $(\alpha\diamond\beta)^{(0)}\neq0$.
        \item[\normalfont(c)] The following equations hold: $$\alpha \diamond \beta \diamond \alpha=2^{\frac{d}{2}}(\alpha\diamond\beta)^{(0)}\alpha,\qquad(\pi^{\frac{1-s}{2}}\circ\tau)(\bar\kappa\alpha) = \kappa \bar{\alpha}$$ for every exterior form $\beta\in\wedge V^*_\C$.
    \end{enumerate}

    If in addition the spinor $\eta\in\Sigma$ is chiral of chirality $\mu\in\Z_2$, then we have to add the condition: $$i^{q+\frac{d}{2}}*(\pi\circ\tau)(\alpha)=\mu\alpha.$$
\end{thm}

Let $\{e^1,\ldots, e^d\}$ be an orthonormal basis of $(V^*,h^*)$ and define $\gamma^{i_1\cdots i_k}:=\Psi_\gamma(e^{i_1})\circ\cdots\circ\Psi_\gamma(e^{i_k})$. Then we can give an explicit form of the Hermitian square $\widehat{\alpha}_\eta$ of an irreducible complex spinor $\eta\in\Sigma$ in terms of an explicit orthonormal basis of $(V^*,h^*)$, see \cite[Eq.\ 10]{Algebraic_Complex_Square_2025}: \begin{equation*}\label{eq:explicitHermitian}
    \widehat{\alpha}_{\eta} = \frac{\kappa}{2^{\frac{d}{2}}} \scS(\eta,\eta) + \frac{\kappa}{2^{\frac{d}{2}}} \sum_{k=1}^d \sum_{i_1 < \cdots < i_k} \scS((\gamma^{i_1 \cdots i_k})^{-1}\eta , \eta)\, e^{i_1}\wedge  \cdots \wedge e^{i_k}.
\end{equation*}

Moreover, in \cite[Section 3.3]{Algebraic_Complex_Square_2025} it is shown that the quadratic maps $\hcE_\gamma^\kappa \colon \Sigma \to \wedge V^*_\C$ are equivariant with respect to the natural action of $\Spin_o^c(p,q)$ on $\Sigma$ and $\wedge V^*_\C$ defined via $\gamma$ and the adjoint representation $\Ad\colon \Spin_o^c(p,q) \to \SO_o(p,q)$, respectively.\medskip

We now give the algebraic characterization of the complex spinorial forms associated with a complex-bilinear paired Clifford module $(\Sigma,\gamma,\scB)$.

\begin{thm}[{\cite[Thm.\ 3.22 and Cor.\ 3.24]{Algebraic_Complex_Square_2025}}]\label{thm:bilinearsquare}
    Let $(\Sigma,\gamma)$ be an irreducible complex Clifford module equipped with an admissible complex-bilinear pairing $\scB$ of adjoint type $s\in \mathbb{Z}_2$ and symmetry type $\sigma\in \mathbb{Z}_2$. Then the following statements are equivalent for a complex exterior form $\alpha\in\wedge V^*_\C$: \begin{enumerate}
        \item[\normalfont(a)] $\alpha$ is the complex-bilinear square of a spinor $\eta\in\Sigma$, that is, $\alpha=\cE_\gamma(\eta)$.
        \item[\normalfont(b)] $\alpha$ satisfies the following equations: $$\alpha\diamond\alpha = 2^{\frac{d}{2}} \alpha^{(0)}\alpha,\qquad \alpha \diamond \beta \diamond \alpha = 2^{\frac{d}{2}}(\alpha\diamond\beta)^{(0)}\alpha,\qquad (\pi^{\frac{1-s}{2}}\circ\tau) (\alpha) = \sigma \alpha$$ for an exterior form $\beta\in\wedge V^*_\C$ satisfying $(\alpha\diamond\beta)^{(0)} \neq 0$.
        \item[\normalfont(c)] The following equations hold: $$\alpha \diamond \beta \diamond \alpha = 2^{\frac{d}{2}}(\alpha\diamond\beta)^{(0)}\alpha,\qquad (\pi^{\frac{1-s}{2}}\circ\tau) (\alpha) = \sigma \alpha$$ for every exterior form $\beta\in\wedge V^*_\C$.
    \end{enumerate}
    
    If in addition the spinor $\eta\in\Sigma$ is chiral of chirality $\mu\in\Z_2$, then we have to add the condition: $$i^{q+\frac{d}{2}}*(\pi\circ\tau)(\alpha)=\mu\alpha.$$
\end{thm}

As in the Hermitian case, for the complex-bilinear square $\alpha_\eta$ of a spinor $\eta\in\Sigma$, we have the following form in terms of an orthonormal basis (see \cite[Eq.\ 13]{Algebraic_Complex_Square_2025}): \begin{equation*}\label{eq:explicitbilinear}
    \alpha_{\eta} = \frac{1}{2^{\frac{d}{2}}} \scB(\eta,\eta) + \frac{1}{2^{\frac{d}{2}}} \sum_{k=1}^d \sum_{i_1 < \cdots < i_k} \scB((\gamma^{i_1 \cdots i_k})^{-1}\eta , \eta)\, e^{i_1}\wedge  \cdots \wedge e^{i_k}.
\end{equation*}

Let $\Ad\colon\Spin_o(p,q)\to\SO_o(p,q)$ be the standard double cover of $\SO_o(p,q)$ and define the following map: $$\Ad^c\colon \Spin_o^c(p,q) \to \SO_o(p,q) \times \U(1),\qquad [u,z]\mapsto \Ad^c_{[u,z]} := (\Ad_u,z^2).$$

Then, by \cite[Section 3.5]{Algebraic_Complex_Square_2025}, the quadratic map $\cE_\gamma \colon \Sigma \to \wedge V^*_\C$ is equivariant with respect to the natural action of $\Spin_o^c(p,q)$ on $\Sigma$ and $\wedge V^*_\C$ defined via $\gamma$ and the double cover representation $\Ad^c \colon \Spin_o^c(p,q) \to \SO_o(p,q)\times \U(1)$, respectively.


\subsection{Spinorial forms in odd dimensions}


Let $V$ be an oriented real vector space of odd dimension $d$ equipped with a non-degenerate metric of signature $(p,q)$. Let $(V^*,h^*)$ be the quadratic vector space dual to $(V,h)$ and denote by $\mathrm{Cl}(V^*,h^*)$ the Clifford algebra of $(V^*,h^*)$. Consider the complexification $\mCl(V^*,h^*) := \mathrm{Cl}(V^*,h^*) \otimes_\R\C$. Denote by $\nu$ the pseudo-Riemannian volume form of $\mathrm{Cl}(V^*,h^*)$ and define the \emph{complex volume form} in $\mCl(V^*,h^*)$ by: $$\nu_\C:=i^{q+\frac{d-1}{2}}\nu\in\mCl(V^*,h^*).$$

It can be easily seen that the complex volume form satisfies $\nu_\C^2=1$ and lies in the center of $\mCl(V^*,h^*)$. Therefore, $\mCl(V^*,h^*)$ is non-simple and splits as a direct sum of unital and associative algebras: \begin{equation}\label{eq:splitting_Clifford}
    \mCl(V^*,h^*)=\mCl_+(V^*,h^*)\oplus\mCl_-(V^*,h^*),
\end{equation} where: $$\mCl_\ell(V^*,h^*)=\{z\in\mCl(V^*,h^*)\mid\nu_\C z=\ell z\}=\frac{1}{2}(1+\ell\nu_\C)\mCl(V^*,h^*),\quad\ell\in\mathbb{Z}_2.$$

The algebra $\mCl_\ell(V^*,h^*)$ is simple and isomorphic to $\Mat(2^{\frac{d-1}{2}},\C)$. Therefore, $\mCl(V^*,h^*)$ admits two irreducible left Clifford modules of complex dimension $2^{\smash{\frac{d-1}{2}}}$: $$\gamma_\ell\colon\mCl(V^*,h^*)\to\End(\Sigma)$$ that correspond to the projection of $\mCl(V^*,h^*)$ to the factor $\mCl_\ell(V^*,h^*)$ composed with an isomorphism of the later to the algebra $\End(\Sigma)$ of endomorphisms of a complex vector space $\Sigma$ of dimension $2^{\smash{\frac{d-1}{2}}}$. These two Clifford modules are distinguished by the value they take at the complex volume form $\nu_\C\in\mCl(V^*,h^*)$, namely: $$\gamma_\ell(\nu_\C)=\ell\,\Id\in\End(\Sigma).$$

We transport the splitting \eqref{eq:splitting_Clifford} of $\mCl(V^*,h^*)$ to the Kähler-Atiyah algebra $(\wedge V^*_\C,\diamond)$ through the Chevalley-Riesz isomorphism $\Psi\colon(\wedge V^*_\C,\diamond)\to\mCl(V^*,h^*)$. This gives: $$(\wedge V^*_\C,\diamond)=(\wedge_+V^*_\C,\diamond)\oplus(\wedge_- V^*_\C,\diamond),$$ where: $$\wedge_\ell V^*_\C=\{\alpha\in\wedge V^*_\C\mid\nu_\C\diamond\alpha=\ell\alpha\}=\{\alpha\in\wedge V^*_\C\mid i^{q+\frac{d-1}{2}}*\tau(\alpha)=\ell\alpha\},\quad\ell\in\mathbb{Z}_2.$$

Here we have used the identity: 
\begin{equation}\label{eq:mult_volume_form_odd}
    \alpha\diamond\nu_\C=\nu_\C\diamond\alpha=i^{q+\frac{d-1}{2}}*\tau(\alpha)
\end{equation} 
given in Proposition \ref{prop:product_volume_form}. By composing the Chevalley-Riesz isomorphism with the irreducible complex representation $\gamma_\ell\colon\mCl(V^*,h^*)\to\End(\Sigma)$ we obtain a surjective morphism of unital and associative complex algebras that we denote by: $$\Psi_\ell:=\gamma_\ell\circ\Psi\colon(\wedge V^*_\C,\diamond)\to\End(\Sigma).$$

Let $\cP_\ell\colon\wedge V^*_\C\to\wedge_\ell V^*_\C$ be the natural projection given explicitly by: \begin{equation}\label{eq:projection_l}
    \cP_\ell(\alpha):=\frac{1}{2}(1+\ell\nu_\C)\diamond\alpha=\frac{1}{2}(\alpha+i^{q+\frac{d-1}{2}}\ell*\tau(\alpha)),
\end{equation} where we have used the identity \eqref{eq:mult_volume_form_odd}. Consider the canonical linear inclusion $\iota_\ell\colon\wedge_\ell V_\C^*\hookrightarrow\wedge V_\C^*$, which is a right inverse to $\cP_\ell$. In particular: $$\Psi_\ell\circ\iota_\ell\colon(\wedge_\ell V^*_\C,\diamond)\to\End(\Sigma)$$ is an isomorphism of unital and associative algebras. Following \cite{LBC13,LB13,LBC16}, we define: \begin{equation*}\label{eq:vee_product}
    \wedge^<V^*_\C:=\bigoplus_{k=0}^{\frac{d-1}{2}}\wedge^kV^*_\C.
\end{equation*}

Note that $\wedge V^*_\C=\wedge^<V_\C^*\oplus*\wedge^<V^*_\C$. Then, restricting $\cP_\ell\colon\wedge V^*_\C\to\wedge_\ell V^*_\C$ to $\wedge^<V^*_\C$, we obtain an isomorphism of vector spaces that we use to transport the algebra product in $(\wedge_\ell V^*_\C,\diamond)$ to $\wedge^<V^*_\C$. For every pair $\alpha,\beta\in\wedge^<V^*_\C$ we define the \emph{truncated geometric product} by: $$\alpha\vee\beta:=2\cP_<(\cP_\ell(\alpha\diamond\beta)),$$ where $\cP_<\colon\wedge V^*_\C\to\wedge^<V^*_\C$ is the natural projection. By construction, $(\wedge_\ell V^*_\C,\diamond)$ and the \emph{truncated Kähler-Atiyah algebra} $(\wedge^<V^*_\C,\vee)$ are naturally isomorphic as unital and associative complex algebras. For further reference we introduce the following linear map: \begin{equation}\label{eq:truncated_iso}
    \Psi^<_\ell:=\Psi_\ell\circ\iota_\ell\circ\cP_\ell\vert_{\wedge^<V^*_\C}\colon(\wedge^<V^*_\C,\vee)\to\End(\Sigma),
\end{equation} which, by the previous discussion, is an isomorphism of unital and associative algebras. Altogether, we obtain the following commutative diagram of unital and associative algebras:
$$\begin{tikzcd}
                                                                                                                                                        &  & {\mathbb{C}\mathrm{l}_\ell(V^*,h^*)} \arrow[d]                                                                                                                                   &  &                                                                                                                                           \\
{\mathbb{C}\mathrm{l}(V^*,h^*)} \arrow[rru, "\mathrm{pr}_\ell"] \arrow[rr, "\gamma_\ell"]                                                  &  & \mathrm{End}(\Sigma)                                                                                                                                                             &  & {(\wedge^<V^*_{\mathbb{C}},\vee)} \arrow[ll, "\Psi^<_\ell"'] \arrow[lldd, "\mathcal{P}_\ell|_{\wedge^<V^*_{\mathbb{C}}}"', shift right] \\
                                                                                                                                                        &  &                                                                                                                                                                                  &  &                                                                                                                                           \\
{(\wedge V^*_{\mathbb{C}},\diamond)} \arrow[uu, "\Psi"] \arrow[rruu, "\Psi_\ell"] \arrow[rr, "\mathcal{P}_\ell", shift left] &  & {(\wedge_\ell V^*_{\mathbb{C}},\diamond)} \arrow[uu, "\Psi_\ell\circ\iota_\ell"] \arrow[ll, "\iota_\ell", shift left] \arrow[rruu, "2\mathcal{P}_<"', shift right] &  &                                                                                                                                          
\end{tikzcd}$$

Recall that, similarly to the even-dimensional case, a Hermitian pairing $\scS\colon\Sigma\times\Sigma\to\C$ is called admissible if it satisfies: $$\scS(\gamma_\ell(z)\eta_1,\eta_2)=\scS(\eta_1,\gamma_\ell(\overline{(\pi^{\frac{1-s}{2}}\circ\tau)(z)})\eta_2),\quad s\in\mathbb{Z}_2,$$ whereas a complex-bilinear pairing $\scB\colon\Sigma\times\Sigma\to\C$ is called admissible if it satisfies: $$\scB(\gamma_\ell(z)\eta_1,\eta_2)=\scB(\eta_1,\gamma_\ell((\pi^{\frac{1-s}{2}}\circ\tau)(z))\eta_2),\quad s\in\mathbb{Z}_2$$ for all $z\in\mCl(V^*,h^*)$ and $\eta_1,\eta_2\in\Sigma$. Proof of the existence of admissible Hermitian and complex-bilinear pairings can be found in \cite{AC97,ACDVP05,HarveyBook,Mei13}, and is given explicitly in \cite[Prop.\ 4.3]{Algebraic_Complex_Square_2025} and \cite[Props.\ 4.8 and 4.11]{Algebraic_Complex_Square_2025}, respectively. We will always assume that, given an irreducible complex Clifford module $(\Sigma,\gamma_\ell)$, the square maps $\cE\colon\Sigma\to\End(\Sigma)$ and $\hcE_\kappa\colon\Sigma\to\End(\Sigma)$ are constructed using an admissible pairing.\medskip

Let $(\Sigma,\gamma_\ell)$ be an irreducible complex Clifford module equipped with a complex-bilinear pairing $\scB$ and a Hermitian pairing $\scS$. Using the isomorphism $\Psi^<_\ell\colon(\wedge^<V^*_\C,\vee)\to\End(\Sigma)$ given in \eqref{eq:truncated_iso}, we define the \emph{complex-bilinear} and \emph{Hermitian square spinor maps} respectively as follows: $$\cE_\ell:=(\Psi_\ell^<)^{-1}\circ\cE\colon\Sigma\to\wedge^<V^*_\C,\qquad\hcE^\kappa_\ell:=(\Psi_\ell^<)^{-1}\circ\hcE_\kappa\colon\Sigma\to\wedge^<V^*_\C.$$

As in the even-dimensional case, we have the following characterization of \emph{constrained} spinors.

\begin{lemma}[{\cite[Lemma 4.2]{Algebraic_Complex_Square_2025}}]\label{lemma:constrainedspinorodd}
    Let $Q\in \End(\Sigma)$. Then, the spinor $\eta\in \Sigma$ satisfies $Q(\eta) = 0$ if and only if $\frq\vee \cE_{\ell}(\eta) = 0$ if and only if $\frq\vee \hcE^{\kappa}_{\ell}(\eta) = 0$, where $\frq:=(\Psi_\ell^<)^{-1}(Q)$ is the \emph{dequantization} of $Q$.
\end{lemma}

We now give the algebraic characterization of the complex spinorial forms associated to a Hermitian paired Clifford module $(\Sigma,\gamma_\ell,\scS)$.

\begin{thm}[{\cite[Thm.\ 4.4]{Algebraic_Complex_Square_2025}}]\label{thm_Hermitian_square_odd}
    Let $(\Sigma,\gamma_\ell)$ be an irreducible complex Clifford module equipped with an admissible Hermitian pairing $\scS$ of adjoint type $s\in\mathbb{Z}_2$. Then the following statements are equivalent for a complex exterior form $\alpha\in\wedge^<V^*_\C$: \begin{enumerate}
        \item[\normalfont(a)] $\alpha$ is the Hermitian square of a spinor $\eta\in\Sigma$, that is, $\alpha=\hcE_\ell^\kappa(\eta)$ for some $\kappa\in\U(1)$.
        \item[\normalfont(b)] $\alpha$ satisfies the following equations: $$\alpha\vee\alpha = 2^{\frac{d-1}{2}}\alpha^{(0)}\alpha,\qquad\alpha\vee\beta\vee\alpha=2^{\frac{d-1}{2}}(\alpha\vee\beta)^{(0)}\alpha,\qquad(\pi^{\frac{1-s}{2}}\circ\tau)(\bar\kappa\alpha) = \kappa \bar{\alpha}$$ for an exterior form $\beta\in\wedge^<V^*_\C$ satisfying $(\alpha\vee\beta)^{(0)}\neq0$.
        \item[\normalfont(c)] The following equations hold: $$\alpha\vee\beta\vee\alpha=2^{\frac{d-1}{2}}(\alpha\vee\beta)^{(0)}\alpha,\qquad(\pi^{\frac{1-s}{2}}\circ\tau)(\bar\kappa\alpha) = \kappa \bar{\alpha}$$ for every exterior form $\beta\in\wedge^<V^*_\C$.
    \end{enumerate}
\end{thm}

As it happened in the even-dimensional case, we can explicitly expand the Hermitian square $\widehat{\alpha}_\eta\in\wedge^<V^*_\C$ of an irreducible complex spinor $\eta\in\Sigma$ in terms of an orthonormal basis $\{e^1,\ldots,e^d\}$ of $(V^*,h^*)$: \begin{equation*}\label{eq:oddhermitianexpansion}
    \widehat{\alpha}_{\eta} = \frac{\kappa}{2^{\frac{d-1}{2}}} \scS(\eta,\eta) + \frac{\kappa}{2^{\frac{d-1}{2}}} \sum_{k=1}^{\frac{d-1}{2}} \sum_{i_1 < \cdots < i_k} \scS((\gamma_\ell^{i_1 \cdots i_k})^{-1}\eta , \eta)\, e^{i_1}\wedge  \cdots \wedge e^{i_k},
\end{equation*} where $\gamma_\ell^{i_1\cdots i_k}:=\Psi^<_\ell(e^{i_1})\circ\cdots\circ\Psi^<_\ell(e^{i_k})$. Furthermore, the Hermitian spinor square maps in odd dimensions are also equivariant with respect to the natural action of $\Spin^c_o(p,q)$.\medskip

We now give the algebraic characterization of the complex spinorial forms associated to a complex-bilinear paired Clifford module $(\Sigma,\gamma_\ell,\scB)$.

\begin{thm}[{\cite[Thm.\ 4.12]{Algebraic_Complex_Square_2025}}]\label{thm:odd_complex-bilinear_square}
    Let $(\Sigma,\gamma_\ell)$ be an irreducible complex Clifford module equipped with an admissible complex-bilinear pairing $\scB$ of adjoint type $s\in\mathbb{Z}_2$ and symmetry type $\sigma\in\mathbb{Z}_2$. Then the following statements are equivalent for a complex exterior form $\alpha\in\wedge^<V^*_\C$: \begin{enumerate}
        \item[\normalfont(a)] $\alpha$ is the complex-bilinear square of a spinor $\eta\in\Sigma$, that is, $\alpha=\cE_\ell(\eta)$.
        \item[\normalfont(b)] $\alpha$ satisfies the following equations: $$\alpha\vee\alpha=2^{\frac{d-1}{2}}\alpha^{(0)}\alpha,\qquad\alpha\vee\beta\vee\alpha=2^{\frac{d-1}{2}}(\alpha\vee\beta)^{(0)}\alpha,\qquad(\pi^{\frac{1-s}{2}}\circ\tau)(\alpha)=\sigma\alpha$$ for an exterior form $\beta\in\wedge^<V^*_\C$ satisfying $(\alpha\vee\beta)^{(0)}\neq0$.
        \item[\normalfont(c)] The following equations hold: $$\alpha\vee\beta\vee\alpha=2^{\frac{d-1}{2}}(\alpha\vee\beta)^{(0)}\alpha,\qquad(\pi^{\frac{1-s}{2}}\circ\tau)(\alpha)=\sigma\alpha$$ for every exterior form $\beta\in\wedge^<V^*_\C$.
    \end{enumerate}
\end{thm}

Analogously to the Hermitian case, the complex-bilinear square $\alpha_\eta\in\wedge^<V^*_\C$ of an irreducible complex spinor $\eta\in\Sigma$ can be expanded as follows: \begin{equation*}\label{eq:odd_explicit_complexbilinearsquare}
    \alpha_{\eta} = \frac{1}{2^{\frac{d-1}{2}}} \scB(\eta,\eta) + \frac{1}{2^{\frac{d-1}{2}}} \sum_{k=1}^{\frac{d-1}{2}} \sum_{i_1 < \cdots < i_k} \scB((\gamma_\ell^{i_1 \cdots i_k})^{-1}\eta , \eta)\, e^{i_1}\wedge  \cdots \wedge e^{i_k}
\end{equation*} in terms of any orthonormal basis $\{e^1,\ldots,e^d\}$ of $(V^*,h^*)$. Furthermore, the complex-bilinear spinor square map in odd dimensions is also equivariant with respect to the natural action of $\Spin^c_o(p,q)$.


\section{Constrained parallel complex spinors}
\label{sec:constrained_parallel_spinors}


In this section, we recall the spinorial background that we will use throughout the article to then introduce the notion of \emph{constrained parallel spinor}, which is our main object of study.


\subsection{Bundles of irreducible complex Clifford modules}


Let $(M,g)$ denote a connected and \emph{oriented} pseudo-Riemannian manifold of signature $(p,q)$ and dimension $d=p+q$. The bundle of Clifford algebras of $(M,g)$ will be denoted by $\mathrm{Cl}(M,g)$, whereas its complexification will be denoted by $\mCl(M,g)$. Let $\pi$ and $\tau$ be the canonical automorphism and anti-automorphism of the Clifford bundle, given by fiberwise extension of the corresponding objects defined fiberwise.

\begin{definition}
A \emph{bundle of complex Clifford modules}, or a \emph{complex spinor bundle} for short, on $(M,g)$ is a pair $(S,\gamma)$, where $S$ is a complex vector bundle on $M$ and $\gamma\colon\mCl(M,g)\to\End(S)$ is a morphism of bundles of unital and associative complex algebras, where $\End(S)$ denotes the bundle of endomorphisms of $S$.
\end{definition}

Given a complex spinor bundle $(S,\gamma)$, we obtain a natural morphism of unital and associative \emph{real} algebras:
\begin{equation*}
\gamma \colon \mathrm{Cl}(M,g) \to \End(S)
\end{equation*}

obtained by restricting $\gamma$ to $\mathrm{Cl}(M,g) \hookrightarrow \mCl(M,g) = \mathrm{Cl}(M,g)\otimes_{\R}\mathbb{C}$. Since $M$ is connected, any bundle of complex Clifford modules $(S,\gamma)$ is modeled on a given Clifford module called its \emph{type}, which is well-defined modulo \emph{unbased} isomorphisms of Clifford modules, see \cite{LS18,LS19} for more details. If this type is irreducible, then we say that $(S,\gamma)$ is a bundle of irreducible complex spinors or an irreducible complex spinor bundle for short. In this case, the rank of a bundle of irreducible complex Clifford modules is $\rk_\C S=\dim_\C\Sigma=2^{\smash{[\frac{d}{2}]}}$, where $[\frac{d}{2}]$ denotes the integer part of $\frac{d}{2}$ and $d$ is the dimension of $M$. In this situation, it is shown in \cite{LS18,LS19,LS22_complex_type} that an orientable pseudo-Riemannian manifold $(M,g)$ admits a bundle of irreducible complex Clifford modules if and only if it admits a spin-c structure $Q \to (M,g)$, in which case $S$ may be obtained from $Q$ via the associated vector bundle construction through an irreducible representation of $\Spin^c(p,q)$. In the following, given an irreducible complex spinor bundle $(S,\gamma)$, we will always assume that $S$ is given as an associated vector bundle to a given spin structure $Q$, that is, we fix a presentation of the form:
\begin{equation*}
S = Q\times_{\varrho} \Sigma,
\end{equation*}
where $\varrho\colon \Spin^c(p,q)\to \mathrm{GL}(\Sigma)$ is an irreducible representation of $\Spin^c(p,q)$ on the complex vector space $\Sigma$. In particular, we have a canonically defined principal $\U(1)$-bundle:
\begin{equation*}
P = Q\times_{\xi} \U(1)
\end{equation*}
associated to $Q$ via the \emph{square} morphism $\xi \colon \Spin^c(p,q) \to \U(1)$ defined by $\xi([u,z])=z^2$. Equivalently, we can consider $P$ as a Hermitian complex line bundle $\cL\to M$. We will respectively refer to $P$ and $\cL$ as the \emph{characteristic} $\U(1)$ and complex line bundles associated to $S$.\medskip

To proceed further, we need to distinguish between even and odd dimensions.


\subsubsection{The even-dimensional case}


We denote by $(\wedge T^*_\C M,\diamond)$ the complex exterior bundle: $$\wedge T^*_\C M:=\bigoplus_{k=0}^d\wedge^kT^*_\C M$$ equipped with the pointwise extension of the geometric product $\diamond$, which depends on the metric $g$. This bundle of unital and associative algebras is called the \emph{(complexified) Kähler-Atiyah bundle} of $(M,g)$ (see \cite{LBC13,LBC16}). The Chevalley-Riesz map extends to a unital isomorphism of bundles of algebras: $$\Psi\colon(\wedge T^*_\C M,\diamond)\to\mCl(M,g),$$ which allows us to view the Kähler-Atiyah bundle as a model for the Clifford bundle. We again denote by $\pi$ and $\tau$ the (anti-)automorphisms of the Kähler-Atiyah bundle obtained by transporting the corresponding objects from the Clifford bundle through $\Psi$.\medskip

The following is a well-known property of the Kähler-Atiyah bundle $(\wedge T^*_\C M,\diamond)$, which also follows from the definition of the geometric product $\diamond$ (see \cite{LBC13,LBC16}).

\begin{prop}\label{prop:geometric_product_derivations}
    The canonical extension to $\wedge T^*_\C M$ of the Levi-Civita connection $\nabla^g$ of $(M,g)$, which we again denote by $\nabla^g$, acts by derivations of the geometric product, that is: $$\nabla^g(\alpha\diamond\beta)=(\nabla^g\alpha)\diamond\beta+\alpha\diamond(\nabla^g\beta)$$ for all $\alpha,\beta\in\Omega_\C(M):=\Gamma(\wedge T^*_\C M)$.
\end{prop}

\begin{proof}
    Let $X,Y\in\Gamma(TM)$. It is well-known that: $$\nabla^g_X(\alpha\wedge\beta)=(\nabla^g_X\alpha)\wedge\beta+\alpha\wedge(\nabla^g_X\beta),\qquad \nabla^g_X(\iota_Y\alpha)=\iota_{\nabla^g_XY}\alpha+\iota_Y(\nabla^g_X\alpha).$$

    Using the above formulas and the definition of the geometric product, see \eqref{eq:def_geometric_product}, it follows that $\nabla^g(\theta\diamond\alpha)=(\nabla^g\theta)\diamond\alpha+\theta\diamond(\nabla^g\alpha)$ for all $\theta\in\Omega^1_\C(M)$ and $\alpha\in\Omega_\C(M)$. The general statement follows by linear and associative extension.
\end{proof}

The map $\Psi_\gamma$ defined in \eqref{eq:Psi_gamma} extends to a unital isomorphism of bundles of algebras which we denote by the same symbol: $$\Psi_\gamma:=\gamma\circ\Psi\colon(\wedge T^*_\C M,\diamond)\to(\End(S),\circ).$$

This map allows us to identify bundles $(S,\gamma)$ of modules over $\mCl(M,g)$ with bundles of modules $(S,\Psi_\gamma)$ over the Kähler-Atiyah bundle $(\wedge T^*_\C M,\diamond)$. For ease of notation, we denote by \emph{dot} the Clifford multiplication of $(S,\Psi_\gamma)$, whose action on global sections is: $$\alpha\cdot\eta:=\Psi_\gamma(\alpha)\eta$$ for all $\alpha\in\Omega_\C(M)$ and $\eta\in\Gamma(S)$.

\begin{definition}
    Let $(S,\gamma)$ be a complex spinor bundle on $(M,g)$ and let $W$ be any vector bundle on $M$. The \emph{symbol} of a section $\cQ\in\Gamma(\End(S)\otimes W)$ is the section $\frq\in\Gamma(\wedge T^*_\C M\otimes W)$ defined through: $$\frq:=(\Psi_\gamma\otimes\Id_W)^{-1}(\cQ),$$ where $\Id_W$ is the identity endomorphism of $W$.
\end{definition}

\begin{remark}
In particular, the symbol of an endomorphism $\cQ\in\Gamma(\End(S))$ is an exterior form $\frq\in\Omega_\C(M)$, while the symbol of an $\End(S)$-valued one-form $\cA\in\Omega^1(M,\End(S))$ is an element $\fra\in\Omega^1(M,\wedge T^*_\C M)$, that is, a one-form taking values on the complex exterior algebra of $M$.
\end{remark}


\subsubsection{The odd-dimensional case}


Here we consider the \emph{(complexified) truncated Kähler-Atiyah bundle} $(\wedge^<T^*_\C M,\vee)$ of $(M,g)$, where: $$\wedge^<T^*_\C M:=\bigoplus_{k=0}^{\frac{d-1}{2}}\wedge^kT^*_\C M$$ and for $\alpha,\beta\in\wedge^<T^*_\C M$ the \emph{truncated geometric product} is given by: $$\alpha\vee\beta:=\cP_<(\alpha\diamond\beta+i^{q+\frac{d-1}{2}}\ell*\tau(\alpha\diamond\beta)),$$ where $\cP_<\colon\wedge T^*_\C M\to\wedge^<T^*_\C M$ is the natural projection and we have used \eqref{eq:projection_l}. We have the following result analogous to Proposition \ref{prop:geometric_product_derivations}.

\begin{prop}
    The canonical extension to $\wedge^<T^*_\C M$ of the Levi-Civita connection $\nabla^g$ of $(M,g)$, which we again denote by $\nabla^g$, acts by derivations of the truncated geometric product, that is: $$\nabla^g(\alpha\vee\beta)=(\nabla^g\alpha)\vee\beta+\alpha\vee(\nabla^g\beta)$$ for all $\alpha,\beta\in\Omega^<_\C(M):=\Gamma(\wedge^< T^*_\C M)$.
\end{prop}

Since the dimension $d$ of $M$ is odd, a bundle of irreducible complex Clifford modules on $(M,g)$ has two non-equivalent irreducible types distinguished by the value they take at the complex volume form $\nu_\C$. We denote by $(S,\gamma_\ell)$ the corresponding bundle of irreducible complex Clifford modules, where $\gamma_\ell\colon\mCl(M,g)\to\End(S)$ satisfies $\gamma_\ell(\nu_\C)=\ell\,\Id_S$ for $\ell\in\Z_2$.\medskip

The map $\Psi^<_\ell$ defined in \eqref{eq:truncated_iso} extends to a unital isomorphism of bundles of algebras which we denote by the same symbol: $$\Psi^<_\ell\colon(\wedge^<T^*_\C M,\vee)\to(\End(S),\circ).$$

This map allows us to identify bundles $(S,\gamma_\ell)$ of modules over $\mCl(M,g)$ with bundles of modules $(S,\Psi^<_\ell)$ over the truncated Kähler-Atiyah bundle $(\wedge^<T^*_\C M,\vee)$. For ease of notation, we denote by \emph{dot} the Clifford multiplication of $(S,\Psi^<_\ell)$, whose action on global sections is: $$\alpha\cdot\eta:=\Psi^<_\ell(\alpha)\eta$$ for all $\alpha\in\Omega^<_\C(M)$ and $\eta\in\Gamma(S)$.


\subsection{Hermitian spinor bundles}


In order to define the Hermitian square of a spinor $\eta\in\Gamma(S)$, we need to consider the appropriate notion of Hermitian pairing on the bundle $S$. In order to simultaneously deal with the even and odd cases, whenever necessary in the following $\Psi^{\bullet}_{\gamma}$ will denote $\Psi_{\gamma}$ if $d$ is even and $\Psi^{<}_{\ell}$ if $d$ is odd, and similarly for $\wedge^{\bullet}T^*_\C M$ and $\Omega^{\bullet}_{\mathbb{C}}(M)$. Then $(\wedge^{\bullet}T^*_\C M,\diamond)$ denotes the Kähler-Atiyah model if $d$ is even, and the truncated Kähler-Atiyah model $(\wedge^<T^*_\C M,\vee)$ if $d$ is odd.

\begin{definition}
Let $(S,\gamma)$ be a complex spinor bundle on $(M,g)$. A fiberwise Hermitian pairing $\scS$ on $S$ is called \emph{admissible} if $\scS_p\colon S_p\times S_p\to\C$ is an admissible Hermitian pairing on the complex Clifford module $(S_p,\gamma_p)$ for all $p\in M$. A \emph{Hermitian spinor bundle} on $(M,g)$ is a tuple $(S,\gamma,\scS)$, where $(S,\gamma)$ is a complex spinor bundle on $(M,g)$ and $\scS$ is an admissible Hermitian pairing on $S$.
\end{definition}

Since $M$ is connected, the adjoint type $s\in\Z_2$ of the admissible Hermitian pairings $\scS_p$, which are non-degenerate by definition, is constant on $M$. Then $s\in\Z_2$ is called the \emph{adjoint type} of $\scS$ or $(S,\gamma,\scS)$. Since $M$ is paracompact, the defining algebraic properties of an admissible Hermitian pairing can be formulated equivalently as follows using global sections when viewing $(S,\gamma)$ as a bundle $(S,\Psi^{\bullet}_{\gamma})$ of modules over the Kähler-Atiyah bundle $(\wedge^{\bullet}T^*_\C M,\diamond)$ of $(M,g)$: 
\begin{itemize}
\item $\scS(\eta_1,\eta_2)=\overline{\scS(\eta_2,\eta_1)}$ for all $\eta_1,\eta_2\in\Gamma(S)$.
\item $\scS(\Psi^{\bullet}_\gamma(\alpha)\eta_1,\eta_2)=\scS(\eta_1,\Psi^{\bullet}_\gamma((\pi^{\frac{1-s}{2}}\circ\tau)(\overline{\alpha}))\eta_2)$ for all $\alpha\in\Omega^{\bullet}_{\mathbb{C}}(M)$ and $\eta_1,\eta_2\in\Gamma(S)$.
\end{itemize}

To guarantee the existence of an admissible Hermitian pairing on an irreducible complex spinor bundle, we introduce the following terminology.

\begin{definition}
Let $(M,g)$ be a pseudo-Riemannian manifold. We say that $(M,g)$ is: \begin{itemize}
\item \emph{Strongly orientable} if it admits a \emph{strong orientation}, that is, if the orthonormal frame bundle of $(M,g)$ reduces to the identity component $\SO_o(p,q)$ of $\SO(p,q)$.
\item \emph{Strongly spin} if it admits a \emph{strong spin structure}, that is, a spin structure that admits a reduction to the identity component $\Spin_o(p,q)$ of the spin group $\Spin(p,q)$.
\item \emph{Strongly spin-c} if it admits a \emph{strong spin-c structure}, that is, a spin-c structure that admits a reduction to the identity component $\Spin^c_o(p,q)$ of the spin-c group $\Spin^c(p,q)$.
\end{itemize}
\end{definition}

\begin{remark}
When $pq=0$, the special orthogonal and spin groups are connected. In this case, orientation and strong orientation are equivalent, as are the properties of being spin and strongly spin. When $pq\neq0$, the groups $\SO(p,q)$ and $\Spin(p,q)$ have two connected components. In this case, $(M,g)$ is strongly orientable if and only if it is orientable and in addition the principal $\Z_2$-bundle associated to its bundle of oriented orthonormal frames $F(M,g)$ through the group morphism $\SO(p,q)\to\SO(p,q)/\SO_o(p,q)$ is trivial, while a spin structure $Q$ reduces to a strong spin structure if and only if the principal $\Z_2$-bundle associated to $Q$ through the group morphism $\Spin(p,q)\to\Spin(p,q)/\Spin_o(p,q)$ is trivial. When $(M,g)$ is strongly spin, the short exact sequence: $$1\to\Z_2\to\Spin_o(p,q)\to\SO_o(p,q)\to1$$ induces a sequence in \u{C}ech cohomology which implies that the set of strong spin structures on $(M,g)$ is a $H^1(M,\Z_2)$-torsor. A particularly simple case arises when $H^1(M,\Z_2)=0$, e.g.\ when $M$ is simply connected. In this situation, $M$ is strongly orientable and any spin structure on $(M,g)$ reduces to a strong spin structure\footnote{This is because the set of isomorphism classes of principal $\Z_2$-bundles over $M$ is $[M,B\Z_2]=[M,K(\Z_2,1)]\cong H^1(M,\Z_2)$.}. Up to isomorphism, in this special case, there exists at most one spin structure, one strong spin structure, and one complex spinor bundle on $(M,g)$.
\end{remark}

\begin{remark}\label{remark:H2(M,Z)-torsor}
When $pq=0$, the spin-c group is connected, so in this case, the properties of being spin-c and strongly spin-c are equivalent. When $pq\neq0$, the group $\Spin^c(p,q)$ has two connected components. In this case, $(M,g)$ is strongly spin-c if and only if it is spin-c and strongly orientable. When $(M,g)$ is strongly spin-c, the short exact sequence: $$1\to\U(1)\to\Spin^c_o(p,q)\to\SO_o(p,q)\to1$$ induces a sequence in \u{C}ech cohomology which implies that the set of strong spin-c structures on $(M,g)$ is a $H^2(M,\Z)$-torsor. A particularly simple case arises when $H^1(M,\Z_2)=0$ and $H^2(M,\Z)=0$. In this situation, $M$ is strongly orientable and any spin-c structure on $(M,g)$ reduces to a strong spin-c structure. Up to isomorphism, in this special case, there exists at most one spin-c structure, one strong spin-c structure, and one complex spinor bundle on $(M,g)$.
\end{remark}

An irreducible complex spinor bundle $(S,\gamma)$ associated to a strong spin-c structure admits an admissible Hermitian pairing, since the latter is by construction preserved by the identity component $\Spin_o^c(p,q)$ of the spin-c group. On the other hand, every spin-c structure on a strongly orientable pseudo-Riemannian manifold $(M,g)$ is strongly spin-c, namely, reduces to $\Spin_o^c(p,q)$. Hence, strong orientability of $(M,g)$ suffices to guarantee that every irreducible complex spinor bundle on $(M,g)$ admits an admissible Hermitian pairing. 

\begin{prop}
\label{prop:Hermitianexistence}
Let $(M,g)$ be a strongly oriented pseudo-Riemannian manifold and let $(S,\gamma)$ be an irreducible complex spinor bundle over $(M,g)$. Then, $(S,\gamma)$ admits an admissible Hermitian pairing.
\end{prop}

Since $S$ is associated to a strong spin-c structure, the spinorial lift of the Levi-Civita connection $\nabla^g$ of $(M,g)$ to $S$ and the choice of a connection $A\in\Omega^1(P,i\R)$ on $P$ define a covariant derivative $\nabla^{g,A}$ on $S$. The connection $\nabla^{g,A}$ acts on $\Gamma(S)$ by module derivations: \begin{equation}\label{eq:nabla_on_alpha.eta}
    \nabla^{g,A}_X(\alpha\cdot\eta)=(\nabla^g_X\alpha)\cdot\eta+\alpha\cdot(\nabla^{g,A}_X\eta)
\end{equation} for all $X\in\Gamma(TM)$, $\alpha\in\Omega^{\bullet}_{\mathbb{C}}(M)$, and $\eta\in\Gamma(S)$; and is compatible with $\scS$: $$X(\scS(\eta_1,\eta_2))=\scS(\nabla^{g,A}_X\eta_1,\eta_2)+\scS(\eta_1,\nabla^{g,A}_X\eta_2)$$ for all $X\in\Gamma(TM)$ and $\eta_1,\eta_2\in\Gamma(S)$ (see \cite[Section 3.1]{Friedrich2000}).\medskip

The connection $\nabla^{g,A}$ induces a linear connection, which we denote by the same symbol for ease of notation, on the bundle of endomorphisms $\End(S)$. Given $X\in\Gamma(TM)$, by definition we have: $$(\nabla^{g,A}_XT)\eta:=\nabla^{g,A}_X(T\eta)-T(\nabla^{g,A}_X\eta)$$ for all $T\in\Gamma(\End(S))$ and $\eta\in\Gamma(S)$. Moreover, the connection $\nabla^{g,A}$ acts on $\Gamma(\End(S))$ by derivations: $$\nabla^{g,A}_X(T_1\circ T_2)=(\nabla^{g,A}_XT_1)\circ T_2+T_1\circ(\nabla^{g,A}_XT_2)$$ for all $X\in\Gamma(TM)$ and $T_1,T_2\in\Gamma(\End(S))$.

\begin{prop}
\label{prop:nabla_compatible_Psi_gamma}
Let $(S,\gamma,\scS)$ be a Hermitian spinor bundle. Then the map $\Psi^{\bullet}_\gamma$ induces a unital isomorphism of complex algebras $(\Omega^{\bullet}_\C(M),\diamond)\cong\Gamma(\End(S))$ which is compatible with $\nabla^g$ and $\nabla^{g,A}$. In other words, the following equation holds: $$\nabla^{g,A}_X(\Psi^{\bullet}_\gamma(\alpha))=\Psi^{\bullet}_\gamma(\nabla^g_X\alpha)$$ for all $X\in\Gamma(TM)$ and $\alpha\in\Omega^{\bullet}_\C(M)$.
\end{prop}

\begin{proof}
    Recall that $\alpha\cdot\eta:=\Psi^{\bullet}_\gamma(\alpha)\eta$. Using \eqref{eq:nabla_on_alpha.eta} we obtain: \begin{align*}
        \big(\nabla^{g,A}_X(\Psi^{\bullet}_\gamma(\alpha))\big)\eta&=\nabla^{g,A}_X(\Psi^{\bullet}_\gamma(\alpha)\eta)-\Psi^{\bullet}_\gamma(\alpha)(\nabla^{g,A}_X\eta)\\
        &=\Psi^{\bullet}_\gamma(\nabla^g_X\alpha)\eta+\Psi^{\bullet}_\gamma(\alpha)(\nabla^{g,A}_X\eta)-\Psi^{\bullet}_\gamma(\alpha)(\nabla^{g,A}_X\eta)\\
        &=\Psi^{\bullet}_\gamma(\nabla^g_X\alpha)\eta
    \end{align*} for all $\eta\in\Gamma(S)$.
\end{proof} 


\subsection{Complex-bilinear paired spinor bundles}


Whereas admissible Hermitian pairings on $(\Sigma,\gamma_0)$ are $\Spin^c_o(p,q)$-invariant, admissible complex-bilinear pairings $\scB$ are not. Instead, as a consequence of $\scB$ being $\Spin_o(p,q)$-invariant, it satisfies the following equation: 
\begin{equation*}
\scB(\gamma_0([u,z])\eta_1,\gamma_0([u,z])\eta_2)=\scB(z\gamma_0(u)\eta_1,z\gamma_0(u)\eta_2)=z^2\scB(\gamma_0(u)\eta_1,\gamma_0(u)\eta_2)=z^2\scB(\eta_1,\eta_2)    
\end{equation*} 
for all $u\in\Spin_o(p,q)$, $z\in\U(1)$, and $\eta_1,\eta_2\in\Sigma$. Therefore, a complex-bilinear pairing $\scB$ on $(S,\gamma)$ takes values on the complex line bundle $\cL$ naturally associated to $S$. This motivates the following definition.

\begin{definition}
Let $(S,\gamma)$ be a complex spinor bundle on $(M,g)$ with associated characteristic line bundle $\cL$. An \emph{admissible complex-bilinear pairing} $\scB$ on $(S,\gamma)$ is a complex-bilinear pairing $\scB$ taking values on $\cL$ and such that $\scB_m\colon S_m\times S_m\to\cL_m\cong\C$ is an admissible complex-bilinear pairing on $(S_m,\gamma_m)$ for all $m\in M$. A \emph{complex-bilinear paired spinor bundle} on $(M,g)$ is a tuple $(S,\gamma,\scB)$, where $(S,\gamma)$ is a complex spinor bundle on $(M,g)$ and $\scB$ is an admissible complex-bilinear pairing on $S$.
\end{definition}

Since $M$ is connected, the symmetry and adjoint type $\sigma,s\in\Z_2$ of the admissible complex-bilinear pairings $\scB_m$, which are non-degenerate by definition, are constant on $M$. Then $\sigma\in\Z_2$ and $s\in\Z_2$ are called the \emph{symmetry type} and \emph{adjoint type} of $\scB$ or $(S,\gamma,\scB)$, respectively. Since $M$ is paracompact, the defining algebraic properties of an admissible complex-bilinear pairing can be formulated equivalently as follows using global sections when viewing $(S,\gamma)$ as a bundle $(S,\Psi^{\bullet}_\gamma)$ of modules over the Kähler-Atiyah bundle $(\wedge^{\bullet}T^*_\C M,\diamond)$ of $(M,g)$: 
\begin{itemize}
\item $\scB(\eta_1,\eta_2)=\sigma\scB(\eta_2,\eta_1)$ for all $\eta_1,\eta_2\in\Gamma(S)$.
\item $\scB(\Psi^{\bullet}_\gamma(\alpha)\eta_1,\eta_2)=\scB(\eta_1,\Psi^{\bullet}_\gamma((\pi^{\frac{1-s}{2}}\circ\tau)(\alpha))\eta_2)$ for all $\alpha\in\Omega^{\bullet}_\C(M)$ and $\eta_1,\eta_2\in\Gamma(S)$.
\end{itemize} 

Similarly to Proposition \ref{prop:Hermitianexistence}, we obtain the following result.

\begin{prop}
\label{prop:ComplexBilinealexistence}
Let $(M,g)$ be a strongly oriented pseudo-Riemannian manifold and let $(S,\gamma)$ be an irreducible complex spinor bundle over $(M,g)$. Then, $(S,\gamma)$ admits an admissible complex-bilinear pairing.
\end{prop}

Since $S = Q\times_{\varrho} \Sigma$ is associated to a spin-c structure $Q$, the combination of the Levi-Civita connection $\nabla^g$ of $(M,g)$ together with a choice of a connection $A\in\Omega^1(P,i\R)$ on the characteristic $\U(1)$-bundle $P$ defines a covariant derivative $\nabla^{g,A}$ on $S$. The connection $\nabla^{g,A}$ acts on $\Gamma(S)$ by module derivations: 
\begin{equation*}
\nabla^{g,A}_X(\alpha\cdot\eta)=(\nabla^g_X\alpha)\cdot\eta+\alpha\cdot(\nabla^{g,A}_X\eta)    
\end{equation*}
for all $X\in\Gamma(TM)$, $\alpha\in\Omega^{\bullet}_\C(M)$, and $\eta\in\Gamma(S)$. The connection $A\in\Omega^1(P,i\R)$ induces a Hermitian connection $\nabla^{A}$ on $\cL$. Then the compatibility with $\scB$ reads as follows: $$\nabla^{A}_X(\scB(\eta_1,\eta_2))=\scB(\nabla^{g,A}_X\eta_1,\eta_2)+\scB(\eta_1,\nabla^{g,A}_X\eta_2)$$ for all $X\in\Gamma(TM)$ and $\eta_1,\eta_2\in\Gamma(S)$. Therefore, $\scB\in\Gamma(S^*\otimes S^*\otimes\cL)$ is parallel with respect to the tensor product connection $\nabla^{g,A}\otimes\nabla^{A}$.\medskip

To obtain a result analogous to Proposition \ref{prop:nabla_compatible_Psi_gamma} for complex-bilinear paired spinor bundles, we tensor $\End(S)$ and $\wedge^{\bullet}T^*_\C M$ with the complex line bundle $\cL$ and extend the corresponding products. That is, if $T_1,T_2\in\End(S)$ and $\lambda_1,\lambda_2\in\cL$, then: $$(T_1\otimes\lambda_1)\circ(T_2\otimes\lambda_2):=(T_1\circ T_2)\otimes(\lambda_1\otimes\lambda_2)\in\End(S)\otimes\cL\otimes\cL.$$

Similarly, if $\alpha_1,\alpha_2\in\wedge^{\bullet}T^*_\C M$ and $\lambda_1,\lambda_2\in\cL$, then: $$(\alpha_1\otimes\lambda_1)\diamond(\alpha_2\otimes\lambda_2):=(\alpha_1\diamond\alpha_2)\otimes(\lambda_1\otimes\lambda_2)\in\wedge^{\bullet}T^*_\C M\otimes\cL\otimes\cL.$$

We define an action of $\cL$-valued complex exterior forms on spinors as follows. Let $\alpha_0\in\Omega^{\bullet}_\C(M)$, $\lambda\in\Gamma(\cL)$, and $\eta\in\Gamma(S)$. Then: $$(\alpha_0\otimes\lambda)\cdot\eta:=(\Psi^{\bullet}_\gamma\otimes\Id_{\cL})(\alpha_0\otimes\lambda)\eta=(\Psi^{\bullet}_\gamma(\alpha_0)\eta)\otimes\lambda=(\alpha_0\cdot\eta)\otimes\lambda\in\Gamma(S\otimes\cL),$$ where $\Psi^{\bullet}_\gamma\otimes\Id_{\cL}\colon\wedge^{\bullet}T^*_\C M\otimes\cL\to\End(S)\otimes\cL$. The connection $\nabla^{g,A}\otimes\nabla^{A}$ acts on $\Gamma(S\otimes\cL)$ by: $$(\nabla^{g,A}_X\otimes\nabla^{A}_X)((\alpha_0\otimes\lambda)\cdot\eta)=((\nabla^g_X\alpha_0)\cdot\eta)\otimes\lambda+(\alpha_0\cdot(\nabla^{g,A}_X\eta))\otimes\lambda+(\alpha_0\cdot\eta)\otimes\nabla^{A}_X\lambda.$$

\begin{prop}\label{prop:L-valued_compatibility}
Let $(S,\gamma,\scB)$ be a complex-bilinear paired spinor bundle. Then the map $\Psi^{\bullet}_\gamma\otimes\Id_{\cL}$ induces an isomorphism of $C^{\infty}(M,\C)$-modules $\Omega^{\bullet}_\C(M,\cL) \cong \Gamma(\End(S)\otimes\cL)$ which intertwines the extended products into $\cL^{\otimes 2}$ and is compatible with $\nabla^{g,A}$ and $\nabla^{A}$. In other words, the following equation holds: 
\begin{equation*}
(\nabla^{g,A}_X\otimes\nabla^{A}_X)((\Psi^{\bullet}_\gamma\otimes\Id_{\cL})(\alpha))=(\Psi^{\bullet}_\gamma\otimes\Id_{\cL})(\nabla^{g,A}_X\alpha)
\end{equation*}
for all $X\in\Gamma(TM)$ and $\alpha\in\Omega^{\bullet}_\C(M,\cL)$.
\end{prop}

\begin{proof}
Let $\alpha_0\in\Omega^{\bullet}_\C(M)$ and $\lambda\in\Gamma(\cL)$. Set $E := \Psi^{\bullet}_\gamma(\alpha_0)\otimes\lambda\in\Gamma(\End(S)\otimes\cL)$ and $\nabla:=\nabla^{g,A}\otimes\nabla^{A}$. ~Then: \begin{align*}
        (\nabla_XE)\eta&=\nabla_X(E(\eta))-E(\nabla^{g,A}_X\eta)\\
        &=\nabla_X((\alpha_0\cdot\eta)\otimes\lambda)-(\alpha_0\cdot(\nabla^{g,A}_X\eta))\otimes\lambda\\
        &=(\nabla^{g,A}_X(\alpha_0\cdot\eta))\otimes\lambda+(\alpha_0\cdot\eta)\otimes\nabla^{A}_X\lambda-(\alpha_0\cdot(\nabla^{g,A}_X\eta))\otimes\lambda\\
        &=((\nabla^g_X\alpha_0)\cdot\eta)\otimes\lambda+(\alpha_0\cdot\eta)\otimes\nabla^{A}_X\lambda\\
        &=(\Psi^{\bullet}_\gamma\otimes\Id_{\cL})(\nabla^g_X\alpha_0\otimes\lambda)\eta+(\Psi^{\bullet}_\gamma\otimes\Id_{\cL})(\alpha_0\otimes\nabla^{A}_X\lambda)\eta\\
        &=(\Psi^{\bullet}_\gamma\otimes\Id_{\cL})((\nabla^g_X\otimes\nabla^{A}_X)(\alpha_0\otimes\lambda))\eta
    \end{align*} for all $\eta\in\Gamma(S)$.
\end{proof}


\subsection{Constrained parallel spinors}


We proceed to introduce the main character of this article.

\begin{definition}
Let $(S,\gamma)$ be a complex spinor bundle on $(M,g)$ equipped with a connection $\cD$ and let $\cQ\in\Gamma(\End(S)\otimes W)$ with $W$ a vector bundle on $M$. A section $\eta\in\Gamma(S)$ is a \emph{constrained parallel spinor} with respect to $(\cD,\cQ)$ if: $$\cD\eta=0,\qquad\cQ(\eta)=0.$$
\end{definition}

Suppose that $(S,\gamma)$ is an irreducible complex spinor bundle. Since it is associated to a spin-c structure, we can write $\cD=\nabla^{g,A}-\cA$ with $\cA\in\Omega^1(M,\End(S))$. In this case, the equation satisfied by a constrained parallel spinor can be written as: $$\nabla^{g,A}\eta=\cA(\eta),\qquad\cQ(\eta)=0,$$ and their solutions are called constrained parallel spinors \emph{relative} to $(\cA,\cQ)$. Using connectedness of $M$ and the parallel transport of $\cD$, the equation $\cD\eta=0$ implies that the space of constrained parallel spinors relative to $(\cA,\cQ)$ is finite-dimensional and that a constrained parallel spinor which is not zero at some point of $M$ is automatically nowhere vanishing on $M$.


\section{Parallel Hermitian spinorial forms}
\label{sec:Hermitian_parallel_forms}


To study constrained parallel complex spinors, we will extend the algebraic theory developed in \cite{Algebraic_Complex_Square_2025}, whose main results we have recalled in Section \ref{sec:spinorialforms}, to bundles of irreducible complex Clifford modules equipped with an arbitrary connection $\cD$. In this section, we focus on the Hermitian case, while we devote Section \ref{sec:Complex_bilinear_parallel_forms} to studying the complex-bilinear case. As in the algebraic case, we consider the even-dimensional and odd-dimensional cases separately.


\subsection{Hermitian spinorial forms in even dimensions}


Let $(S,\gamma,\scS)$ be a Hermitian spinor bundle on $(M,g)$. The admissible Hermitian pairing $\scS$ allows us to construct pointwise extensions to $M$ of the Hermitian square maps $\hcE_\kappa\colon\Sigma\to\End(\Sigma)$ and the Hermitian square spinor maps $\hcE^\kappa_\gamma\colon\Sigma\to\wedge V^*_\C$ of Section \ref{sec:spinorialforms}. We denote these by the same symbol: $$\hcE_\kappa\colon S\to\End(S),\qquad \hcE^\kappa_\gamma\colon S\to\wedge T^*_\C M.$$

We have the following diagram of vector bundles: $$\begin{tikzcd}
S \arrow[rr, "\widehat{\mathcal{E}}_\kappa"] \arrow[dd, "\widehat{\mathcal{E}}^\kappa_\gamma"'] &  & \mathrm{End}(S)                                   \\
                                                                                                &  &                                                   \\
{(\wedge T^*_{\mathbb{C}}M,\diamond)} \arrow[rr, "\Psi"] \arrow[rruu, "\Psi_\gamma"]            &  & {\mathbb{C}\mathrm{l}(M,g)} \arrow[uu, "\gamma"']
\end{tikzcd}$$

which extends to maps of sections the algebraic maps, denoted with the same symbol, that we introduced in Section \ref{sec:spinorialforms}. Note that although $\hcE^\kappa_\gamma$ preserves fibers, it is not a morphism of vector bundles since it is fiberwise quadratic.

\begin{definition}
The \emph{Hermitian spinor square maps} of the Hermitian spinor bundle $(S,\gamma,\scS)$ are the maps $\hcE^\kappa_\gamma\colon\Gamma(S)\to\Omega_\C(M)$, $\kappa\in\U(1)$, induced by $\hcE^\kappa_\gamma$ on sections. The \emph{Hermitian square} of an irreducible complex spinor $\eta\in\Gamma(S)$ is the complex exterior form $\hcE^\kappa_\gamma(\eta)\in\Omega_\C(M)$ for $\kappa\in\U(1)$. Elements in the image of $\hcE^\kappa_\gamma$ are generically called \emph{Hermitian spinorial forms}.
\end{definition}

By the results of Section \ref{sec:spinorialforms}, we have that $\hcE^\kappa_\gamma$ are real quadratic maps and satisfy: $$\mathrm{supp}(\hcE^\kappa_\gamma(\eta))=\mathrm{supp}(\eta)$$ for all $\eta\in\Gamma(S)$. Define: $$\Im(\hcE_\gamma):=\bigcup_{\kappa\in\U(1)}\Im(\hcE^\kappa_\gamma)\subset\Omega_\C(M)$$ as the disjoint union of the images of the Hermitian square maps for all $\kappa\in\U(1)$.

\begin{prop}
\label{prop:spin_cohomology_class}
Let $(S,\gamma,\scS)$ be an irreducible Hermitian spinor bundle. Then every nowhere vanishing form in $\Im(\hcE_\gamma)$ determines a cohomology class:
\begin{equation*}
\mathfrak{s}_{\gamma}(\alpha)\in H^2(M,\Z)    
\end{equation*}
encoding the obstruction to the existence of a globally defined irreducible complex spinor $\eta\in\Gamma(S)$, necessarily nowhere vanishing, such that $\alpha=\hcE^\kappa_\gamma(\eta)$ for some $\kappa\in\U(1)$. In particular, such $\eta\in\Gamma(S)$ exists if and only if $\mathfrak{s}_{\gamma}(\alpha)=0$.
\end{prop}

\begin{proof}
For every $\kappa\in\U(1)$, the map $\hcE^\kappa_\gamma\colon S\setminus\{0\} \to \wedge T^*_\C M\setminus\{0\}$ induces a base-point preserving smooth map: 
\begin{equation*}
\P\hcE^{\gamma}_{\kappa} \colon\P(S)\to\P(\wedge T^*_\C M),\qquad [\eta] \mapsto[\alpha_{\eta}]    
\end{equation*} 

\noindent
between the complex projectivizations $\P(S)$ and $\P(\wedge T^*_\C M)$, of $S$ and $\wedge T^*_\C M$, respectively. This map is bijective onto its image. Since $\alpha\in\Im(\hcE^\kappa_\gamma)$, it follows that $[\alpha] \in \Im(\P\hcE^\kappa_\gamma)$ and therefore: 
\begin{equation*}
(\P\hcE^\kappa_\gamma)^{-1}([\alpha]) \colon M \to  \P(S)
\end{equation*}

\noindent
defines a smooth section of $\P(S)$, which is equivalent to a complex line subbundle $L_{\alpha}\subset S$ whose first Chern class determines the cohomology class $\mathfrak{s}_{\gamma}(\alpha):=c_1(L_\alpha)\in H^2(M,\Z)$ that we associate to $\alpha$. A section $\eta$ of $S$ such that $\hcE^\kappa_\gamma(\eta)=\alpha$ would be a nowhere vanishing section of $L_{\alpha}$, which therefore exists if and only if $\mathfrak{s}_{\gamma}(\alpha) = 0$.
\end{proof}

We will refer to the cohomology class $\mathfrak{s}_{\gamma}(\alpha)\in H^2(M,\Z)$ given by Proposition \ref{prop:spin_cohomology_class} as the \emph{spinor class} of the nowhere vanishing exterior form $\alpha\in\Im(\hcE_\gamma)$. Note that $\mathfrak{s}_{\gamma}(\alpha)$ depends only on $(S,\gamma,\scS)$ and $\alpha$, since all admissible Hermitian pairings on $(S,\gamma)$ are related to each other by automorphisms of $S$, see \cite[Prop.\ 3.10]{Algebraic_Complex_Square_2025}. In particular, $\mathfrak{s}_{\gamma}(\alpha)$ is not a characteristic class of $S$ since it depends on $\alpha$.\medskip

The category of bundles of irreducible complex spinors on $(M,g)$ is well-known to be a torsor over the category of complex line bundles on $M$ with the action given by the tensor product \cite{FriedrichTrautman,LS18}. The action of such a complex line bundle $L$ on a given $(S,\gamma,\scS)$ induces naturally an admissible Hermitian pairing $\scS_L$ on $(S_L:=S\otimes L, \gamma_L) $ by setting:
\begin{equation}\label{eq:Hermitial_HL}
    \scS_L(\eta_1\otimes\lambda_1,\eta_2\otimes\lambda_2):=h_L(\lambda_1,\lambda_2)\scS(\eta_1,\eta_2)   
\end{equation}
for every Hermitian inner product $h_L$ on $L$. With this definition, $(S_L,\gamma_L,\scS_L)$ is a Hermitian spinor bundle. Using this torsor structure on the category of bundles of irreducible complex spinors, we can conveniently make the spinor class $\mathfrak{s}_{\gamma}(\alpha)$ vanish.

\begin{prop}
\label{prop:spin_spinor_class=0}
Let $(S_o,\gamma_o,\scS_o)$ be an irreducible Hermitian spinor bundle. For every nowhere vanishing form $\alpha\in \Im(\hcE_{\gamma_o})$ there exists an irreducible Hermitian spinor bundle $(S,\gamma,\scS)$ on $(M,g)$, unique modulo isomorphism, such that $\mathfrak{s}_{\gamma}(\alpha)=0$.
\end{prop}

\begin{proof}
The irreducible Hermitian spinor bundle $(S_o,\gamma_o,\scS_o)$ yields $\mathfrak{s}_{\gamma_o}(\alpha)\in H^2(M,\mathbb{Z})$, generally non-vanishing. Consider a complex line bundle $L\to M$ whose first Chern class is $c_1(L)=\mathfrak{s}_{\gamma_o}(\alpha)$. We compute now the spinor class $\mathfrak{s}_{\gamma}(\alpha)$ for the irreducible Hermitian spinor bundle $(S,\gamma,\scS)$, where $S=S_o\otimes L^{\ast}$, $L^*$ is the dual line bundle of $L$, $\gamma$ is the trivial $L^*$-extension of $\gamma_o$, and $\scS$ is defined as in \eqref{eq:Hermitial_HL} for any choice of Hermitian inner product on $L^*$. First, we note that $\P(S)=\P(S_o\otimes L^{\ast})=\P(S_o)$, and furthermore $\hcE^\kappa_\gamma$ and $\hcE^\kappa_{\gamma_o}$ induce the same smooth map at the projectivized level, that is: $$\P\hcE^\kappa_\gamma=\P\hcE^\kappa_{\gamma_o}\colon\P(S)=\P(S_o)\to\P(\wedge T^*_\C M).$$

Hence: $$(\P\hcE^\kappa_\gamma)^{-1}([\alpha])=(\P\hcE^\kappa_{\gamma_o})^{-1}([\alpha])\colon M\to\P(S)=\P(S_o).$$
    
Hence, if this section defines the line bundle $L^{\prime}\subset S_o$, necessarily isomorphic to $L$ since $c_1(L')=\mathfrak{s}_{\gamma_o}(\alpha)$, then it defines the line bundle $L^{\prime}\otimes L^{\ast}\subset S_o\otimes L^{\ast}=S$. However, we have $L^{\prime}\otimes L^{\ast}\cong L\otimes L^{\ast}\cong\End(L)$, which is therefore trivial. Hence, Proposition \ref{prop:spin_cohomology_class} implies that there exists a nowhere vanishing section $\eta\in\Gamma(S)$ such that $\hcE^\kappa_\gamma(\eta)=\alpha$, which implies $\mathfrak{s}_\gamma(\alpha)=0$.
\end{proof} 

Let $(S,\gamma,\scS)$ be a Hermitian spinor bundle on $(M,g)$. Recall that if $\cQ\in\Gamma(\End(S)\otimes W)$, then its symbol is defined as $\frq=(\Psi_\gamma\otimes\Id_W)^{-1}(\cQ)\in\Gamma(\wedge T^*_\C M\otimes W)$. Lemma \ref{lemma:constrainedspinoreven} and Remark \ref{remark:equivalent_constrained_eq} imply:

\begin{lemma}
\label{lemma:constrained_spinor_as_section}
A spinor $\eta\in\Gamma(S)$ satisfies $\cQ(\eta)=0$ if and only if one of the following equivalent equations holds: $$\frq\diamond\alpha=0,\qquad\alpha\diamond(\pi^{\frac{1-s}{2}}\circ\tau)(\overline{\frq})=0,$$ where $\alpha=\hcE^\kappa_\gamma(\eta)$ for some $\kappa\in\U(1)$.
\end{lemma}

Let $(S,\gamma,\scS)$ be an \emph{irreducible} Hermitian spinor bundle on $(M,g)$ equipped with a connection $\cD$. Set $\cA:=\nabla^{g,A}-\cD\in\Omega^1(M,\End(S))$ and let: 
\begin{equation*}
\fra:=(\Psi_\gamma\otimes\Id_{T^*M})^{-1}(\cA)\in\Omega^1(M,\wedge T^*_\C M)    
\end{equation*}
be the symbol of $\cA$. We have:

\begin{lemma}
\label{lemma:nabla_square_alpha}
A nowhere vanishing spinor $\eta\in\Gamma(S)$ satisfies $\cD\eta=0$ only if: $$\nabla^g\alpha=\fra\diamond\alpha+\alpha\diamond(\pi^{\frac{1-s}{2}}\circ\tau)(\overline{\fra}),$$ where $\alpha=\hcE^\kappa_\gamma(\eta)$ for some $\kappa\in\U(1)$.
\end{lemma}

\begin{proof}
    Assume that $\eta$ satisfies $\nabla^{g,A}\eta=\cA(\eta)$. Fix $\kappa\in\U(1)$ and set $E_\eta:=\hcE_\kappa(\eta)\in\Gamma(\End(S))$. Then: \begin{align*}
        (\nabla^{g,A}E_\eta)\chi&=\nabla^{g,A}(E_\eta(\chi))-E_\eta(\nabla^{g,A}\chi)\\
        &=\nabla^{g,A}(\kappa\scS(\chi,\eta)\eta)-\kappa\scS(\nabla^{g,A}\chi,\eta)\eta\\
        &=\kappa\scS(\chi,\nabla^{g,A}\eta)\eta+\kappa\scS(\chi,\eta)\nabla^{g,A}\eta\\
        &=\kappa\scS(\chi,\cA(\eta))\eta+\kappa\scS(\chi,\eta)\cA(\eta)\\
        &=E_\eta(\cA^\dagger(\chi))+\cA(E_\eta(\chi)),
    \end{align*} for all $\chi\in\Gamma(S)$, where $\cA^\dagger$ denotes the adjoint of $\cA$ taken with respect to $\scS$. The previous equation implies: $$\nabla^{g,A}E_\eta=\cA\circ E_\eta+E_\eta\circ\cA^\dagger.$$
    
    Applying $\Psi_\gamma^{-1}$ to this equation, and using Proposition \ref{prop:nabla_compatible_Psi_gamma} and $\Psi_\gamma^{-1}(\cA^\dagger)=\fra^\dagger=(\pi^{\frac{1-s}{2}}\circ\tau)(\overline{\fra})$, see Remark \ref{remark:equivalent_constrained_eq}, gives the result from the statement.
\end{proof}

We arrive at the following set of necessary conditions for a complex exterior form $\alpha\in\Omega_\C(M)$ to be the Hermitian square of a constrained parallel spinor relative to $(\cA,\cQ)$.

\begin{thm}
\label{thm:even_constrained_parallel_spinors}
Let $(S,\gamma,\scS)$ be an irreducible Hermitian spinor bundle on $(M,g)$ of adjoint type $s\in\Z_2$. Let $\cA \in \Omega^1(M,\End(S))$ and $\cQ\in\Gamma(\End(S)\otimes W)$ for a vector bundle $W$ on $M$. Then there exists a nowhere vanishing constrained parallel spinor $\eta\in\Gamma(S)$ relative to $(\cA,\cQ)$ only if: 
\begin{itemize}
\item There exists a nowhere vanishing complex exterior form $\alpha\in\Omega_\C(M)$ with vanishing spinor class $\mathfrak{s}_{\gamma}(\alpha)$.
\item The complex exterior form $\alpha\in\Omega_\C(M)$ satisfies the following algebraic and differential equations: \begin{gather*}
\alpha\diamond\alpha = 2^{\frac{d}{2}}\alpha^{(0)}\alpha,\qquad \alpha \diamond \beta \diamond \alpha=2^{\frac{d}{2}}(\alpha\diamond\beta)^{(0)}\alpha,\qquad(\pi^{\frac{1-s}{2}}\circ\tau)(\bar\kappa\alpha) = \kappa \bar{\alpha},\\
\nabla^g\alpha=\fra\diamond\alpha+\alpha\diamond(\pi^{\frac{1-s}{2}}\circ\tau)(\overline{\fra}),\qquad\frq\diamond\alpha=0
\end{gather*} 
for a complex exterior form $\beta\in\Omega_\C(M)$ satisfying $(\alpha\diamond\beta)^{(0)}\neq0$ or, equivalently, satisfies the equations: 
\begin{gather*}
\alpha \diamond \beta \diamond \alpha=2^{\frac{d}{2}}(\alpha\diamond\beta)^{(0)}\alpha,\qquad(\pi^{\frac{1-s}{2}}\circ\tau)(\bar\kappa\alpha) = \kappa \bar{\alpha},\\
\nabla^g\alpha=\fra\diamond\alpha+\alpha\diamond(\pi^{\frac{1-s}{2}}\circ\tau)(\overline{\fra}),\qquad\frq\diamond\alpha=0
\end{gather*} 
for every complex exterior form $\beta\in\Omega_\C(M)$.
\end{itemize}
    
If in addition the spinor $\eta\in\Gamma(S)$ is chiral of chirality $\mu\in\Z_2$, then we have to add the equation: $$i^{q+\frac{d}{2}}*(\pi\circ\tau)(\alpha)=\mu\alpha.$$

The complex exterior form $\alpha\in\Omega_\C(M)$ as above is determined by $\eta\in\Gamma(S)$ through the relation: $$\alpha=\hcE^\kappa_\gamma(\eta)$$ for some $\kappa\in\U(1)$. Moreover, $\alpha\in\Omega_\C(M)$ satisfying the algebraic equations of the theorem determines a nowhere vanishing spinor $\eta\in\Gamma(S)$, unique up to a $\U(1)$-valued function, satisfying the relation $\alpha=\hcE^\kappa_\gamma(\eta)$.
\end{thm}

\begin{proof}
Condition $\mathfrak{s}_{\gamma}(\alpha)=0$ is required by Proposition \ref{prop:spin_cohomology_class}. The algebraic conditions of the theorem follow from the pointwise extension of Theorem \ref{thm:even_Hermitian_forms} together with Lemma \ref{lemma:constrained_spinor_as_section}. The necessary differential condition follows from Lemma \ref{lemma:nabla_square_alpha}. Note that $\eta\in\Gamma(S)$ vanishes at a point $p\in M$ if and only if its Hermitian square $\alpha$ satisfies $\alpha_p=0$.
\end{proof}

\begin{remark}
Suppose that $(S,\gamma,\scS)$ is an irreducible Hermitian spinor bundle and $\alpha\in\Omega_\C(M)$ is a nowhere vanishing complex exterior form for which $\mathfrak{s}_{\gamma}(\alpha)\neq 0$, but satisfies the remaining algebraic and differential conditions of Theorem \ref{thm:even_constrained_parallel_spinors}. Then Proposition \ref{prop:spin_spinor_class=0} implies that there exists another irreducible Hermitian spinor bundle that admits a nowhere vanishing section whose Hermitian square is precisely $\alpha$. 
\end{remark}

In practical applications, such as in the study of supersymmetric configurations of supergravity theories, the specific form of $\cA$ and $\cQ$ as vector-valued endomorphisms of $S$ is usually not fundamental. Instead, only their \emph{symbols} are relevant. This is because we are interested in the consequences of such a parallel spinor existing with respect to \emph{some} irreducible Hermitian spinor bundle, and for this, we only need to be concerned with studying constrained parallel spinors relative to pairs $(\cA,\cQ)$ that have a given fixed symbol $(\fra,\frq)$. In this regard, the following form of the previous theorem is the most convenient for our purposes.

\begin{cor}
\label{cor:even_constrained_parallel_spinors}
Let $(M,g)$ be a strongly spin-c pseudo-Riemannian manifold and let $\fra\in\Omega^1(M,\wedge T^*_\C M)$ and $\frq\in\Gamma(\wedge T^*_\C M\otimes W)$ be given. Then $(M,g)$ admits a nowhere vanishing constrained parallel spinor relative to a pair $(\cA,\cQ)$ whose symbol is $(\fra,\frq)$ only if there exists a nowhere vanishing complex exterior form $\alpha\in\Omega_\C(M)$ satisfying the following algebraic and differential equations: 
\begin{gather*}
\alpha\diamond\alpha = 2^{\frac{d}{2}}\alpha^{(0)}\alpha,\qquad \alpha \diamond \beta \diamond \alpha=2^{\frac{d}{2}}(\alpha\diamond\beta)^{(0)}\alpha,\qquad(\pi^{\frac{1-s}{2}}\circ\tau)(\bar\kappa\alpha) = \kappa \bar{\alpha},\\
\nabla^g\alpha=\fra\diamond\alpha+\alpha\diamond(\pi^{\frac{1-s}{2}}\circ\tau)(\overline{\fra}),\qquad\frq\diamond\alpha=0
\end{gather*} 
for a complex exterior form $\beta\in\Omega_\C(M)$ satisfying $(\alpha\diamond\beta)^{(0)}\neq0$.
\end{cor}


\subsection{Hermitian spinorial forms in odd dimensions}


In this section, we present the characterization of Hermitian spinorial forms associated with irreducible Hermitian spinor bundles when $M$ is odd-dimensional. The theory is completely analogous to the even-dimensional case considered in the previous subsection, and thus we will omit explicit proofs. The main difference in the odd-dimensional case is the use of the \emph{truncated Kähler-Atiyah model}, as reviewed in Section \ref{sec:spinorialforms} and explained in detail in \cite{Algebraic_Complex_Square_2025,LBC13,Sha24}. Let $(S,\gamma_\ell,\scS)$ be an irreducible Hermitian spinor bundle on $(M,g)$. The admissible Hermitian pairing $\scS$ allows us to construct pointwise extensions to $M$ of the Hermitian square maps $\hcE_\kappa\colon\Sigma\to\End(\Sigma)$ and the Hermitian square spinor maps $\hcE^\kappa_\ell \colon \Sigma \to \wedge^<V^*_\C$ of Section \ref{sec:spinorialforms}. For simplicity in the exposition, we denote these by the same symbol, that is: 
\begin{eqnarray*}
\hcE_\kappa\colon S\to\End(S),\qquad \hcE^\kappa_\ell\colon S\to\wedge^<T^*_\C M,
\end{eqnarray*}

\noindent
where $\wedge^<T^*_\C M$ is the natural pointwise extension of $\wedge^<V^*_\C$.

\begin{definition}
The \emph{Hermitian spinor square maps} of the Hermitian spinor bundle $(S,\gamma_\ell,\scS)$ are the maps $\hcE^\kappa_\ell\colon\Gamma(S)\to\Omega^<_\C(M)$, $\kappa\in\U(1)$, induced by $\hcE^\kappa_\ell$ on sections. The \emph{Hermitian square} of an irreducible complex spinor $\eta\in\Gamma(S)$ is the complex exterior form $\hcE^\kappa_\ell(\eta)\in\Omega^<_\C(M)$ for $\kappa\in\U(1)$. Elements in the image of $\hcE^\kappa_\ell$ are generically called \emph{Hermitian spinorial forms}.
\end{definition}

Let $(S,\gamma_\ell,\scS)$ be an \emph{irreducible} Hermitian spinor bundle on $(M,g)$, which is necessarily associated to a strong spin-c structure. Set $\cA:=\nabla^{g,A}-\cD\in\Omega^1(M,\End(S))$ and let: $$\fra:=(\Psi^<_\ell\otimes\Id_{T^*M})^{-1}(\cA)\in\Omega^1(M,\wedge^<T^*_\C M)$$ be the symbol of $\cA$. Recall that if $\cQ\in\Gamma(\End(S)\otimes W)$ for a vector bundle $W$ on $M$, then its symbol is defined as $\frq:=(\Psi^<_\ell\otimes\Id_W)^{-1}(\cQ)\in\Gamma(\wedge^<T^*_\C M\otimes W)$. Then we have the following result, whose proof is completely analogous to that of Theorem \ref{thm:even_constrained_parallel_spinors}.

\begin{thm}\label{thm:odd_constrained_parallel_spinors}
Let $(S,\gamma_\ell,\scS)$ be an irreducible Hermitian spinor bundle on $(M,g)$ of adjoint type $s\in\Z_2$. Let $\cA\in\Omega^1(M,\End(S))$ and $\cQ\in\Gamma(\End(S)\otimes W)$ for a vector bundle $W$ on $M$. Then there exists a nowhere vanishing constrained parallel spinor $\eta\in\Gamma(S)$ relative to $(\cA,\cQ)$ only if: \begin{itemize}
\item There exists a nowhere vanishing complex exterior form $\alpha\in\Omega^<_\C(M)$ with vanishing spinor class $\mathfrak{s}_{\gamma_{\ell}}(\alpha)$.
\item The complex exterior form $\alpha\in\Omega^<_\C(M)$ satisfies the following algebraic and differential equations: \begin{gather*}
\alpha\vee\alpha = 2^{\frac{d-1}{2}}\alpha^{(0)}\alpha,\qquad \alpha\vee\beta\vee\alpha=2^{\frac{d-1}{2}}(\alpha\vee\beta)^{(0)}\alpha,\qquad(\pi^{\frac{1-s}{2}}\circ\tau)(\bar\kappa\alpha) = \kappa \bar{\alpha},\\
\nabla^g\alpha=\fra\vee\alpha+\alpha\vee(\pi^{\frac{1-s}{2}}\circ\tau)(\overline{\fra}),\qquad\frq\vee\alpha=0
\end{gather*} 
for a complex exterior form $\beta\in\Omega^<_\C(M)$ satisfying $(\alpha\vee\beta)^{(0)}\neq0$ or, equivalently, satisfies the equations: 
\begin{gather*}
\alpha \vee \beta \vee \alpha=2^{\frac{d-1}{2}}(\alpha\vee\beta)^{(0)}\alpha,\qquad(\pi^{\frac{1-s}{2}}\circ\tau)(\bar\kappa\alpha) = \kappa \bar{\alpha},\\
\nabla^g\alpha=\fra\vee\alpha+\alpha\vee(\pi^{\frac{1-s}{2}}\circ\tau)(\overline{\fra}),\qquad\frq\vee\alpha=0
\end{gather*} 
for every complex exterior form $\beta\in\Omega^<_\C(M)$.
\end{itemize}
    
The complex exterior form $\alpha\in\Omega^<_\C(M)$ as above is determined by $\eta\in\Gamma(S)$ through the relation: $$\alpha=\hcE^\kappa_\ell(\eta)$$ for some $\kappa\in\U(1)$. Moreover, $\alpha\in\Omega^<_\C(M)$ satisfying the algebraic equations of the theorem determines a nowhere vanishing spinor $\eta\in\Gamma(S)$, unique up to a $\U(1)$-valued function, satisfying the relation $\alpha=\hcE^\kappa_\ell(\eta)$.
\end{thm}

If we do not worry about studying constrained parallel spinors of a fixed irreducible Hermitian spinor bundle, we obtain the odd-dimensional analog of Corollary \ref{cor:even_constrained_parallel_spinors}.

\begin{cor}
\label{cor:odd_constrained_parallel_spinors}
Let $(M,g)$ be a strongly spin-c pseudo-Riemannian manifold and let $\fra\in\Omega^1(M,\wedge^< T^*_\C M)$ and $\frq\in\Gamma(\wedge^< T^*_\C M\otimes W)$ be given. Then $(M,g)$ admits a nowhere vanishing constrained parallel spinor relative to a pair $(\cA,\cQ)$ whose symbol is $(\fra,\frq)$ only if there exists a nowhere vanishing complex exterior form $\alpha\in\Omega^<_\C(M)$ satisfying the following algebraic and differential equations: 
\begin{gather*}
\alpha\vee\alpha = 2^{\frac{d-1}{2}}\alpha^{(0)}\alpha,\qquad \alpha\vee\beta\vee\alpha=2^{\frac{d-1}{2}}(\alpha\vee\beta)^{(0)}\alpha,\qquad(\pi^{\frac{1-s}{2}}\circ\tau)(\bar\kappa\alpha) = \kappa \bar{\alpha},\\
\nabla^g\alpha=\fra\vee\alpha+\alpha\vee(\pi^{\frac{1-s}{2}}\circ\tau)(\overline{\fra}),\qquad\frq\vee\alpha=0
\end{gather*} 
for a complex exterior form $\beta\in\Omega^<_\C(M)$ satisfying $(\alpha\vee\beta)^{(0)}\neq0$.
\end{cor}


\section{Parallel complex-bilinear spinorial forms}
\label{sec:Complex_bilinear_parallel_forms}


In Section \ref{sec:Hermitian_parallel_forms}, we have translated differential equations on an irreducible complex spinor into differential equations on its Hermitian square. However, these equations give only \emph{necessary} conditions. This is because, as we explained in Section \ref{sec:spinorialforms}, the Hermitian square of a spinor determines it only up to phase, while the complex-bilinear square determines it up to sign. In this section, we obtain \emph{necessary and sufficient} conditions on a complex-bilinear spinorial form for it to be the complex-bilinear square of an irreducible complex spinor parallel under a general connection on the spinor bundle. As in the Hermitian case, we will consider separately the cases where the dimension $d$ of $M$ is even or odd.


\subsection{Complex-bilinear spinorial forms in even dimensions}


As in Section \ref{sec:Hermitian_parallel_forms}, we consider the Kähler-Atiyah bundle $(\wedge T^*_\C M,\diamond)$ of $(M,g)$ as a model for the Clifford algebra $\mCl(M,g)$. We will also consider bundles of irreducible complex Clifford modules $(S,\gamma)$ on $(M,g)$, which we identify with bundles of modules $(S,\Psi_\gamma)$ over the Kähler-Atiyah bundle $(\wedge T^*_\C M,\diamond)$, where: 
\begin{eqnarray*}
\Psi_\gamma=\gamma\circ\Psi\colon(\wedge T^*_\C M,\diamond)\to(\End(S),\circ)    
\end{eqnarray*}

is the composition of the Chevalley-Riesz bundle map $\Psi\colon(\wedge T^*_\C M,\diamond)\to\mCl(M,g)$ and the unital morphism of bundles of algebras $\gamma\colon\mCl(M,g)\to\End(S)$.\medskip

Let $(S,\gamma,\scB)$ be a complex-bilinear paired spinor bundle on $(M,g)$. Recall that $\scB$ is a complex-bilinear pairing taking values in the characteristic line bundle $\cL$ canonically associated to the spin-c structure $Q$ in terms of which we have $S = Q \times_{\varrho} \Sigma$. The admissible complex-bilinear pairing $\scB$ allows us to construct global extensions to $M$ of the complex-bilinear square map $\cE\colon\Sigma\to\End(\Sigma)$ and the complex-bilinear square spinor map $\cE_\gamma\colon\Sigma\to\wedge V^*_\C$ of Section \ref{sec:spinorialforms}. However, these global maps are decorated with a complex line bundle $\cL$, which is absent in the algebraic theory laid out in Section \ref{sec:spinorialforms} but is crucial to develop the theory globally. In particular, we have:
\begin{equation*}
\cE\colon S\to\End(S)\otimes\cL,\qquad \cE_\gamma\colon S\to\wedge T^*_\C M\otimes\cL.
\end{equation*}

Furthermore, we have the following diagram of vector bundles: $$\begin{tikzcd}
S \arrow[rr, "\mathcal{E}"] \arrow[dd, "\mathcal{E}_\gamma"'] &  & \mathrm{End}(S)\otimes\cL                                   \\
                                                                                                &  &                                                   \\
{(\wedge T^*_{\mathbb{C}}M\otimes\cL,\diamond)} \arrow[rr, "\Psi\otimes\Id_{\cL}"] \arrow[rruu, "\Psi_\gamma\otimes\Id_{\cL}"]            &  & {\mathbb{C}\mathrm{l}(M,g)\otimes\cL} \arrow[uu, "\gamma\otimes\Id_{\cL}"']
\end{tikzcd}$$

which extends to maps of sections that we denote by the same symbol.

\begin{definition}
    The \emph{complex-bilinear spinor square map} of the complex-bilinear paired spinor bundle $(S,\gamma,\scB)$ is the map $\cE_\gamma\colon\Gamma(S)\to\Omega_\C(M,\cL)$ induced by $\cE_\gamma$ on sections. The \emph{complex-bilinear square} of an irreducible complex spinor $\eta\in\Gamma(S)$ is the $\cL$-valued complex exterior form $\cE_\gamma(\eta)\in\Omega_\C(M,\cL)$. Elements in the image of $\cE_\gamma$ are generically called \emph{complex-bilinear spinorial forms}.
\end{definition}

Let $(S,\gamma,\scB)$ be a complex-bilinear paired spinor bundle on $(M,g)$. Recall that if $\cQ\in\Gamma(\End(S)\otimes W)$, then its symbol is defined as $\frq=(\Psi_\gamma\otimes\Id_W)^{-1}(\cQ)\in\Gamma(\wedge T^*_\C M\otimes W)$. Lemma \ref{lemma:constrainedspinoreven} and Remark \ref{remark:equivalent_constrained_eq} imply the following.

\begin{lemma}\label{lemma:cx_bilinear_constrained_spinor}
    A spinor $\eta\in\Gamma(S)$ satisfies $\cQ(\eta)=0$ if and only if one of the following equivalent equations holds: $$\frq\diamond\alpha=0,\qquad\alpha\diamond(\pi^{\frac{1-s}{2}}\circ\tau)(\frq)=0,$$ where $\alpha=\cE_\gamma(\eta)\in\Omega_\C(M,\cL)$. These are equations as sections of $\Omega_\C(M,W\otimes\cL)$.
\end{lemma}

Since $(S,\gamma,\scB)$ is an irreducible complex-bilinear paired spinor bundle on $(M,g)$, it is necessarily associated to a strong spin-c structure \cite{FriedrichTrautman,LS18}, and thus we can write $\cA:=\nabla^{g,A}-\cD\in\Omega^1(M,\End(S))$ upon choosing a connection $A$ on $\cL$. Let: $$\fra:=(\Psi_\gamma\otimes\Id_{T^*M})^{-1}(\cA)\in\Omega^1(M,\wedge T^*_\C M)$$ be the symbol of $\cA$. We have:

\begin{lemma}
\label{lemma:complex_bilinear_parallel_spinor}
A nowhere vanishing spinor $\eta\in\Gamma(S)$ satisfies $\cD\eta=0$ if and only if: 
\begin{equation}
\label{eq:complex_bilinear_parallel_spinor}
\nabla^{g,A}\alpha=\fra\diamond\alpha+\alpha\diamond(\pi^{\frac{1-s}{2}}\circ\tau)(\fra) \in \Omega^1(M,\wedge T^*_\C M\otimes\cL),
\end{equation} 
where $\alpha=\cE_\gamma(\eta)\in\Omega_\C(M,\cL)$ and $\nabla^{g,A} = \nabla^g\otimes\nabla^{A}$.  
\end{lemma}

\begin{proof}
    Assume that the spinor $\eta$ satisfies $\nabla^{g,A}\eta=\cA(\eta)$. Set $E_\eta:=\cE(\eta)=\eta\otimes\scB(-,\eta)\in\Gamma(\End(S)\otimes\cL)$ and $\nabla:=\nabla^{g,A}\otimes\nabla^{A}$. Then: \begin{align*}
        (\nabla E_\eta)\chi&=\nabla(E_\eta(\chi))-E_\eta(\nabla^{g,A}\chi)\\
        &=\nabla(\eta\otimes\scB(\chi,\eta))-\eta\otimes\scB(\nabla^{g,A}\chi,\eta)\\
        &=\nabla^{g,A}\eta\otimes\scB(\chi,\eta)+\eta\otimes\nabla^{A}(\scB(\chi,\eta))-\eta\otimes\scB(\nabla^{g,A}\chi,\eta)\\
        &=\nabla^{g,A}\eta\otimes\scB(\chi,\eta)+\eta\otimes\scB(\chi,\nabla^{g,A}\eta)\\
        &=\cA(\eta)\otimes\scB(\chi,\eta)+\eta\otimes\scB(\chi,\cA(\eta))\\
        &=\cA(\eta)\otimes\scB(\chi,\eta)+\eta\otimes\scB(\cA^t(\chi),\eta)
    \end{align*} for all $\chi\in\Gamma(S)$, where $\cA^t$ denotes the adjoint of $\cA$ taken with respect to $\scB$. The equation above implies: \begin{equation}\label{eq:cov_dev_E_eta}
        (\nabla^{g,A}\otimes\nabla^{A})E_\eta=(\cA\otimes\Id_{\cL})\circ E_\eta+E_\eta\circ\cA^t.
    \end{equation}
    
    Set $\alpha=\cE_\gamma(\eta)=(\Psi_\gamma\otimes\Id_{\cL})^{-1}\circ E_\eta$. Taking $(\Psi_\gamma\otimes\Id_{\cL})^{-1}$ on Equation \eqref{eq:cov_dev_E_eta} and using Proposition \ref{prop:L-valued_compatibility} we obtain Equation \eqref{eq:complex_bilinear_parallel_spinor}.\medskip
    
    Conversely, assume that $\alpha\in\Omega_\C(M,\cL)$ satisfies Equation \eqref{eq:complex_bilinear_parallel_spinor}. Applying $\Psi_\gamma\otimes\Id_{\cL}$ to it gives Equation \eqref{eq:cov_dev_E_eta}, which reads: $$\cD_X\eta\otimes\scB(\chi,\eta)+\eta\otimes\scB(\chi,\cD_X\eta)=0$$ for all $\chi\in\Gamma(S)$ and $X\in\Gamma(TM)$. Hence $\cD_X\eta=\beta(X)\eta$ for some $\beta\in\Omega^1_\C(M)$. Plugging this into the above expression gives us: $$\beta(X)\eta\otimes\scB(\chi,\eta)=0$$ for all $\chi\in\Gamma(S)$ and $X\in\Gamma(TM)$. This implies $\beta=0$ since $\scB$ is non-degenerate and $\eta\in\Gamma(S)$ is nowhere vanishing. Hence $\cD\eta=0$.
\end{proof}

To globalize Theorem \ref{thm:bilinearsquare}, we need to acknowledge that the complex-bilinear square of an irreducible complex spinor $\eta\in\Gamma(S)$ is a $\cL$-valued complex exterior form $\alpha=\cE_\gamma(\eta)\in\Omega_\C(M,\cL)$. To deal with this fact, we extend the geometric product $\diamond$ trivially along $\cL$, so given two forms $\alpha,\beta \in \Omega_{\C}(M,\cL)$ we have:
\begin{equation*}
\alpha \diamond \beta \in \Omega_\C(M,\cL^{\otimes 2}).
\end{equation*}

Therefore, the algebraic equations of Theorem \ref{thm:bilinearsquare} now read as follows: $$\alpha\diamond\alpha = 2^{\frac{d}{2}}\alpha^{(0)}\otimes\alpha,\qquad \alpha\diamond\beta\diamond\alpha=2^{\frac{d}{2}}(\alpha\diamond\beta)^{(0)}\otimes\alpha,\qquad (\pi^{\frac{1-s}{2}}\circ\tau)(\alpha)=\sigma\alpha$$ for a $\cL$-valued complex exterior form $\beta\in\Omega_\C(M,\cL)$ satisfying $(\alpha\diamond\beta)^{(0)}\neq0$. Note the following: \begin{itemize}
    \item Since $\alpha^{(0)}\in\Omega_\C^0(M,\cL)=\Gamma(\cL)$, the first equation is as sections of $\Omega_\C(M,\cL^{\otimes2})$.
    \item Since $(\alpha\diamond\beta)^{(0)}\in\Omega_\C^0(M,\cL^{\otimes2})=\Gamma(\cL^{\otimes2})$, the second equation is as sections of $\Omega_\C(M,\cL^{\otimes3})$.
\end{itemize}

We proceed now to associate a cohomology class in $H^1(M,\mathbb{Z}_2)$ to every nowhere vanishing element $\alpha \in \Omega_{\C}(M,\cL)$. Given $\alpha \in \Omega_{\C}(M,\cL)$ together with a good open cover $\left\{ U_a \right\}$ for some index $a$, choose local sections $\eta_a \colon U_a \to S$ such that $\cE_{\gamma}(\eta_a) = \alpha\vert_{U_a}$. This defines a family of functions:
\begin{equation*}
s_{ab} \colon U_a\cap U_b \to \mathbb{Z}_2
\end{equation*}

\noindent
defined by the condition $\eta_a = s_{ab} \eta_b$ for all indices $a$ and $b$ such that $U_a\cap U_b \neq \emptyset$. This family of functions can be seen to satisfy, by verifying the consistency of its defining relation on triple overlaps, that $s_{ab}s_{bc}s_{ca} = 1$. Hence, the collection of $s_{ab}$ defines a \u{C}ech cocycle with values in $\mathbb{Z}_2$. Furthermore, different choices of lifts define cohomologous cocycles, whence we obtain a well-defined class in $H^1(M,\mathbb{Z}_2)$, which we denote by $\mathfrak{s}_{\gamma}(\alpha)$ and refer to as the $\mathbb{Z}_2$-\emph{spinor class} of $\alpha$. Clearly, we have the following.

\begin{lemma}
\label{lemma:Z2spinorclass}
Let $(S,\gamma,\scB)$ be a complex-bilinear paired spinor bundle on $(M,g)$ and $\alpha \in \Omega_{\C}(M,\cL)$. There exists a nowhere vanishing spinor $\eta\in\Gamma(S)$ such that $\alpha = \cE_{\gamma}(\eta)$ if and only if $\mathfrak{s}_{\gamma}(\alpha) = 0$.
\end{lemma}

The $\mathbb{Z}_2$-spinor class of a nowhere vanishing form $\alpha \in \Im(\cE_{\gamma})$ can be made to vanish by tuning the complex-bilinear paired spinor bundle without changing $\cL$.

\begin{prop}
Let $(S_o,\gamma_o,\scB_o)$ be a complex-bilinear paired spinor bundle and let $\alpha \in \Im(\cE_{\gamma_o})$ be nowhere vanishing. Then, there exists a complex-bilinear paired spinor bundle $(S,\gamma,\scB)$ and a nowhere vanishing section $\eta \in \Gamma(S)$ such that $\cE_{\gamma}(\eta) = \alpha$.
\end{prop}

\begin{proof}
Let $\mathfrak{s}_{\gamma_o}(\alpha)$  be the $\mathbb{Z}_2$-spinor class defined by $\alpha$, and fix a spin-c structure $Q_o$ to which $(S_o,\gamma_o,\scB_o)$ is associated naturally. Denote by:
\begin{equation*}
[x_{ab} , z_{ab}]\colon U_a \cap U_b \to \Spin^c_o(p,q)
\end{equation*}

\noindent
a \u{C}ech cocycle for $Q_o$. Then:
\begin{equation}
\label{eq:twistedtransition}
[x_{ab} , s_{ab} z_{ab}]\colon U_a \cap U_b \to \Spin^c_o(p,q)
\end{equation}

\noindent
is again a $\Spin^c_o(p,q)$-valued \u{C}ech cocycle, and hence defines a $\Spin^c_o(p,q)$ structure $Q$ which has precisely the same characteristic bundle as $Q_o$. Via the associated bundle construction, $Q$ defines a complex-bilinear paired spinor bundle $(S,\gamma,\scB)$. If $\eta_a$ is the family of local sections of $S_o$ used to define the maps $s_{ab}\colon U_a\cap U_b \to \mathbb{Z}_2$, then it becomes a well-defined, nowhere vanishing section of $S$, since multiplying by $s_{ab}$ is acting with the structure functions of $S$ as determined in Equation \eqref{eq:twistedtransition}.
\end{proof}

We arrive at the final characterization of a constrained parallel spinor relative to $(\cA,\cQ)$ in terms of its associated complex-bilinear square.

\begin{thm}\label{thm:even_CB_CPS}
Let $(S,\gamma,\scB)$ be an irreducible complex-bilinear paired spinor bundle of symmetry type $\sigma\in\Z_2$ and adjoint type $s\in\Z_2$ on $(M,g)$. Let $\cA\in\Omega^1(M,\End(S))$ and $\cQ\in\Gamma(\End(S)\otimes W)$ for a vector bundle $W$ on $M$. Then the following statements are equivalent: 
\begin{enumerate}
\item[\normalfont(a)] There exists a nowhere vanishing constrained parallel spinor $\eta\in\Gamma(S)$ relative to $(\cA,\cQ)$.
\item[\normalfont(b)] There exists a nowhere vanishing $\cL$-valued complex exterior form $\alpha\in\Omega_\C(M,\cL)$ with vanishing $\mathbb{Z}_2$-spinor class $\mathfrak{s}_{\gamma}(\alpha)\in H^1(M,\Z_2)$ which satisfies the following algebraic and differential equations: \begin{gather*}
            \alpha\diamond\alpha=2^{\frac{d}{2}}\alpha^{(0)}\otimes\alpha,\qquad \alpha\diamond\beta\diamond\alpha=2^{\frac{d}{2}}(\alpha\diamond\beta)^{(0)}\otimes\alpha,\qquad (\pi^{\frac{1-s}{2}}\circ\tau)(\alpha)=\sigma\alpha,\\
            \nabla^{g,A}\alpha=\fra\diamond\alpha+\alpha\diamond(\pi^{\frac{1-s}{2}}\circ\tau)(\fra),\qquad \frq\diamond\alpha=0
        \end{gather*} for a $\cL$-valued complex exterior form $\beta\in\Omega_\C(M,\cL)$ satisfying $(\alpha\diamond\beta)^{(0)}\neq0$ or, equivalently, satisfies the equations: \begin{gather*}
            \alpha\diamond\beta\diamond\alpha=2^{\frac{d}{2}}(\alpha\diamond\beta)^{(0)}\otimes\alpha,\qquad (\pi^{\frac{1-s}{2}}\circ\tau)(\alpha)=\sigma\alpha,\\
            \nabla^{g,A}\alpha=\fra\diamond\alpha+\alpha\diamond(\pi^{\frac{1-s}{2}}\circ\tau)(\fra),\qquad \frq\diamond\alpha=0
        \end{gather*} for every $\cL$-valued complex exterior form $\beta\in\Omega_\C(M,\cL)$.
    \end{enumerate}

    If in addition the spinor $\eta\in\Gamma(S)$ is chiral of chirality $\mu\in\Z_2$, then we have to add the equation: $$i^{q+\frac{d}{2}}*(\pi\circ\tau)(\alpha)=\mu\alpha.$$

    The $\cL$-valued complex exterior form $\alpha\in\Omega_\C(M,\cL)$ as above is determined by $\eta\in\Gamma(S)$ through the relation: $$\alpha=\cE_\gamma(\eta).$$
    
    Moreover, $\alpha\in\Omega_\C(M,\cL)$ satisfying the algebraic equations of the theorem determines a nowhere vanishing spinor $\eta\in\Gamma(S)$, unique up to sign, satisfying the relation $\alpha=\cE_\gamma(\eta)$.
\end{thm}

\begin{proof}
Condition $\mathfrak{s}_{\gamma}(\alpha)=0$ is required by Lemma \ref{lemma:Z2spinorclass}. The algebraic conditions of the theorem follow from the pointwise extension of Theorem \ref{thm:bilinearsquare} together with Lemma \ref{lemma:cx_bilinear_constrained_spinor}. The necessary and sufficient differential condition follows from Lemma \ref{lemma:complex_bilinear_parallel_spinor}. 
\end{proof}

As in the Hermitian case considered in Section \ref{sec:Hermitian_parallel_forms}, we are interested in the necessary and sufficient conditions for a manifold $(M,g)$ to admit a constrained parallel spinor, regardless of which spinor bundle it is a section of. 

\begin{cor}
\label{cor:even_constrained_parallel_spinorsbilinear}
Let $(M,g)$ be a strongly spin-c pseudo-Riemannian manifold and let $\fra\in\Omega^1(M,\wedge T^*_\C M)$ and $\frq\in\Gamma(\wedge T^*_\C M\otimes W)$ be given. Then $(M,g)$ admits a nowhere vanishing constrained parallel spinor relative to a pair $(\cA,\cQ)$ whose symbol is $(\fra,\frq)$ if and only if there exists a nowhere vanishing $\cL$-valued complex exterior form $\alpha\in\Omega_\C(M,\cL)$ satisfying the following algebraic and differential equations: 
\begin{gather*}
    \alpha\diamond\alpha=2^{\frac{d}{2}}\alpha^{(0)}\otimes\alpha,\qquad \alpha\diamond\beta\diamond\alpha=2^{\frac{d}{2}}(\alpha\diamond\beta)^{(0)}\otimes\alpha,\qquad (\pi^{\frac{1-s}{2}}\circ\tau)(\alpha)=\sigma\alpha,\\
    \nabla^{g,A}\alpha=\fra\diamond\alpha+\alpha\diamond(\pi^{\frac{1-s}{2}}\circ\tau)(\fra),\qquad \frq\diamond\alpha=0
\end{gather*} 
for a $\cL$-valued complex exterior form $\beta\in\Omega_\C(M,\cL)$ satisfying $(\alpha\diamond\beta)^{(0)}\neq0$.
\end{cor}

If the characteristic bundle $\cL$ is topologically trivial, then the spin-c structure reduces to a spin structure. Moreover, if the connection $A$ on $\cL$ is trivial, then the connection $\nabla^{g,A}$ on $S$ corresponds to the spinorial Levi-Civita connection $\nabla^g$ on $S$. In this case, we obtain the following corollary.

\begin{cor}\label{cor:spin_CB_CPS}
    Let $(S,\gamma,\scB)$ be an irreducible complex-bilinear paired spinor bundle of symmetry type $\sigma\in\Z_2$ and adjoint type $s\in\Z_2$. Let $\cA\in\Omega^1(M,\End(S))$ and $\cQ\in\Gamma(\End(S)\otimes W)$ for a vector bundle $W$ on $M$. Then the following statements are equivalent: 
    \begin{enumerate}
        \item[\normalfont(a)] There exists a nowhere vanishing constrained parallel spinor $\eta\in\Gamma(S)$ relative to $(\cA,\cQ)$.
        \item[\normalfont(b)] There exists a nowhere vanishing complex exterior form $\alpha\in\Omega_\C(M)$ with vanishing $\mathbb{Z}_2$-spinor class $\mathfrak{s}_{\gamma}(\alpha)\in H^1(M,\Z_2)$ which satisfies the following algebraic and differential equations: 
        \begin{gather*}
            \alpha\diamond\alpha=2^{\frac{d}{2}}\alpha^{(0)}\alpha,\qquad \alpha\diamond\beta\diamond\alpha=2^{\frac{d}{2}}(\alpha\diamond\beta)^{(0)}\alpha,\qquad (\pi^{\frac{1-s}{2}}\circ\tau)(\alpha)=\sigma\alpha,\\
            \nabla^g\alpha=\fra\diamond\alpha+\alpha\diamond(\pi^{\frac{1-s}{2}}\circ\tau)(\fra),\qquad \frq\diamond\alpha=0
        \end{gather*} for a complex exterior form $\beta\in\Omega_\C(M)$ satisfying $(\alpha\diamond\beta)^{(0)}\neq0$ or, equivalently, satisfies the equations: \begin{gather*}
            \alpha\diamond\beta\diamond\alpha=2^{\frac{d}{2}}(\alpha\diamond\beta)^{(0)}\alpha,\qquad (\pi^{\frac{1-s}{2}}\circ\tau)(\alpha)=\sigma\alpha,\\
            \nabla^g\alpha=\fra\diamond\alpha+\alpha\diamond(\pi^{\frac{1-s}{2}}\circ\tau)(\fra),\qquad \frq\diamond\alpha=0
        \end{gather*} for every complex exterior form $\beta\in\Omega_\C(M)$.
    \end{enumerate}

    If in addition the spinor $\eta\in\Gamma(S)$ is chiral of chirality $\mu\in\Z_2$, then we have to add the equation: $$i^{q+\frac{d}{2}}*(\pi\circ\tau)(\alpha)=\mu\alpha.$$

    The complex exterior form $\alpha\in\Omega_\C(M)$ as above is determined by $\eta\in\Gamma(S)$ through the relation: $$\alpha=\cE_\gamma(\eta).$$
    
    Moreover, $\alpha\in\Omega_\C(M)$ satisfying the algebraic equations of the theorem determines a nowhere vanishing spinor $\eta\in\Gamma(S)$, unique up to sign, satisfying the relation $\alpha=\cE_\gamma(\eta)$.
\end{cor}


\subsection{Complex-bilinear spinorial forms in odd dimensions}


Let $(S,\gamma_\ell,\scB)$ be a complex-bilinear paired spinor bundle on $(M,g)$. The admissible complex-bilinear pairing $\scB$ allows us to construct global extensions to $M$ of the complex-bilinear square map $\cE\colon\Sigma\to\End(\Sigma)$ and the complex-bilinear square spinor map $\cE_\ell\colon\Sigma\to\wedge^<V^*_\C$ of Section \ref{sec:spinorialforms}. As in the even-dimensional case, these global maps are valued in the complex line bundle $\cL$, that is:
\begin{equation*}
\cE\colon S\to\End(S)\otimes\cL,\qquad \cE_\ell\colon S\to\wedge^<T^*_\C M\otimes\cL.
\end{equation*}

\begin{definition}
The \emph{complex-bilinear spinor square map} of the complex-bilinear paired spinor bundle $(S,\gamma_\ell,\scB)$ is the map $\cE_\ell\colon\Gamma(S)\to\Omega^<_\C(M,\cL)$ induced by $\cE_\ell$ on sections. The \emph{complex-bilinear square} of an irreducible complex spinor $\eta\in\Gamma(S)$ is the $\cL$-valued complex exterior form $\cE_\ell(\eta)\in\Omega^<_\C(M,\cL)$. Elements in the image of $\cE_\ell$ are generically called \emph{complex-bilinear spinorial forms}.
\end{definition}

Let $(S,\gamma_\ell,\scB)$ be an \emph{irreducible} complex-bilinear paired spinor bundle on $(M,g)$, which, as explained earlier, we consider to be associated to a given spin-c structure $Q$. Hence, after choosing a connection $A$ on the characteristic $\U(1)$-bundle $P$ of $Q$, we can set $\cA:=\nabla^{g,A}-\cD\in\Omega^1(M,\End(S))$  and let: $$\fra:=(\Psi^<_\ell\otimes\Id_{T^*M})^{-1}(\cA)\in\Omega^1(M,\wedge^<T^*_\C M)$$ be the symbol of $\cA$. Recall that if $\cQ\in\Gamma(\End(S)\otimes W)$ for a vector bundle $W$ on $M$, then its symbol is defined as $\frq:=(\Psi^<_\ell\otimes\Id_W)^{-1}(\cQ)\in\Gamma(\wedge^<T^*_\C M\otimes W)$. Then we have the following result, whose proof is completely analogous to that of Theorem \ref{thm:even_CB_CPS}.

\begin{thm}\label{thm:odd_CB_CPS}
Let $(S,\gamma_\ell,\scB)$ be an irreducible complex-bilinear paired spinor bundle of symmetry type $\sigma\in\Z_2$ and adjoint type $s\in\Z_2$ on $(M,g)$. Let $\cA\in\Omega^1(M,\End(S))$ and $\cQ\in\Gamma(\End(S)\otimes W)$ for a vector bundle $W$ on $M$. Then the following statements are equivalent: 
\begin{enumerate}
        \item[\normalfont(a)] There exists a nowhere vanishing constrained parallel spinor $\eta\in\Gamma(S)$ relative to $(\cA,\cQ)$.
        \item[\normalfont(b)] There exists a nowhere vanishing $\cL$-valued complex exterior form $\alpha\in\Omega^<_\C(M,\cL)$ with vanishing $\mathbb{Z}_2$-spinor class $\mathfrak{s}_{\gamma}(\alpha)\in H^1(M,\Z_2)$ which satisfies the following algebraic and differential equations: \begin{gather*}
            \alpha\vee\alpha=2^{\frac{d-1}{2}}\alpha^{(0)}\otimes\alpha,\qquad \alpha\vee\beta\vee\alpha=2^{\frac{d-1}{2}}(\alpha\vee\beta)^{(0)}\otimes\alpha,\qquad (\pi^{\frac{1-s}{2}}\circ\tau)(\alpha)=\sigma\alpha,\\
            \nabla^{g,A}\alpha=\fra\vee\alpha+\alpha\vee(\pi^{\frac{1-s}{2}}\circ\tau)(\fra),\qquad \frq\vee\alpha=0
        \end{gather*} for a $\cL$-valued complex exterior form $\beta\in\Omega^<_\C(M,\cL)$ satisfying $(\alpha\vee\beta)^{(0)}\neq0$ or, equivalently, satisfies the equations: \begin{gather*}
            \alpha\vee\beta\vee\alpha=2^{\frac{d-1}{2}}(\alpha\vee\beta)^{(0)}\otimes\alpha,\qquad (\pi^{\frac{1-s}{2}}\circ\tau)(\alpha)=\sigma\alpha,\\
            \nabla^{g,A}\alpha=\fra\vee\alpha+\alpha\vee(\pi^{\frac{1-s}{2}}\circ\tau)(\fra),\qquad \frq\vee\alpha=0
        \end{gather*} for every $\cL$-valued complex exterior form $\beta\in\Omega^<_\C(M,\cL)$.
    \end{enumerate}
    
    The $\cL$-valued complex exterior form $\alpha\in\Omega^<_\C(M,\cL)$ as above is determined by $\eta\in\Gamma(S)$ through the relation: $$\alpha=\cE_\ell(\eta).$$
    
    Moreover, $\alpha\in\Omega^<_\C(M,\cL)$ satisfying the algebraic equations of the theorem determines a nowhere vanishing spinor $\eta\in\Gamma(S)$, unique up to sign, satisfying the relation $\alpha=\cE_\ell(\eta)$.
\end{thm}
 
The analog of Corollary \ref{cor:even_constrained_parallel_spinorsbilinear} is given as follows.

\begin{cor}
\label{cor:odd_constrained_parallel_spinorsbilinear}
Let $(M,g)$ be a strongly spin-c pseudo-Riemannian manifold and let $\fra\in\Omega^1(M,\wedge^< T^*_\C M)$ and $\frq\in\Gamma(\wedge^< T^*_\C M\otimes W)$ be given. Then $(M,g)$ admits a nowhere vanishing constrained parallel spinor relative to a pair $(\cA,\cQ)$ whose symbol is $(\fra,\frq)$ if and only if there exists a nowhere vanishing $\cL$-valued complex exterior form $\alpha\in\Omega^<_\C(M,\cL)$ satisfying the following algebraic and differential equations:
\begin{gather*}
    \alpha\vee\alpha=2^{\frac{d-1}{2}}\alpha^{(0)}\otimes\alpha,\qquad \alpha\vee\beta\vee\alpha=2^{\frac{d-1}{2}}(\alpha\vee\beta)^{(0)}\otimes\alpha,\qquad (\pi^{\frac{1-s}{2}}\circ\tau)(\alpha)=\sigma\alpha,\\
    \nabla^{g,A}\alpha=\fra\vee\alpha+\alpha\vee(\pi^{\frac{1-s}{2}}\circ\tau)(\fra),\qquad \frq\vee\alpha=0
\end{gather*} 
for a $\cL$-valued complex exterior form $\beta\in\Omega^<_\C(M,\cL)$ satisfying $(\alpha\vee\beta)^{(0)}\neq0$.
\end{cor}


\section{Spin-c Killing spinors in low dimensions}
\label{sec:lowKilling}


In this section, we apply the results obtained in the previous sections to study real and imaginary Killing spinors on spin-c Riemannian manifolds $(M,g)$ of dimension two, three, and four. These cases will be used to exemplify the applicability of the formalism of complex spinorial forms developed in this paper. We will recover some well-known results of \cite{Moroianu1997,Herzlich_Moroianu_1999,Grosse_Nakad_2015} without requiring $(M,g)$ to be simply connected or complete.

\begin{definition}
    Let $(M,g)$ be a spin-c Riemannian manifold. A nowhere vanishing irreducible complex spinor $\eta\in\Gamma(S)$ satisfying: $$\nabla^{g,A}_X\eta=i\lambda X^\flat\cdot\eta$$ for all $X\in\Gamma(TM)$ and some non-zero complex number $\lambda=\lambda_R+i\lambda_I\in\C\setminus\{0\}$ is called a \emph{Killing spinor} with \emph{Killing number} $\lambda$. It turns out that $\lambda$ must be either real or imaginary. Hence, such $\eta\in\Gamma(S)$ is called a \emph{real Killing spinor} if $\lambda=\lambda_R\in\R\setminus\{0\}$ and an \emph{imaginary Killing spinor} if $\lambda=i\lambda_I\in i\R\setminus\{0\}$.
\end{definition}

\begin{remark}
    Note the factor $i$ in the definition of a Killing spinor. In the mathematical literature, and in particular in the articles cited above, Clifford multiplication is defined using the convention opposite to ours. The inclusion of the factor $i$ in our definition ensures that our notions and results of real and imaginary spin-c Killing spinors are consistent with those appearing in the literature.
\end{remark}

For the following, recall that for one-forms $\theta,\vartheta\in\Omega^1(M)$ we define: $$\theta\wedge\vartheta:=\theta\otimes\vartheta-\vartheta\otimes\theta,\qquad \theta\odot\vartheta:=\theta\otimes\vartheta+\vartheta\otimes\theta$$ and: \begin{align*}
    \mathrm{Alt}(\nabla^g\theta)(X,Y)&:=(\nabla^g\theta)(X,Y)-(\nabla^g\theta)(Y,X),\\
    \mathrm{Sym}(\nabla^g\theta)(X,Y)&:=(\nabla^g\theta)(X,Y)+(\nabla^g\theta)(Y,X)
\end{align*} for all $X,Y\in\Gamma(TM)$. It is not difficult to show that: $$\mathrm{Alt}(\nabla^g\theta)=\dd\theta,\qquad\mathrm{Sym}(\nabla^g\theta)=\mathscr{L}_{\theta^\sharp}g.$$


\subsection{Spin-c Killing spinors in dimension two}


Let $(S,\gamma)$ be a bundle of irreducible complex Clifford modules on a two-dimensional spin-c Riemannian manifold $(M,g)$. Assume that $(S,\gamma)$ is equipped with an admissible Hermitian pairing $\scS$ of positive adjoint type. Set $\kappa=1$ for simplicity. By the global version of Theorem \ref{thm:even_Hermitian_forms}, a complex exterior form $\widehat{\alpha}\in\Omega_\C(M)$ is the Hermitian square of a nowhere vanishing irreducible complex spinor $\eta\in\Gamma(S)$ if and only if: $$\widehat{\alpha}\diamond\widehat{\alpha}=2\widehat{\alpha}^{(0)}\widehat{\alpha},\qquad \widehat{\alpha}\diamond\beta\diamond\widehat{\alpha}=2(\widehat{\alpha}\diamond\beta)^{(0)}\widehat{\alpha},\qquad \tau(\widehat{\alpha})=\overline{\widehat{\alpha}}$$ for a complex exterior form $\beta\in\Omega_\C(M)$ satisfying $(\widehat{\alpha}\diamond\beta)^{(0)}\neq0$.

\begin{prop}\label{prop:square_(2,0)}
    Let $(M,g)$ be a two-dimensional spin-c Riemannian manifold. A complex exterior form $\widehat{\alpha}\in\Omega_\C(M)$ is the Hermitian square of a nowhere vanishing irreducible complex spinor $\eta\in\Gamma(S)$ if and only if: $$\widehat{\alpha}=r+\theta+if\nu,$$ where $r\in C^\infty(M)$ is nowhere vanishing, $f\in C^\infty(M)$ and $\theta\in\Omega^1(M)$ satisfying $r^2=f^2+\escal{\theta,\theta}$.
\end{prop}

\begin{proof}
    Let $\widehat{\alpha}=\widehat{\alpha}^{(0)}+\widehat{\alpha}^{(1)}+\widehat{\alpha}^{(2)}$, where $\widehat{\alpha}^{(k)}\in\Omega^k_\C(M)$. The linear equation $\tau(\widehat{\alpha})=\overline{\widehat{\alpha}}$ implies that $\widehat{\alpha}^{(0)}$ and $\widehat{\alpha}^{(1)}$ are real and that $\widehat{\alpha}^{(2)}$ is imaginary. Set $r:=\widehat{\alpha}^{(0)}\in C^\infty(M)$, $\theta:=\widehat{\alpha}^{(1)}\in\Omega^1(M)$, and $i\omega:=\widehat{\alpha}^{(2)}$ for $\omega\in\Omega^2(M)$. Since $\omega$ is a two-form, we can write $\omega=f\nu$ for some $f\in C^\infty(M)$, where $\nu$ is the Riemannian volume form. Hence $\widehat{\alpha}=r+\theta+if\nu$. We compute: $$\widehat{\alpha}\diamond\widehat{\alpha}=(r+\theta+if\nu)\diamond(r+\theta+if\nu)=r^2+2r\theta+2irf\nu+\escal{\theta,\theta}+f^2.$$
    
    Then the equation $\widehat{\alpha}\diamond\widehat{\alpha}=2\widehat{\alpha}^{(0)}\widehat{\alpha}$ amounts to $r^2=f^2+\escal{\theta,\theta}$. Since $\scS$ is a positive-definite admissible Hermitian pairing and $\eta\in\Gamma(S)$ is nowhere vanishing, we have that $r=\frac{1}{2}\scS(\eta,\eta)$ is nowhere vanishing as well, thus it suffices to take $\beta=1$ in the quadratic equation $\widehat{\alpha}\diamond\beta\diamond\widehat{\alpha}=2(\widehat{\alpha}\diamond\beta)^{(0)}\widehat{\alpha}$ to conclude.
\end{proof}

\begin{remark}\label{remark:(2,0)_chiral}
    If $\theta=0$, then $r^2=f^2$ and $f=\mu r$ for $\mu\in\Z_2$. In this situation $\widehat{\alpha}=r+i\mu r\nu$ is the Hermitian square of a \emph{chiral} spinor with chirality $\mu\in\Z_2$, see \cite[Cor.\ 5.1]{Algebraic_Complex_Square_2025}. Note that $\theta$ need not be nowhere vanishing; hence there may exist points $p\in M$ where $\theta_p=0$. At such points $\widehat{\alpha}_p=r_p+i\mu r_p\nu_p$, so the spinor $\eta$ is pointwise chiral at $p$, even if $\eta$ is not chiral on all of $M$.
\end{remark}

\begin{prop}\label{prop:(2,0)_Killing}
    Let $(M,g)$ be a two-dimensional spin-c Riemannian manifold and let $\eta\in\Gamma(S)$ be a Killing spinor with Killing number $\lambda=\lambda_R+i\lambda_I$. Then: $$\dd r=-2\lambda_I\theta,\qquad \nabla^g\theta=-2\lambda_Irg-2\lambda_Rf\nu,\qquad \dd f=-2\lambda_R*\theta,$$ where $\widehat{\alpha}=r+\theta+if\nu$ is the Hermitian square of $\eta$.
\end{prop}

\begin{proof}
    In the case of a Killing spinor, we have that $\cA_X=i\lambda\Psi_\gamma(X^\flat)\in\Gamma(\End(S))$, so its symbol is just $\fra_X=i\lambda X^\flat\in\Omega^1_\C(M)$. We compute: \begin{align*}
        \fra_X\diamond\widehat{\alpha}&=i\lambda X^\flat\diamond(r+\theta+if\nu)=i\lambda(rX^\flat+X^\flat\wedge\theta+\theta(X)+if*X^\flat),\\
        \widehat{\alpha}\diamond\tau(\overline{\fra}_X)&=(r+\theta+if\nu)\diamond(-i\bar{\lambda}X^\flat)=-i\bar{\lambda}(rX^\flat+\theta\wedge X^\flat+\theta(X)-if*X^\flat).
    \end{align*}
    
    By Corollary \ref{cor:even_constrained_parallel_spinors}, if the spinor $\eta\in\Gamma(S)$ satisfies $\nabla^{g,A}_X\eta=\cA_X(\eta)$, then $\nabla^g_X\widehat{\alpha}=\fra_X\diamond\widehat{\alpha}+\widehat{\alpha}\diamond\tau(\overline{\fra}_X)$. Separating by degrees this equation we obtain: $$\nabla^g_Xr=-2\lambda_I\theta(X),\qquad \nabla^g_X\theta=-2\lambda_IrX^\flat-2\lambda_Rf*X^\flat,\qquad(\nabla^g_Xf)\nu=2\lambda_RX^\flat\wedge\theta$$ for all $X\in\Gamma(TM)$. From the first equation, we obtain $\dd r=-2\lambda_I\theta$, and from the second equation, we obtain $\nabla^g\theta=-2\lambda_Irg-2\lambda_Rf\nu$. Taking the Hodge star operator of the third equation and using: $$*(X^\flat\wedge\theta)=-*(\theta\wedge X^\flat)=-\iota_X(*\theta)=-(*\theta)(X)$$ we obtain $\nabla^g_Xf=-2\lambda_R(*\theta)(X)$, which in turn implies that $\dd f=-2\lambda_R*\theta$.
\end{proof}

Using $\dd r=-2\lambda_I\theta$ and $\nabla^g\theta=-2\lambda_Irg-2\lambda_Rf\nu$ we compute the Hessian of the function $r$: $$\Hess(r):=\nabla^g\dd r=4\lambda_I^2rg+4\lambda_R\lambda_If\nu.$$

Since the Hessian of a function is a symmetric tensor (the Levi-Civita connection is torsion-free), the term $\lambda_R\lambda_If\nu\in\Omega^2(M)$ must vanish. Suppose that $\lambda_R\lambda_I\neq0$. Then $f=0$ and from $\dd f=-2\lambda_R*\theta$ we obtain $\theta=0$, thus $r=0$, which is not possible. Therefore $\lambda_R\lambda_I=0$, as expected.

\begin{cor}
    Let $(M,g)$ be a two-dimensional spin-c Riemannian manifold and let $\eta\in\Gamma(S)$ be a real Killing spinor with Killing number $\lambda_R$. Then $(M,g)$ has constant positive Gaussian curvature $K^g\equiv4\lambda_R^2$.
\end{cor}

\begin{proof}
    By Proposition \ref{prop:(2,0)_Killing} we have: $$\dd r=0,\qquad \nabla^g\theta=-2\lambda_Rf\nu,\qquad \dd f=-2\lambda_R*\theta.$$

    Let $\omega:=f\nu$. Then $\nabla^g_X\omega=2\lambda_R X^\flat\wedge\theta$ and $\nabla^g_X\theta=-2\lambda_R\iota_X\omega$. A computation gives us: \begin{equation}\label{eq:R(X,Y)theta}
        R^g(X,Y)\theta=2\lambda_R(\iota_X(\nabla^g_Y\omega)-\iota_Y(\nabla^g_X\omega))=4\lambda_R^2(\theta(Y)X^\flat-\theta(X)Y^\flat).
    \end{equation}
    
    The curvature of a two-dimensional Riemannian manifold $(M,g)$ is completely determined by its Gaussian curvature $K^g\in C^\infty(M)$. For a one-form $\theta$, the curvature acts as $(R^g(X,Y)\theta)Z=-\theta(R^g(X,Y)Z)$. Then a computation gives: $$R^g(X,Y)\theta=K^g(\theta(Y)X^\flat-\theta(X)Y^\flat).$$
    
    Comparing this expression with \eqref{eq:R(X,Y)theta} we obtain $(K^g-4\lambda_R^2)(\theta(Y)X^\flat-\theta(X)Y^\flat)=0$ for all $X,Y\in\Gamma(TM)$.\medskip

    Let $U:=\{p\in M\mid \theta_p\neq0\}$. The subset $U\subset M$ is open. Let $p\in U$, take $X_p=\theta^\sharp_p\in T_pM$ and $Y_p\in T_pM$ non-zero and orthogonal to $X_p$. Then $(\theta(Y)X^\flat-\theta(X)Y^\flat)_p=-\escal{\theta_p,\theta_p}Y_p\neq0$ and $K^g_p=4\lambda_R^2$. Since the point $p\in U$ was arbitrary, we conclude that $K^g\equiv4\lambda_R^2$ on $U$.\medskip

    Suppose that $U\subset M$ is not dense, then its complement $U^c:=M\setminus U=\{p\in M\mid\theta_p=0\}$ contains an open set $W$. On $W$, $\theta=0$ implies that $\dd f=0$, hence $f$ is constant on $W$ and $\Hess(f)=0$ on $W$. The relation $\Hess(f)=-4\lambda_R^2fg$ implies that $\lambda_R^2f=0$ on $W$. Since $\lambda_R\neq0$, then $f=0$ on $W$. This implies that $r=0$ on $W$ since $r^2=f^2+\escal{\theta,\theta}$, but this is not possible since $r$ is nowhere vanishing. Therefore, $U^c$ has empty interior, i.e.\ $U$ is dense. Finally, since $K^g\equiv4\lambda_R^2$ on a dense open subset $U\subset M$, we conclude that $K^g\equiv4\lambda_R^2$ on $M$.
\end{proof}

\begin{cor}
    Let $(M,g)$ be a two-dimensional spin-c Riemannian manifold and let $\eta\in\Gamma(S)$ be an imaginary Killing spinor with Killing number $i\lambda_I$. Then $(M,g)$ has constant negative Gaussian curvature $K^g\equiv-4\lambda_I^2$.
\end{cor}

\begin{proof}
    By Proposition \ref{prop:(2,0)_Killing} we have: $$\dd r=-2\lambda_I\theta,\qquad \nabla^g\theta=-2\lambda_Irg,\qquad \dd f=0.$$

    These equations imply that $\Hess(r)=4\lambda_I^2rg$. Let $V_r:=\mathrm{grad}(r)=(\dd r)^\sharp\in\Gamma(TM)$. Then: $$g(\nabla^g_XV_r,Y)=\Hess(r)(X,Y)=4\lambda_I^2rg(X,Y)=g(4\lambda_I^2rX,Y)$$ for all $X,Y\in\Gamma(TM)$, thus $\nabla^g_XV_r=4\lambda_I^2rX$. A computation gives us: \begin{equation}\label{eq:R(X,Y)V_r}
        R^g(X,Y)V_r:=(\nabla^g_X\nabla^g_Y-\nabla^g_Y\nabla^g_X-\nabla^g_{[X,Y]})V_r=4\lambda_I^2(X(r)Y-Y(r)X).
    \end{equation}

    Recall that the curvature of a two-dimensional Riemannian manifold $(M,g)$ is completely determined by its Gaussian curvature $K^g\in C^\infty(M)$, which is given by: $$R^g(X,Y)Z=K^g(g(Y,Z)X-g(X,Z)Y).$$

    Setting $Z=V_r$ we obtain $R^g(X,Y)V_r=K^g(Y(r)X-X(r)Y)$. Comparing this expression with \eqref{eq:R(X,Y)V_r} we obtain $(K^g+4\lambda_I^2)(X(r)Y-Y(r)X)=0$ for all $X,Y\in\Gamma(TM)$.\medskip
    
    Let $U:=\{p\in M\mid (V_r)_p\neq0\}$. The subset $U\subset M$ is open. Let $p\in U$, take $X_p=(V_r)_p\in T_pM$ and $Y_p\in T_pM$ non-zero and orthogonal to $X_p$. Then $(X(r)Y-Y(r)X)_p=g_p((V_r)_p,(V_r)_p)Y_p\neq0$ and $K^g_p=-4\lambda_I^2$. Since the point $p\in U$ was arbitrary, we conclude that $K^g\equiv-4\lambda_I^2$ on $U$.\medskip

    Suppose that $U\subset M$ is not dense, then its complement $U^c:=M\setminus U=\{p\in M\mid(V_r)_p=0\}$ contains an open set $W$. On $W$, $V_r=0$ implies that $\dd r=0$, hence $r$ is constant on $W$ and $\Hess(r)=0$ on $W$. The relation $\Hess(r)=4\lambda_I^2rg$ implies that $\lambda_I^2r=0$ on $W$. Since $r$ is nowhere vanishing and $\lambda_I\neq0$, this is not possible. Therefore, $U^c$ has empty interior, i.e.\ $U$ is dense. Finally, since $K^g\equiv-4\lambda_I^2$ on a dense open subset $U\subset M$, we conclude that $K^g\equiv-4\lambda_I^2$ on $M$.
\end{proof}


\subsection{Spin-c Killing spinors in dimension three}


Let $(S,\gamma_\ell)$ be a bundle of irreducible complex Clifford modules on a three-dimensional spin-c Riemannian manifold $(M,g)$. Assume that $(S,\gamma_\ell)$ is equipped with an admissible Hermitian pairing $\scS$ of positive adjoint type. The truncated Kähler-Atiyah bundle in this dimension and signature is defined on: $$\wedge^<T^*_\C M=\C\oplus T^*_\C M,$$ where here $\C$ denotes the trivial complex line bundle, endowed with the product: $$\alpha\vee\beta=\cP_<(\alpha\diamond\beta+i\ell*\tau(\alpha\diamond\beta)),\quad \ell\in\Z_2$$ for all $\alpha,\beta\in\wedge^<T^*_\C M$. Set $\kappa=1$ for simplicity. By the global version of Theorem \ref{thm_Hermitian_square_odd}, a complex exterior form $\widehat{\alpha}\in\Omega^<_\C(M)$ is the Hermitian square of a nowhere vanishing irreducible complex spinor $\eta\in\Gamma(S)$ if and only if: $$\widehat{\alpha}\vee\widehat{\alpha}=2\widehat{\alpha}^{(0)}\widehat{\alpha},\qquad \widehat{\alpha}\vee\beta\vee\widehat{\alpha}=2(\widehat{\alpha}\vee\beta)^{(0)}\widehat{\alpha},\qquad \tau(\widehat{\alpha})=\overline{\widehat{\alpha}}$$ for a complex exterior form $\beta\in\Omega^<_\C(M)$ satisfying $(\widehat{\alpha}\vee\beta)^{(0)}\neq0$.

\begin{prop}\label{prop:square_(3,0)}
    Let $(M,g)$ be a three-dimensional spin-c Riemannian manifold. A complex exterior form $\widehat{\alpha}\in\Omega^<_\C(M)$ is the Hermitian square of a nowhere vanishing irreducible complex spinor $\eta\in\Gamma(S)$ if and only if: $$\widehat{\alpha}=r+\theta,$$ where $r\in C^\infty(M)$ is nowhere vanishing and $\theta\in\Omega^1(M)$ satisfying $r^2=\escal{\theta,\theta}$.
\end{prop}

\begin{proof}
    Let $\widehat{\alpha}=\widehat{\alpha}^{(0)}+\widehat{\alpha}^{(1)}\in\Omega^<_\C(M)$. The linear equation $\tau(\widehat{\alpha})=\overline{\widehat{\alpha}}$ implies that $\widehat{\alpha}^{(0)}$ and $\widehat{\alpha}^{(1)}$ are real. Set $r:=\alpha^{(0)}\in C^\infty(M)$ and $\theta:=\alpha^{(1)}\in\Omega^1(M)$. Hence $\widehat{\alpha}=r+\theta$. Using that $\widehat{\alpha}\diamond\widehat{\alpha}=r^2+2r\theta+\escal{\theta,\theta}$, we obtain that the quadratic equation $\widehat{\alpha}\vee\widehat{\alpha}=2\widehat{\alpha}^{(0)}\widehat{\alpha}$ is equivalent to $r^2=\escal{\theta,\theta}$. Since $\scS$ is a positive-definite admissible Hermitian pairing and $\eta\in\Gamma(S)$ is nowhere vanishing, we have that $r=\frac{1}{2}\scS(\eta,\eta)$ is nowhere vanishing as well, thus it suffices to take $\beta=1$ in the quadratic equation $\smash{\widehat{\alpha}\vee\beta\vee\widehat{\alpha}=2(\widehat{\alpha}\vee\beta)^{(0)}\widehat{\alpha}}$ to conclude.
\end{proof}

\begin{remark}
    Since $\widehat{\alpha}$ and $r$ are nowhere vanishing, and the relation $r^2=\escal{\theta,\theta}$ holds, the one-form $\theta$ is nowhere vanishing too.
\end{remark}

\begin{prop}\label{prop:(3,0)_Killing}
    Let $(M,g)$ be a three-dimensional spin-c Riemannian manifold and let $\eta\in\Gamma(S)$ be a Killing spinor with Killing number $\lambda=\lambda_R+i\lambda_I$. Then: $$\dd r=-2\lambda_I\theta,\qquad \nabla^g\theta=-2\lambda_Irg-2\ell\lambda_R*\theta,$$ where $\widehat{\alpha}=r+\theta$ is the Hermitian square of $\eta$.
\end{prop}

\begin{proof}
    In the case of a Killing spinor, we have that $\cA_X=i\lambda\Psi^<_\ell(X^\flat)\in\Gamma(\End(S))$, so its symbol is just $\fra_X=i\lambda X^\flat\in\Omega^1_\C(M)$. We compute: \begin{align*}
        \fra_X\vee\widehat{\alpha}&=i\lambda X^\flat\vee(r+\theta)=i\lambda(rX^\flat+\theta(X)-i\ell*(X^\flat\wedge\theta)),\\
        \widehat{\alpha}\vee\tau(\overline{\fra}_X)&=(r+\theta)\vee(-i\bar{\lambda}X^\flat)=-i\bar{\lambda}(rX^\flat+\theta(X)+i\ell*(X^\flat\wedge\theta)).
    \end{align*}

    By Corollary \ref{cor:odd_constrained_parallel_spinors}, if the spinor $\eta\in\Gamma(S)$ satisfies $\nabla^{g,A}_X\eta=\cA_X(\eta)$, then $\nabla^g_X\widehat{\alpha}=\fra_X\vee\widehat{\alpha}+\widehat{\alpha}\vee\tau(\overline{\fra}_X)$. Separating by degrees this equation we obtain: $$\nabla^g_Xr=-2\lambda_I\theta(X),\qquad \nabla^g_X\theta=-2\lambda_IrX^\flat+2\ell\lambda_R*(X^\flat\wedge\theta)$$ for all $X\in\Gamma(TM)$. These two equations are equivalent to those of the statement.
\end{proof}

Using $\dd r=-2\lambda_I\theta$ and $\nabla^g\theta=-2\lambda_Irg-2\ell\lambda_R*\theta$ we compute the Hessian of the function $r$: $$\Hess(r)=\nabla^g\dd r=4\lambda_I^2rg+4\ell\lambda_R\lambda_I*\theta.$$

Since $\Hess(r)$ is a symmetric tensor and $\theta$ is nowhere vanishing, we have that $\lambda_R\lambda_I=0$, as expected.

\begin{cor}
A three-dimensional Riemannian manifold $(M,g)$ admits a real spin-c Killing spinor with Killing number $\lambda_R$ if and only if it is an $\alpha$-Sasakian manifold with $\alpha=2\lambda_R$.
\end{cor}

\begin{proof}
By Proposition \ref{prop:(3,0)_Killing} we have: $$\dd r=0,\qquad \nabla^g\theta=-2\ell\lambda_R*\theta.$$

The first equation implies that $r$ is a non-zero constant function. The second equation implies that $\theta^\sharp$ is a Killing vector field since $\mathscr{L}_{\theta^\sharp}g=\Sym(\nabla^g\theta)=0$. Set $\xi:=\frac{1}{r}\theta^\sharp$. Then $\xi$ is a Killing vector field of unit length. Set $\varphi:=-\frac{1}{2\lambda_R}\nabla^g\xi$ and $\zeta:=\xi^\flat$. Then: $$\nabla^g_X\xi=\tfrac{1}{r}\nabla^g_X\theta^\sharp=\tfrac{1}{r}(\nabla^g_X\theta)^\sharp=\tfrac{1}{r}(-2\ell\lambda_R\iota_X(*\theta))^\sharp=-2\ell\lambda_R(\iota_X(*\zeta))^\sharp$$ for all $X\in\Gamma(TM)$. Setting $J(X):=(\iota_X(*\zeta))^\sharp$, we get $\varphi(X)=-\frac{1}{2\lambda_R}\nabla^g_X\xi=\ell J(X)$, thus $\varphi^2=J^2$. Now let $\{e_1,e_2,\xi\}$ be an orthonormal frame of $(M,g)$. Then $J(e_1)=e_2$, $J(e_2)=-e_1$, and $J(\xi)=0$. Hence: $$\varphi^2=-\Id+\zeta\otimes\xi$$ since $\zeta(e_1)=\zeta(e_2)=0$ and $\zeta(\xi)=1$. Finally, we compute: 
\begin{align*}
((\nabla^g_X\varphi)Y)^\flat&=\nabla^g_X\varphi(Y)^\flat-\varphi(\nabla^g_XY)^\flat=\ell\nabla^g_X(\iota_Y(*\zeta))-\ell\iota_{\nabla^g_XY}(*\zeta)=\ell\iota_Y(*\nabla^g_X\zeta)\\
&=-2\lambda_R\iota_Y(*\iota_X(*\zeta))=-2\lambda_R\iota_Y(**(\zeta\wedge X^\flat))=2\lambda_R\iota_Y(X^\flat\wedge\zeta)\\
&=2\lambda_R(g(X,Y)\zeta-\zeta(Y)X^\flat),
\end{align*} 
which gives $(\nabla^g_X\varphi)Y=2\lambda_R(g(X,Y)\xi-\zeta(Y)X)$. Therefore, we conclude that $(M,g)$ is an $\alpha$-Sasakian manifold with $\alpha=2\lambda_R$.\medskip

To prove the converse, we use the complex-bilinear square of an irreducible complex spinor $\eta$, which is computed in \cite[Cor.\ 5.5]{Algebraic_Complex_Square_2025} using an admissible complex-bilinear pairing of negative adjoint type. Its global version yields a nowhere vanishing isotropic $\cL$-valued complex one-form $\beta\in\Omega^1_\C(M,\cL)$. By Corollary \ref{cor:odd_constrained_parallel_spinorsbilinear}, the equation $\nabla^{g,A}_X\eta=i\lambda X^\flat\cdot\eta$ is equivalent to $\nabla^{g,A}_X\beta=\fra_X\vee\beta+\beta\vee(\pi\circ\tau)(\fra_X)$, where $\fra_X=i\lambda X^\flat$. In the case that $\eta$ is a real Killing spinor, i.e.\ $\lambda=\lambda_R$, we obtain: 
\begin{equation}
\label{eq:bilinearKilling3d}
\nabla^{g,A}_X\beta=i\lambda_R(X^\flat\vee\beta-\beta\vee X^\flat)=i\lambda_R(2i\ell*(\beta\wedge X^\flat))= -2\ell\lambda_R\iota_X(*\beta)
\end{equation}
for all $X\in\Gamma(TM)$. Hence, we have to prove that every $\alpha$-Sasakian three-manifold, which automatically admits a $\Spin^c(3)$-structure, admits a nowhere vanishing isotropic $\cL$-valued one-form $\beta\in\Omega^1_\C(M,\cL)$ for which Equation \eqref{eq:bilinearKilling3d} is satisfied.\medskip

Assume $(M,g)$ is an $\alpha$-Sasakian manifold with $\alpha = 2\lambda_R$. Let $(\xi, \zeta, \varphi, g)$ be the $\alpha$-Sasakian structure, where $\zeta = \xi^\flat$ and the fundamental endomorphism satisfies:
\begin{equation*}
\nabla^g_X \xi = - \alpha \varphi(X) = -2\lambda_R \varphi(X).
\end{equation*}

\noindent
We identify $\varphi = \ell J$, where $J$ is the complex structure defined by the induced orientation in the contact distribution $\ker (\zeta)\subset TM$. Thus, we have:
\begin{equation*}
\nabla^g_X \xi = - 2\ell \lambda_R J(X^{\perp_{\xi}}),
\end{equation*}
where $X^{\perp_{\xi}} \in \ker(\zeta)$ denotes the orthogonal projection  of $X$ to $\ker(\zeta)$. The Levi-Civita connection does not, in general, preserve the contact distribution $\ker(\zeta)\subset TM$. Instead, by projecting orthogonally, we obtain an induced connection that we denote by $\nabla^{\perp_{\xi}}$ and is defined as follows:
\begin{equation*}
\nabla^{\perp_{\xi}}_X w := (\nabla_X^g w)^{\perp_{\xi}}\, , \quad X\in\Gamma(TM)\, , \quad w \in \Gamma(\ker(\zeta)).
\end{equation*}

\noindent
Furthermore, $\nabla^{\perp_{\xi}}$ preserves the complex structure $J$ on $\ker(\zeta)$, and therefore it defines a unitary connection on the complex line bundle $(\ker(\zeta),J)$ equipped with the connection $\nabla^{\perp_{\xi}}$. We identify:
\begin{equation*}
\cL := (\ker(\zeta)^{\ast},J^{\ast})
\end{equation*}
as complex line bundles, where $\ker(\zeta)^{\ast}$ is the vector bundle dual to $\ker(\zeta)$ equipped with the induced complex structure $J^*$. Since $\nabla^{\perp_{\xi}}$ preserves $J$, it induces a unitary connection on $\cL$ which is naturally associated to a connection $A$ on a principal $\U(1)$-bundle $P$ inducing $\cL$. The pair $(P,A)$ is unique modulo isomorphism of principal $\U(1)$-bundles with connection. Let $\{e_1, e_2, \xi\}$ be a local orthonormal frame such that $J(e_1) = e_2$ and $J(e_2) = -e_1$, with dual coframe $\{e^1, e^2, \zeta\}$. We define $\beta\in \Omega^1_{\mathbb{C}}(M,\cL)$ as follows:
\begin{equation*}
\beta = (e^1 + i e^2)\otimes (e^1 - i e^2),
\end{equation*}
where $e^1-ie^2$ is regarded as a local unitary frame of $\cL$. Note that this gives a globally well-defined, nowhere vanishing, isotropic $\cL$-valued complex one-form on $M$. We have:
\begin{equation*}
\nabla^{g}_X (e^1 + i e^2) = -i A(X) (e^1 + i e^2) + 2 \ell \lambda_R (iX_1 - X_2) \zeta,
\end{equation*}
where $X_i = e^i(X)$, $i=1,2$. Define the connection $\nabla^{g,A} := \nabla^g\otimes \nabla^{\perp_{\xi}}$ on the tensor product bundle $T^{\ast}M\otimes \cL$. By the previous calculation, we have:
\begin{align*}
\nabla^{g,A}_X\beta&=\nabla^g_X(e^1+ie^2)\otimes(e^1-ie^2) + (e^1+ie^2)\otimes\nabla^{\perp_\xi}_X(e^1-ie^2)\\
&=2\ell\lambda_R(iX_1-X_2)\zeta\otimes(e^1-ie^2)\\
&=-2\ell\lambda_R(*\beta)(X^{\perp_\xi}),
\end{align*}
where $X^{\perp_{\xi}}$ is the projection of $X$ to the orthogonal complement of $\xi$. Defining a new connection $\nabla^{g,\bar{A}}$ as follows:
\begin{equation*}
\nabla^{g,\bar{A}}_X \beta := \nabla^{g,A}_X \beta - 2i\ell \lambda_R\zeta(X)\beta
\end{equation*}
we immediately obtain:
\begin{equation*}
\nabla^{g,\bar{A}}_X \beta = -2\ell\lambda_R\iota_X(*\beta)
\end{equation*}
and thus we obtain a solution to Equation \eqref{eq:bilinearKilling3d}. Hence, applying Theorem \ref{thm:odd_CB_CPS}, we conclude the existence of a real spin-c Killing spinor with Killing number $\lambda_R$ on $(M,g)$.
\end{proof}

\begin{cor}
    Let $(M,g)$ be a three-dimensional spin-c Riemannian manifold and let $\eta\in\Gamma(S)$ be an imaginary Killing spinor with Killing number $i\lambda_I$. Then the metric $g$ is locally as follows: $$g=\dd t\otimes\dd t+e^{-4\lambda_It}g_N,$$ where $g_N$ is a Riemannian metric on the two-dimensional manifold $N$.
\end{cor}

\begin{proof}
    By Proposition \ref{prop:(3,0)_Killing} we have: $$\dd r=-2\lambda_I\theta,\qquad \nabla^g\theta=-2\lambda_Irg.$$

    These equations imply that $\Hess(r)=4\lambda_I^2rg$. Hence, the metric $g$ is locally a warped product: $$g=\dd t\otimes\dd t+F(t)^2g_N$$ for a positive function $F(t)$, see e.g.\ \cite[Thm.\ 4.3.3]{Petersen2016}. The coordinate $t$ is defined by $t:=-(2\lambda_I)^{-1}\log(r)$, which satisfies that $\dd t=\frac{1}{r}\theta$ is nowhere vanishing and $\partial_t=(\dd t)^\sharp$ is a unit vector field such that $\dd t(\partial_t)=1$. The Riemannian metric $g_N$ is defined on a suitable domain $N\subset\{p\in M\mid t(p)=t_0\in\R\}$. It is a standard exercise to show that the Hessian of a function $r(t)$ in this warped product situation is given by: \begin{equation}\label{eq:Hessian_warped_product}
        \Hess(r)=r''(t)\dd t\otimes\dd t+r'(t)F'(t)F(t)g_N.
    \end{equation}
    
    In our case $r(t)=e^{-2\lambda_It}$, so $r'(t)=-2\lambda_Ir(t)$ and $r''(t)=4\lambda_I^2r(t)$. Comparing with $\Hess(r)=4\lambda_I^2r(t)g$ we obtain that $F'(t)=-2\lambda_IF(t)$, thus $F(t)=e^{-2\lambda_It}$, where we have absorbed the constant into $g_N$.
\end{proof}


\subsection{Spin-c Killing spinors in dimension four}


Let $(S,\gamma)$ be a bundle of irreducible complex Clifford modules on a four-dimensional spin-c Riemannian manifold $(M,g)$. Assume that $(S,\gamma)$ is equipped with an admissible Hermitian pairing $\scS$ of positive adjoint type. Set $\kappa=1$ for simplicity. By the global version of Theorem \ref{thm:even_Hermitian_forms}, a complex exterior form $\widehat{\alpha}\in\Omega_\C(M)$ is the Hermitian square of a nowhere vanishing irreducible complex spinor $\eta\in\Gamma(S)$ if and only if: $$\widehat{\alpha}\diamond\widehat{\alpha}=4\widehat{\alpha}^{(0)}\widehat{\alpha},\qquad \widehat{\alpha}\diamond\beta\diamond\widehat{\alpha}=4(\widehat{\alpha}\diamond\beta)^{(0)}\widehat{\alpha},\qquad \tau(\widehat{\alpha})=\overline{\widehat{\alpha}}$$ for a complex exterior form $\beta\in\Omega_\C(M)$ satisfying $(\widehat{\alpha}\diamond\beta)^{(0)}\neq0$.

\begin{prop}\label{prop:square_(4,0)}
    Let $(M,g)$ be a four-dimensional spin-c Riemannian manifold. A complex exterior form $\widehat{\alpha}\in\Omega_\C(M)$ is the Hermitian square of a nowhere vanishing irreducible complex spinor $\eta\in\Gamma(S)$ if and only if: $$\widehat{\alpha}=r+\theta+i\omega+i*\vartheta+f\nu,$$ where $r\in C^\infty(M)$ is nowhere vanishing, $f\in C^\infty(M)$, $\theta,\vartheta\in\Omega^1(M)$, and $\omega\in\Omega^2(M)$ satisfying: \begin{gather*}
        \escal{\theta,\theta}+\escal{\omega,\omega}+\escal{\vartheta,\vartheta}+f^2=3r^2,\qquad *(\omega\wedge\vartheta)=r\theta,\\
        *(\theta\wedge\vartheta)+f*\omega=-r\omega,\qquad \theta\wedge\omega=r*\vartheta,\qquad \omega\wedge\omega=-2rf\nu.
    \end{gather*}
\end{prop}

\begin{proof}
    Let $\widehat{\alpha}=\sum_{k=0}^4\widehat{\alpha}^{(k)}$, where $\widehat{\alpha}^{(k)}\in\Omega^k_\C(M)$. The linear equation $\tau(\widehat{\alpha})=\overline{\widehat{\alpha}}$ implies that $\widehat{\alpha}^{(0)}$, $\widehat{\alpha}^{(1)}$, and $\widehat{\alpha}^{(4)}$ are real; and $\widehat{\alpha}^{(2)}$ and $\widehat{\alpha}^{(3)}$ are imaginary. Set $r:=\widehat{\alpha}^{(0)}\in C^\infty(M)$, $\theta:=\widehat{\alpha}^{(1)}\in\Omega^1(M)$, $i\omega:=\widehat{\alpha}^{(2)}$ for $\omega\in\Omega^2(M)$, and $i*\vartheta:=\widehat{\alpha}^{(3)}$ for $\vartheta\in\Omega^1(M)$. The term $\widehat{\alpha}^{(4)}$ is a multiple of the volume form $\nu$, that is, $\widehat{\alpha}^{(4)}=f\nu$ for some $f\in C^\infty(M)$. Hence $\widehat{\alpha}=r+\theta+i\omega+i*\vartheta+f\nu$. Note that $*\vartheta=\vartheta\diamond\nu$. A computation gives us: \begin{align*}
        \widehat{\alpha}\diamond\widehat{\alpha}&=(r+\theta+i\omega+i\vartheta\diamond\nu+f\nu)\diamond(r+\theta+i\omega+i\vartheta\diamond\nu+f\nu)\\
        &=r^2+2r\theta+2ir\omega+2ir\vartheta\diamond\nu+2rf\nu\\
        &\quad+\escal{\theta,\theta}+i(\theta\diamond\omega+\omega\diamond\theta)+i(\theta\diamond\vartheta-\vartheta\diamond\theta)\diamond\nu\\
        &\quad-\omega\diamond\omega-(\omega\diamond\vartheta+\vartheta\diamond\omega)\diamond\nu+2if\omega\diamond\nu+\escal{\vartheta,\vartheta}+f^2\\
        &=r^2+\escal{\theta,\theta}+\escal{\omega,\omega}+\escal{\vartheta,\vartheta}+f^2\\
        &\quad+2r\theta+2*(\omega\wedge\vartheta)+2ir\omega-2i*(\theta\wedge\vartheta)-2if*\omega\\
        &\quad+2ir*\vartheta+2i\theta\wedge\omega+2rf\nu-\omega\wedge\omega,
    \end{align*} where we have used that $\nu\diamond\alpha=\pi(\alpha)\diamond\nu$ for all $\alpha\in\Omega_\C(M)$ and $\nu\diamond\nu=1$. Therefore, the quadratic equation $\widehat{\alpha}\diamond\widehat{\alpha}=4\widehat{\alpha}^{(0)}\widehat{\alpha}$ becomes the system of equations of the statement. Since $\scS$ is a positive-definite admissible Hermitian pairing and $\eta\in\Gamma(S)$ is nowhere vanishing, we have that $r=\frac{1}{4}\scS(\eta,\eta)$ is nowhere vanishing as well, thus it suffices to take $\beta=1$ in the quadratic equation $\widehat{\alpha}\diamond\beta\diamond\widehat{\alpha}=4(\widehat{\alpha}\diamond\beta)^{(0)}\widehat{\alpha}$ to conclude.
\end{proof}

\begin{remark}
    Since $r$ is nowhere vanishing, from the equations of Proposition \ref{prop:square_(4,0)} we conclude that $\omega$ is nowhere vanishing too. Moreover, if $\theta=0$, then $\vartheta=0$ and $f*\omega=-r\omega$, thus $f$ is also nowhere vanishing. Hence, using that $*^2=\Id$ on two-forms, we obtain that $f^2=r^2$, i.e.\ $f=-\mu r$ for some $\mu\in\Z_2$. In this situation $\widehat{\alpha}=r+i\omega-\mu r\nu$ satisfying $*\omega=\mu\omega$ and $\escal{\omega,\omega}=2r^2$ is the Hermitian square of a \emph{chiral} spinor with chirality $\mu\in\Z_2$, see \cite[Cor.\ 5.7]{Algebraic_Complex_Square_2025}. The same conclusion holds if $\vartheta=0$.
\end{remark}

\begin{lemma}\label{lemma:square_(4,0)_scalar_part}
    Let $\widehat{\alpha}=r+\theta+i\omega+i*\vartheta+f\nu$ be the Hermitian square of $\eta\in\Gamma(S)$. Then: $$\escal{\vartheta,\theta}=0,\qquad \escal{\vartheta,\vartheta}=\escal{\theta,\theta}=r^2-f^2,\qquad \escal{\omega,\omega}=r^2+f^2.$$
\end{lemma}

\begin{proof}
    From Proposition \ref{prop:square_(4,0)} we have $r*\theta=-\vartheta\wedge\omega$ and $r*\vartheta=\theta\wedge\omega$. Hence $r\escal{\theta,\vartheta}\nu=\theta\wedge(r*\vartheta)=0$, which implies $\escal{\vartheta,\theta}=0$ since $r$ is nowhere vanishing. Similarly: $$r\escal{\vartheta,\vartheta}\nu=\vartheta\wedge(r*\vartheta)=\vartheta\wedge\theta\wedge\omega=r\theta\wedge*\theta=r\escal{\theta,\theta}\nu,$$ which implies $\escal{\vartheta,\vartheta}=\escal{\theta,\theta}$. Recall that for $\sigma_1,\sigma_2\in\Omega^2(M)$ we have $(*\sigma_1)\wedge(*\sigma_2)=\sigma_1\wedge\sigma_2$. Using this, $r\omega=-f*\omega-*(\theta\wedge\vartheta)$, and $\omega\wedge\omega=-2rf\nu$ we obtain: \begin{align*}
        r\escal{\omega,\omega}\nu&=r\omega\wedge*\omega=-f(*\omega)\wedge(*\omega)-*(\theta\wedge\vartheta)\wedge(*\omega)\\
        &=-f\omega\wedge\omega-\theta\wedge\vartheta\wedge\omega=2rf^2\nu+r\escal{\theta,\theta}\nu,
    \end{align*} which implies $\escal{\omega,\omega}=2f^2+\escal{\theta,\theta}$. Finally, plugging all these relations in the degree zero equation of Proposition \ref{prop:square_(4,0)} yields $\escal{\theta,\theta}=r^2-f^2$, which in turn gives $\escal{\omega,\omega}=r^2+f^2$.
\end{proof}

\begin{prop}\label{prop:(4,0)_Killing}
    Let $(M,g)$ be a four-dimensional spin-c Riemannian manifold and let $\eta\in\Gamma(S)$ be a Killing spinor with Killing number $\lambda=\lambda_R+i\lambda_I$. Then: $$\dd r=-2\lambda_I\theta,\qquad \dd f=-2\lambda_R\vartheta,\qquad \nabla^g\theta=-2\lambda_Irg-2\lambda_R\omega,\qquad \nabla^g\vartheta=2\lambda_I*\omega+2\lambda_Rfg,$$ and: $$\nabla^g_X\omega=2\lambda_RX^\flat\wedge\theta-2\lambda_I*(\vartheta\wedge X^\flat)$$ for all $X\in\Gamma(TM)$, where $\widehat{\alpha}=r+\theta+i\omega+i*\vartheta+f\nu$ is the Hermitian square of $\eta$.
\end{prop}

\begin{proof}
    In the case of a Killing spinor, we have that $\cA_X=i\lambda\Psi_\gamma(X^\flat)\in\Gamma(\End(S))$, so its symbol is just $\fra_X=i\lambda X^\flat\in\Omega^1_\C(M)$. We compute: \begin{align*}
        \fra_X\diamond\widehat{\alpha}&=i\lambda X^\flat\diamond(r+\theta+i\omega+i*\vartheta+f\nu)\\
        &=i\lambda(rX^\flat+X^\flat\wedge\theta+\theta(X)+iX^\flat\wedge\omega+i\iota_X\omega+i\vartheta(X)\nu+i*(\vartheta\wedge X^\flat)+f*X^\flat),\\
        \widehat{\alpha}\diamond\tau(\overline{\fra}_X)&=(r+\theta+i\omega+i*\vartheta+f\nu)\diamond(-i\bar\lambda X^\flat)\\
        &=-i\bar{\lambda}(rX^\flat+\theta\wedge X^\flat+\theta(X)+i\omega\wedge X^\flat-i\iota_X\omega+i*(\vartheta\wedge X^\flat)-i\vartheta(X)\nu-f*X^\flat).
    \end{align*}
    
    By Corollary \ref{cor:even_constrained_parallel_spinors}, if the spinor $\eta\in\Gamma(S)$ satisfies $\nabla^{g,A}_X\eta=\cA_X(\eta)$, then $\nabla^g_X\widehat{\alpha}=\fra_X\diamond\widehat{\alpha}+\widehat{\alpha}\diamond\tau(\overline{\fra}_X)$. Separating by degrees this equation we obtain: \begin{gather*}
        \nabla^g_Xr=-2\lambda_I\theta(X),\qquad \nabla^g_X\theta=-2\lambda_IrX^\flat-2\lambda_R\iota_X\omega,\qquad \nabla^g_X\omega=2\lambda_RX^\flat\wedge\theta-2\lambda_I*(\vartheta\wedge X^\flat),\\
        \nabla^g_X(*\vartheta)=-2\lambda_IX^\flat\wedge\omega+2\lambda_Rf*X^\flat,\qquad \nabla^g_X(f\nu)=-2\lambda_R\vartheta(X)\nu
    \end{gather*} for all $X\in\Gamma(TM)$. From the first equation we obtain $\dd r=-2\lambda_I\theta$, from the second equation we obtain $\nabla^g\theta=-2\lambda_Irg-2\lambda_R\omega$. From the fourth one we get $\nabla^g\vartheta=2\lambda_I*\omega+2\lambda_Rfg$ since the Levi-Civita connection $\nabla^g$ commutes with the Hodge star operator and $*^2=-\Id$ on one-forms, and from the last one we get $\dd f=-2\lambda_R\vartheta$ since $\nabla^g\nu=0$.
\end{proof}

Using $\dd r=-2\lambda_I\theta$ and $\nabla^g\theta=-2\lambda_Irg-2\lambda_R\omega$ we compute the Hessian of the function $r$: $$\Hess(r)=\nabla^g\dd r=4\lambda_I^2rg+4\lambda_R\lambda_I\omega.$$

Since $\Hess(r)$ is a symmetric tensor and $\omega$ is nowhere vanishing, we have that $\lambda_R\lambda_I=0$, as expected.

\begin{cor}
    Let $(M,g)$ be a four-dimensional spin-c Riemannian manifold and let $\eta\in\Gamma(S)$ be a real Killing spinor with Killing number $\lambda_R$. Then the metric $g$ is locally as follows: $$g=\dd t\otimes\dd t+\sin^2(2\lambda_Rt)g_N$$ around any point where $\dd f\neq0$. Moreover, if $\dd f_p=0$ for $p\in M$, then $(M,g)$ is locally isometric to the standard sphere of constant positive sectional curvature $\mathrm{sec}^g\equiv 4\lambda_R^2$ on some neighborhood of $p$.
\end{cor}

\begin{proof}
    By Proposition \ref{prop:(4,0)_Killing} we have: $$\dd r=0,\qquad \dd f=-2\lambda_R\vartheta,\qquad \nabla^g\theta=-2\lambda_R\omega,\qquad \nabla^g\vartheta=2\lambda_Rfg,\qquad \nabla^g_X\omega=2\lambda_RX^\flat\wedge\theta.$$

    From $\dd f=-2\lambda_R\vartheta$ and $\nabla^g\vartheta=2\lambda_Rfg$ we get $\Hess(f)=-4\lambda_R^2fg$. Let $U:=\{p\in M\mid \dd f_p\neq0\}$. The function $f$ has no critical points on $U$, hence $\vartheta$ is nowhere vanishing on $U$. From $\Hess(f)=-4\lambda_R^2fg$ we get that the metric $g$ is locally a warped product: $$g=\dd t\otimes\dd t+F(t)^2g_N$$ for a positive function $F(t)$. The coordinate $t$ is defined by: $$\dd t:=\tfrac{1}{\sqrt{r^2-f^2}}\vartheta.$$

    Integrating $\dd f=-2\lambda_R\vartheta=-2\lambda_R\sqrt{r^2-f^2}\dd t$ we obtain $\arccos(\frac{f}{r})=2\lambda_R(t-t_0)$ for some $t_0\in\R$. Then, after shifting $t$ to absorb the factor $t_0$, we obtain $f(t)=r\cos(2\lambda_Rt)$ and $\smash{\sqrt{r^2-f(t)^2}}=r\sin(2\lambda_Rt)$. Comparing $\Hess(f)=-4\lambda_R^2fg$ with the expression \eqref{eq:Hessian_warped_product}, we obtain: $$f''(t)=-4\lambda_R^2f(t),\qquad f'(t)F'(t)=-4\lambda_R^2f(t)F(t),$$ which is solved by $F(t)=\sin(2\lambda_Rt)$, where we have absorbed the constant into $g_N$.\medskip

    For every point in the complement $U^c=\{p\in M\mid \dd f_p=0\}$ we have $f_p\neq0$. Indeed, if $\dd f_p=0$, then $\vartheta_p=0$, which implies that $f^2_p=r^2_p$ by Lemma \ref{lemma:square_(4,0)_scalar_part}, thus $f_p$ is non-zero since $r$ is nowhere vanishing. Hence $g=\dd t\otimes\dd t+G(t)^2g_{S^3}$ on some open neighborhood of $p\in M$. Comparing $\Hess(f)=-4\lambda_R^2fg$ with the expression \eqref{eq:Hessian_warped_product}, we obtain: $$f''(t)=-4\lambda_R^2f(t),\qquad f'(t)G'(t)=-4\lambda_R^2f(t)G(t),$$ with the initial conditions $f(0)=f_0>0$, $f'(0)=0$, $G(0)=0$, and $G'(0)=1$. The first equation above gives $f(t)=f_0\cos(2\lambda_Rt)$ and the second equation gives $G(t)=\frac{1}{2\lambda_R}\sin(2\lambda_Rt)$. Then: $$g=\dd t\otimes\dd t+\frac{1}{4\lambda_R^2}\sin^2(2\lambda_Rt)g_{S^3}.$$
    
    This is precisely the standard metric of the four-dimensional sphere, which has constant positive sectional curvature $\mathrm{sec}^g\equiv4\lambda_R^2$.
\end{proof}

Let $q_\eta:=r^2-\escal{\theta,\theta}$. By Lemma \ref{lemma:square_(4,0)_scalar_part} we have $q_\eta=f^2$. If we assume that $\eta\in\Gamma(S)$ is an imaginary Killing spinor, then $\dd f=0$ by Proposition \ref{prop:(4,0)_Killing}. This implies that $q_\eta$ is constant. We distinguish two cases depending on whether $q_\eta$ is zero or not.

\begin{cor}
    Let $(M,g)$ be a four-dimensional spin-c Riemannian manifold and let $\eta\in\Gamma(S)$ be an imaginary Killing spinor with Killing number $i\lambda_I$ and $q_\eta=0$. Then the metric $g$ is locally as follows: $$g=\dd t\otimes\dd t+e^{-4\lambda_It}g_N,$$ where $g_N$ is a Riemannian metric on the three-dimensional manifold $N$.
\end{cor}

\begin{proof}
    By Proposition \ref{prop:(4,0)_Killing} we have: $$\dd r=-2\lambda_I\theta,\qquad \dd f=0,\qquad \nabla^g\theta=-2\lambda_Irg,\qquad \nabla^g\vartheta=2\lambda_I*\omega,\qquad \nabla^g_X\omega=-2\lambda_I*(\vartheta\wedge X^\flat).$$

    Note that in this case we have $\escal{\theta,\theta}=r^2$ by Lemma \ref{lemma:square_(4,0)_scalar_part}, which implies that $\theta$ is nowhere vanishing, as well as $\dd r$. We have seen that $\Hess(r)=4\lambda_I^2rg$. Hence, the metric $g$ is locally a warped product: $$g=\dd t\otimes\dd t+F(t)^2g_N$$ for a positive function $F(t)$. The coordinate $t$ is defined by $t:=-(2\lambda_I)^{-1}\log(r)$, which satisfies that $\dd t=\frac{1}{r}\theta$ is nowhere vanishing and $\partial_t=(\dd t)^\sharp$ is a unit vector field such that $\dd t(\partial_t)=1$. Comparing $\Hess(r)=4\lambda_I^2r(t)g$ with the expression \eqref{eq:Hessian_warped_product}, we obtain that $F(t)=e^{-2\lambda_It}$, where we have absorbed the constant into $g_N$.
\end{proof}

\begin{cor}
    Let $(M,g)$ be a four-dimensional spin-c Riemannian manifold and let $\eta\in\Gamma(S)$ be an imaginary Killing spinor with Killing number $i\lambda_I$ and $q_\eta>0$. Then the metric $g$ is locally as follows: $$g=\dd t\otimes\dd t+\sinh^2(2\lambda_It)g_N$$ around any point where $\dd r\neq0$. Moreover, if $\dd r_p=0$ for $p\in M$, then $(M,g)$ is locally isometric to the hyperbolic space of constant negative sectional curvature $\mathrm{sec}^g\equiv-4\lambda_I^2$ on some neighborhood of $p$.
\end{cor}

\begin{proof}
    By Proposition \ref{prop:(4,0)_Killing} we have: $$\dd r=-2\lambda_I\theta,\qquad \dd f=0,\qquad \nabla^g\theta=-2\lambda_Irg,\qquad \nabla^g\vartheta=2\lambda_I*\omega,\qquad \nabla^g_X\omega=-2\lambda_I*(\vartheta\wedge X^\flat).$$
    
    Let $U:=\{p\in M\mid\dd r_p\neq0\}$. The function $r$ has no critical points on $U$, hence $\theta$ is nowhere vanishing on $U$. From $\Hess(r)=4\lambda_I^2rg$ we get that the metric $g$ is locally a warped product: $$g=\dd t\otimes\dd t+F(t)^2g_N$$ for a positive function $F(t)$. The coordinate $t$ is defined by: $$\dd t:=\tfrac{1}{\sqrt{r^2-f^2}}\theta.$$
    
    Integrating $\dd r=-2\lambda_I\theta=-2\lambda_I\sqrt{r^2-f^2}\dd t$ we obtain $\operatorname{arcosh}(\frac{r}{f})=-2\lambda_I(t-t_0)$ for some $t_0\in\R$. Then, after shifting $t$ to absorb the factor $t_0$, we obtain $r(t)=f\cosh(2\lambda_It)$ and $\smash{\sqrt{r(t)^2-f^2}}=-f\sinh(2\lambda_It)$. Comparing $\Hess(r)=4\lambda_I^2rg$ with the expression \eqref{eq:Hessian_warped_product}, we obtain: $$r''(t)=4\lambda_I^2r(t),\qquad r'(t)F'(t)=4\lambda_I^2r(t)F(t),$$ which is solved by $F(t)=-\sinh(2\lambda_It)$, where we have absorbed the constant into $g_N$.\medskip

    For every point in the complement $U^c=\{p\in M\mid\dd r_p=0\}$ we have $\dd r_p=0$ and $4\lambda_I^2r_p\neq0$, hence $g=\dd t\otimes\dd t+G(t)^2g_{S^3}$ on some open neighborhood of $p\in M$, where $g_{S^3}$ is the standard metric on the unit sphere $S^3\subset\R^4$, see e.g.\ \cite[Thm.\ 4.3.3]{Petersen2016}. Comparing $\Hess(r)=4\lambda_I^2rg$ with the expression \eqref{eq:Hessian_warped_product}, we obtain: $$r''(t)=4\lambda_I^2r(t),\qquad r'(t)G'(t)=4\lambda_I^2r(t)G(t),$$ with the initial conditions $r(0)=r_0>0$, $r'(0)=0$, $G(0)=0$, and $G'(0)=1$. The first equation above gives $r(t)=r_0\cosh(2\lambda_It)$ and the second equation gives $G(t)=\frac{1}{2\lambda_I}\sinh(2\lambda_It)$. Then: $$g=\dd t\otimes\dd t+\frac{1}{4\lambda_I^2}\sinh^2(2\lambda_It)g_{S^3}.$$
    
    This is precisely the metric of the four-dimensional hyperbolic space, which has constant negative sectional curvature $\mathrm{sec}^g\equiv-4\lambda_I^2$.
\end{proof}


\section{Supersymmetric Brinkmann four-manifolds}
\label{sec:sugraspinors}


In this section, we consider a natural spinorial differential system that arises as the supersymmetry conditions for the celebrated Freedman's minimal gauged four-dimensional supergravity, or Freedman's gauged supergravity for short \cite{Freedman_1977}. Let $M$ be an oriented four-manifold and let $P$ be a principal $\U(1)$-bundle defined on $M$. A \emph{supersymmetric configuration} of Freedman's gauged supergravity on $(M,P)$ is, by definition, a tuple $(g,A,\eta)$ consisting of a Lorentzian metric $g$ of \emph{mostly plus} signature on $M$, a connection $A$ on $P$, and an irreducible complex chiral spinor $\eta\in\Gamma(S^\mu)$ with chirality $\mu\in\Z_2$ satisfying the following equations:
\begin{equation}
\label{eq:susytrans}
\nabla^{g,A} \eta = 0\, , \qquad (F^{\mu}_A -  \lambda) \cdot \eta = 0,
\end{equation}

\noindent
where:
\begin{itemize}
\item $S = S^{+}\oplus S^{-}$ is a bundle of irreducible complex spinors on $(M,g)$, which is associated to a given $\Spin^c_o(3,1)$-structure $Q$, whose characteristic $\U(1)$-bundle is $P$.
\item $\nabla^{g,A}$ is the canonical lift of the Levi-Civita connection of $g$ and $A$ to $P$.
\item $F_A\in \Omega^2(M)$ is the curvature of $A$, understood as a real two-form on $M$.
\item  $\mu \in \mathbb{Z}_2$ and $F_A^{\mu} \in \Omega^2_\C(M)$ is the following complex combination:
\begin{equation*}
F_A^{\mu} := \frac{1}{2}(F_A + i \mu\ast_g F_A).
\end{equation*}

\noindent
Note that we have $\ast_g F_A^{\mu} = - i\mu F_A^{\mu}$ since $\ast_g^2 = -\Id$ on two-forms. Therefore, $F_A^{+}$ is complex anti-self-dual, whereas $F_A^{-}$ is complex self-dual.
\item $\lambda\in \mathbb{C}^{\ast}$ is a non-zero complex constant.
\end{itemize}

\noindent
The first equation in \eqref{eq:susytrans} is typically called the \emph{gravitino equation}, since it arises from imposing the vanishing of the infinitesimal supersymmetry transformation of the gravitino on a bosonic background, whereas the second equation in \eqref{eq:susytrans} is commonly called the \emph{gaugino equation}, since it arises from imposing the vanishing of the infinitesimal supersymmetry transformation of the gaugino on a bosonic background. The gravitino equation implies, therefore, the existence of a spin-c parallel spinor on a Lorentzian four-manifold. The study of such parallel spinors in pseudo-Riemannian signature has been pioneered in \cite{Ikemakhen06,Ikemakhen07}, where the holonomy groups of all simply connected irreducible non-locally symmetric pseudo-Riemannian spin-c manifolds that admit parallel spinors are described, obtaining a classification in the Lorentzian symmetric case. In contrast to \emph{opere citato}, the $\U(1)$-connection occurring in the parallelicity condition of the gravitino equation has to satisfy its own coupled equation.\medskip

Let $\Lor_o(M)$ denote the set of time-orientable, strongly spin-c, Lorentzian metrics on $M$, and let $C(P)$ denote the affine space of connections on $P$. It is convenient to refer to supersymmetric solutions simply as pairs $(g,A)\in \Lor_o(M)\times C(P)$, meaning that $(g,A)$ is supersymmetric for some choice of irreducible complex spinor bundle $S$ and section $\eta\in \Gamma(S^{\mu})$. A supersymmetric configuration $(g,A)$ is a \emph{supersymmetric solution} of Freedman's gauged supergravity on $(M,P)$ if $(g,A)$ satisfies the following differential system:
\begin{equation}
\label{eq:Freedmaneqs}
\mathrm{Ric}^g - \frac{1}{2} s_g g  = \Lambda g + \mathfrak{e}^2\, F_A \circ_g F_A - \frac{\mathfrak{e}^2}{2}  \vert F_A \vert^2_g g\, , \qquad \dd \ast_g F_A =0 
\end{equation}

\noindent
for a pair of non-zero real constants $\Lambda, \mathfrak{e}\in\mathbb{R}^{\ast}$. These are precisely the equations of the Lorentzian-signature Einstein-Maxwell system with a possibly non-vanishing cosmological constant $\Lambda$ and coupling constant $\mathfrak{e}$.

\begin{remark}
The proper supersymmetry conditions and equations of motion of Freedman's gauged supergravity fix $\mathfrak{e}$ and $\Lambda$ in terms of $\lambda$. Here, we have opted for greater initial flexibility in order to encounter the relation between $\mathfrak{e}$ and $\lambda$ as a consistency condition derived from imposing that a given supersymmetric configuration is a supersymmetric solution. Note that, crucially, if $\eta\in \Gamma(S^{\mu})$, then $F_A^{-\mu} \cdot \eta = 0$ identically on $M$.     
\end{remark}

\begin{lemma}
\label{lemma:squareslorentzian}
Let $\eta \in \Gamma(S^{\mu})$ be an irreducible complex chiral spinor with chirality $\mu\in\Z_2$ on an orientable Lorentzian four-manifold $(M,g)$. Then:
\begin{itemize}
    \item A complex exterior form $\widehat{\alpha} \in \Omega_\C(M)$ is the Hermitian square of $\eta$ with $\kappa=1$ if and only if there exists an isotropic real one-form $\widehat{u}\in \Omega^1(M)$ such that:
    \begin{equation*}
    \widehat{\alpha} = \widehat{u} + i\mu \ast \widehat{u}.
    \end{equation*}

    \noindent
    The one-form $\widehat{u}$ is the so-called \emph{Dirac current} of $\eta$.
    \item A $\cL$-valued complex exterior form $\alpha \in \Omega_\C(M,\cL)$ is the complex-bilinear square of $\eta$ if and only if it is a decomposable complex $\mu$-self-dual two-form valued in $\cL$. That is, around every point in $M$ there exists a local section $\mathfrak{l}$ of $\cL$ such that:
    \begin{equation*}
    \alpha = (\theta_1\wedge \theta_2) \otimes \mathfrak{l} \, , \qquad i\ast (\theta_1\wedge \theta_2 ) = \mu \theta_1\wedge \theta_2
    \end{equation*} 

    \noindent
    for a pair of local isotropic and orthogonal complex one-forms $\theta_1 , \theta_2 \in \Omega^1_\C(M)$.
\end{itemize}
\end{lemma}

\begin{proof}
The Hermitian square of $\eta$ is computed in \cite[Prop.\ 3.26]{Algebraic_Complex_Square_2025} using an admissible Hermitian pairing of positive adjoint type. We proceed to compute the complex-bilinear square of $\eta$ pointwise, thus $S^\mu_p\cong\Sigma^\mu$ and $(\wedge T^*_\C M\otimes\cL)_p\cong\wedge V^*_\C$ for all $p\in M$. We equip $\Sigma$ with a skew-symmetric admissible complex-bilinear pairing of positive adjoint type, see \cite[Prop.\ 3.18 and Table 3]{Algebraic_Complex_Square_2025}. Applying Theorem \ref{thm:bilinearsquare} to four Lorentzian dimensions, it follows that $\alpha\in\wedge V^*_\C$ is the complex-bilinear square of $\eta \in \Sigma^\mu$ if and only if:
\begin{eqnarray}
\label{eq:algebraic31}
\alpha\diamond\alpha = 4 \alpha^{(0)}\alpha,\quad \alpha \diamond \beta \diamond \alpha = 4(\alpha\diamond\beta)^{(0)}\alpha,\quad \tau(\alpha) = -\alpha , \quad  \ast (\pi\circ\tau)(\alpha) = i\mu \alpha
\end{eqnarray}

for a complex exterior form $\beta\in\wedge V^*_\C$ satisfying $(\alpha\diamond\beta)^{(0)} \neq 0$. The linear equations in \eqref{eq:algebraic31} are solved by setting $\alpha = \omega$ for a complex two-form $\omega\in\wedge^2V^*_\C$ satisfying:
\begin{equation*}
i\ast \omega =  \mu \omega.
\end{equation*}

The first equation in \eqref{eq:algebraic31} becomes:
\begin{equation*}
\alpha\diamond\alpha=\omega\diamond\omega=\omega\wedge\omega-\escal{\omega,\omega}=0,
\end{equation*}

which in turn is equivalent to $\alpha = \theta_1\wedge \theta_2$ and $\escal{\theta_1,\theta_1}\escal{\theta_2,\theta_2}=\escal{\theta_1,\theta_2}^2$ for a pair of complex one-forms $\theta_1 , \theta_2 \in V^*_\C$. Then the equation $i*\omega=\mu\omega$ translates into $i*(\theta_1\wedge\theta_2)=\mu\theta_1\wedge\theta_2$. Plugging $\theta_1^{\smash{\sharp}}$ and $\theta_2^{\smash{\sharp}}$ into this equation, we conclude that the one-forms $\theta_1$ and $\theta_2$ are isotropic and mutually orthogonal. We complete $\theta_1$ and $\theta_2$ into a basis $\{\theta_1,\theta_2,\vartheta_1,\vartheta_2\}$ of $V^*_\C$, where $\vartheta_1$ and $\vartheta_2$ are isotropic and orthogonal, and furthermore are \emph{conjugate} to $\theta_1$ and $\theta_2$, that is:
\begin{equation*}
\escal{\theta_1,\vartheta_1}=\escal{\theta_2,\vartheta_2}=1,\qquad \escal{\theta_1,\vartheta_2}=\escal{\theta_2,\vartheta_1}=0.
\end{equation*}

Setting $\beta = \vartheta_1\wedge\vartheta_2$, the second equation \eqref{eq:algebraic31} becomes:
\begin{equation*}
(\theta_1\wedge\theta_2)\diamond(\vartheta_1\wedge\vartheta_2)\diamond(\theta_1\wedge\theta_2)=-4\theta_1\wedge\theta_2,
\end{equation*}

\noindent
where we have used that $((\theta_1\wedge\theta_2)\diamond(\vartheta_1\wedge\vartheta_2))^{(0)}=-1$. Expanding the above equation yields an identity; hence, we conclude.
\end{proof}

\noindent
The geometry of decomposable complex self-dual isotropic two-forms in four Lorentzian dimensions, as they occur in the previous lemma, has been extensively studied within the Newman-Penrose formalism. We use this characterization to refine the description of the complex-bilinear square of $\eta$.

\begin{prop}
A $\cL$-valued complex exterior form $\alpha \in \Omega_\C(M,\cL)$ is the complex-bilinear square of $\eta\in\Gamma(S^\mu)$ if and only if there exists a nowhere vanishing isotropic one-form $u\in \Omega^1(M)$ such that locally around every point in $M$ we have:
\begin{equation*}
\alpha =  u\wedge(l+i\mu n)\otimes \mathfrak{l},
\end{equation*} 

\noindent
where $\mathfrak{l}$ is a local section of $\cL$, and $l, n$ are unit-norm orthogonal local one-forms orthogonal to $u$.    
\end{prop}

\begin{proof}
By Lemma \ref{lemma:squareslorentzian}, there exists a local section $\mathfrak{l}$ of $\cL$ and local complex one-forms, isotropic and orthogonal, such that $\alpha = (\theta_1\wedge \theta_2) \otimes \mathfrak{l}$ and $i\ast (\theta_1\wedge \theta_2 ) = \mu \theta_1\wedge \theta_2$. Write:
\begin{equation*}
\theta_1\wedge \theta_2  = \omega_R + i \omega_I
\end{equation*}

\noindent
for real two-forms $\omega_R , \omega_I \in \Omega^2(M)$. The complex self-duality condition translates into:
\begin{equation*}
\ast\omega_R = \mu \omega_I
\end{equation*}

\noindent
and therefore $\theta_1\wedge\theta_2 = \omega_R + i\mu\ast\omega_R$. From $\theta_1\wedge\theta_2\wedge\theta_1\wedge\theta_2=0$ we obtain $\omega_R\wedge \ast\omega_R = 0$, and from the complex self-duality condition we automatically have $(\theta_1\wedge\theta_2)\wedge(\bar{\theta}_1\wedge\bar{\theta}_2)=0$, which yields $\omega_R \wedge \omega_R = 0$. These equations are solved by: 
\begin{eqnarray*}
\theta_1\wedge \theta_2  = u\wedge l + i\mu \ast(u\wedge l)
\end{eqnarray*}

\noindent
for local real one-forms $u,l$ such that $u$ is isotropic and $l$ is of unit norm and orthogonal to $u$. We can introduce another, unit-norm, local one-form $n$ that is orthogonal to both $u$ and $l$ and satisfies:
\begin{equation*}
\ast(u\wedge l) = u\wedge n \, , \qquad \ast(u\wedge n) = - u\wedge l.
\end{equation*}

\noindent
Hence:
\begin{equation*}
\theta_1\wedge \theta_2  = u\wedge (l + i\mu n)
\end{equation*}

\noindent
and thus we can identify $\theta_1 = u$ and $\theta_2 = l+i\mu n$.\medskip
 
Given a good open cover $\{U_a\}$ of $M$ labeled by an index $a$, we can perform the previous construction at every $U_a$, obtaining tuples $\{u_a,l_a,n_a,\mathfrak{l}_a\}$ such that:
\begin{equation*}
u_a \wedge (l_a + i\mu n_a) \otimes \mathfrak{l}_a = u_b \wedge (l_b + i\mu n_b) \otimes \mathfrak{l}_b
\end{equation*}
 
\noindent
on non-trivial overlaps $U_{ab}:=U_a\cap U_b \neq \emptyset$. Hence:
\begin{equation}
\label{eq:transitionL}
u_a = r_{ab} u_b\, , \qquad l_a + f^l_{ab} u_a + i \mu (n_a + f^n_{ab} u_a) = e^{i\theta_{ab}} (l_b + i \mu n_b)
\end{equation}

\noindent
for families of functions defined on overlaps:
\begin{equation*}
r_{ab}\in C^{\infty}(U_{ab}, \mathbb{R}_{>0}),\qquad f^l_{ab},f^n_{ab},\theta_{ab}\in C^{\infty}(U_{ab}, \mathbb{R})
\end{equation*}

\noindent
such that $\mathfrak{l}_{a} = r^{-1}_{ab} e^{i\theta_{ab}} \mathfrak{l}_b$. That is, $r^{-1}_{ab} e^{i\theta_{ab}}\colon U_{ab} \to \mathbb{C}^{\ast}$ are transition functions for $\cL$, which therefore satisfy the corresponding cocycle condition. Since the sheaf of smooth positive functions on $M$ is soft, its cohomology vanishes, implying that the family of isotropic real one-forms $u_a$ assembles into a well-defined isotropic one-form $u\in\Omega^1(M)$.  
\end{proof}

\begin{remark}
The decomposition of the complex-bilinear square of $\eta$ given in the previous proposition provides an alternative perspective on the notion of \emph{isotropic parallelism} introduced in \cite{ShahbaziThesis} for an irreducible real spinor on a four-dimensional Lorentzian manifold.
\end{remark}

\noindent
The Hermitian and complex-bilinear squares of the same irreducible chiral spinor $\eta\in \Gamma(S^{\mu})$ are not independent. Instead, it follows from \cite[Prop.\ 3.30]{Algebraic_Complex_Square_2025} that $\widehat{u} = c\, u$ for a non-zero real constant $c$ which is unimportant for our purposes. Hence, in the following, we will only refer to $u\in \Omega^1(M)$ as the Dirac current of $\eta$. As a corollary to the previous proposition, we have the following characterization of $\cL$.

\begin{cor}
Let $(S,\gamma,\scB)$ be a complex-bilinear paired spinor bundle. Then, $S$ admits a nowhere vanishing section $\eta\in \Gamma(S)$ with Dirac current $u\in \Omega^1(M)$ only if:
\begin{eqnarray*}
\mathfrak{S}_u \xrightarrow{\cong} \cL,
\end{eqnarray*}

\noindent
that is, only if the screen bundle determined by $u$ is isomorphic to $\cL$ as rank-two real vector bundles. 
\end{cor}

\begin{remark}
We remind the reader that the \emph{screen bundle} defined by a nowhere vanishing isotropic one form $u\in \Omega^1(M)$ is the following abstract vector bundle:
\begin{equation*}
\mathfrak{S}_u := \frac{\ker(u)}{\langle \mathbb{R} u^{\sharp_g}\rangle}.
\end{equation*}
\end{remark}
 
\begin{proof}
Follows from Equation \eqref{eq:transitionL} after taking the quotient by the span of $u$, which eliminates the terms proportional to $u$ and therefore implies that the collection of local basis $(l_a,n_a)$ of the screen bundle is related on overlaps by the transition functions of $\cL$, once these are reduced to $\U(1)$. 
\end{proof}

\noindent
The Hermitian square of $\eta$ suffices to obtain convenient necessary and sufficient conditions for the existence of a solution to the gaugino equation in \eqref{eq:susytrans}.

\begin{lemma}
A pair $(M,P)$ admits a solution $(g,A)\in \Lor_o(M)\times C(P)$ to the gaugino equation for a nowhere vanishing spinor $\eta\in \Gamma(S^{\mu})$ if and only if there exists a nowhere vanishing isotropic one-form $u\in \Omega^1(M)$ such that:
\begin{equation}
\label{eq:Asugra}
F_A(u^{\sharp_g}) + \lambda_R u =  0\, , \qquad F_A\wedge u + \mu\,  \lambda_I \ast_g u = 0.
\end{equation}
\end{lemma}

\begin{proof}
By Lemma \ref{lemma:constrained_spinor_as_section}, a pair $(g,A)\in \Lor_o(M)\times C(P)$ satisfies the gaugino equation if and only if:
\begin{align*}
0 &= (F^{\mu}_A -  \lambda) \diamond (u + i \mu \ast_g u)\\
&= 2 (F^{\mu}_A \wedge u - F_A^{\mu}(u^{\sharp_g})) - (\lambda_R + i \lambda_I) u - \mu (i \lambda_R - \lambda_I) \ast_g u\\
&= F_A\wedge u + i \mu (\ast_g F_A) \wedge u - F_A(u^{\sharp_g}) - i \mu \ast_g (F_A\wedge u)- (\lambda_R + i \lambda_I) u - \mu (i \lambda_R - \lambda_I) \ast_g u,
\end{align*}

\noindent
where we have written $\lambda = \lambda_R + i\lambda_I$. Isolating by degree and real and imaginary part, this equation is equivalent to:
\begin{eqnarray*}
& F_A\wedge u + \mu\,  \lambda_I \ast_g u = 0\, , \qquad (\ast_g F_A)\wedge u - \lambda_R \ast_g u = 0\, ,\\
& F_A(u^{\sharp_g}) + \lambda_R u =  0\, , \qquad \ast_g(F_A\wedge u) + \mu\, \lambda_I u = 0\, .
\end{eqnarray*}

\noindent
The first and fourth equations are equivalent via applying the Hodge operator. Similarly, the second and third equations are equivalent via the same mechanism, upon using $\ast_g^2 = -\Id$ on two-forms.
\end{proof}

\noindent
Equations \eqref{eq:Asugra} are a natural generalization of the \emph{Lorentzian instanton equations} obtained for an equation of the form $F_A\cdot \eta = 0$ and studied briefly in \cite{ShahbaziThesis} for a real irreducible spinor. The Hermitian square of $\eta$ can be further exploited to obtain a neat necessary condition for the first equation in \eqref{eq:susytrans}.

\begin{lemma}
A pair $(M,P)$ admits a solution $(g,A)\in \Lor_o(M)\times C(P)$ to the gravitino equation for a nowhere vanishing spinor $\eta\in \Gamma(S^{\mu})$ only if its Dirac current $u\in \Omega^1(M)$ satisfies:
\begin{equation}
\label{eq:nablasugra}
\nabla^g u = 0,
\end{equation}

that is, only if its Dirac current is parallel with respect to the Levi-Civita connection on $(M,g)$. In particular, $(M,g,u)$ is a Brinkmann four-manifold. 
\end{lemma}

\begin{proof}
Follows directly from Theorem \ref{thm:even_constrained_parallel_spinors}, after setting $\cA = 0$, which is the case for the first equation in \eqref{eq:susytrans}.
\end{proof}

Equations \eqref{eq:Asugra} and \eqref{eq:nablasugra} contain all the information that can be extracted from the supersymmetry transformations \eqref{eq:susytrans} through the Hermitian square of $\eta$. In other words, the Hermitian square of $\eta$ satisfies the system of Theorem \ref{thm:even_constrained_parallel_spinors} for the supersymmetry transformations \eqref{eq:susytrans} if and only if it satisfies equations \eqref{eq:Asugra} and \eqref{eq:nablasugra}. This leads us to introduce the following weaker notion of supersymmetry.

\begin{definition}
A pair $(g,A)\in \Lor_o(M)\times C(P)$ on $(M,P)$ is a \emph{quasi-supersymmetric} configuration of Freedman's gauged supergravity if and only if there exists a nowhere vanishing isotropic one-form $u\in \Omega^1(M)$ such that $(g,A,u)$ satisfies equations  \eqref{eq:Asugra} and \eqref{eq:nablasugra}.
\end{definition}

\begin{remark}
The notion of \emph{quasi-supersymmetry}, or \emph{quasi-supersymmetric solutions}, can be applied to every supergravity with complex infinitesimal supersymmetry generators, and hence defines a weaker notion of supersymmetry which we are very intrigued to explore in the future for other, more general, supergravities. 
\end{remark}

Studying the full supersymmetry system of Freedman's gauged supergravity is a subtle problem that we will consider in a separate publication. For the remainder of this section, we consider the necessary and sufficient conditions for a quasi-supersymmetric configuration of Freedman's gauged supergravity to be a solution of the theory, focusing on a class of solutions naturally adapted to the presence of a nowhere vanishing parallel isotropic vector field.

\begin{definition}
Let $X$ be an oriented two-dimensional manifold equipped with a principal $\U(1)$-bundle $\pi \colon P\to X$, and let $\mathfrak{p} \colon \cI_\fru\times \cI_\frv \times X \to X$ be the canonical projection, where $\cI_\fru$ and $\cI_\frv$ are intervals with coordinates $\fru$ and $\frv$, respectively. A tuple $(g,A,u)$, where $(g,A) \in \Lor_o(\cI_\fru\times \cI_\frv \times X)\times C(\mathfrak{p}^{\ast}P)$ and $u\in \Omega^1(\cI_\fru\times\cI_\frv\times X)$ is isotropic on $(\cI_\fru \times\cI_\frv\times X,g)$, is \emph{standard Brinkmann} if:
\begin{equation}
\label{eq:standardBrinkmann}
g = \cH_{\fru} \dd \fru \otimes \dd \fru +  \dd \fru \odot (\dd \frv + \cA_{\fru}) + h_{\fru}\, , \qquad  u = \dd \fru\, , \qquad \partial_{\frv}A = 0,
\end{equation}

\noindent
where $\cH_{\fru} \in C^{\infty}(X)$ is a family of functions on $X$, $\cA_{\fru}$ is a family of one-forms on $X$, and $h_{\fru}$ is a family of complete Riemannian metrics on $X$, all parametrized by $\fru\in \cI_{\fru}$.   
\end{definition}

\begin{remark}
It is well-known that a Brinkmann space-time \cite{brinkmann1925einstein}, that is, a Lorentzian four-manifold equipped with a nowhere vanishing parallel vector field, is locally isometric to an open set in $\mathbb{R}^4$ equipped with a metric of the form given in \eqref{eq:standardBrinkmann}, hence the terminology introduced in the previous definition. Brinkmann space-times describe idealized gravitational waves and are therefore of enormous importance in general relativity and mathematical physics. In fact, each of the ingredients has a neat interpretation:
\begin{itemize}
\item $\mathcal{H}_{\fru}$ is the tidal potential. That is, in the geodesic equations for a test particle, the negative spatial gradient of $\mathcal{H}_{\fru}$ appears exactly as an external force pushing or pulling the particle across the transverse space $X$.

\item $\cA_{\fru}$ encodes the gravito-magnetic twist of the space-time, which induces Coriolis-like forces on the transverse spatial slices. This represents the gravito-magnetic twist, transforming the spacetime into a \emph{gyraton} \cite{Frolov:2005in}. It signifies that the source of the geometry is a beam of radiation carrying intrinsic angular momentum. For a test particle moving on the spatial slice $X$, $\mathcal{A}_{\fru}$ acts identically to a magnetic vector potential, generating a velocity-dependent Coriolis-like force that twists the particle's trajectory.

\item $(X,h_{\fru})$ is the dynamical wave-front of the wave. The coordinate $\fru$ acts as the retarded time of the wave, and the $ \fru$-constant slices define a family of parallel null hypersurfaces. Each of these null hypersurfaces represents the spacetime history of a single phase of the wave. Taking a spatial cross-section of this propagating wave isolates the wavefront itself. As $\fru$ advances, $h_{\fru}$ allows this screen to expand or shear. This dynamic geometry acts as a time-dependent mass matrix for particles, leading to phenomena such as kinetic damping as the space stretches beneath them.
\end{itemize}
\end{remark}

\noindent 
Let $(g,A,u)$ be a standard Brinkmann triple on $(\cI_u\times \cI_v\times X , \mathfrak{p}^{\ast}P)$. Then, we have the canonical identification:
\begin{equation*}
\mathfrak{p}^{\ast}P = \cI_u\times \cI_v \times P \to \cI_u\times \cI_v \times X.
\end{equation*}

\noindent
Using this identification, we write:
\begin{equation}
\label{eq:Aexpansion}
A = (\psi_{\fru}\circ \mathfrak{p})\, \dd \fru + (\phi_{\fru}\circ \mathfrak{p}) \, \dd \frv + a_{\fru}
\end{equation}

\noindent
for uniquely determined families of functions $\psi_{\fru}, \phi_{\fru}\in C^{\infty}(X)$ and a family of connections $a_{\fru} \in \Omega^1(P)$, all parametrized by $\fru$. Note that condition $\partial_{\frv} A = 0$ is equivalent to none of the terms in the above expansion depending on $\frv$. Hence, we will equivalently refer to standard Brinkmann triples $(g,A,u)$ on $(\cI_{\fru}\times\cI_{\frv} \times X, \mathfrak{p}^{\ast}P)$ as tuples $(h_{\fru}, a_{\fru}, \cA_{\fru}, \psi_{\fru}, \phi_{\fru}, \cH_{\fru})$ on $(X,P)$, where $h_{\fru}$ is a family of complete Riemannian metrics on $X$, $a_{\fru}$ is a family of connections on $P$, and $\psi_{\fru},\phi_{\fru},\cH_{\fru}\in C^{\infty}(X)$ are families of functions on $X$. \medskip

Here, and for the remainder of this section, $\dd \colon \Omega(X) \to \Omega(X)$ will denote the exterior derivative on $X$.

\begin{prop}
\label{prop:brinkmannsusy}
A standard Brinkmann configuration $(h_{\fru}, a_{\fru}, \cA_{\fru}, \psi_{\fru}, \phi_{\fru}, \cH_{\fru})$ on $(X,P)$ is quasi-supersymmetric if and only if it satisfies the following differential system:
\begin{equation*}
\phi_{\fru} = \lambda_R \fru + \lambda_R^o \, , \qquad \ast_{h_{\fru}} F_{a_{\fru}} = \mu \lambda_I,
\end{equation*}
where $\lambda_R^o\in \mathbb{R}$ is a real constant.
\end{prop}

\begin{proof}
Let $(g,A,u)$ be such a standard Brinkmann configuration on $(\cI_u\times \cI_v \times X, \mathfrak{p}^{\ast}P)$, with associated tuple $(h_{\fru}, a_{\fru}, \cA_{\fru}, \psi_{\fru}, \phi_{\fru}, \cH_{\fru})$. Taking the exterior derivative of Equation \eqref{eq:Aexpansion} as a one-form on the $\U(1)$-bundle $\mathfrak{p}^{\ast}P = \cI_u\times \cI_v \times P$, we obtain:
\begin{equation*}
F_A =  - \dd \fru \wedge \dd\psi_{\fru}  -  \dd \frv\wedge \dd\phi_{\fru} + \partial_{\fru} \phi_{\fru}\, \dd \fru\wedge  \dd \frv + \dd \fru \wedge \partial_{\fru} a_{\fru} + F_{a_{\fru}}.
\end{equation*}

\noindent
We quickly compute:
\begin{align*}
F_A(u^{\sharp_g}) &= F_A(\partial_{\frv}) = - \dd \phi_{\fru} - \partial_{\fru}\phi_{\fru} \, \dd \fru,\\
F_A \wedge u &= F_A\wedge \dd \fru = -\dd \fru\wedge \dd \frv \wedge \dd \phi_{\fru} + \dd \fru \wedge F_{a_{\fru}}.
\end{align*}
Plugging these expressions into both equations in \eqref{eq:Asugra}, we obtain:
\begin{equation*}
\dd \phi_{\fru} = 0\, , \qquad \partial_{\fru}\phi_{\fru} = \lambda_R\, , \qquad F_{a_{\fru}} = \mu \lambda_I \nu_{h_{\fru}},
\end{equation*}
where $\nu_{h_{\fru}}$ is the Riemannian volume form of $(X,h_{\fru})$, and we have used $\ast_g u = \ast_g \dd \fru = - \dd \fru \wedge \nu_{h_{\fru}}$. Hence, $\phi_{\fru} = \lambda_R \, \fru + \lambda_R^o$ for $\lambda_R^o\in\R$ and the only non-trivial remaining condition is $ F_{a_{\fru}} = \mu \lambda_I \nu_{h_{\fru}}$, or equivalently:
\begin{equation*}
\ast_{h_{\fru}} F_{a_{\fru}} = \mu \lambda_I.
\end{equation*}

\noindent
In particular, $a_{\fru}$ is a family of Yang-Mills connections on $P$ over $(X,h_{\fru})$.
\end{proof}

\begin{lemma}
Let $(h_{\fru}, a_{\fru},\cA_{\fru}, \psi_{\fru}, \phi_{\fru}, \cH_{\fru})$ be a quasi-supersymmetric configuration on $(X,P)$. Then:
\begin{eqnarray}
\label{eq:FAdual}
\ast_g F_A =  \dd \fru \wedge \ast_{h_{\fru}} (\partial_{\fru} a_{\fru} - \dd \psi_{\fru} - \lambda_R \cA_{\fru})  - \lambda_R \nu_{h_{\fru}} + \mu\lambda_I \dd \fru \wedge (\dd \frv + \cA_{\fru}),
\end{eqnarray}

\noindent
where $A$ is given as in Equation \eqref{eq:Aexpansion}. 
\end{lemma}

\begin{proof}
Set:
\begin{equation*}
u := \dd \fru \, , \qquad v:= \frac{1}{2}\cH_{\fru} \dd \fru + \dd \frv + \cA_{\fru},
\end{equation*}

\noindent
in terms of which we have:
\begin{equation*}
F_A = u\wedge (\partial_{\fru} a_{\fru} - \dd \psi_{\fru} - \lambda_R \cA_{\fru})   + \lambda_R u\wedge v + F_{a_{\fru}}.
\end{equation*}

\noindent
The result follows from applying the following general identities to the previous formula:
\begin{equation*}
\ast_g (u\wedge v) = - \nu_{h_{\fru}}\, , \qquad \ast_g (u \wedge \alpha) = u\wedge \ast_{h_{\fru}} \alpha,
\end{equation*}

\noindent
where $\alpha \in \Omega^1(X)$ and we have used that, by Proposition \ref{prop:brinkmannsusy}, $F_{a_{\fru}} = \mu\lambda_I \nu_{h_{\fru}}$, which implies that $\ast_g F_{a_{\fru}} = \mu\lambda_I \ast_g \nu_{h_{\fru}} = \mu\lambda_I u \wedge v$.
\end{proof}

\begin{prop}
\label{prop:susyYM}
A quasi-supersymmetric configuration $(h_{\fru}, a_{\fru}, \cA_{\fru}, \psi_{\fru}, \phi_{\fru}, \cH_{\fru})$ on $(X,P)$ satisfies the Maxwell equation in \eqref{eq:Freedmaneqs} if and only if the following equation holds:
\begin{equation*}
\dd \ast_{h_{\fru}} (\partial_{\fru} a_{\fru} - \dd \psi_{\fru} - \lambda_R  \cA_{\fru})  + \lambda_R   \partial_{\fru} \nu_{h_{\fru}} + \mu\lambda_I \dd \cA_{\fru} = 0.
\end{equation*} 
\end{prop}

\begin{proof}
We apply the exterior derivative to Equation \eqref{eq:FAdual}, obtaining:
\begin{equation*}
0 = \dd \ast_g F_A = - \dd \fru \wedge \dd \ast_{h_{\fru}} (\partial_{\fru} a_{\fru} - \dd \psi_{\fru} - \lambda_R  \cA_{\fru})  - \lambda_R \dd\fru \wedge \partial_{\fru} \nu_{h_{\fru}} - \mu\lambda_I \dd \fru \wedge \dd \cA_{\fru},
\end{equation*}

\noindent
from which the result follows.
\end{proof}

\noindent
We now introduce a natural class of standard Brinkmann configurations, which will be our main focus for the remainder of this section. 

\begin{definition}
A standard Brinkmann configuration $(h_{\fru}, a_{\fru}, \cA_{\fru}, \psi_{\fru}, \phi_{\fru}, \cH_{\fru})$ on $(X,P)$ is \emph{stationary} if it is independent of $\fru$, in which case we will drop the subscript and write $(h, a, \cA, \psi, \phi, \cH)$.
\end{definition}

We proceed now to elucidate the necessary and sufficient conditions for a stationary quasi-supersymmetric configuration $(h, a, \cA, \psi, \phi, \cH)$ on $(X,P)$ to be an honest solution of Freedman's gauged supergravity. The general, \emph{non-stationary}, case leads to an interesting evolution problem in the null coordinate $\fru$, and will be considered elsewhere.

\begin{lemma}
\label{lemma:curvatureBrinkmann}
Let $(h, a, \cA, \psi, \phi, \cH)$ be a stationary configuration on $(X,P)$. Then:
\begin{eqnarray*}
& \mathrm{Ric}^g = \mathrm{Ric}^{h} + \frac{1}{2}(\nabla^{h\ast} F_{\cA}) \odot \dd \fru + \frac{1}{2} \left( \vert F_{\cA}\vert^2_{h} + \nabla^{h\ast}\dd \cH  \right) \dd \fru \otimes \dd \fru \, , \\%
& s_g = s_{h} \, ,
\end{eqnarray*}

\noindent
where $g$ is given as in Equation \eqref{eq:standardBrinkmann} in terms of the given tuple $(h, a, \cA, \psi, \phi, \cH)$, and for convenience we have set $F_{\cA} := \dd\cA$. 
\end{lemma}

\begin{proof}
For further reference, recall that the components of the metric inverse to \eqref{eq:standardBrinkmann} are given by:
\begin{eqnarray*}
g^{\fru\fru} = 0 \, , \qquad g^{\fru\frv} = 1 \, , \qquad g^{\fru i} = 0 \, , \qquad
g^{\frv\frv} = \vert\cA\vert^2_h -\cH \, , \qquad g^{\frv i} = -\cA^i \, , \qquad g^{ij} = h^{ij}.
\end{eqnarray*}

\noindent
To obtain the Ricci curvature of $g$, we first compute its non-zero Christoffel symbols in a coordinate basis $(\partial_{\fru},\partial_{\frv},\partial_i)$ for $i=1,2$:
\begin{eqnarray*}
\label{eq:christoffel_stationary}
& 2\Gamma^{\frv}_{\fru\fru} = \dd \cH(\cA^{\sharp_h}) \, , \qquad  2 \Gamma^{\frv}_{\fru i} = \partial_i \cH + \dd \cA (\cA^{\sharp_h},\partial_i) \, , \qquad 2 \Gamma^{\frv}_{ij}= (\delta^h\cA)(\partial_i , \partial_j) \, , \\%
& 2\Gamma_{\fru\fru}^i = -  h^{ik} \partial_k\cH \, , \qquad 2 \Gamma^i_{\fru j} =   h^{ik} \left( \partial_j \cA_{k} - \partial_k \cA_{j} \right) \, , \qquad
\Gamma^{i}_{jk} = (\Gamma^{h})^{i}_{jk} \, ,
\end{eqnarray*}

\noindent
where $\delta^{h}\colon \Omega^1(X) \to \Gamma(TX\odot TX)$ denotes the symmetrization of the Levi-Civita connection $\nabla^{h}$ on $(X , h)$, and $\Gamma^{h}$ denotes the Christoffel symbols of $\nabla^{h}$. We proceed to compute the Riemann curvature tensor $\cR^g$ in the coordinate basis $(\partial_{\fru},\partial_{\frv},\partial_i)$, which we label by Greek indices. We have:
\begin{eqnarray*}
\cR^g_{\mu\nu}\partial_{\rho} = \nabla^g_{\mu}\nabla^g_{\nu} \partial_{\rho} - \nabla^g_{\nu}\nabla^g_{\mu} \partial_{\rho}.
\end{eqnarray*}

\noindent
Since $\partial_\frv$ is parallel, we immediately obtain $\cR^g_{\mu\nu}\partial_{\frv} = 0$. Furthermore, because all Christoffel symbols with a lower $\frv$ index identically vanish and the metric is independent of $\frv$, it follows that $\cR^g_{\frv\mu}\partial_{\rho} = 0$. For the remaining elements in the coordinate basis, we obtain:
\begin{eqnarray*}
& \cR^g_{ij}\partial_{k} = \cR^{h}_{ij}\partial_{k} + ( \partial_i \Gamma^{\frv}_{jk} - \partial_j \Gamma^{\frv}_{ik} + \Gamma^m_{jk}\Gamma^{\frv}_{im} - \Gamma^m_{ik}\Gamma^{\frv}_{jm} ) \partial_{\frv}  \, , \\
& \cR^g_{\fru i}\partial_{j} = ( \Gamma^m_{ij}\Gamma^{\frv}_{\fru m} - \partial_i \Gamma^{\frv}_{\fru j} - \Gamma^m_{\fru j}\Gamma^{\frv}_{im} ) \partial_{\frv} + (\Gamma^k_{ij}\Gamma^m_{\fru k} - \partial_i \Gamma^{m}_{\fru j}  - \Gamma^k_{\fru j}\Gamma^m_{ik} ) \partial_{m} \, , \\
& \cR^g_{ij}\partial_{\fru } = ( \partial_i \Gamma^{\frv}_{j\fru } - \partial_j \Gamma^{\frv}_{i\fru } + \Gamma^m_{j\fru }\Gamma^{\frv}_{im} - \Gamma^m_{i\fru }\Gamma^{\frv}_{jm} ) \partial_{\frv} + ( \partial_i \Gamma^{m}_{j\fru } - \partial_j \Gamma^{m}_{i\fru } + \Gamma^k_{j\fru }\Gamma^m_{ik} - \Gamma^k_{i\fru }\Gamma^m_{jk} ) \partial_{m} \, , \\
& \cR^g_{\fru i}\partial_{\fru } = ( \Gamma^m_{i\fru }\Gamma^{\frv}_{\fru m} - \partial_i \Gamma^{\frv}_{\fru \fru } - \Gamma^m_{\fru \fru }\Gamma^{\frv}_{im} ) \partial_{\frv} + (\Gamma^k_{i\fru }\Gamma^m_{\fru k} - \partial_i \Gamma^{m}_{\fru \fru } - \Gamma^k_{\fru \fru }\Gamma^m_{ik} ) \partial_{m} \, ,
\end{eqnarray*}
 
\noindent
where $\cR^{h}_{ij}\partial_{k}$ is the Riemann curvature tensor of the transverse Riemannian metric $h$. From these explicit equations, we obtain the following non-zero components of the Ricci tensor of $g$. First, we consider $\mathrm{Ric}^g (\partial_{\fru} , \partial_{\fru})$:
\begin{align*}
\Ric^g (\partial_{\fru} , \partial_{\fru}) &= g( \cR^g_{\fru i}\partial_{j} , \partial_{\fru}) h^{ij}  \\
&= ( \Gamma^m_{ij}\Gamma^\frv_{\fru m} - \partial_i \Gamma^{\frv}_{\fru j}  - \Gamma^m_{\fru j}\Gamma^{\frv}_{im} ) h^{ij} + (\Gamma^k_{ij}\Gamma^m_{\fru k} - \partial_i \Gamma^{m}_{\fru j}  - \Gamma^k_{\fru j}\Gamma^m_{ik} ) \cA_m h^{ij} \\
& =  \frac{1}{2} h^{ij} \Gamma^k_{ij} \partial_k \cH - \partial_i \Gamma^{\frv}_{\fru j}  h^{ij}  - \partial_i \Gamma^{m}_{\fru j}  \cA_m h^{ij} \\
&= \frac{1}{2} \left( \vert F_{\cA}\vert^2_{h} + \nabla^{h\ast}\dd \cH  \right),
\end{align*}

\noindent
where $\nabla^{h\ast}\alpha = - \mathrm{Tr}_h (\nabla^h \alpha)$ for any one-form $\alpha\in \Omega^1(X)$. For the non-zero \emph{mixed} components, we obtain:
\begin{align*}
\Ric^g (\partial_{\fru} , \partial_{k}) &= g(\cR^g_{\fru i} \partial_j , \partial_k) h^{ij} \\
&=  h_{mk} \left( \Gamma^l_{ij}\Gamma^m_{\fru l} - \partial_i \Gamma^{m}_{\fru j} - \Gamma^l_{\fru j}\Gamma^m_{il} \right) h^{ij} \\
&=  h_{mk} \left( \Gamma^l_{ij}\Gamma^m_{\fru l} - \partial_i \Gamma^{m}_{\fru j}  \right) h^{ij} \\
&= \frac{1}{2}(\nabla^{h\ast}F_{\cA})_k.
\end{align*}

\noindent
For the \emph{purely spatial} components, we have:
\begin{equation*}
\Ric^g (\partial_{i} , \partial_{m}) = g(\cR^g_{ij}\partial_k , \partial_m) h^{jk} = h(\cR^h_{ij}\partial_k , \partial_m) h^{jk} = \Ric^h (\partial_{i} , \partial_{m})
\end{equation*}

\noindent
and thus we conclude.
\end{proof}

\begin{prop}
\label{prop:equivalentBrinkmann}
A quasi-supersymmetric stationary configuration  $(h, a, \cA, \psi, \phi, \cH)$ on $(X,P)$ satisfies Freedman's gauged supergravity equations \eqref{eq:Freedmaneqs} on $\cI_u\times \cI_v\times X$ if and only if:
\begin{eqnarray*}
& \vert F_{\cA}\vert^2_{h} + \nabla^{h\ast}\dd \cH     =   2 \mathfrak{e}^2 \vert \dd \psi\vert^2_h \, , \qquad \nabla^{h\ast} F_{\cA}   =     2 \mathfrak{e}^2 \mu \lambda_I \ast_h \dd\psi \, , \qquad \dd\ast_h \dd \psi   = \mu\lambda_I F_{\cA}\, ,\\
&   s_h = 2 \mathfrak{e}^2 \lambda_I^2\, , \qquad    2 \Lambda  +  \mathfrak{e}^2 \lambda_I^2 = 0\, ,
\end{eqnarray*}

\noindent
where $g$ is given as in Equation \eqref{eq:standardBrinkmann} in terms of the given tuple $(h, a, \cA, \psi, \phi, \cH)$ and $F_{\cA} = \dd\cA$.
\end{prop}

\begin{proof}
By Proposition \ref{prop:brinkmannsusy}, a stationary configuration $(h, a, \cA, \psi, \phi, \cH)$ on the pair $(X,P)$ is quasi-supersymmetric if and only if $\phi$ is constant, whence $\lambda_R = 0$, and $a$ satisfies $\ast_h F_a = \mu \lambda_I$. The corresponding connection $A$ on $\cI_{\fru}\times \cI_{\frv} \times X$ is given by Equation \eqref{eq:Aexpansion}, and consequently reads:
\begin{equation*}
F_A =  - \dd \fru \wedge \dd\psi + F_a.
\end{equation*}

\noindent
Recall that $(F_A\circ_g F_A)(w_1,w_2):=\escal{F_A(w_1),F_A(w_2)}_g$. We compute:
\begin{eqnarray*}
& (F_A\circ_g F_A)(\partial_{\fru} , \partial_{\fru}) = \vert \dd \psi\vert^2_h\, , \qquad (F_A\circ_g F_A)(\partial_{\frv} , \partial_{\frv}) = 0\, ,\qquad (F_A\circ_g F_A)(\partial_{\fru} , \partial_{\frv}) = 0\, ,\\
& (F_A\circ_g F_A)(\partial_{\fru} , w) = - F_a(w,(\dd\psi)^{\sharp_h}) \, , \qquad (F_A\circ_g F_A)(\partial_{\frv} , w) = 0\, ,\\
& (F_A\circ_g F_A)(w_1 , w_2) = \escal{F_a(w_1) , F_a(w_2)}_h\, ,
\end{eqnarray*}

\noindent
where $w, w_1 , w_2 \in\Gamma(TX)$. Hence:
\begin{equation*}
F_A \circ_g F_A = \vert \dd \psi\vert^2_h \dd \fru \otimes \dd \fru + \dd\fru \odot F_a((\dd\psi)^{\sharp_h}) + F_a\circ_h F_a.
\end{equation*}

\noindent
Using that $s_g = s_h$ and $\vert F_A\vert_g^2 = \vert F_a \vert_h^2$, together with the previous formulae and Lemma \ref{lemma:curvatureBrinkmann}, we substitute in the Einstein equation of \eqref{eq:Freedmaneqs}, obtaining:
\begin{eqnarray*}
& \vert F_{\cA}\vert^2_{h} + \nabla^{h\ast}\dd \cH    - \cH s_h =  2\Lambda \cH + \mathfrak{e}^2 (2\vert \dd \psi\vert^2_h -  \cH \vert F_a\vert^2_h)\, ,\\
&  2\Lambda + s_h = \mathfrak{e}^2 \vert F_a \vert^2_h\, , \qquad  \nabla^{h\ast} F_{\cA} - s_h \cA = 2\Lambda \cA + 2\mathfrak{e}^2F_a((\dd\psi)^{\sharp_h}) -  \mathfrak{e}^2 \vert F_a\vert^2_h \cA\, ,\\
& 0 = 2 \Lambda h + 2 \mathfrak{e}^2\, F_a \circ_h F_a -  \mathfrak{e}^2 \vert F_a \vert^2_g h\, .
\end{eqnarray*}

\noindent
Substituting $F_a = \mu\lambda_I \nu_h$ and manipulating, the previous set of equations reduces to:
\begin{eqnarray*}
& \vert F_{\cA}\vert^2_{h} + \nabla^{h\ast}\dd \cH     =   2 \mathfrak{e}^2 \vert \dd \psi\vert^2_h \, , \qquad \nabla^{h\ast} F_{\cA}   =     2\mathfrak{e}^2\mu \lambda_I \ast_h \dd\psi\, ,\\
&  2\Lambda + s_h = \mathfrak{e}^2 \lambda_I^2\, , \qquad    2 \Lambda  +  \mathfrak{e}^2\, \lambda_I^2 = 0\, .
\end{eqnarray*}
Combining these equations with Proposition \ref{prop:susyYM}, we obtain the conditions in the statement.
\end{proof}

\begin{thm}
\label{thm:NoGoFreedman}
A standard stationary configuration  $(h, a, \cA, \psi, \phi, \cH)$ on the pair $(X,P)$ is a quasi-supersymmetric solution of Freedman's gauged supergravity equations \eqref{eq:Freedmaneqs} on $\mathbb{R}^2 \times X$ if and only if all the following conditions hold:
\begin{itemize}
    \item $X = S^2$, that is, $X$ is diffeomorphic to the two-dimensional sphere.
    \item $h$ is the round metric of scalar curvature $s_h = 2 \mathfrak{e}^2 \lambda_I^2$, that is, of radius $R=\vert\mathfrak{e} \lambda_I\vert^{-1}$.
    \item The Chern number $c(P)$ satisfies $\lambda_I \mathfrak{e}^2 c(P) = 2\mu$, and $a$ is a Yang-Mills connection on $P$.  
    \item $\psi \in C^{\infty}(S^2)$ is an eigenfunction of the Laplacian on $(S^2,h)$ with the lowest non-zero eigenvalue. 
    \item $\cA \in \Omega^1(S^2)$ is given by $\cA = \mu \lambda^{-1}_I (\ast_h \dd \psi + \dd \mathfrak{f})$ for a function $\mathfrak{f}\in C^{\infty}(S^2)$.
    \item $\cH\in C^{\infty}(S^2)$ is given by $\lambda_I^2\cH = -(R^{-2}\psi^2 + \frc)$ for a real constant $\frc$.
    \item The cosmological constant $\Lambda$ is determined by $2\Lambda = -\mathfrak{e}^2 \lambda_I^2$ and $\phi$ is constant.
\end{itemize}
In particular, the aforementioned Poisson equation for $\cH$ is guaranteed to have a solution.
\end{thm}

\begin{proof}
Let $(h, a, \cA, \psi, \phi, \cH)$ be a stationary configuration of Freedman's gauged supergravity on $\mathbb{R}^2 \times X$. By Proposition \ref{prop:brinkmannsusy}, the quasi-supersymmetry conditions reduce to $\lambda_R = 0$, hence $\phi$ is constant, and $a$ being a Yang-Mills connection on $P$ with normalization:
\begin{equation*}
F_a = \mu\lambda_I \nu_h.
\end{equation*}

\noindent
Equivalently, the Chern number $c(P)$ of $P$ satisfies:
\begin{equation*}
c(P) = \frac{\mu \chi(X)}{\lambda_I \mathfrak{e}^2},
\end{equation*}

\noindent
where $\chi(X)$ is the Euler characteristic of $X$. These conditions, together with Proposition \ref{prop:equivalentBrinkmann}, give the necessary and sufficient conditions for the tuple $(h, a, \cA, \psi, \phi, \cH)$ being a stationary quasi-supersymmetric solution of Freedman's gauged supergravity equations \eqref{eq:Freedmaneqs}. By condition $s_h = 2\mathfrak{e}^2 \lambda_I^2$ in Proposition \ref{prop:equivalentBrinkmann} it immediately follows that $(X,h)$ is the round sphere of radius $R := \vert \mathfrak{e} \lambda_I \vert^{-1}$ since $h$ is complete by definition. Hence $\chi(X)=2$. On the other hand, the equation $\dd\ast_h \dd \psi   = \mu\lambda_I F_{\cA}$ in Proposition \ref{prop:equivalentBrinkmann} can be explicitly integrated by setting:
\begin{equation*}
\cA = \mu \lambda^{-1}_I (\ast_h \dd \psi + \dd \mathfrak{f})
\end{equation*}

\noindent
for a smooth function $\frf \in C^{\infty}(X)$. Substituting in the remaining conditions contained in Proposition \ref{prop:equivalentBrinkmann}, and setting $2 \Lambda  := -  \mathfrak{e}^2 \lambda_I^2$, we obtain:
\begin{equation*}
\lambda_I^{-2}\vert \dd\ast_h \dd\psi \vert^2_{h} + \nabla^{h\ast}\dd \cH     =   2 \mathfrak{e}^2 \vert \dd \psi\vert^2_h \, , \qquad \nabla^{h\ast} \dd \ast_h \dd\psi   =     2 \mathfrak{e}^2   \lambda^2_I \ast_h \dd\psi  
\end{equation*}

\noindent
as the remaining system that needs to be solved. We observe that the second equation above can be rewritten as follows:
\begin{equation*}
\nabla^{h\ast} \dd \ast_h \dd\psi  =  \dd^{h\ast}\dd \ast_h \dd \psi =    (\dd^{h\ast}\dd + \dd \dd^{h\ast}) \ast_h \dd \psi = \Delta_h \ast_h \dd \psi = 2 \mathfrak{e}^2   \lambda^2_I \ast_h \dd\psi = \frac{2}{R^2} \ast_h \dd\psi,
\end{equation*}

\noindent
where $\dd^{h\ast}$ is the formal adjoint of the exterior derivative, $R$ is the radius of $(S^2,h)$, and $\Delta_h$ is the positive Laplace operator on $(S^2,h)$. Hence, we conclude that $\psi$ is, modulo a constant, a linear combination of the three linearly independent eigenfunctions of the Laplacian with the lowest eigenvalue. Using standard notation for the spherical harmonics $Y^m_l$ on $(S^2,h)$, we have: 
\begin{equation*}
\psi = \sum_{m =-1}^1 Y_1^m\, c_m
\end{equation*}

\noindent
for constants $c_1 , c_0, c_{-1}$. It remains to solve:
\begin{equation*}
\nabla^{h\ast}\dd \cH     =   2 \mathfrak{e}^2 \vert \dd \psi\vert^2_h   - \lambda_I^{-2}\vert \dd\ast_h \dd\psi \vert^2_{h}.
\end{equation*}

\noindent
Integrating, we obtain:
\begin{equation*}
\int_{S^2} \vert \dd\ast_h \dd\psi \vert^2_{h} \nu_h = \int_{S^2} \langle  \dd^{h\ast}\dd\psi , \dd^{h\ast}\dd\psi \rangle_{h} = \int_{S^2} \langle   \dd\psi , \dd \dd^{h\ast}\dd\psi \rangle_{h}= \int_{S^2} \langle  \ast_h  \dd\psi , \Delta_h \ast_h\dd\psi \rangle_{h} = 2 \mathfrak{e}^2 \lambda_I^2  \vert \dd \psi\vert_h^2.
\end{equation*}

\noindent
Hence:
\begin{equation*}
\int_{S^2} \nabla^{h\ast}\dd \cH  \nu_h   =  \int_{S^2}    (2 \mathfrak{e}^2 \vert \dd \psi\vert^2_h   - \lambda_I^{-2}\vert \dd\ast_h \dd\psi \vert^2_{h}) \nu_h  =  \int_{S^2}   (2\mathfrak{e}^2 \vert \dd \psi\vert^2_h   - 2 \mathfrak{e}^2   \vert \dd \psi\vert_h^2 ) \nu_h = 0
\end{equation*}
and thus, crucially, the standard theory of elliptic operators implies the existence of a solution $\cH$ that is unique up to an additive constant. Furthermore, we have:
\begin{equation*}
\lambda_I^{-2}\vert \dd\ast_h \dd\psi \vert^2_{h} = \lambda_I^{-2} (\Delta_h \psi)^2 = \frac{4}{R^4 \lambda^2_I} \psi^2.
\end{equation*}

\noindent
Hence, the Poisson equation for $\psi$ reduces to:
\begin{equation*}
\lambda^2_I \nabla^{h\ast}\dd \cH     =   \frac{2}{R^2} \vert \dd \psi\vert^2_h   - \frac{4}{R^4}\psi^2.
\end{equation*}

\noindent
Now, we observe that the following choice:
\begin{equation*}
\lambda_I^2\cH=-(R^{-2}\psi^2+\frc),
\end{equation*}
where $\frc$ is a real constant, is a solution of the Poisson equation for $\cH$. Since solutions to such an equation are unique up to an additive constant, this is the general solution.
\end{proof}

After a change of coordinates in the $(\fru,\frv)$-plane to absorb the factor $\dd\frf$ and the constants $\lambda_I$ and $\mu$, we conclude the following.

\begin{cor}
Every standard stationary quasi-supersymmetric solution of Freedman's gauged supergravity is isometric to a Lorentzian four-manifold $(M,g)$ of the form:
\begin{equation}
\label{eq:Freedmanmetric}
M = \mathbb{R}^2\times S^2 \, , \qquad g =  -(R^{-2}\psi^2+\frc)\dd\fru\otimes \dd\fru + \dd \fru \odot ( \dd\frv +  \ast_h\dd\psi) + h\, ,
\end{equation}

\noindent
where $(\fru,\frv)$ are the Cartesian coordinates of $\mathbb{R}^2$, $h$ is the round metric on $S^2$ of radius $R$, $\frc$ is a real constant, and $\psi$ is a first eigenfunction of the Laplacian on $(S^2,h)$, that is, $\Delta_h\psi=2R^{-2}\psi$.
\end{cor} 

In local standard angular coordinates $(\theta,\phi)$ on the sphere, the function $\psi$ can be written as follows:
\begin{equation*}
\psi(\theta,\phi)= c_1 \sin(\theta) \cos (\phi) + c_2 \sin(\theta)\sin(\phi) + c_3 \cos(\theta)
\end{equation*}

\noindent
for constants $c_1, c_2, c_3$. Plugging this expression in Equation \eqref{eq:Freedmanmetric}, we obtain its complete expression in local coordinates. We also find that standard stationary quasi-supersymmetric solutions of Freedman's gauged supergravity are very well-behaved causally.

\begin{cor}
Every standard stationary quasi-supersymmetric solution of Freedman's gauged supergravity is geodesically complete and globally hyperbolic.
\end{cor}

\begin{proof}
Geodesic completeness follows from \cite{Candela:2002rr,Samann:2016svf}, after observing that $\psi$ and its gradient are both bounded on $M = \mathbb{R}^2 \times S^2$, and hence have sublinear behavior, satisfying thus the conditions for completeness obtained in \emph{opere citato}. To establish the global hyperbolicity of the space-time \eqref{eq:Freedmanmetric}, we apply the general theory developed in \cite{Flores:2002fx} together with an isocausal metric comparison argument to bound the \emph{gyratonic} twist term. Let $\mathcal{A} = \ast_h \dd\psi$ denote the twist one-form on $S^2$. For any constant $\delta \in (0,1)$, we define a twist-free comparison metric $g_{co}$:
\begin{equation*}
g_{co} :=  (\cH - \frac{1}{\delta} \vert\mathcal{A}\vert^2_h ) \dd\fru \otimes \dd\fru + \dd\fru \odot \dd\frv + (1-\delta) h.
\end{equation*}

\noindent
Evaluating the difference $g - g_{co}$ on an arbitrary tangent vector $w = w^\fru \partial_\fru + w^\frv \partial_\frv + w^\perp$, we obtain:
\begin{equation*}
(g - g_{co})(w,w) = \frac{1}{\delta} \vert\mathcal{A}\vert^2_h (w^\fru)^2 + 2 w^\fru \mathcal{A}(w^\perp) + \delta \vert w^\perp \vert^2_h = \left\vert \frac{1}{\sqrt{\delta}} w^\fru \mathcal{A} + \sqrt{\delta} w^\perp \right\vert^2_h \ge 0.
\end{equation*}
Thus, $g \ge g_{co}$ everywhere on $M$. The comparison metric $g_{co}$ represents a standard generalized plane-fronted wave om $\mathbb{R}\times S^2$ as defined in \cite{Flores:2002fx}. Because the transverse manifold $(S^2, h)$ is compact, the effective profile function $\cH_{co} = \cH - \frac{1}{\delta} \vert\mathcal{A}\vert^2_h$ is globally bounded, satisfying the conditions of \cite{Flores:2002fx} and ensuring $g_{co}$ is globally hyperbolic on $M = \mathbb{R}^2 \times S^2$. The algebraic bound $g \ge g_{co}$ implies that for any $g$-causal vector $w \in TM$, we necessarily have $g_{co}(w,w) \le g(w,w) \le 0$. The light cones of $g$ are therefore contained within the light cones of $g_{co}$. Since $g_{co}$ admits a global Cauchy hypersurface, this spatial slice is intersected exactly once by any inextendible causal curve of $g_{co}$, and by causal containment, exactly once by any inextendible causal curve of $g$. Thus, $(M,g)$ inherits strict global hyperbolicity.
\end{proof}

\begin{remark}
Imposing strict supersymmetry requires the existence of an isomorphism between $P$ and a certain power of the unit bundle in $TS^2$. The quasi-supersymmetric solutions of the theory will, in general, not satisfy this condition and therefore will be strictly non-supersymmetric.
\end{remark}

The class of solutions obtained in Theorem \ref{thm:NoGoFreedman} gives a vast generalization of the solution found by Meessen and Ortín in \cite{Meessen:2010ph}, which inspired this section.


\section{Skew-torsion parallel spinors on Lorentzian six-manifolds}
\label{sec:skewtorsionparallel}


In this section, we characterize Lorentzian six-dimensional manifolds equipped with an irreducible complex and chiral spinor parallel under a metric connection with totally skew-symmetric torsion $H\in \Omega^3(M)$, with the goal of proving a structural result regarding their foliated structure. We will refer to such spinors as \emph{skew-torsion parallel spinors}, and we will denote a metric connection on $(M,g)$ with totally skew-symmetric torsion $H\in \Omega^3(M)$ by $\nabla^{g,H}$. Hence, a skew-torsion parallel spinor on $(M,g)$ is an irreducible complex and chiral spinor $\eta\in \Gamma(S^{\mu})$ satisfying:
\begin{equation*}
\nabla^{g,H} \eta = 0,
\end{equation*}
where we are denoting by the same symbol the lift of $\nabla^{g,H}$ to $S^{\mu}$. Since every metric connection $\nabla^{g,H}$ with totally skew-symmetric torsion $H\in \Omega^3(M)$ can be expanded as: 
\begin{equation*}
 \nabla^{g,H}_{w_1}w_2=\nabla^g_{w_1}w_2+\frac{1}{2}H(w_1,w_2)^{\sharp_g}  \, , \qquad w_1 , w_2 \in\Gamma(TM),
\end{equation*}
it follows that equation $\nabla^{g,H}\eta=0$ is equivalent to $\nabla^g_w\eta=\cA_w(\eta)$ with:
\begin{equation*}
\cA_w = -\frac{1}{4}\Psi_\gamma(H(w)) \in\Gamma(\End(S))\, , \qquad w\in\Gamma(TM),
\end{equation*}
whose symbol is given by $4\fra_w = -H(w)\in\Omega^2(M)$. This identification will be important later in this section to apply Theorems \ref{thm:even_constrained_parallel_spinors} and \ref{thm:even_CB_CPS}.\medskip

In Section \ref{sec:sugra6d} we will apply the results of this section to the study of the quasi-supersymmetric and supersymmetric solutions of six-dimensional minimal supergravity. 


\subsection{Algebraic squares in six dimensions}


In this subsection, we compute the complex-bilinear and Hermitian squares of an irreducible complex chiral spinor in six Lorentzian dimensions. The algebraic equations characterizing the Hermitian and complex-bilinear square of such spinor were obtained in \cite[Prop.\ 7.4]{Algebraic_Complex_Square_2025} and \cite[Cor.\ 7.8]{Algebraic_Complex_Square_2025}, respectively. We state the global characterization in the following lemma.

\begin{lemma}\label{lemma:(5,1)_squares}
Let $\eta\in\Gamma(S^\mu)$ be an irreducible complex chiral spinor with chirality $\mu\in\Z_2$ on an oriented Lorentzian six-manifold $(M,g)$. Then: 
\begin{itemize}
\item A complex exterior form $\widehat{\alpha}\in\Omega_\C(M)$ is the Hermitian square of $\eta$ if and only if: $$\widehat{\alpha}=u+iu\wedge\omega-\mu*_{g}u$$ for a unique nowhere vanishing isotropic one-form $u\in\Omega^1(M)$ and a two-form $\omega\in\Omega^2(M)$ satisfying: $$*_g(u\wedge\omega)=\mu u\wedge\omega,\qquad\omega(u^{\sharp_{g}})=0,\qquad \escal{\omega,\omega}_g=2,\qquad 2\mu*_{g}u=u\wedge\omega\wedge\omega.$$
\item A complex exterior form $\alpha\in\Omega_\C(M)$ is the complex-bilinear square of $\eta$ if and only if it is a decomposable $\mu$-self-dual three-form. That is, around every point in $M$ we have: $$\alpha=\theta_1\wedge\theta_2\wedge\theta_3,\qquad *_{g}(\theta_1\wedge\theta_2\wedge\theta_3)=\mu\theta_1\wedge\theta_2\wedge\theta_3$$ 
in terms of three isotropic and orthogonal complex local one-forms $\theta_1,\theta_2,\theta_3\in\Omega^1_\C(M)$.
\end{itemize}
\end{lemma}

The two-form $\omega$ in the Hermitian square $\widehat{\alpha}=u+iu\wedge\omega-\mu*_{g}u$ is determined up to transformations of the form $\omega\mapsto\omega+u\wedge\theta$ with $\theta\in\Omega^1(M)$ orthogonal to $u$. Given an isotropic one-form $u\in\Omega^1(M)$ we choose a \emph{conjugate} one-form $v\in\Omega^1(M)$, which by definition is isotropic and satisfies $\escal{u,v}_g=1$. Note that the existence of $v$ is guaranteed since $(M,g)$ is assumed to be oriented and time-oriented. For each choice of one-form $v$ conjugate to $u$, we have a canonical codimension-two distribution in $TM$ defined by:
\begin{equation*}
D_{uv}:=\ker(u)\cap\ker(v)=(\escal{\R u^{\sharp_{g}}}\oplus\escal{\R v^{\sharp_{g}}})^{\perp_g}\subset TM.
\end{equation*}

\noindent
We will occasionally refer to $D_{uv}$ as the \emph{transverse} distribution. In terms of $D_{uv} \subset TM$, the tangent bundle splits orthogonally as follows:
\begin{equation*} 
TM=\escal{\R u^{\sharp_g}}\oplus\escal{\R v^{\sharp_g}}\oplus D_{uv} 
\end{equation*}
in terms of which the metric $g$ splits accordingly:
\begin{equation*}
g = u\odot v + h_{uv}\, , \qquad h_{uv} = g\vert_{D_{uv}}.
\end{equation*}

\noindent 
Note that $h_{uv}$ is a positive-definite metric on $D_{uv}$. Given $u$ and a conjugate one-form $v$, we may choose the unique representative for $\omega$ satisfying $\omega(v^{\sharp_g})=0$, which we denote by $\omega_{uv}$. With this choice, we have the following identities: 
\begin{equation}
\label{eq:Hodge_(u,v,omega)}
\begin{gathered}
    *_gu=\frac{\mu}{2}u\wedge\omega_{uv}\wedge\omega_{uv},\qquad*_gv=-\frac{\mu}{2}v\wedge\omega_{uv}\wedge\omega_{uv},\qquad *_g(u\wedge v)=\frac{\mu}{2}\omega_{uv}\wedge\omega_{uv},\\
    *_g\omega_{uv}=-\mu u\wedge v\wedge\omega_{uv},\qquad*_g(u\wedge\omega_{uv})=\mu u\wedge\omega_{uv},\qquad *_g(v\wedge\omega_{uv})=-\mu v \wedge\omega_{uv}.
\end{gathered}
\end{equation}

\noindent
Since $\omega_{uv} (u^{\sharp_g}) = \omega_{uv}(v^{\sharp_g}) = 0$, it follows that $\omega_{uv}$ descends to $D_{uv}$, and hence we can consider it as a section of the form $\omega_{uv} \in \Gamma(\wedge^2 D^{\ast}_{uv})$, where $D_{uv}^{\ast}$ denotes the vector bundle dual to $D_{uv}$. One-forms conjugate to a given isotropic one-form $u\in\Omega^1(M)$ are not unique. Indeed, if $v,\tilde{v}\in\Omega^1(M)$ are both conjugate to $u$, then there exists a unique $\mathfrak{w}\in\Omega^1(M)$ orthogonal to $u$ and $v$ such that: 
\begin{equation*}
\tilde{v}=v-\frac{\vert \mathfrak{w}\vert^2_g}{2}u+\mathfrak{w} \, .    
\end{equation*}

\noindent
The transverse distribution changes accordingly via the following isomorphism of vector bundles:
\begin{equation*}
D_{uv} \to D_{u\tilde{v}} \, , \qquad w \mapsto w - \mathfrak{w}(w) u.
\end{equation*}

\noindent
In particular, we have the following isomorphisms of vector bundles: 
\begin{equation*}
\mathrm{Ker}(u) =  \langle \mathbb{R} u^{\sharp_g} \rangle \oplus D_{uv} =  \langle \mathbb{R} u^{\sharp_g} \rangle \oplus D_{u\tilde{v}}.
\end{equation*}

\noindent
The two-form $\omega_{uv}$ also changes accordingly: 
\begin{equation*}
\omega_{u\tilde{v}} =\omega_{uv} -u\wedge\omega_{uv}(\mathfrak{w}^{\sharp_g}),
\end{equation*}

\noindent
where we have used that $\omega_{u\tilde{v}}(\tilde{v}^{\sharp_g})=0$. We induce a natural orientation on $D_{uv}$ by defining a volume form $\nu_{uv}$ on $D_{uv}$ via the relation $\nu_g =u\wedge v\wedge\nu_{uv}$, where $\nu_g$ is the Lorentzian volume form on $(M,g)$. In turn, $\nu_{uv}$ defines a Hodge star operator $*_{uv}$ on $D_{uv}$. In particular, it follows that: 
\begin{equation*}
\ast_{uv}\omega_{uv} = -\mu \omega_{uv}.
\end{equation*} 

\noindent
Since $\omega_{uv}$ is a non-degenerate two-form on $D_{uv}$, which is also equipped with a Euclidean metric $h_{uv}$, there exists a unique almost complex structure $J_{uv} \colon D_{uv} \to D_{uv}$ on $D_{uv}$ such that:
\begin{equation*}
h_{uv}(w_1 , w_2) = \omega_{uv}( w_1 , J_{uv} w_2).
\end{equation*}

\noindent
Furthermore, by the duality relation satisfied by $\omega_{uv}$, it follows that $(\omega_{uv},J_{uv})$ defines an almost Hermitian structure on $D_{uv}$, compatible with the given orientation if $\mu = -1$, and opposite to the given orientation if $\mu = 1$. Equivalently, we can consider the pair $(\omega_{uv},J_{uv})$ as defining a $\U(2)$-structure on $D_{uv}$. In particular, the six-dimensional Lorentzian metric $g$ can be written as follows:
\begin{equation*}
g = u\odot v + \omega_{uv} \circ J_{uv},
\end{equation*}

\noindent
and therefore the metric $g$ is completely determined by $u, v, \omega_{uv}$ and $J_{uv}$. As expected, changing the conjugate one-form $v$ changes $J_{uv}$ accordingly.
\begin{lemma} 
The following relation holds:
\begin{equation*}
J_{u\tilde{v}} =J_{uv} -u\otimes J_{uv}(\mathfrak{w}^{\sharp_g}) + \omega_{uv}(\mathfrak{w}^{\sharp_g})\otimes u^{\sharp_g}
\end{equation*}
for $\mathfrak{w}\in\Omega^1(M)$ orthogonal to $u$ and $v$. 
\end{lemma}

\begin{proof}
To prove that the given endomorphism $J_{u\tilde{v}}$ is the correct almost complex structure on the new transverse distribution $D_{u\tilde{v}}$, we must verify that it annihilates the null vectors $u^{\sharp_g}$ and $\tilde{v}^{\sharp_g}$, preserves the transverse space $D_{u\tilde{v}} = \ker(u) \cap \ker(\tilde{v})$, and satisfies the metric compatibility condition $g = u \odot \tilde{v} + \omega_{u\tilde{v}}( - , J_{u\tilde{v}} - )$. By evaluating the proposed formula for $J_{u\tilde{v}}$ on $u^{\sharp_g}$, we immediately get:
\begin{equation*}
J_{u\tilde{v}}(u^{\sharp_g}) = J_{uv}(u^{\sharp_g}) - u(u^{\sharp_g})J_{uv}(\mathfrak{w}^{\sharp_g}) + \omega_{uv}(\mathfrak{w}^{\sharp_g}, u^{\sharp_g}) u^{\sharp_g} = 0.
\end{equation*} 

\noindent
Next, we evaluate $J_{u\tilde{v}}$ on the new conjugate one-form $\tilde{v} = v - \frac{1}{2}\vert\mathfrak{w}\vert_g^2 u + \mathfrak{w}$:
\begin{equation*}
J_{u\tilde{v}}(\tilde{v}^{\sharp_g}) = J_{uv}(\tilde{v}^{\sharp_g}) - u(\tilde{v}^{\sharp_g})J_{uv}(\mathfrak{w}^{\sharp_g}) + \omega_{uv}(\mathfrak{w}^{\sharp_g}, \tilde{v}^{\sharp_g}) u^{\sharp_g}   = J_{uv}(\mathfrak{w}^{\sharp_g}) - J_{uv}(\mathfrak{w}^{\sharp_g}) = 0.
\end{equation*}

\noindent
For any vector field $w \in D_{u\tilde{v}}$, we have $u(w) = 0$ and $\tilde{v}(w) = 0$. Applying $J_{u\tilde{v}}$ yields:
\begin{equation*}
J_{u\tilde{v}}(w) = J_{uv}(w) + \omega_{uv}(\mathfrak{w}^{\sharp_g}, w) u^{\sharp_g}.
\end{equation*}

\noindent
We must check that $J_{u\tilde{v}}(w) \in D_{u\tilde{v}}$ for every $w\in D_{u\tilde{v}}$. We have:
\begin{equation*}
J_{uv}(w) =  J_{uv}(w - \mathfrak{w}^{\sharp_g}(w) u^{\sharp_g})  \in D_{uv} \subset \ker(u).
\end{equation*}

\noindent
We compute:
\begin{align*}
\tilde{v}(J_{u\tilde{v}}(w)) &= \tilde{v}(J_{uv}(w)) + \omega_{uv}(\mathfrak{w}^{\sharp_g}, w) \tilde{v}(u^{\sharp_g}) \\
&= h_{uv}(\mathfrak{w}^{\sharp_g} , J_{uv}(w - \mathfrak{w}(w) u^{\sharp_g})) + \omega_{uv}(\mathfrak{w}^{\sharp_g}, w)  \\
&= - \omega_{uv}(\mathfrak{w}^{\sharp_g}, w)   + \omega_{uv}(\mathfrak{w}^{\sharp_g}, w) \\
&= 0.
\end{align*}
Hence $J_{u\tilde{v}}(w) \in \ker(u) \cap \ker(\tilde{v})=D_{u\tilde{v}}$.
\end{proof}

Given two tuples $(u,v,\omega_{uv},J_{uv})$ and $(u,\tilde{v},\omega_{u\tilde{v}},J_{u\tilde{v}})$ associated to the same chiral complex spinor $\eta$, one can check that:
\begin{equation*}
g = u\odot v + \omega_{uv} \circ J_{uv} = u\odot \tilde{v} + \omega_{u\tilde{v}} \circ J_{u\tilde{v}} 
\end{equation*}
as required by the internal consistency of the construction.  The previous discussion shows that every pair $(g,\eta)$ consisting of a Lorentzian metric $g$ on $M$ and an irreducible complex and chiral spinor $\eta\in \Gamma(S^{\mu})$ defines a tuple $(u,v,\omega_{uv},J_{uv})$ as defined above, which is unique modulo the choice of conjugate one-form $v$. Hence, the pair $(g,\eta)$ defines a natural equivalence class $[u,v,\omega,J]$ of tuples $(u,v,\omega,J)$ consisting of an isotropic vector field $u$, a conjugate one-form $v$, and a $\U(2)$-structure $(\omega,J)$ on $D_{uv}$ compatible with the induced orientation if $\mu = -1$ and opposite otherwise. Two such tuples $(u,v,\omega,J)$ and $(\tilde{u},\tilde{v},\tilde{\omega},\tilde{J})$ are considered as equivalent if and only if:
\begin{equation}
\label{eq:equivalence_relation_(u,v,omega)}
(\tilde{u},\tilde{v},\tilde{\omega},\tilde{J})=(u,v-\frac{\vert \mathfrak{w}\vert_g^2}{2}u+\mathfrak{w},\omega-u\wedge\omega(\mathfrak{w}^{\sharp_g}), J  -u\otimes J(\mathfrak{w}^{\sharp_g}) + \omega(\mathfrak{w}^{\sharp_g})\otimes u^{\sharp_g})
\end{equation}

for a unique section $\mathfrak{w} \in \Gamma(D_{uv})$ of the transverse distribution $D_{uv}\subset TM$. We will refer to such a tuple $[u,v,\omega,J]$ as an \emph{isotropic almost Hermitian class}. They represent a natural generalization of the notion of isotropic parallelism, introduced in \cite{ShahbaziThesis}, from four to six dimensions. Note that there is a one-to-one correspondence between isotropic almost Hermitian classes and pairs $(g,\widehat{\alpha})$, where $g$ is a Lorentzian metric on $M$ and $\widehat{\alpha}$ is the Hermitian square of an irreducible complex and chiral spinor. Hence, the Hermitian square of such a spinor can be equivalently studied in terms of the notion of isotropic almost Hermitian class.


\subsection{The Hermitian square of a skew-torsion parallel spinor}


In this subsection, we study the differential system satisfied, via Theorem \ref{thm:even_constrained_parallel_spinors}, by the Hermitian square of every skew-torsion parallel spinor. This will provide a set of necessary conditions for a Lorentzian six-manifold $(M,g)$ to admit a skew-torsion parallel spinor that we will interpret in terms of an associated foliation of $M$.  

\begin{lemma}
A complex exterior form $\widehat{\alpha}\in \Omega_{\C}(M)$ is the Hermitian square of a skew-torsion parallel spinor with torsion $H\in \Omega^3(M)$ if and only if it satisfies the following differential system:
\begin{equation}
\label{eq:quasisusy6d}
\nabla^g_w\widehat{\alpha}=\fra_w\diamond\widehat{\alpha}+\widehat{\alpha}\diamond\tau(\overline{\fra}_w)\, , \qquad 4\fra_w = -H(w)\in\Omega^2(M)\, , \qquad w\in\Gamma(TM)\, ,
\end{equation}
where $\widehat{\alpha} = u + i u \wedge \omega - \mu \ast_g u$ is the Hermitian square of a chiral complex spinor $\eta\in \Gamma(S^{\mu})$.
\end{lemma}

\begin{proof}
Follows from Corollary \ref{cor:even_constrained_parallel_spinors} after noticing that the Hermitian square $\widehat{\alpha}$ is computed with respect to a Hermitian pairing of positive adjoint type.  
\end{proof}

\noindent
The goal now is to expand Equation \eqref{eq:quasisusy6d} into its components, which can be studied and understood geometrically.

\begin{lemma}
\label{lemma:skew_torsion_Hermitian}
The Hermitian square $\widehat{\alpha} = u + i u\wedge \omega - \mu \ast_g u$ of an irreducible complex and chiral spinor $\eta\in \Gamma(S^{\mu})$ satisfies Equation \eqref{eq:quasisusy6d} if and only if:
\begin{equation*}
\nabla^g_wu=\frac{1}{2}H(w,u^{\sharp_g}),\qquad\nabla^g_w\omega=\frac{1}{2}H(w)\triangle_1\omega+\kappa_w\wedge u    
\end{equation*} 
for a certain $\kappa\in\Omega^1(M,T^*M)$ such that $\escal{\kappa_w,u}_g=0$ for all $w\in\Gamma(TM)$.
\end{lemma}

\begin{proof}
We expand Equation \eqref{eq:quasisusy6d}. We have: 
\begin{align*}
\fra_w\diamond\widehat{\alpha}&=\fra_w\diamond(u+iu\wedge\omega-\mu u\diamond\nu_g)\\
&=\fra_w\wedge u-\fra_w(u^{\sharp_{g}})+\mu*_{g}(\fra_w\wedge u)+\mu*_{g}(\fra_w(u^{\sharp_{g}}))\\
&\quad+i(\fra_w\wedge u\wedge\omega-\fra_w\triangle_1(u\wedge\omega)-\fra_w\triangle_2(u\wedge\omega)),\\
\widehat{\alpha}\diamond\tau(\overline{\fra}_w)&=(u+iu\wedge\omega+\mu\nu_g\diamond u)\diamond(-\fra_w)\\
&=-u\wedge\fra_w-\fra_w(u^{\sharp_{g}})-\mu*_{g}(u\wedge\fra_w)+\mu*_{g}(\fra_w(u^{\sharp_{g}}))\\
&\quad-i(u\wedge\omega\wedge\fra_w+(u\wedge\omega)\triangle_1\fra_w-(u\wedge\omega)\triangle_2\fra_w),
\end{align*} 
where we have used that $*_{g}u=u\diamond\nu_g=-\nu_g\diamond u$. Separating by degrees this equation we obtain: $$\nabla^g_wu=-2\fra_w(u^{\sharp_{g}}),\qquad \nabla^g_w(u\wedge\omega)=-2\fra_w\triangle_1(u\wedge\omega)$$ for all $w\in\Gamma(TM)$. Let $\{e_1,\ldots,e_6\}$ be a frame of $M$. Then: 
\begin{align*}
        \fra_w\triangle_1(u\wedge\omega)&=g^{ij}\iota_{e_i}\fra_w\wedge\iota_{e_j}(u\wedge\omega)\\
        &=g^{ij}\iota_{e_i}\fra_w\wedge(\iota_{e_j}u\wedge\omega-u\wedge\iota_{e_j}\omega)\\
        &=(g^{ij}\iota_{e_i}\fra_w\wedge\iota_{e_j}u)\wedge\omega+u\wedge(g^{ij}\iota_{e_i}\fra_w\wedge\iota_{e_j}\omega)\\
        &=\fra_w(u^{\sharp_g})\wedge\omega+u\wedge(\fra_w\triangle_1\omega),
    \end{align*} where we have used $\fra_w\triangle_1u=\fra_w(u^{\sharp_g})$. Thus $\nabla^g_w(u\wedge\omega)=-2\fra_w(u^{\sharp_g})\wedge\omega-2u\wedge(\fra_w\triangle_1\omega)$, which implies: $$u\wedge\nabla^g_w\omega=-2u\wedge(\fra_w\triangle_1\omega)$$ since $\nabla^g_wu=-2\fra_w(u^{\sharp_{g}})$. Hence: $$\nabla^g_w\omega=-2\fra_w\triangle_1\omega+\kappa_w\wedge u$$ for some $\kappa\in\Omega^1(M,T^*M)$. Substituting $\fra_w=-\frac{1}{4}H(w)$ yields the equations from the statement. To get the condition $\escal{\kappa_w,u}_g=0$ for all $w\in\Gamma(TM)$, we differentiate $\omega(u^{\sharp_g})=0$. Then we get: \begin{align*}
        0&=(\nabla^g_w\omega)(u^{\sharp_g})+\omega((\nabla^g_wu)^{\sharp_g})\\
        &=\frac{1}{2}\iota_{u^{\sharp_g}}(H(w)\triangle_1\omega)+\iota_{u^{\sharp_g}}(\kappa_w\wedge u)+\frac{1}{2}\omega\triangle_1H(w,u^{\sharp_g})\\
        &=-\frac{1}{2}H(w,u^{\sharp_g})\triangle_1\omega+\escal{\kappa_w,u}_gu+\frac{1}{2}\omega\triangle_1H(w,u^{\sharp_g})\\
        &=\escal{\kappa_w,u}_gu,
    \end{align*} which implies that $\escal{\kappa_w,u}_g=0$ since $u$ is nowhere vanishing.
\end{proof}

As we explained in the previous subsection, pairs $(g,\widehat{\alpha})$ are in one-to-one correspondence with isotropic almost Hermitian classes $[u,v,\omega,J]$.

\begin{prop}
\label{prop:covariant_derivative_(u,v,omega)}
A Lorentzian six-dimensional manifold $(M,g)$ admits a skew-torsion parallel spinor with torsion $H\in\Omega^3(M)$ only if it admits an isotropic almost Hermitian class $[u,v,\omega,J]$ satisfying:
\begin{gather*}
\nabla^g_w u=\frac{1}{2}H(w,u^{\sharp_{g}}),\qquad \nabla^g_w v=\frac{1}{2}H(w,v^{\sharp_g})-\kappa_w\triangle_1\omega,\\
\nabla^g_w\omega=\frac{1}{2}H(w)\triangle_1\omega+\kappa_w\wedge u,\qquad \nabla^g_w J = \frac{1}{2} ( H(w) \circ J - J \circ H(w) ) + \kappa_w\otimes u^{\sharp_g}-u\otimes\kappa_w^{\sharp_g}
\end{gather*}
for all $w\in\Gamma(TM)$ and some $\kappa\in\Omega^1(M,T^*M)$ such that $\escal{\kappa_w,u}_g=0$.
\end{prop}

\begin{proof}
We have computed $\nabla^g_wu$ and $\nabla^g_w\omega$ in Lemma \ref{lemma:skew_torsion_Hermitian}. It remains to compute $\nabla^g_w v$ and $\nabla^g_w J$. First note that from $\escal{v,v}_g=0$ we obtain $\escal{\nabla_w^g v,v}_g=0$, that is, $\iota_{v^{\sharp_g}}(\nabla^g_wv)=0$. Set $2T_w:=H(w)\triangle_1\omega\in\Omega^2(M)$ to simplify the notation. We will use equations \eqref{eq:Hodge_(u,v,omega)} repeatedly. Using the relation $2 u\wedge v=-\mu \ast_g(\omega\wedge\omega)$ we compute on one hand: $$\nabla^g_w(u\wedge v)=\nabla^g_wu\wedge v+u\wedge\nabla^g_wv=\frac{1}{2}H(w,u^{\sharp_g})\wedge v+u\wedge\nabla^g_wv.$$
    
    Taking the interior product with $v^{\sharp_g}$ yields: $$\iota_{v^{\sharp_g}}(\nabla^g_w(u\wedge v))=\frac{1}{2}H(w,u^{\sharp_g},v^{\sharp_g})v+\nabla^g_wv.$$

    On the other hand: $$\nabla^g_w(-\frac{\mu}{2}*_g(\omega\wedge\omega))=-\mu*(\nabla^g_w\omega\wedge\omega)=-\mu*_g(T_w\wedge\omega+\kappa_w\wedge u\wedge\omega).$$
    
    Taking the interior product with $v^{\sharp_g}$ yields: $$\iota_{v^{\sharp_g}}(\nabla^g_w(-\frac{\mu}{2}*_g(\omega\wedge\omega)))=-\mu*_g(T_w\wedge\omega\wedge v+\kappa_w\wedge u\wedge\omega\wedge v).$$
    
    Using $*_g(v\wedge\omega)=-\mu v\wedge\omega$ and Proposition \ref{prop:appendix_properties} we obtain: $$-\mu*_g(T_w\wedge\omega\wedge v)=-\mu*_g(v\wedge\omega\wedge T_w)=-\mu T_w\triangle_2(*_g(v\wedge\omega))=T_w\triangle_2(v\wedge\omega).$$

    A computation using the definition of $\triangle_2$ and the fact that $\omega(v^{\sharp_g})=0$ gives us: $$T_w\triangle_2(v\wedge\omega)=\escal{T_w,\omega}v+T_w(v^{\sharp_g})\triangle_1\omega.$$

    Note also that $T_w(v^{\sharp_g})=\frac{1}{2}\iota_{v^{\sharp_g}}(H(w)\triangle_1\omega)=-\frac{1}{2}H(w,v^{\sharp_g})\triangle_1\omega$. Now we need to compute two identities: \begin{itemize}
        \item Let $\theta_1,\theta_2\in\Omega^1(M)$. Then $\escal{\theta_1\triangle_1\omega,\theta_2}=-\escal{\theta_1,\theta_2\triangle_1\omega}$. Indeed: $$\escal{\theta_1\triangle_1\omega,\theta_2}=\escal{\omega(\theta_1^{\sharp_g}),\theta_2}=\omega(\theta_1^{\sharp_g},\theta_2^{\sharp_g})=-\omega(\theta_2^{\sharp_g},\theta_1^{\sharp_g})=-\escal{\theta_1,\theta_2\triangle_1\omega}.$$
        \item Let $w_1,w_2\in\Gamma(D_{uv})$. Then $\escal{\iota_{w_1}\omega,\iota_{w_2}\omega}=g(w_1,w_2)$. Indeed, using that the exterior product is adjoint to the interior product, $\escal{\omega,\omega}=2$, and $*_{uv}\omega=-\mu\omega$ on $D_{uv}$ we have: \begin{align*}
            \escal{\iota_{w_1}\omega,\iota_{w_2}\omega}&=\escal{\omega\wedge w_1^\flat,\omega\wedge w_2^\flat}\\
            &=\escal{\omega,\iota_{w_1}\omega\wedge w_2^\flat}+\escal{\omega,\omega}g(w_1,w_2)\\
            &=-\escal{\iota_{w_1}\omega,\iota_{w_2}\omega}+2g(w_1,w_2).
        \end{align*}
    \end{itemize}

    A one-form $\beta\in\Omega^1(M)$ can be decomposed as $\beta=\beta(v^{\sharp_g})u+\beta(u^{\sharp_g})v+\beta^\perp$, where $\beta^\perp(u^{\sharp_g})=\beta^\perp(v^{\sharp_g})=0$, i.e.\ $\beta^\perp\in\Gamma(D_{uv}^*)$. Thus $\beta\triangle_1\omega=\beta^\perp\triangle_1\omega$. Using the two identities above we compute: \begin{align*}
        \escal{(\beta\triangle_1\omega)\triangle_1\omega,\theta}&=\escal{(\beta^\perp\triangle_1\omega)\triangle_1\omega,\theta}\\
        &=-\escal{\beta^\perp\triangle_1\omega,\theta\triangle_1\omega}\\
        &=-\escal{\beta^\perp\triangle_1\omega,\theta^\perp\triangle_1\omega}\\
        &=-\escal{\beta^\perp,\theta^\perp}\\
        &=-\escal{\beta^\perp,\theta}
    \end{align*} for all $\theta\in\Omega^1(M)$. Hence: $$(\beta\triangle_1\omega)\triangle_1\omega=-\beta^\perp=-\beta+\beta(v^{\sharp_g})u+\beta(u^{\sharp_g})v.$$

    In our situation we have: $$-\frac{1}{2}(H(w,v^{\sharp_g})\triangle_1\omega)\triangle_1\omega=\frac{1}{2}H(w,v^{\sharp_g})-\frac{1}{2}H(w,v^{\sharp_g},u^{\sharp_g})v.$$

    One can show that $\escal{\sigma_1\triangle_1\omega,\sigma_2}=-\escal{\sigma_1,\sigma_2\triangle_1\omega}$ for all $\sigma_1,\sigma_2\in\Omega^2(M)$. Using this and $\omega\triangle_1\omega=0$ we get: $$\escal{T_w,\omega}=\frac{1}{2}\escal{H(w)\triangle_1\omega,\omega}=-\frac{1}{2}\escal{H(w),\omega\triangle_1\omega}=0.$$
    
    Finally, using $*_g(u\wedge v\wedge\omega)=\mu\omega$ and Proposition \ref{prop:appendix_properties} we obtain: $$-\mu*_g(\kappa_w\wedge u\wedge v\wedge\omega)=-\mu*_g(u\wedge v\wedge\omega\wedge\kappa_w)=-\mu\kappa_w\triangle_1(*_g(u\wedge v\wedge\omega))=-\kappa_w\triangle_1\omega.$$

    Putting everything together, we obtain the expression for $\nabla^g_w v$ from the statement. Finally, the expression for $\nabla^g_wJ$ is obtained from that of $\nabla^g_w\omega$ and the relation $\omega=g(J-,-)$.
\end{proof}

\begin{remark}
We note that if a tuple $(u,v,\omega,J)$ satisfies the equations of Proposition \ref{prop:covariant_derivative_(u,v,omega)} for a certain $\kappa\in \Gamma (T^{\ast} M\otimes T^{\ast}M$), then another tuple $(u,\tilde{v},\tilde{\omega},\tilde{J})$ as in Equation \eqref{eq:equivalence_relation_(u,v,omega)} satisfies the same equations for: $$\tilde{\kappa}_w=\kappa_w+\nabla^g_w(\omega(\mathfrak{w}^{\sharp_g}))-\frac{1}{2}H(w)\triangle_1(\omega(\mathfrak{w}^{\sharp_g})).$$
\end{remark}

Let $\eta$ be an irreducible complex and chiral spinor on $(M,g)$, with Hermitian square $\widehat{\alpha}=u+iu\wedge\omega-\mu*_{g}u$. Since $u\in \Omega^1(M)$ is an isotropic one-form, it canonically determines a \emph{screen bundle} over $(M,g)$, that is, a vector bundle $\mathfrak{S}_u \to (M,g)$ defined by:
\begin{equation*}
\mathfrak{S}_u := \frac{\ker(u)}{\langle \mathbb{R} u^{\sharp_g}\rangle}.
\end{equation*}
In our case, $\mathfrak{S}_u$ is of real rank $4$. Since $u$ is canonically determined by $\eta$, it follows that $\mathfrak{S}_u$ is also canonically determined by $\eta$. Note that every choice of one-form $v$ conjugate to $u$ determines a distribution $D_{uv} \subset TM$ that is naturally isomorphic, as a vector bundle, to $\mathfrak{S}_u$.

\begin{definition}
The screen bundle $\mathfrak{S}_u$ is $\emph{integrable}$ if there exists a one-form $v$ conjugate to $u$ such that $D_{uv}\subset TM$ is Frobenius integrable. 
\end{definition}

Let us study under which conditions the distribution $D_{uv}=\ker(u)\cap\ker(v)\subset TM$ is integrable. For that, we use the Frobenius criterion, which says that $D_{uv}$ is integrable if and only if:
\begin{equation*}
u\wedge v\wedge \dd u=0,\qquad u\wedge v\wedge\dd v=0.
\end{equation*}
 
\noindent
To evaluate this criteria, it is convenient to decompose $H\in\Omega^3(M)$ with respect to: $$T^*M=\escal{\R u}\oplus\escal{\R v}\oplus D^*_{uv}$$ as follows:
\begin{equation}
\label{eq:Hdecomposition}
H = u\wedge v\wedge\beta + u\wedge\chi_u + v\wedge\chi_v + H^\perp    
\end{equation}
for uniquely determined sections:
\begin{equation*}
\beta\in\Gamma(D_{uv}^*),\qquad \chi_u,\chi_v\in\Gamma(\wedge^2D_{uv}^*),\qquad H^\perp\in\Gamma(\wedge^3D_{uv}^*).    
\end{equation*}

\begin{remark}
Note that in the decomposition of $H$ given in Equation \eqref{eq:Hdecomposition}, it follows that $H$ is $\mu$-self-dual if and only if:
\begin{equation*}
*_{uv}\chi_u=-\mu\chi_u,\qquad *_{uv}\chi_v=\mu\chi_v,\qquad *_{uv}H^\perp=\mu\beta.
\end{equation*}
\end{remark}
 
\begin{lemma}\label{lemma:D_integrable}
The distribution $D_{uv}\subset TM$ is integrable if and only if: $$\chi_v=0,\qquad \chi_u(w_1,w_2)=-\omega(\kappa_{w_1}^{\sharp_g},w_2)+\omega(\kappa_{w_2}^{\sharp_g},w_1)$$ for all $w_1,w_2\in\Gamma(D_{uv})$.
\end{lemma}

\begin{proof}
    By Proposition \ref{prop:covariant_derivative_(u,v,omega)}, the skew-symmetrization of $\nabla^gu$ is: $$\dd u=-H(u^{\sharp_g})=u\wedge\beta-\chi_v.$$

    Hence $u\wedge v\wedge\dd u=0$ if and only if $\chi_v=0$. We now compute $\dd v$: \begin{align*}
        \dd v(w_1,w_2)&=-H(v^{\sharp_g},w_1,w_2)-(\kappa_{w_1}\triangle_1\omega)(w_2)+(\kappa_{w_2}\triangle_1\omega)(w_1)\\
        &=-H(v^{\sharp_g},w_1,w_2)-\omega(\kappa_{w_1}^{\sharp_g},w_2)+\omega(\kappa_{w_2}^{\sharp_g},w_1)\\
        &=-v(w_1)\beta(w_2)+v(w_2)\beta(w_1)-\chi_u(w_1,w_2)-\omega(\kappa_{w_1}^{\sharp_g},w_2)+\omega(\kappa_{w_2}^{\sharp_g},w_1),
    \end{align*} where we have used that $H(v^{\sharp_g})=v\wedge\beta+\chi_u$. Since $u\wedge v\wedge\dd v=0$ if and only if $\dd v(w_1,w_2)=0$ for all $w_1,w_2\in\Gamma(D_{uv})$, we conclude that: $$\chi_u(w_1,w_2)=-\omega(\kappa_{w_1}^{\sharp_g},w_2)+\omega(\kappa_{w_2}^{\sharp_g},w_1)$$ for all $w_1,w_2\in\Gamma(D_{uv})$.
\end{proof}

Using the aforementioned notion of integrable screen bundle together with the previous lemmata, we extract the following structural result.

\begin{thm}
\label{thm:transverse_LCK}
An oriented Lorentzian six-manifold admits a skew-torsion parallel spinor with an integrable screen bundle only if it admits a codimension-two foliation by locally conformally Kähler complex surfaces.
\end{thm}

\begin{proof}
Let $\eta$ be such a parallel spinor, with associated isotropic almost Hermitian class $[u,v,\omega,J]$. We first show that $\dd\omega=\beta\wedge\omega$. Let $w_1,w_2,w_3\in\Gamma(D_{uv})$. Then: \begin{align*}
        (H(w_1)\triangle_1\omega)(w_2,w_3)&=\iota_{w_3}\iota_{w_2}(H(w_1)\triangle_1\omega)\\
        &=\iota_{w_3}(-H(w_1,w_2)\triangle_1\omega+H(w_1)\triangle_1\omega(w_2))\\
        &=H(w_1,w_2)\triangle_1\omega(w_3)-H(w_1,w_3)\triangle_1\omega(w_2)\\
        &=H(w_1,w_2,Jw_3)+H(w_1,Jw_2,w_3),
    \end{align*} where we have used $\omega(w)^{\sharp_g}=Jw$. Hence: $$(\nabla^g_{w_1}\omega)(w_2,w_3)=\frac{1}{2}(H(w_1,w_2,Jw_3)+H(w_1,Jw_2,w_3)).$$
    
    We now compute: \begin{align*}
        \dd\omega(w_1,w_2,w_3)&=(\nabla^g_{w_1}\omega)(w_2,w_3)+(\nabla^g_{w_2}\omega)(w_3,w_1)+(\nabla^g_{w_3}\omega)(w_1,w_2)\\
        &=H(Jw_1,w_2,w_3)+H(w_1,Jw_2,w_3)+H(w_1,w_2,Jw_3).
    \end{align*}
    
    Using that $*_{uv}\omega=-\mu\omega$, $*_{uv}\beta=-\mu H^\perp$, and $\nu_{uv}\diamond\nu_{uv}=1$ on $D_{uv}$, we compute: $$H^\perp\diamond\omega=(-\mu*_{uv}\beta)\diamond(-\mu*_{uv}\omega)=(\beta\diamond\nu_{uv})\diamond(-\nu_{uv}\diamond\omega)=-\beta\diamond\omega.$$
    
    In particular, $H^\perp\triangle_1\omega=-\beta\wedge\omega$. Then: $$-(\beta\wedge\omega)(w_1)=\iota_{w_1} (H^\perp\triangle_1\omega)=-H^\perp(w_1)\triangle_1\omega-H^\perp\triangle_1\omega(w_1)$$ and: \begin{align*}
        (H^\perp(w_1)\triangle_1\omega)(w_2,w_3)&=H(w_1,w_2,Jw_3)+H(w_1,Jw_2,w_3),\\
        (H^\perp\triangle_1\omega(w_1))(w_2,w_3)&=H(Jw_1,w_2,w_3)
    \end{align*} for all $w_1,w_2,w_3\in\Gamma(D_{uv})$. Hence, we conclude that $\dd\omega=\beta\wedge\omega$ on $D_{uv}$.\medskip

    Now we show that $\beta$ is closed. From $\dd\omega=\beta\wedge\omega$ we obtain $\dd\beta\wedge\omega=0$. Moreover, from the proof of Lemma \ref{lemma:D_integrable} we get $\dd u=-H(u^{\sharp_g})=u\wedge\beta$, which in turn implies $\dd\beta\wedge u=0$. Write: $$\dd\beta=\gamma_0u\wedge v+u\wedge\gamma_u+v\wedge\gamma_v+\gamma^\perp,$$ where $\gamma_0$ is a function, $\gamma_u,\gamma_v\in\Gamma(D_{uv}^*)$ and $\gamma^\perp\in\Gamma(\wedge^2D_{uv}^*)$. The equation $\dd\beta\wedge u=0$ implies that $\gamma_v=0$ and $\gamma^\perp=0$. The equation $\dd\beta\wedge\omega=0$ implies that $\gamma_0=0$ and $\gamma_u\wedge\omega=0$. Since $\omega$ is non-degenerate on $D_{uv}$, there exists $w\in\Gamma(D_{uv})$ such that $\gamma_u=\iota_w\omega$. Hence: $$0=\gamma_u\wedge\omega=\iota_w\omega\wedge\omega=\tfrac{1}{2}\iota_w(\omega\wedge\omega)=-\mu\iota_w\nu_{uv}=-\mu*_{uv}w^\flat$$ if and only if $w=0$, i.e.\ $\gamma_u=0$. Therefore $\dd\beta=0$.\medskip

    Finally, we show that $J$ is integrable. From $\omega=g(J-,-)$ and the equation for $(\nabla^g_{w_1}\omega)(w_2,w_3)$ we get: $$g((\nabla^g_{w_1}J)w_2,w_3)=\frac{1}{2}(H(w_1,w_2,Jw_3)+H(w_1,Jw_2,w_3))$$ for all $w_1,w_2,w_2\in\Gamma(D_{uv})$. Using the standard formula for the Nijenhuis tensor $N_J$ in terms of the Levi-Civita connection $\nabla^g$, namely: $$N_J(w_1,w_2)=(\nabla^g_{w_1}J)Jw_2-(\nabla^g_{w_2}J)Jw_1+(\nabla^g_{Jw_1}J)w_2-(\nabla^g_{Jw_2}J)w_1,$$ we obtain: $$g(N_J(w_1,w_2),w_3)=H(Jw_1,Jw_2,w_3)+H(Jw_1,w_2,Jw_3)+H(w_1,Jw_2,Jw_3)-H(w_1,w_2,w_3).$$
    
    Now note that: $$\omega(\beta^{\sharp_g})\wedge\omega=\tfrac{1}{2}\iota_{\beta^{\sharp_g}}(\omega\wedge\omega)=-\mu*_{uv}\beta=H^\perp.$$
    
    Using $\omega(\beta^{\sharp_g},w)=g(J\beta^{\sharp_g},w)=-g(\beta^{\sharp_g},Jw)=-\beta(Jw)$ we compute: $$H(w_1,w_2,w_3)=H^\perp(w_1,w_2,w_3)=-\beta(Jw_1)\omega(w_2,w_3)-\beta(Jw_2)\omega(w_3,w_1)-\beta(Jw_3)\omega(w_1,w_2)$$ for all $w_1,w_2,w_3\in\Gamma(D_{uv})$. Plugging this into $g(N_J(w_1,w_2),w_3)$ we obtain that $N_J=0$ on $D_{uv}$. Hence, the almost complex structure $J$ on $D_{uv}$ is integrable.
\end{proof}


\subsection{The complex-bilinear square of a skew-torsion parallel spinor}


In this subsection, we compute, via Theorem \ref{thm:even_CB_CPS} with $4\fra_w=-H(w)\in\Omega^2(M)$, the necessary and sufficient conditions for a Lorentzian six-dimensional manifold $(M,g)$ to admit a skew-torsion parallel spinor with totally skew-symmetric torsion $H\in \Omega^3(M)$. Since we have already obtained the conditions that arise from the Hermitian square via Theorem \ref{thm:even_constrained_parallel_spinors}, for internal consistency across sections, it is convenient to first obtain the necessary and sufficient conditions for a Hermitian and complex-bilinear square to be the square of the \emph{same} spinor $\eta$.

\begin{lemma}
\label{lemma:squares_same_spinor_(5,1)}
The Hermitian $\widehat{\alpha}\in\Omega_\C(M)$ and complex-bilinear $\alpha\in\Omega_\C(M)$ squares are the squares of the same irreducible complex chiral spinor $\eta\in\Gamma(S^\mu)$ with chirality $\mu\in\Z_2$ if and only if: $$u=\frac{\mu}{4}*_g(\alpha(v^{\sharp_g})\wedge\overline{\alpha}),\qquad\omega=-\frac{i}{4}\alpha(v^{\sharp_g})\triangle_1\overline{\alpha}(v^{\sharp_g}),$$ where $v$ is a one-form conjugate to $u$.
\end{lemma}

\begin{proof}
In order to compute the Hermitian and the complex-bilinear squares of $\eta\in\Gamma(S)$, we needed to equip the bundle of irreducible complex Clifford modules $(S,\gamma)$ with a Hermitian pairing $\scS$ of positive adjoint type, and a skew-symmetric complex-bilinear pairing $\scB$, also of positive adjoint type. These two pairings are related through a quaternionic structure $K\in\Gamma(\End(S))$, namely $\scB(\eta_1,\eta_2)=\scS(\eta_1,K\eta_2)$ for all $\eta_1,\eta_2\in\Gamma(S)$.\medskip
    
Let $\widehat{\alpha}_\eta,\alpha_\eta\in\Omega_\C(M)$ and $\widehat{\alpha}_{K\eta},\alpha_{K\eta}\in\Omega_\C(M)$ denote the Hermitian and complex-bilinear squares of a chiral spinor $\eta\in\Gamma(S^\mu)$ and its conjugate $K\eta\in\Gamma(S^\mu)$, respectively. From Lemma \ref{lemma:(5,1)_squares} and \cite[Prop.\ 3.29]{Algebraic_Complex_Square_2025} we have that $\alpha_{K\eta}=\overline{\alpha}_\eta$. By \cite[Prop.\ 3.30]{Algebraic_Complex_Square_2025}, the Hermitian and complex-bilinear squares are the squares of the same spinor $\eta\in\Gamma(S^\mu)$ if and only if the following algebraic relations are satisfied: $$\widehat{\alpha}_\eta\diamond\beta\diamond\alpha_\eta=8(\widehat{\alpha}_\eta\diamond\beta)^{(0)}\alpha_\eta,\qquad\alpha_\eta\diamond\beta\diamond\alpha_{K\eta}=-8(\widehat{\alpha}_\eta\diamond\tau(\beta))^{(0)}\widehat{\alpha}_\eta$$ for a complex exterior form $\beta\in\Omega_\C(M)$ satisfying $(\widehat{\alpha}_\eta\diamond\beta)^{(0)}\neq0$. Taking $\beta=v$ a conjugate one-form of $u$ is sufficient since $(\widehat{\alpha}_\eta\diamond v)^{(0)}=\escal{u,v}_g=1$. We compute $\alpha_\eta\diamond v\diamond\alpha_{K\eta}$. First of all note that $\alpha_\eta\diamond\alpha_{K\eta}=0$. Indeed: $$\alpha_\eta\diamond\alpha_{K\eta}=(\mu*_g\alpha_\eta)\diamond(\mu*_g\alpha_{K\eta})=(-\alpha_\eta\diamond\nu_g)\diamond(\nu_g\diamond\alpha_{K\eta})=-\alpha_\eta\diamond\alpha_{K\eta},$$ where we have used $*_g\alpha_\eta=\mu\alpha_\eta$, $*_g\alpha_{K\eta}=\mu\alpha_{K\eta}$, $\nu_g\diamond\nu_g=1$, and Proposition \ref{prop:product_volume_form}. Then: $$\alpha_\eta\diamond v\diamond\alpha_{K\eta}=(2\alpha_\eta(v^{\sharp_g})-v\diamond\alpha_\eta)\diamond\alpha_{K\eta}=2(\alpha_\eta(v^{\sharp_g})\wedge\alpha_{K\eta}-\alpha_\eta(v^{\sharp_g})\triangle_1\alpha_{K\eta}-\alpha_\eta(v^{\sharp_g})\triangle_2\alpha_{K\eta}).$$

From $\alpha_\eta \diamond\beta\diamond\alpha_{K\eta} = -8(\widehat{\alpha}_\eta \diamond \tau(\beta))^{(0)} \widehat{\alpha}_\eta$ we obtain: $$u=\frac{1}{4}\alpha_\eta(v^{\sharp_g})\triangle_2\alpha_{K\eta},\qquad iu\wedge\omega=\frac{1}{4}\alpha_\eta(v^{\sharp_g})\triangle_1\alpha_{K\eta},\qquad\mu*_gu=\frac{1}{4}\alpha_\eta(v^{\sharp_g})\wedge\alpha_{K\eta}.$$

The first and last equations are equivalent. Taking the interior product with $v^{\sharp_g}$ on the second equation gives $i\omega=\frac{1}{4}\alpha_\eta(v^{\sharp_g})\triangle_1\alpha_{K\eta}(v^{\sharp_g})$. Therefore, we obtain the equations from the statement.
\end{proof}

As explained in detail in \cite{Algebraic_Complex_Square_2025}, the Hermitian square of a spinor $\eta$ contains \emph{less} information than its complex-bilinear square, which is equivalent to $\eta$ modulo a sign. Hence, it is natural to wonder about which extra piece of information we need to add to the Hermitian square $\widehat{\alpha}$ of $\eta$ to uniquely determine its complex-bilinear square. This is determined in the following lemma.

\begin{lemma}
\label{lemma:alpha=u_wedge_Omega}
Let $\widehat{\alpha}\in\Omega_\C(M)$ and $\alpha\in\Omega_\C(M)$ be, respectively, the Hermitian and complex-bilinear squares of the same irreducible complex chiral spinor $\eta\in\Gamma(S^\mu)$ with chirality $\mu\in\Z_2$. Then: $$\alpha=u\wedge\Omega,\qquad\Omega:=\alpha(v^{\sharp_g})\in\Omega_\C^2(M).$$

Two complex two-forms $\Omega$ and $\tilde{\Omega}$ are related if and only if: $$\tilde{\Omega}=\Omega-u\wedge\Omega(\mathfrak{w}^{\sharp_g})$$ for $\mathfrak{w}\in\Omega^1(M)$ orthogonal to $u$ and $v$.
\end{lemma}

\begin{proof}
Locally, the complex-bilinear square can be written as $\alpha=\theta_1\wedge\theta_2\wedge\theta_3$ for three isotropic and orthogonal complex local one-forms $\theta_1,\theta_2,\theta_3\in\Omega^1_\C(M)$. Set $t_i:=\escal{v,\theta_i}_g$ for $i=1,2,3$. Then: $$\alpha(v^{\sharp_g})=\iota_{v^{\sharp_g}}(\theta_1\wedge\theta_2\wedge\theta_3)=t_1\theta_2\wedge\theta_3-t_2\theta_1\wedge\theta_3+t_3\theta_1\wedge\theta_2.$$

Taking the wedge product with $\overline{\alpha}=\overline{\theta}_1\wedge\overline{\theta}_2\wedge\overline{\theta}_3$ yields: $$\alpha(v^{\sharp_g})\wedge\overline{\alpha}=t_1\theta_2\wedge\theta_3\wedge\overline{\theta}_1\wedge\overline{\theta}_2\wedge\overline{\theta}_3-t_2\theta_1\wedge\theta_3\wedge\overline{\theta}_1\wedge\overline{\theta}_2\wedge\overline{\theta}_3+t_3\theta_1\wedge\theta_2\wedge\overline{\theta}_1\wedge\overline{\theta}_2\wedge\overline{\theta}_3.$$

Hence $*_g(\alpha(v^{\sharp_g})\wedge\overline{\alpha})\in\mathrm{span}\{\theta_1,\theta_2,\theta_3\}$, which implies that $*_g(\alpha(v^{\sharp_g})\wedge\overline{\alpha})\wedge\alpha=0$. By Lemma \ref{lemma:squares_same_spinor_(5,1)} we have: $$u\wedge\alpha=\frac{\mu}{4}*_g(\alpha(v^{\sharp_g})\wedge\overline{\alpha})\wedge\alpha=0,$$ thus $0=\iota_{v^{\sharp_g}}(u\wedge\alpha)=\alpha-u\wedge\alpha(v^{\sharp_g})$, which yields the claimed result. Note also that: $$\alpha(u^{\sharp_g})=\iota_{u^{\sharp_g}}\alpha=\mu\iota_{u^{\sharp_g}}(*_g\alpha)=\mu*_g(\alpha\wedge u)=0$$ since $*_g\alpha=\mu\alpha$. Finally, a direct computation shows that if $\tilde{v} = v - \frac{1}{2}\vert\mathfrak{w}\vert_g^2 u + \mathfrak{w}$ for a unique $\mathfrak{w}\in\Omega^1(M)$ orthogonal to $u$ and $v$, then $\tilde{\Omega}=\Omega-u\wedge\Omega(\mathfrak{w}^{\sharp_g})$.
\end{proof}

By Lemma \ref{lemma:alpha=u_wedge_Omega}, we see that the pair $(\widehat{\alpha},\alpha)$ canonically determines a class of complex two-forms, and that every one-form $v$ conjugate to $u$ determines a complex two-form $\Omega$ on $D_{uv}$, which in turn defines a complex structure $J_\Omega$ on $D_{uv}$ since $\Omega$ is decomposable and satisfies $\Omega\wedge\overline{\Omega}=2\omega\wedge\omega\neq0$. Indeed, from Lemma \ref{lemma:squares_same_spinor_(5,1)} and equations \eqref{eq:Hodge_(u,v,omega)} we obtain: $$\Omega\wedge\overline{\Omega}=\alpha(v^{\sharp_g})\wedge\overline{\alpha}(v^{\sharp_g})=\iota_{v^{\sharp_g}}(\alpha(v^{\sharp_g})\wedge\overline{\alpha})=4\mu\iota_{\v^{\sharp_g}}(*_gu)=4\mu*_g(u\wedge v)=2\omega\wedge\omega.$$

Hence, $\Omega$ is a complex two-form of type $(2,0)$ with respect to the complex structure $J_\Omega$. Using $J_\Omega$ and the metric $g$ we define $\omega_\Omega:=g(J_\Omega-,-)$. Therefore, the tuple $(\omega_\Omega,J_\Omega,\Omega)$ defines an $\mathrm{SU}(2)$-structure on $D_{uv}$.

\begin{prop}
Let $(M,g)$ be an oriented Lorentzian six-manifold. An irreducible complex chiral spinor $\eta\in\Gamma(S^\mu)$ with chirality $\mu\in\Z_2$ satisfies $\nabla^{g,H}\eta=0$ if and only if: $$\nabla^g_wu=\frac{1}{2}H(w,u^{\sharp_g}),\qquad\nabla^g_w\Omega=\frac{1}{2}H(w)\triangle_1\Omega+\Theta_w\wedge u$$ for all $w\in\Gamma(TM)$ and some $\Theta\in\Omega^1(M,T_\C^*M)$.
\end{prop}

\begin{proof}
The complex-bilinear square $\alpha$ is computed with respect to a complex-bilinear pairing of positive adjoint type. Hence we compute: 
\begin{align*}
\fra_w\diamond\alpha&=\fra_w\wedge\alpha-\fra_w\triangle_1\alpha-\fra_w\triangle_2\alpha,\\
\alpha\diamond\tau(\fra_w)&=-\alpha\wedge\fra_w-\alpha\triangle_1\fra_w+\alpha\triangle_2\fra_w.
\end{align*}

By Corollary \ref{cor:spin_CB_CPS}, the spinor $\eta\in\Gamma(S)$ satisfies $\nabla^g_w\eta=\cA_w(\eta)$ if and only if $\nabla^g_w\alpha=\fra_w\diamond\alpha+\alpha\diamond\tau(\fra_w)$. This equation gives us: $$\nabla^g_w\alpha=-2\fra_w\triangle_1\alpha.$$

We now write $\alpha=u\wedge\Omega$ as in Lemma \ref{lemma:alpha=u_wedge_Omega}. A computation gives us: $$\fra_w\triangle_1(u\wedge\Omega)=\fra_w(u^{\sharp_g})\wedge\Omega+u\wedge(\fra_w\triangle_1\Omega).$$

Using that $\nabla^g_wu=-2\fra_w(u^{\sharp_g})$ we obtain: $$\nabla^g_w\Omega=-2\fra_w\triangle_1\Omega+\Theta_w\wedge u$$ for some $\Theta\in\Omega^1(M,T_\C^*M)$. Substituting $\fra_w=-\frac{1}{4}H(w)$ yields the equations from the statement.
\end{proof}
 
\noindent
The previous proposition, which guarantees the existence of the complex two-form $\Omega$ of type $(2,0)$, can be used to refine Theorem \ref{thm:transverse_LCK} in terms of a foliation by locally conformally Kähler surfaces with topologically trivial canonical bundle. This and other associated results on skew-torsion parallel spinors on Lorentzian six-manifolds will be presented in a separate publication.


\section{Supersymmetric self-dual bundle gerbes}
\label{sec:sugra6d}


In this section, we characterize the quasi-supersymmetric and supersymmetric solutions of minimal supergravity in six dimensions, whose bosonic sector naturally couples a Lorentzian metric on a six-dimensional oriented manifold $M$ to a self-dual curving on a bundle gerbe with connective structure defined on $M$.  


\subsection{Minimal six-dimensional supergravity}


In the following, the triple $(\cP,\cA,Y)$ will denote a bundle gerbe \cite{Murray1996}, with principal $\U(1)$-bundle $\cP\to Y \times_M Y$ over the fibered product of a submersion $Y\to M$, equipped with a fixed connective structure $\cA \in \Omega^1(\cP)$ and defined on the given six-dimensional manifold $M$. Given a curving $b\in \Omega^2(Y)$ on $(\cP,\cA,Y)$ we denote its curvature by $H_b\in \Omega^3(M)$.  A triple $(\cP,\cA,Y)$ determines a unique minimal six-dimensional bosonic supergravity on $M$, as prescribed in the following definition. For simplicity in the exposition, we will omit the term \emph{bosonic} in the following.
 
\begin{definition}
\label{label:NSNSsystem}
The \emph{six-dimensional minimal supergravity system} determined by $(\cP,\cA,Y)$ on $M$ is the following differential system:
\begin{equation}
\label{eq:NSNSsystem}
\Ric^g =  \frac{1}{2} H_b \circ_g H_b \, , \qquad \ast_g H_b = \mu H_b 
\end{equation}
 	
\noindent
for pairs $(g,b)$ consisting of a Lorentzian metric $g$ on $M$ and a curving $b\in \Omega^2(Y)$ on $(\cP,\cA,Y)$, where $H_b\in \Omega^3(M)$ is the curvature of $b$, and $\mu\in\Z_2$.
\end{definition}
 
\noindent
By the previous definition, we can consider the six-dimensional minimal supergravity system as a natural \emph{gauge-theoretic} system that couples Lorentzian metrics to curvings on a gerbe.

\begin{remark}
Solutions $(g,b)$ to equations \eqref{eq:NSNSsystem} will be called \emph{minimal supergravity solutions}. The first equation in \eqref{eq:NSNSsystem} is the so-called \emph{Einstein equation}, whereas the second equation in \eqref{eq:NSNSsystem} is called the self-duality condition for the curving $b$. Curvings on $(\cP,\cA,Y)$ are usually called \emph{b-fields}, a terminology that we will use occasionally. 
\end{remark}
 
\noindent
Given $(\cP,\cA,Y)$, the \emph{configuration space} of the NS-NS system on $(\cP,\cA,Y)$ is the set of pairs $(g,b)$ consisting of a Lorentzian metric $g$ on $M$, and a curving $b$ on $(\cP,\cA,Y)$. Given an element $(g,b)$ we denote by $\nabla^{g,b}$ the unique metric connection on $(M,g)$ with totally skew-symmetric torsion given by $H_b\in \Omega^3(M)$, the curvature of $b\in\Omega^2(Y)$. For ease of notation, we denote by the same symbol $\nabla^{g,b}$ its lift to any irreducible complex spinor bundle defined on $(M,g)$.

\begin{definition}
A \emph{supersymmetric configuration} on $(\cP,\cA,Y)$ is a pair $(g,b)$ consisting of a Lorentzian metric $g$ on $M$, and a curving $b\in \Omega^2(Y)$ satisfying the following differential system:
\begin{equation}
\label{eq:susyparallel}
\nabla^{g,b} \eta = 0
\end{equation}
	
\noindent
for a non-vanishing section $\eta\in\Gamma(S^{\mu})$ of an irreducible complex chiral spinor bundle $S^{\mu}$, of chirality $\mu\in \mathbb{Z}_2$, defined on $(M,g)$ and associated to a strong spin structure.
\end{definition}

\noindent
Therefore, the supersymmetry parameter of a supersymmetric configuration is a particular instance of an irreducible spinor parallel under a metric connection with totally skew-symmetric, closed, torsion.\medskip

In this section, we consider the quasi-supersymmetric and supersymmetric solutions of six-dimensional minimal supergravity whose underlying Lorentzian metric belongs to the following class of standard Kundt six-manifolds, which is especially well-adapted to the existence of an isotropic Killing vector field.

\begin{definition}
\label{def:standardKundt}
A six-dimensional oriented Lorentzian manifold $(M,g)$ is \emph{standard Kundt}, or \emph{Kundt} for short, if it has the following isometry type:
\begin{equation}
\label{eq:standardKundt}
M = \mathcal{I}_{\mathfrak{u}}\times \mathcal{I}_{\mathfrak{v}}\times N\, , \qquad g = \cH_{\mathfrak{u}} \dd \mathfrak{u} \otimes \dd \mathfrak{u} +  \dd \mathfrak{u} \odot (e^{\cF_{\mathfrak{u}}}\dd \mathfrak{v} + \cA_{\mathfrak{u}}) + h_{\mathfrak{u}},
\end{equation}

\noindent
where, $\mathcal{I}_{\mathfrak{u}}\times \mathcal{I}_{\mathfrak{v}}$ is a direct product with Cartesian coordinates $(\mathfrak{u},\mathfrak{v})$, $N$ is an oriented four-manifold, $\cH_{\mathfrak{u}}, \cF_{\mathfrak{u}} \in C^{\infty}(N)$ are families of smooth functions on $N$, $\cA_{\mathfrak{u}} \in \Omega^1(N)$ is a family of one-forms on $N$, and $h_{\mathfrak{u}}$ is a family of complete Riemannian metrics on $N$. 
\end{definition}

\noindent
Similarly to the standard Brinkmann case considered in Section \ref{sec:sugraspinors}, we say that a standard Kundt manifold is \emph{stationary} if all elements in the tuple $(\cH_{\mathfrak{u}}, \cF_{\mathfrak{u}} , \cA_{\fru} , h_{\fru} )$ are independent of $\fru$, in which case we will drop the subscript $\fru$. On the other hand, we refer to a standard Kundt manifold as \emph{non-twisting} if $\cA_{\fru}=0$.\medskip

\noindent
Studying the quasi-supersymmetric or supersymmetric systems of minimal six-dimensional for such a class of metrics yields an evolution differential system defined on the transverse wave front of the underlying Kundt six-manifold. We begin with the corresponding reduction of the self-dual bundle gerbe that defines the given six-dimensional minimal supergravity. 


\subsection{Self-dual bundle gerbes on standard Kundt six-manifolds}


In the following, we fix a Kundt six-manifold of the form \eqref{eq:standardKundt}. Let $(\bar{\cP} , \bar{A}, \bar{Y})$ be a bundle gerbe on $M = \cI_{\fru}\times \cI_{\frv} \times N$. Since the goal is to reduce the minimal supergravity system on $(\bar{\cP} , \bar{A}, \bar{Y})$, we will assume that the \emph{topological data} contained in $(\bar{\cP} , \bar{A}, \bar{Y})$, namely $(\bar{\cP},\bar{Y})$, is given by the pull-back of a bundle gerbe $(\cP,Y)$ defined over the four-dimensional manifold $N$, while we will allow the connective structure $\bar{A}$ to be genuinely $\fru$-\emph{dependent}, that is, not necessarily given by the pull-back of a fixed connective structure on $(\cP,Y)$. We will refer to such bundle gerbes as \emph{reducible} bundle gerbes on $\cI_{\fru}\times \cI_{\frv} \times N$. This setup captures the dynamics of the $\fru$-evolution problem in its full generality. Since the submersion $\bar{Y} \to M$ of a reducible bundle gerbe on $\cI_{\fru}\times \cI_{\frv} \times N$ is, by assumption, a pull-back submersion, it follows that it is isomorphic to:
\begin{equation*}
\bar{Y} = \cI_{\fru}\times \cI_{\frv} \times Y \to \cI_{\fru}\times \cI_{\frv} \times N,
\end{equation*}

\noindent
where the projection on the interval factors is the identity map and $Y\to N$ is the submersion underlying the bundle gerbe $(\cP , Y)$. Similarly:
\begin{equation*}
\bar{Y}\times_M  \bar{Y} = \cI_{\fru}\times \cI_{\frv}\times Y \times_N Y \to \cI_{\fru}\times \cI_{\frv} \times N,
\end{equation*}

\noindent
where the projection restricts to the identity map on $\cI_{\fru}\times \cI_{\frv}$. Since $\bar{\cP}$ is the pull-back of $\cP\to Y\times_{N} Y$ to $\cI_{\fru}\times \cI_{\frv} \times Y\times_{N} Y$, we have:
\begin{eqnarray*}
\bar{\cP}= \cI_{\fru}\times \cI_{\frv} \times \cP \to \cI_{\fru}\times \cI_{\frv} \times Y\times_{N} Y
\end{eqnarray*}

\noindent
again restricting to the identity on $\cI_{\fru}\times \cI_{\frv}$. The principal action of the group $\U(1)$ on $\bar{\cP} = \cI_{\fru}\times \cI_{\frv} \times \cP$ is the one induced by the principal $\U(1)$-action carried by $\cP$ and the trivial action on $\cI_{\fru}\times \cI_{\frv}$. Let $\bar{A} \in \Omega^1(\cI_{\fru}\times \cI_{\frv}\times \cP , \mathbb{R})$ be a connective structure on $\bar{\cP}$. Then we can write:
\begin{equation*}
\bar{A} = \psi_{\fru} \, \dd \fru + \phi_{\fru} \dd \frv  + A_{\fru},
\end{equation*}

\noindent
where $\psi_{\fru}$ and $\phi_{\fru}$ are families of functions on $\cP$ and $A_{\fru}$ is a family of connections on $\cP$. Recall that $\fru$ denotes the Cartesian coordinate on $\cI_{\fru}$. Since $\bar{A}$ is a connection on $\bar{\cP}$, it is in particular invariant under the $\U(1)$-action of the principal bundle $\bar{\cP}$, and since this action is ineffective on $\cI_{\fru}\times \cI_{\frv}$ we conclude that $\psi_{\fru}$ and $\phi_{\fru}$ define in fact families of invariant functions on $\cP$, whence they descend to families of invariant functions on $Y\times_{N} Y$ that we denote by the same symbol for ease of notation. Since $(\bar{\cP} , \bar{A}, \bar{Y})$ is a bundle gerbe, it comes equipped with an isomorphism:
\begin{equation*}
(\bar{\pi}^{\ast}_{12}\bar{\cP}\otimes\bar{\pi}^{\ast}_{23}\bar{\cP} , \bar{\pi}^{\ast}_{12}\bar{A}\otimes \bar{\pi}^{\ast}_{23}\bar{A})  \xrightarrow{\cong} (\bar{\pi}^{\ast}_{13} \bar{\cP} , \bar{\pi}^{\ast}_{13}\bar{A}),
\end{equation*}

\noindent
where $\bar{\pi}\colon \bar{Y} \to M$ and the notation $\bar{\pi}_{ij} \colon \bar{Y}\times_M \bar{Y}\times_M \bar{Y} \to \bar{Y}\times_M \bar{Y}$ forgets the entry labeled neither by $i$ nor $j$. This implies $\delta \psi_{\fru} = \delta \phi_{\fru} = 0$, where $\delta$ is the simplicial differential of the \u{C}ech nerve of $\pi\colon Y \to N$. Since the \u{C}ech complex for smooth functions is exact, it follows that both $\psi_{\fru}$ and $\phi_{\fru}$ descend through $\delta$ to a pair of families of functions on $Y$ that we denote again by $\psi_{\fru}$ and $\phi_{\fru}$, respectively. Hence, we obtain the following correspondence. 

\begin{prop}
There is a natural one-to-one correspondence between reducible bundle gerbes $(\bar{\cP},\bar{Y},\bar{A})$ on $M = \cI_{\fru}\times \cI_{\frv} \times N$ and tuples $(\cP,Y,A_{\fru} , \psi_{\fru} , \phi_{\fru})$, consisting of a bundle gerbe $(\cP,Y)$ on $N$, a pair of families of functions $\psi_{\fru}$ and $\phi_{\fru}$ on $Y$, and a family of connective structures $A_{\fru}$ on $(\cP,Y)$.
\end{prop} 

\noindent
We will refer to $(\cP,Y, A_{\fru} , \psi_{\fru} ,\phi_{\fru})$ as the \emph{reduction} of $(\bar{\cP},\bar{Y},\bar{A})$. A direct computation gives the following result.

\begin{lemma}
\label{lemma:gerbecurvaturedecomposition}
The curvature $F_{\bar{A}} \in \Omega^2(\cI_{\fru}\times \cI_{\frv}  \times Y\times_{N} Y)$ of $\bar{A} = \psi_{\fru} \dd \fru + \phi_{\fru} \dd \frv + A_{\fru}$ satisfies the following equation:
\begin{eqnarray*}
F_{\bar{A}} = \dd_{Y^{[2]}} \psi_{\fru} \wedge \dd \fru + \dd_{Y^{[2]}} \phi_{\fru} \wedge \dd \frv + \partial_{\fru}\phi_{\fru} \dd\fru \wedge \dd\frv + \dd \fru \wedge \partial_{\fru} A_{\fru} + F_{A_{\fru}},
\end{eqnarray*}

\noindent
where $\dd_{Y^{[2]}}\colon \Omega(Y\times_{N} Y) \to \Omega(Y\times_{N} Y) $ is the exterior derivative on $\Omega(Y\times_{N} Y)$ and $F_{A_{\fru}}$ is the family of curvatures of $A_{\fru}$.
\end{lemma}

\begin{remark}
Since the set of connections on $\cP$ is an affine space over the one-forms on $Y\times_N Y$, the one-form $\partial_{\fru} A_{\fru}$, which is a priori a one-form on $\cP$, is well-defined as a one-form on the base space $Y\times_{N} Y$.
\end{remark}

\noindent
A curving on $(\bar{\cP},\bar{Y},\bar{A})$ is by definition a two-form $\bar{b}\in \Omega^2(\bar{Y})$ satisfying:
\begin{equation*}
\bar{\delta} \bar{b} = F_{\bar{A}}.
\end{equation*}

\noindent 
Since $\bar{Y} = \cI_{\fru}\times \cI_{\frv}  \times Y$, we can canonically write every curving $\bar{b}$ on $(\bar{\cP},\bar{Y},\bar{A})$ as follows:
\begin{equation}
\label{eq:decompositionbarb}
\bar{b} = \mathfrak{b}_{\fru} \dd \fru \wedge \dd \frv + \dd \fru \wedge \varrho_{\fru} + \dd \frv \wedge \vartheta_{\fru}   + b_{\fru}
\end{equation}

\noindent
for uniquely determined families of functions $\mathfrak{b}_{\fru}$, one forms $\varrho_{\fru}$ and $\vartheta_{\fru}$, and two-forms $b_{\fru}$ on $Y$.

\begin{lemma}
\label{lemma:reduciblegerbe}
Let $(\bar{\cP},\bar{Y},\bar{A})$ be a reducible bundle gerbe with connective structure on $M = \cI_{\fru}\times \cI_{\frv} \times N$ and let  $(\cP,Y,A_{\fru} , \psi_{\fru} , \phi_{\fru})$ be its reduction on $N$. A two-form $\bar{b}\in \Omega^2(\bar{Y})$ is a curving on $(\bar{\cP},\bar{Y},\bar{A})$ if and only if, in the decomposition given in \eqref{eq:decompositionbarb}, $b_{\fru}$ is a family of curvings on $(\cP,Y,A_{\fru})$ and the triplet $(\mathfrak{b}_{\fru}, \varrho_{\fru}, \vartheta_{\fru})$  satisfies:
\begin{equation}
\label{eq:conditionat}
\delta b_{\fru} = F_{A_{\fru}}\, , \qquad \delta \varrho_{\fru} = \partial_{\fru} A_{\fru} - \dd_{Y^{[2]}} \psi_{\fru}\, , \qquad \delta \vartheta_{\fru} = - \dd_{Y^{[2]}}\phi_{\fru}\, , \qquad  \delta \mathfrak{b}_{\fru} =  \partial_{\fru} \phi_{\fru}
\end{equation}

\noindent
for every $(\fru , \frv)\in \cI_{\fru} \times \cI_{\frv}$, where $\delta$ is the simplicial differential of the \u{C}ech nerve of $\pi \colon Y \to N$.
\end{lemma}

\begin{proof}
Plugging Equation \eqref{eq:decompositionbarb} into $\bar{\delta}\bar{b}$ we have:
\begin{align*}
\bar{\delta}\bar{b} &= \bar{\delta}( \mathfrak{b}_{\fru} \dd \fru \wedge \dd \frv + \dd \fru \wedge \varrho_{\fru} + \dd \frv \wedge \vartheta_{\fru}   + b_{\fru}) \\
&= \delta\mathfrak{b}_{\fru} \dd \fru \wedge \dd \frv + \dd \fru \wedge \delta \varrho_{\fru} + \dd \frv \wedge \delta \vartheta_{\fru}   + \delta b_{\fru}.
\end{align*}
Using the expression for $F_{\bar{A}}$ obtained in Lemma \ref{lemma:gerbecurvaturedecomposition} and isolating by tensor type in $\bar{\delta}\bar{b}=F_{\bar{A}}$ we obtain \eqref{eq:conditionat} and thus we conclude.
\end{proof}

\begin{cor}
\label{cor:curvingreduction}
There is a natural one-to-one correspondence between curvings $\bar{b}$ on a reducible bundle gerbe $(\bar{\cP},\bar{Y},\bar{A})$ with reduction $(\cP,Y,A_{\fru} , \psi_{\fru} , \phi_{\fru})$ and tuples $(b_{\fru}, \mathfrak{b}_{\fru}, \varrho_{\fru}, \vartheta_{\fru})$ consisting of a family of curvings $b_{\fru}$ on $(\cP,Y,A_{\fru})$, families of one-forms $\varrho_{\fru}$ and $\vartheta_{\fru}$ on $Y$, and a family of functions $\mathfrak{b}_{\fru}$ on $Y$ satisfying Equation \eqref{eq:conditionat}.
\end{cor}

\begin{remark}
\label{remark:bdecomposition}
Note that the second, third, and fourth equations in \eqref{eq:conditionat} can be equivalently written as follows:
\begin{equation*}
\delta (\varrho_{\fru} + \dd_{Y} \psi_{\fru}) = \partial_{\fru} A_{\fru}   \, , \qquad  \delta (\vartheta_{\fru} + \dd_{Y} \phi_{\fru}) = 0 \, , \qquad   \mathfrak{b}_{\fru} =  \partial_{\fru} \phi_{\fru} + \pi^{\ast} f_{\fru}
\end{equation*}

\noindent
for a unique family of functions $f_{\fru}$ on $N$. Here we use the same symbols $\psi_\fru$ and $\phi_\fru$ for the families of functions of $Y$ whose simplicial differentials are the families of functions in Lemma \ref{lemma:reduciblegerbe}. In particular, from the second equation above, it follows that there exists a unique family of one-forms $\alpha_{\fru}$ on $N$ such that:
\begin{equation*}
\pi^{\ast} \alpha_{\fru} =   \vartheta_{\fru} + \dd_{Y} \phi_{\fru}.
\end{equation*}

\noindent
We will respectively refer to $f_{\fru}$ and $\alpha_{\fru}$ as the \emph{derived families of functions and one-forms} of the curving $\bar{b}$ on the reducible bundle gerbe $(\bar{\cP},\bar{Y},\bar{A})$ with reduction $(\cP,Y,A_{\fru} , \psi_{\fru} , \phi_{\fru})$.
\end{remark} 

\noindent
Let $\bar{b} = \mathfrak{b}_{\fru} \dd \fru \wedge \dd \frv + \dd \fru \wedge \varrho_{\fru} + \dd \frv \wedge \vartheta_{\fru}   + b_{\fru}$ be a curving on the reducible bundle gerbe $(\bar{\cP},\bar{Y},\bar{A})$ with reduction $(\cP,Y,A_{\fru} , \psi_{\fru} , \phi_{\fru})$, defined over the six-dimensional manifold $M = \cI_{\fru}\times \cI_{\frv} \times N$. By Remark \ref{remark:bdecomposition}, it follows that $\mathfrak{b}_{\fru}$ is completely determined by $\phi_{\fru}$ and $f_{\fru}$. Hence, we will refer to the reduction of a curving $\bar{b}\in \Omega^2(\bar{Y})$ as a tuple of the form $(f_{\fru},\varrho_{\fru},\vartheta_{\fru},b_{\fru})$. We have:
\begin{equation*}
\dd_{\bar{Y}}\bar{b} = \dd_Y b_{\fru} + \pi^{\ast}(\partial_{\fru}\alpha_{\fru} + \dd f_{\fru})\wedge \dd\fru \wedge \dd \frv + \dd \fru \wedge (\partial_{\fru} b_{\fru} - \dd_Y \varrho_{\fru})  - \dd\frv \wedge \pi^{\ast}\dd \alpha_{\fru},
\end{equation*}

\noindent
where $\dd_{\bar{Y}}$ is the exterior derivative on $\bar{Y}$ and $\dd_Y$ is the exterior derivative on $Y$.

\begin{lemma}
The family of two-forms $\partial_{\fru} b_{\fru} - \dd_Y \varrho_{\fru} \in \Omega^2(Y)$ satisfies $\delta (\partial_{\fru} b_{\fru} - \dd_Y \varrho_{\fru}) = 0$ for every $\fru \in \cI_{\fru}$, where $\delta$ denotes the simplicial differential of the \u{C}ech nerve of $\pi \colon Y\to N$. 
\end{lemma}

\begin{proof}
We compute:
\begin{equation*}
\delta (\partial_{\fru} b_{\fru} - \dd_Y \varrho_{\fru}) = \partial_{\fru} \delta b_{\fru} - \dd_{Y^{[2]}} \delta \varrho_{\fru} =  \partial_{\fru} F_{A_{\fru}}- \dd_{Y^{[2]}} (\partial_{\fru} A_{\fru} - \dd_{Y^{[2]}} \psi_{\fru}) = 0 \, ,
\end{equation*}

\noindent
where we have used Lemma \ref{lemma:reduciblegerbe} and the fact that $F_{A_{\fru}} = \dd_{Y^{[2]}} A_{\fru}$.
\end{proof}

\noindent
Since the \u{C}ech complex for smooth forms is exact, it follows from the previous lemma that there exists a unique family of two-forms $\Theta_{\fru} \in \Omega^2(N)$ on $N$ such that:
\begin{equation}
\label{eq:buevaluation}
\pi^{\ast} \Theta_{\fru} = \partial_{\fru} b_{\fru} - \dd_Y \varrho_{\fru} \, .
\end{equation}

\noindent
This leads us to introducing the following definition.
\begin{definition}
The family of two-forms $\Theta_{\fru}$ is the \emph{derived family of two-forms} of the curving $\bar{b}$ on the reducible bundle gerbe $(\bar{\cP},\bar{Y},\bar{A})$.
\end{definition}

\noindent
With the previous definition we have:
\begin{equation*}
\dd_{\bar{Y}}\bar{b} = \dd_Y b_{\fru} + \pi^{\ast}(\partial_{\fru}\alpha_{\fru} + \dd f_{\fru}) \wedge \dd\fru \wedge \dd \frv + \dd \fru \wedge \pi^{\ast}\Theta_{\fru} - \dd\frv \wedge \pi^{\ast}\dd \alpha_{\fru}\, .
\end{equation*}

\noindent
Since $\dd_{\bar{Y}} \bar{b} = \bar{\pi}^{\ast} H_{\bar{b}}$ for a uniquely determined three-form $H_{\bar{b}}\in \Omega^3(M)$, and $\dd_Y b_{\fru} = \pi^{\ast} H_{b_{\fru}}$, it follows that:
\begin{equation}
\label{eq:decompositionH}
H_{\bar{b}} =  H_{b_{\fru}} + (\partial_{\fru}\alpha_{\fru} + \dd f_{\fru}) \wedge \dd\fru \wedge \dd \frv + \dd \fru \wedge \Theta_{\fru} - \dd\frv \wedge  \dd \alpha_{\fru},
\end{equation}

\noindent
where $H_{b_{\fru}} \in \Omega^3(N)$ is the curvature of the family of curvings $b_{\fru}\in \Omega^2(Y)$. We conclude that every curving $\bar{b}$ on a reducible bundle gerbe is equivalent to a tuple $(b_{\fru}, \mathfrak{b}_{\fru}, \varrho_{\fru}, \vartheta_{\fru})$ as introduced in Lemma \ref{lemma:reduciblegerbe}, and this tuple defines in turn families of \emph{derived} one-forms and two-forms $\alpha_{\fru}$ and $\Theta_{\fru}$ on $N$ as described above. The fact that $H_{\bar{b}}\in \Omega^3(\cI_{\fru}\times \cI_{\frv} \times N)$ is closed translates into:
\begin{equation}
\label{eq:closedimplication}
\partial_{\fru} H_{b_{\fru}} = \dd \Theta_{\fru}\, , \qquad \dd H_{b_{\fru}} = 0 \, .
\end{equation}

\noindent
Since $b_{\fru}$ is a family of curvings on $(\cP, Y, A_{\fru})$, the second condition holds automatically.

\begin{remark}
Using the decomposition given in \eqref{eq:decompositionH}, we have:
\begin{equation*}
\vert H_{\bar{b}}\vert_g^2 = \vert H_{b_{\fru}}\vert_{h_\fru}^2  - e^{-2\cF_\fru} \vert \partial_{\fru}\alpha_{\fru} +\dd f_\fru \vert^2_{h_\fru}    -2 e^{-\cF_\fru} \langle \Theta_{\fru} , \dd \alpha_{\fru}\rangle_{h_\fru},
\end{equation*}

\noindent
where $\vert - \vert_g^2$ denotes the pseudo-norm defined by $g$ and $\vert - \vert_{h_\fru}^2$ denotes the norm defined by $h_\fru$.
\end{remark}
 
\noindent
Let $\alpha\in \Omega^k(N)$ be a $k$-form on $N$. A direct computation gives the following relations:
\begin{gather*}
\ast_g (u\wedge \alpha) = (-1)^{k+1} u\wedge \ast_{h_\fru}\alpha\, , \qquad \ast_g (v\wedge \alpha) = (-1)^{k} v\wedge \ast_{h_\fru}\alpha\, ,\\
\ast_g (u\wedge v \wedge \alpha) = -\ast_{h_\fru} \alpha\, , \qquad \ast_g \alpha = u\wedge v \wedge \ast_{h_\fru}\alpha\, ,
\end{gather*}

\noindent
where $u$ and $v$ are one-forms on $M = \cI_{\fru}\times \cI_{\frv}\times N$ in terms of which the six-dimensional Lorentzian metric $g$ is given by:
\begin{equation*}
g = u\odot v + h.
\end{equation*}

\noindent
For our standard Kundt metric we can take:
\begin{equation*}
u := \dd \fru \, , \qquad v := \frac{1}{2} \cH_{\fru} u + e^{\cF_{\fru}} \dd\frv + \cA_{\fru}.
\end{equation*}

\noindent
With these definitions, we have:
\begin{align*}
H_{\bar{b}} &= H_{b_{\fru}} + e^{-\cF_{\fru}} (\partial_{\fru}\alpha_{\fru} + \dd f_{\fru})  \wedge u \wedge v \\
&\quad + u\wedge \Big(\Theta_{\fru} + e^{-\cF_{\fru}} (\partial_{\fru}\alpha_{\fru} + \dd f_{\fru}) \wedge \cA_{\fru} + \frac{1}{2}\cH_{\fru} e^{-\cF_{\fru}}   \dd \alpha_{\fru} \Big) \\
&\quad + e^{-\cF_{\fru}} (\cA_{\fru} - v)\wedge \dd \alpha_{\fru}.
\end{align*}

Hence:
\begin{align*}
\ast_g H_{\bar{b}} & = u\wedge v \wedge \ast_{h_{\fru}}  \Big( H_{b_{\fru}} + e^{-\cF_{\fru}} \cA_{\fru}  \wedge \dd \alpha_{\fru}\Big) \\
&\quad - e^{-\cF_{\fru}} \ast_{h_{\fru}} (\partial_{\fru}\alpha_{\fru} + \dd f_{\fru}) - e^{-\cF_{\fru}}   v \wedge \ast_{h_{\fru}} \dd \alpha_{\fru}\\
&\quad - u\wedge \ast_{h_{\fru}} \Big( \Theta_{\fru} + e^{-\cF_{\fru}} (\partial_{\fru}\alpha_{\fru} + \dd f_{\fru}) \wedge \cA_{\fru} + \frac{1}{2} \cH_{\fru} e^{-\cF_{\fru}}   \dd \alpha_{\fru} \Big),
\end{align*}

\noindent
from which we immediately obtain the following result.

\begin{lemma}
\label{lemma:selfdualH}
A curving $\bar{b}$ on a reducible bundle gerbe $(\bar{\cP},\bar{Y},\bar{A})$ with reduction $(\cP,Y,A_{\fru} , \psi_{\fru} , \phi_{\fru})$ is $\mu$-self-dual, that is:
\begin{equation*}
\ast_g H_{\bar{b}} = \mu H_{\bar{b}}\, , \qquad \mu \in \mathbb{Z}_2\, ,
\end{equation*}

\noindent
if and only if:
\begin{eqnarray*}
& \ast_{h_{\fru}}  ( e^{\cF_{\fru}}  H_{b_{\fru}} +  \cA_{\fru}  \wedge \dd \alpha_{\fru}) = \mu  (\partial_{\fru}\alpha_{\fru} + \dd f_{\fru})  \, ,\qquad \ast_{h_{\fru}} \dd\alpha_{\fru} = \mu \dd\alpha_{\fru}\, ,\\
& \Theta_{\fru} + e^{-\cF_{\fru}} (\partial_{\fru}\alpha_{\fru} + \dd f_{\fru}) \wedge \cA_{\fru} + \cH_{\fru} e^{-\cF_{\fru}}   \dd \alpha_{\fru} = -\mu \ast_{h_{\fru}} ( \Theta_{\fru} + e^{-\cF_{\fru}} (\partial_{\fru}\alpha_{\fru} + \dd f_{\fru}) \wedge \cA_{\fru} )\, ,
\end{eqnarray*}

\noindent
where $f_{\fru}\in C^{\infty}(N)$, $\alpha_{\fru}\in \Omega^1(N)$ and $\Theta_{\fru}\in \Omega^2(N)$ are the derived forms of the reduction of $\bar{b}$.
\end{lemma}


\subsection{Minimal six-dimensional supergravity on a wave-front}


In this subsection, we reduce the differential system \eqref{eq:NSNSsystem} of minimal six-dimensional supergravity to the wave front of a stationary \emph{non-twisting} standard Kundt six-manifold. Recall that a stationary standard Kundt six-manifold is non-twisting if, using the notation of Definition \ref{def:standardKundt}, it has $\cA = 0$. The general twisting case is remarkably richer and is therefore reserved for a series of separate future publications. We begin by computing the six-dimensional Ricci and scalar curvatures of the Lorentzian metric: 
\begin{equation}
\label{eq:nontwistingmetric}
    g = \cH  \dd \mathfrak{u} \otimes \dd \mathfrak{u} +  e^{\cF} \dd \mathfrak{u} \odot \dd \mathfrak{v}  + h 
\end{equation}

\noindent
as a differential system on the underlying wave-front $(N,h)$. We first compute the non-zero Christoffel symbols of $g$:
\begin{eqnarray*}
& \Gamma^{\mathfrak{v}}_{\mathfrak{v}i}  = \frac{1}{2} \partial_i \cF \, , \qquad \Gamma^k_{\mathfrak{u}\mathfrak{v}} = -\frac{1}{2} e^{\cF} \partial^k \cF  \, , \qquad  \Gamma^{\mathfrak{v}}_{\mathfrak{u}i}  = \frac{1}{2} e^{-\cF} ( \partial_i \cH -  \cH  \partial_i \cF )\, ,\\
& \Gamma^k_{\mathfrak{u}\mathfrak{u}} =  - \frac{1}{2} \partial^k \cH  \, , \qquad \Gamma^{\mathfrak{u}}_{\mathfrak{u}i}  = \frac{1}{2} \partial_i \cF \, , \qquad  \Gamma^k_{ij}  =  (\Gamma^h)^k_{ij}\, .
\end{eqnarray*}
 
Here we have used the fact that the inverse metric of $g$ is given by: 
\begin{eqnarray*}
& g^{\mathfrak{u}\mathfrak{u}} = g^{\mathfrak{u}i} = 0 \, , \qquad g^{\mathfrak{u}\mathfrak{v}} = g^{\mathfrak{v}\mathfrak{u}} = e^{-\cF} \, , \qquad g^{ij} = h^{ij}\, ,\\
& g^{\mathfrak{v}i} = g^{i\mathfrak{v}} = 0 \, , \qquad g^{\mathfrak{v}\mathfrak{v}} = -e^{-2\cF} \cH\, .
\end{eqnarray*}

\noindent
In particular, we have:
\begin{eqnarray*} 
& \nabla^g_{\frv}\partial_{\fru}   =  \nabla^g_{\fru}\partial_{\frv} =    - \frac{1}{2} e^{\cF} \partial^k \cF \partial_k \, , \quad \nabla^g_{i}\partial_{\fru}   = \nabla^g_{\fru}\partial_{i}  =  \frac{1}{2} \partial_i \cF \partial_{\fru} + \frac{1}{2} e^{-\cF} ( \partial_i \cH -   \cH  \partial_i \cF ) \partial_{\frv}\, ,\\ 
& \nabla^g_{\fru}\partial_{\fru}  =   - \frac{1}{2} \partial^k \cH \partial_k\, , \quad \nabla^g_{\frv} \partial_{i} = \nabla^g_{i} \partial_{\frv} = \frac{1}{2}\partial_i \cF \partial_{\frv}\, , \quad  \nabla^g_{i}\partial_j =  \nabla^g_{j}\partial_i  = \nabla^h_{i}\partial_j\, .
\end{eqnarray*} 

\noindent
From these expressions, a direct computation gives the following non-zero components of the six-dimensional Ricci tensor:
\begin{equation}\label{eq:Riccitensor}
\begin{aligned}
\mathrm{Ric}^g(\partial_{\fru} , \partial_{\fru}) &= \frac{1}{2} \nabla^{h\ast} \dd\cH + \frac{1}{2} \langle \dd\cH, \dd\cF \rangle_h - \frac{1}{2} \cH \vert \dd\cF \vert^2_h \, ,\\
\mathrm{Ric}^g(\partial_{\fru} , \partial_{\frv}) &=  \frac{1}{2} e^{\cF} \left( \nabla^{h\ast}\dd \cF - \vert \dd\cF \vert^2_h \right) \, ,\\
\mathrm{Ric}^g\vert_{TN\times TN} &= \mathrm{Ric}^h - \nabla^h \dd \cF - \frac{1}{2} \dd \cF \otimes \dd \cF\, ,
\end{aligned}
\end{equation}

\noindent
where $\nabla^{h\ast}$ is the formal adjoint of the Levi-Civita connection $\nabla^h$ and $\mathrm{Ric}^h$ denotes the Ricci tensor of $h$. Taking the trace of the Ricci tensor $\mathrm{Ric}^g$ with $g$, we obtain the scalar curvature of $(M,g)$:
\begin{eqnarray*}
s_g = s_h + 2 \nabla^{h\ast}\dd \cF  - \frac{3}{2} \vert \dd\cF \vert^2_h \, ,
\end{eqnarray*}

\noindent
where $s_h$ is the scalar curvature of the wave-front manifold $(N,h)$.

\begin{prop}
\label{prop:wavefrontreduction}
Let $\bar{b}$ be a curving on a reducible bundle gerbe $(\bar{\cP},\bar{A},\bar{Y})$ on $M = \cI_{\fru}\times \cI_{\frv}\times N$ with reduction $(\cP,Y,A_{\fru} , \psi_{\fru} , \phi_{\fru})$, and let $g$ be a non-twisting, stationary, standard Kundt six-dimensional metric. Then, $(g,\bar{b})$ is a solution of six-dimensional minimal supergravity if and only if the following differential system:
\begin{eqnarray}
& \nabla^{h\ast} \dd\cH + \langle \dd\cH, \dd\cF \rangle_h -   \cH \vert \dd\cF \vert^2_h  =   \vert\Theta_{\fru} \vert^2_h - \cH \vert H_{b_{\fru}} \vert^2_h\, , \quad  \nabla^{h\ast}\dd \cF   = \vert \dd\cF \vert^2_h -   \vert H_{b_{\fru}} \vert^2_h\, ,\label{eq:wavefront1}\\
& \mathrm{Ric}^h - \nabla^h \dd \cF - \frac{1}{2} \dd \cF \otimes \dd \cF = \frac{1}{2} (2H_{b_{\fru}}\circ_h H_{b_{\fru}} - h \vert H_{b_{\fru}} \vert^2_h )\, ,\label{eq:wavefront2}\\
& \dd\alpha_{0} = 0\, ,\quad \ast_h\Theta_{\fru} = -\mu \Theta_{\fru}\, , \quad \dd (e^{\cF} \ast_h H_{b_{\fru}}) = 0\, ,\quad \langle \Theta_{\fru} , H_{b_{\fru}}(w) \rangle_h = 0\label{eq:wavefront3}
\end{eqnarray}

\noindent
for all $w\in\Gamma(TN)$, is satisfied by the components of $g$ and $\bar{b}$ in the splitting \eqref{eq:standardKundt} and \eqref{eq:decompositionH}.
\end{prop}

\begin{proof}
A pair $(g,\bar{b})$ is a solution of six-dimensional minimal supergravity on $(\bar{\cP},\bar{A},\bar{Y})$ if and only if:
\begin{equation*}
\mathrm{Ric}^g =\frac{1}{2} H_{\bar{b}} \circ_g H_{\bar{b}}\, , \qquad \ast_g H_{\bar{b}} = \mu H_{\bar{b}}\, .
\end{equation*}

\noindent
Since $g$ is by assumption a standard non-twisting Kundt six-dimensional metric on $\cI_{\fru}\times \cI_{\frv}\times N$ as given in \eqref{eq:nontwistingmetric}, its Ricci tensor is given in Equation \eqref{eq:Riccitensor}. We proceed then to decompose the symmetric tensor $H_{\bar{b}}\circ_g H_{\bar{b}}\in \Gamma(T^{\ast}M\odot T^{\ast}M)$ in the splitting given by $M= \cI_{\fru}\times \cI_{\frv} \times N$. First, we compute:
\begin{align*}
H_{\bar{b}}(\partial_{\fru}) &=  e^{-\cF_{\fru}} (\partial_{\fru}\alpha_{\fru} + \dd f_{\fru})   \wedge (\frac{1}{2} \cH u - v) +    \Theta_{\fru}\, ,\\
H_{\bar{b}}(\partial_{\frv}) &= (\partial_{\fru}\alpha_{\fru} + \dd f_{\fru})  \wedge u   - \dd \alpha_{\fru}\, ,\\
H_{\bar{b}}(w) &= H_{b_{\fru}}(w) + e^{-\cF} (\partial_{\fru}\alpha_{\fru}(w) + \dd f_{\fru}(w) ) u \wedge v \\
&\quad - u\wedge (\Theta_{\fru}(w)  + \frac{1}{2}\cH e^{-\cF}   \dd \alpha_{\fru}(w) ) + e^{-\cF}  v \wedge \dd \alpha_{\fru}(w)
\end{align*}

\noindent
for every $w\in\Gamma(TN)$. Using these expressions, a calculation gives:
\begin{equation}\label{eq:HHstandard}
\begin{aligned}
(H_{\bar{b}}\circ_g H_{\bar{b}})(\partial_{\fru} , \partial_{\fru}) &= \vert\Theta_{\fru}\vert^2_h - \cH e^{-2\cF} \vert \partial_{\fru} \alpha_{\fru} + \dd f_{\fru}\vert_h^2\, ,\\
(H_{\bar{b}}\circ_g H_{\bar{b}})(\partial_{\frv} , \partial_{\frv}) &= \vert \dd \alpha_{\fru} \vert^2_h\, ,\\
(H_{\bar{b}}\circ_g H_{\bar{b}})(\partial_{\fru} , \partial_{\frv}) &= - e^{-\cF} \vert \partial_{\fru} \alpha_{\fru} + \dd f_{\fru}\vert^2_h - \langle \dd\alpha_{\fru}, \Theta_{\fru} \rangle_h\, ,\\
(H_{\bar{b}}\circ_g H_{\bar{b}})(\partial_{\fru} , w) &= \langle \Theta_{\fru} , H_{b_{\fru}}(w) \rangle_h - e^{-\cF}\langle \partial_{\fru}\alpha_{\fru} + \dd f_{\fru} , \Theta_{\fru}(w) + \cH e^{-\cF} \dd\alpha_{\fru}(w)\rangle_h\, ,\\
(H_{\bar{b}}\circ_g H_{\bar{b}})(\partial_{\frv} , w) &= - \langle \dd\alpha_{\fru} , H_{b_{\fru}}(w) \rangle_h - e^{-\cF}\langle \partial_{\fru}\alpha_{\fru} + \dd f_{\fru}, \dd\alpha_{\fru}(w)\rangle_h\, ,\\
(H_{\bar{b}}\circ_g H_{\bar{b}})\vert_{TN\times TN} &= H_{b_{\fru}}\circ_h H_{b_{\fru}} - e^{-2\cF} (\partial_{\fru}\alpha_{\fru} + \dd f_{\fru} ) \circ_h (\partial_{\fru}\alpha_{\fru} + \dd f_{\fru})\\
&\quad - e^{-\cF} (\Theta_{\fru}\circ_h \dd \alpha_{\fru} + \dd \alpha_{\fru} \circ_h \Theta_{\fru}) - \cH e^{-2\cF} \dd \alpha_{\fru} \circ_h \dd \alpha_{\fru}\, .
\end{aligned}
\end{equation}

\noindent
Combining these equations with the previous decomposition of the Ricci tensor, it follows that the equation $\mathrm{Ric}^g = \frac{1}{2} H_{\bar{b}} \circ_g H_{\bar{b}}$ is equivalent to:
 \begin{eqnarray*}
& \frac{1}{2} \nabla^{h\ast} \dd\cH + \frac{1}{2} \langle \dd\cH, \dd\cF \rangle_h - \frac{1}{2} \cH \vert \dd\cF \vert^2_h  =  \frac{1}{2} (\vert\Theta_{\fru} \vert^2_h - \cH e^{-2\cF} \vert \partial_{\fru} \alpha_{\fru} + \dd f_{\fru}\vert_h^2)\, ,\\  
& \frac{1}{2} e^{\cF} \left( \nabla^{h\ast}\dd \cF - \vert \dd\cF \vert^2_h \right)  = - \frac{1}{2} e^{-\cF} \vert \partial_{\fru} \alpha_{\fru} + \dd f_{\fru} \vert^2_h\, ,\\
& \mathrm{Ric}^h - \nabla^h \dd \cF - \frac{1}{2} \dd \cF \otimes \dd \cF =\frac{1}{2} (H_{b_{\fru}}\circ_h H_{b_{\fru}} - e^{-2\cF} (\partial_{\fru}\alpha_{\fru} + \dd f_{\fru})\circ_h (\partial_{\fru}\alpha_{\fru} + \dd f_{\fru}))\, ,\\
&  \dd\alpha_{\fru} = 0\, ,\qquad \langle \Theta_{\fru} , H_{b_{\fru}}(w) \rangle_h = e^{-\cF}\langle \partial_{\fru}\alpha_{\fru} + \dd f_{\fru} , \Theta_{\fru}(w)  \rangle_h\, , \qquad w\in\Gamma(TN)\, .
\end{eqnarray*}

\noindent
On the other hand, by Lemma \ref{lemma:selfdualH} it follows that if $(g,\bar{b})$ is a solution to the Einstein equation $\mathrm{Ric}^g = \frac{1}{2} H_{\bar{b}} \circ_g H_{\bar{b}}$, then $H_{\bar{b}}$ is $\mu$-self-dual if and only if:
\begin{equation*}
\ast_h\Theta_{\fru} = -\mu \Theta_{\fru}\, , \qquad  \partial_{\fru}\alpha_{\fru} + \dd f_{\fru}= \mu e^{\cF} \ast_h H_{b_{\fru}}\, .
\end{equation*}

\noindent
The second equation is immediately solved by:
\begin{equation*}
\alpha_\fru:=\int_{0}^{\fru}(\mu e^{\cF} \ast_h H_{b_{\tau}}-\dd f_{\tau}) \dd \tau + \alpha_0
\end{equation*}

\noindent
for a $\fru$-independent one-form $\alpha_0\in \Omega^1(N)$. Such family of one-forms $\alpha_{\fru}$ is closed if and only if:
\begin{equation*}
\dd(e^{\cF} \ast_h H_{b_{\tau}}) = 0\, , \qquad \dd\alpha_0 = 0\, ,
\end{equation*}

\noindent
in which case it automatically satisfies $\partial_{\fru}\dd\alpha_{\fru} = 0$. Combining the previous two sets of equations, we conclude that $(g,\bar{b})$ is a solution to minimal six-dimensional supergravity if and only if the equations in the statement of the proposition hold. For this, we have used the following identity:
\begin{eqnarray*}
e^{-\cF}\langle \partial_{\fru}\alpha_{\fru} + \dd f_{\fru}, \Theta_{\fru}(w)  \rangle_h = \mu \langle \ast_h H_{b_{\fru}}, \Theta_{\fru}(w)  \rangle_h  = -\langle \ast_h H_{b_{\fru}}, \ast_h (\Theta_{\fru}\wedge w^{\flat_h})  \rangle_h = -\langle  H_{b_{\fru}}(w) ,  \Theta_{\fru}  \rangle_h
\end{eqnarray*}

\noindent
for all $w\in\Gamma(TN)$.
\end{proof}

\begin{prop}
\label{prop:onlycompactwave}
Let $\bar{b}$ be a curving on a reducible bundle gerbe $(\bar{\cP},\bar{A},\bar{Y})$ on $M = \cI_{\fru}\times \cI_{\frv}\times N$ with reduction $(\cP,Y,A_{\fru} , \psi_{\fru} , \phi_{\fru})$, and let $g$ be a non-twisting, stationary, standard Kundt six-dimensional metric. Then, $(g,\bar{b})$ is a solution of six-dimensional minimal supergravity only if:
\begin{equation*}
s_h = 2 \vert H_{b_{\fru}}\vert^2_h - \frac{1}{2}\vert\dd \cF\vert^2_h \, , \qquad \nabla^{h\ast}(e^{\cF} \dd (\cH e^{-\cF})) =   \vert \Theta_{\fru}\vert^2_h \, ,
\end{equation*}

\noindent
where $s_h$ is the scalar curvature of $h$.
\end{prop}

\begin{proof}
Taking the trace of Equation \eqref{eq:wavefront2}, we obtain:    
\begin{equation*}
s_h + \nabla^{h\ast}\dd\cF - \frac{1}{2}\vert\dd \cF\vert^2_h =   \vert H_{b_{\fru}}\vert^2_h \, ,
\end{equation*}

\noindent
which, combined with the second equation in \eqref{eq:wavefront1}, gives $s_h = 2 \vert H_{b_{\fru}}\vert^2_h - \frac{1}{2}\vert\dd \cF\vert^2_h$. On the other hand, combining both equations in \eqref{eq:wavefront1} to eliminate $\vert H_{b_{\fru}}\vert^2_h$ and $\vert \dd\cF \vert^2_h$, we obtain:
\begin{equation*}
0 = \nabla^{h\ast} \dd\cH - \cH \nabla^{h\ast} \dd\cF  + \langle \dd\cH, \dd\cF \rangle_h  -   \vert\Theta_{\fru} \vert^2_h = \nabla^{h\ast} (\dd\cH - \cH  \dd\cF)   -   \vert\Theta_{\fru} \vert^2_h \, ,
\end{equation*}

\noindent
which gives the second equation in the statement since $\dd\cH - \cH \dd\cF = e^{\cF} \dd (\cH e^{-\cF})$.
\end{proof}

\begin{remark}
\label{remark:tuples}
For simplicity in the exposition, we will refer to tuples $(h,\cH,\cF, f_{\fru} , \varrho_{\fru}, \vartheta_{\fru} , b_{\fru})$ satisfying the system of Proposition \ref{prop:wavefrontreduction} as solutions of minimal six-dimensional supergravity on $(\cP,Y,A_{\fru} , \psi_{\fru} , \phi_{\fru})$.
\end{remark}

\begin{lemma}
\label{lemma:Ncompact}
Let $(h,\cH,\cF, f_{\fru} , \varrho_{\fru}, \vartheta_{\fru} , b_{\fru})$ be a solution of minimal six-dimensional supergravity on $(\cP,Y,A_{\fru} , \psi_{\fru} , \phi_{\fru})$, defined over a compact four-manifold $N$ without boundary. Then its derived family of two-forms is zero, that is $\Theta_{\fru}=0$, and consequently $b_{\fru}$ is flat and $\cF$ and $\cH$ are constant.
\end{lemma}

\begin{proof}
Integrating the second equation in Proposition \ref{prop:onlycompactwave} and using the divergence theorem we get $\Theta_{\fru}=0$. On the other hand, the second equation in \eqref{eq:wavefront1} can be written as $\nabla^{h*}\dd e^\cF = -e^{\cF}\vert H_{b_{\fru}} \vert^2_h$. Integrating again we obtain $H_{b_{\fru}}=0$, so $b_{\fru}$ is flat. Moreover, since $N$ is compact, the only harmonic functions are constant, thus $e^\cF$ is constant, which implies that $\cF$ is also constant. Finally, we obtain that $\cH$ is constant from the first equation in \eqref{eq:wavefront1}.
\end{proof}

\noindent
Hence, if $N$ is compact without boundary then the reduced differential system \eqref{eq:wavefront1}, \eqref{eq:wavefront2}, and \eqref{eq:wavefront3} reduces to the Ricci-flatness condition on $(N,h)$.

\begin{example}
\label{ep:susyexample}
As an example of solution to Proposition \ref{prop:wavefrontreduction}, let us consider $\Theta_{\fru} = 0$, $\cH=e^{\cF}$ and a globally conformally flat ansatz for the transverse four-manifold $N = \mathbb{R}^4 \setminus \{0\} = \mathbb{R}_{>0}\times S^3$ with metric:
\begin{equation*}
h = e^{U}(\dd r^2 + r^2 h_{S^3}),
\end{equation*}

\noindent
where $U = U(r)$ is a smooth function depending only on the radial coordinate $r$, and $h_{S^3}$ is the standard round metric on the unit three-sphere $S^3$. By spherical symmetry, we assume the three-form gerbe flux is proportional to the Riemannian volume form $\nu_{S^3}$  on the standard sphere:
\begin{equation*}
H_b = \mathfrak{e} \nu_{S^3},
\end{equation*}

\noindent
where $\mathfrak{e}\in\R$ is a constant magnetic charge. The Hodge dual of the flux with respect to $h$ is given by:
\begin{equation*}
\ast_h H_b = \mathfrak{e} e^{-U} r^{-3} \dd r.
\end{equation*}

\noindent
Substituting this expression into the Maxwell equation in Proposition \ref{prop:wavefrontreduction} (third equation in \eqref{eq:wavefront3}) yields:
\begin{equation*}
\dd(e^{\cF} \ast_h H_b) = \dd(\mathfrak{e} e^{\cF - U} r^{-3} \dd r) = 0,
\end{equation*}

\noindent
which is automatically satisfied for any choice of radial functions $\cF(r)$ and $U(r)$. To continue further, we impose the BPS-like constraint linking the longitudinal wave profile to the transverse conformal factor:
\begin{equation*}
\cF(r) = - U(r).
\end{equation*}

\noindent
Under this assumption, the squared norm of the flux becomes $\vert H_b \vert^2_h = \mathfrak{e}^2 e^{-3U} r^{-6}$. Substituting $\cF = -U$ into the non-linear dilaton and Einstein equations in Proposition \ref{prop:wavefrontreduction} (second equation in \eqref{eq:wavefront1} and equation \eqref{eq:wavefront2}, respectively) causes the non-linear terms to cancel, and the coupled system reduces entirely to the Laplace equation on the flat Euclidean background $\mathbb{R}^4$ together with a constraint on the integration constants:
\begin{equation*}
\partial_r^2 e^{U} + \frac{3}{r} \partial_r e^{U} = 0,\qquad \mathfrak{e}  = 2\mathfrak{m}.
\end{equation*}

\noindent
The general spherically symmetric solution to this equation is the harmonic function:
\begin{equation*}
e^{U(r)} = 1 + \frac{\mathfrak{m}}{r^2},
\end{equation*}

\noindent
for a positive constant $\mathfrak{m}\in\mathbb{R}_{>0}$. The Lorentzian metric on $M$ is given by $g = e^{\cF}(\dd\fru\otimes \dd\fru +  \dd \fru\odot \dd \frv) + h$. However, by the change of coordinates $\frv\mapsto\frv+\frac{1}{2}\fru$, we can get rid of the term $\dd\fru\otimes\dd\fru$. Thus, we obtain the solution:
\begin{equation*}
g = \Big(1+\frac{\mathfrak{m}}{r^2}\Big)^{-1} \dd \fru\odot \dd \frv + \Big(1+\frac{\mathfrak{m}}{r^2}\Big)(\dd r^2 + r^2 h_{S^3}),\qquad  H_b = 2\mathfrak{m}\nu_{S^3},
\end{equation*}

\noindent
which represents a \emph{black brane} in six Lorentzian dimensions and recovers a well-known supersymmetric solution of minimal six-dimensional supergravity \cite{Duff_Lu_1994}.
\end{example}

For the remainder of the article, we consider the case $\Theta_{\fru} = 0$, in which case the system in Proposition \ref{prop:wavefrontreduction} reduces to the following differential system on $(\cP,Y,A_{\fru} , \psi_{\fru} , \phi_{\fru})$:
\begin{eqnarray}
& \nabla^{h\ast} \dd\cH + \langle \dd\cH, \dd\cF \rangle_h -   \cH \vert \dd\cF \vert^2_h  +   \cH \vert H_{b_{\fru}} \vert^2_h = 0\, , \quad  \nabla^{h\ast}\dd \cF   = \vert \dd\cF \vert^2_h -   \vert H_{b_{\fru}} \vert^2_h\, ,\label{eq:wavefront01}\\
& \mathrm{Ric}^h - \nabla^h \dd \cF - \frac{1}{2} \dd \cF \otimes \dd \cF = \frac{1}{2} (2H_{b_{\fru}}\circ_h H_{b_{\fru}} - h \vert H_{b_{\fru}} \vert^2_h )\, ,\label{eq:wavefront02}\\
& \dd\alpha_{0} = 0\, ,\  \quad \dd (e^{\cF} \ast_h H_{b_{\fru}}) = 0\, .\label{eq:wavefront03}
\end{eqnarray}

To proceed further, it is convenient to perform the following conformal transformation of the metric, as implemented in the following lemma.

\begin{lemma}
\label{lemma:conformaltrans}
A tuple $(h,\cH,\cF, f_{\fru} , \varrho_{\fru}, \vartheta_{\fru} , b_{\fru})$ satisfies the system in Proposition \ref{prop:wavefrontreduction} if and only if $(\mathfrak{h} = e^{\cF}h,\cH,\cF, f_{\fru} , \varrho_{\fru}, \vartheta_{\fru} , b_{\fru})$ satisfies:
\begin{eqnarray}
& \nabla^{\mathfrak{h}\ast} \dd\cH + 2\langle \dd\cH, \dd\cF \rangle_{\mathfrak{h}} -   \cH \vert \dd\cF \vert^2_{\mathfrak{h}}  +  \cH e^{2\cF} \vert H_{b_{\fru}} \vert^2_{\mathfrak{h}} = 0\, , \quad  \nabla^{\mathfrak{h}\ast}\dd \cF  +   e^{2\cF} \vert H_{b_{\fru}} \vert^2_{\mathfrak{h}} = 0\, ,\label{eq:wavefront001}\\
& \mathrm{Ric}^{\mathfrak{h}} = \dd\cF \otimes \dd\cF +   e^{2\cF} (H_{b_{\fru}}\circ_{\mathfrak{h}} H_{b_{\fru}} -  \vert H_{b_{\fru}} \vert^2_{\mathfrak{h}}\mathfrak{h} )\, , \quad \dd\alpha_{0} = 0\, ,  \quad \dd (e^{2\cF} \ast_{\mathfrak{h}} H_{b_{\fru}}) = 0\, .\label{eq:wavefront002}
\end{eqnarray}
\end{lemma}

\begin{proof}
Setting $\mathfrak{h} = e^{\cF} h$, we compute:
\begin{equation*}
\nabla^{h} \dd\cH = \nabla^{\mathfrak{h}}\dd \cH + \frac{1}{2} (\dd\cF \otimes \dd\cH + \dd\cH \otimes \dd\cF) - \frac{1}{2} \langle \dd\cF ,\dd\cH\rangle_{\mathfrak{h}} \mathfrak{h}\, ,
\end{equation*}
from which we obtain:
\begin{equation*}
\nabla^{h\ast}\dd\cH = e^{\cF} (\nabla^{\mathfrak{h}\ast}\dd\cH  + \langle \dd\cF , \dd\cH \rangle_{\mathfrak{h}})\, , \qquad \nabla^{h\ast}\dd\cF = e^{\cF} (\nabla^{\mathfrak{h}\ast}\dd\cF  + \vert \dd\cF\vert^2_{\mathfrak{h}})\, .
\end{equation*}

\noindent
Plugging these expressions into the equations in \eqref{eq:wavefront01}, we obtain both equations in \eqref{eq:wavefront001}. On the other hand, a computation gives the following formulas:
\begin{equation*}
\mathrm{Ric}^{h} =  \mathrm{Ric}^{\mathfrak{h}} + \nabla^{\mathfrak{h}}\dd\cF + \frac{1}{2} \dd\cF \otimes \dd\cF - \frac{1}{2} (\nabla^{\mathfrak{h}\ast}\dd\cF + \vert \dd\cF\vert^2_{\mathfrak{h}}) \mathfrak{h}\, , \qquad \ast_h H_{b_{\fru}} = e^{\cF}  \ast_{\mathfrak{h}} H_{b_{\fru}}\, .
\end{equation*}

\noindent
Plugging these expressions into the equations in \eqref{eq:wavefront02} and \eqref{eq:wavefront03}, we obtain both equations in \eqref{eq:wavefront002} after noticing that $H_{b_{\fru}}\circ_h H_{b_{\fru}} =  e^{2\cF}   H_{b_{\fru}}\circ_{\mathfrak{h}} H_{b_{\fru}}$.
\end{proof}

\noindent
We now proceed to decouple the system comprising equations \eqref{eq:wavefront01}, \eqref{eq:wavefront02}, and \eqref{eq:wavefront03}, thereby providing a clear procedure for constructing standard Kundt solutions in minimal six-dimensional supergravity with a vanishing derived family of two-forms.

\begin{thm}
\label{thm:finalreducedsystem}
Every $\fru$-independent and non-twisting standard Kundt solution of minimal six-dimen\-sional supergravity with vanishing derived family of two-forms $\Theta_{\fru} = 0$ on $\mathbb{R}^2\times N$ determines a solution $(\mathfrak{h},b,\cF)$ of the system:
\begin{equation}\label{eq:finalreducedsystem}
\begin{gathered}
\mathrm{Ric}^{\mathfrak{h}} = \dd\cF \otimes \dd\cF +  e^{2\cF} (H_b\circ_{\mathfrak{h}} H_b -  \vert H_b \vert^2_{\mathfrak{h}}\mathfrak{h} )\, ,\\
\nabla^{\mathfrak{h}\ast}\dd \cF  + e^{2\cF} \vert H_b \vert^2_{\mathfrak{h}} = 0\, ,\\
\dd (e^{2\cF} \ast_{\mathfrak{h}} H_b) = 0
\end{gathered}
\end{equation}

\noindent
on $(\cP,Y,A)$ together with a harmonic function $\bar{\cH}$ on $(N,\mathfrak{h})$. Conversely, every tuple $(\mathfrak{h},b,\cF,\bar{\cH})$ consisting of a solution $(h,b,\cF)$ to the previous system together with a harmonic function $\bar{\cH}$ on $(N,\mathfrak{h})$, defines a natural solution to minimal six-dimensional supergravity on $\mathbb{R}^2 \times N$ with vanishing derived family of two-forms, Lorentzian metric: 
\begin{equation*}
g =  \bar{\cH} e^{\cF} \dd\fru\otimes \dd\fru +  e^{\cF} \dd \fru\odot \dd \frv + e^{-\cF} \mathfrak{h}\, ,
\end{equation*}
and curving curvature $H_{\bar{b}} = H_b + \mu e^{2\cF} \dd\fru \wedge \dd\frv \wedge \ast_{\mathfrak{h}} H_b$.
\end{thm}

\begin{remark}
The differential system \eqref{eq:finalreducedsystem} has a strong resemblance to the generalized Ricci soliton system \cite{MarioStreets}, albeit being different. It would be interesting to explore whether some of the techniques used to study the generalized Ricci soliton system, see for instance \cite{StreetsI,StreetsUstinovskiy,StreetsUstinovskiyII}, are also applicable to \eqref{eq:finalreducedsystem}.
\end{remark}

\begin{proof}
Let $(g,\bar{b})$ be a $\fru$-independent non-twisting standard Kundt solution of minimal six-dimensional supergravity on the reducible bundle gerbe $(\bar{\cP},\bar{Y},\bar{A})$ defined on $\mathbb{R}^2\times N$, with reduction $(\cP,Y,A_{\fru} , \psi_{\fru} , \phi_{\fru})$. Then, by Proposition \ref{prop:wavefrontreduction} together with Remark \ref{remark:tuples}, $(g,\bar{b})$  defines a tuple $(h,\cH,\cF, f_{\fru} , \varrho_{\fru}, \vartheta_{\fru} , b_{\fru})$ satisfying the differential system \eqref{eq:wavefront01}, \eqref{eq:wavefront02}, and \eqref{eq:wavefront03}. In turn, by Lemma \ref{lemma:conformaltrans}, the associated tuple $(\mathfrak{h} = e^{\cF}h,\cH,\cF, f_{\fru} , \varrho_{\fru}, \vartheta_{\fru} , b_{\fru})$ satisfies the differential system \eqref{eq:wavefront001} and \eqref{eq:wavefront002}.  Then, it follows by Equation \eqref{eq:buevaluation} that the family of curvings $b_{\fru}$ can be written as follows:
\begin{equation*}
b_{\fru} = b + \dd_Y k_{\fru}
\end{equation*}
for a given curving $b\in \Omega^2(Y)$ and a family of one-forms $k_{\fru}\in C^{\infty}(Y)$. Hence, there exists a $\fru$-dependent family of gauge transformations of $(\cP,Y,A_{\fru})$ such that the transformed curving, denoted by $b$, is $\fru$-independent. We define the new variable:
\begin{equation*}
\cH = \bar{\cH} e^{\cF}\, .
\end{equation*}

\noindent
Plugging this expression into the first equation of \eqref{eq:wavefront001} we obtain:
\begin{equation*}
\nabla^{\mathfrak{h}\ast} \dd \bar{\cH} + \bar{\cH}(\nabla^{\mathfrak{h}\ast}\dd \cF  +   e^{2\cF} \vert H_{b} \vert^2_{\mathfrak{h}}) = 0\, .
\end{equation*}

\noindent
Combining this equation with the second equation in \eqref{eq:wavefront001}, we obtain $\nabla^{\mathfrak{h}\ast} \dd \bar{\cH} = 0$, and thus the Laplacian equations for $\bar{\cH}$ and $\cF$ decouple. In particular, $\bar{\cH}$ is harmonic on $(N,\mathfrak{h})$ and $(\mathfrak{h},b, \cF)$ satisfies the differential system \eqref{eq:finalreducedsystem}.\medskip

Conversely, let $(\mathfrak{h},b,\cF)$ be a solution of the differential system \eqref{eq:finalreducedsystem} on $N$ for some bundle gerbe $(\cP,Y,A)$, whose pull-back to $\mathbb{R}^2\times N$ will be denoted $(\bar{\cP}, \bar{Y}, \bar{A})$, and let $\bar{\cH}$ be a harmonic function on $(N,\mathfrak{h})$. We set:
\begin{equation*}
g = \bar{\cH} e^{\cF} \dd\fru \otimes \dd\fru  + e^{\cF} \dd\fru \odot \dd \fru + e^{-\cF} \mathfrak{h}\, , \qquad \bar{b} = b + \fru \, \dd\frv \wedge \pi^{\ast}(\mu e^{2\cF} \ast_{\mathfrak{h}} H_b)\, ,
\end{equation*}

\noindent
which implies that the curvature of $\bar{b}$ is given by:
\begin{equation*}
H_{\bar{b}} = H_b + \dd\fru \wedge \dd\frv \wedge \mu e^{2\cF} \ast_{\mathfrak{h}} H_b\, .
\end{equation*}

\noindent
Tracing back the computations that lead to the reduced system \eqref{eq:finalreducedsystem} shows that such a choice of $g$ and $\bar{b}$ is, by construction, a non-twisting solution of minimal six-dimensional supergravity.
\end{proof}


\subsection{Families of solutions to six-dimensional minimal supergravity}
\label{eq:familiessolutions}


To go beyond the supersymmetric Example \ref{ep:susyexample}, we consider $N = \cI\times X$ for an interval $\cI$ with Cartesian coordinate $r$, where $X$ is an oriented three-manifold equipped with an Einstein metric $h_X$ with Einstein constant $\lambda\in\R$. We assume that $h$ is given by the following warped metric:
\begin{equation*}
\mathfrak{h} =  e^{K}(\dd r^2 +  h_X),
\end{equation*}

\noindent
where $K := K(r)$ is a smooth function of the \emph{radial} coordinate $r$. We assume the three-form gerbe curvature is proportional to the volume form of $X$, that is:
\begin{equation*}
H_b = \mathfrak{e} \nu_{h_X},
\end{equation*}

\noindent
where $\nu_{h_X}$ is the Riemannian volume form on $(X,h_X)$ and $\mathfrak{e}\in\R$. The Hodge dual of the flux with respect to $\mathfrak{h}$ takes the form:
\begin{equation*}
\ast_{\mathfrak{h}} H_b = \mathfrak{e} e^{-K} \dd r   
\end{equation*}

\noindent
and thus the Maxwell equation $\dd(e^{2\cF} \ast_{\mathfrak{h}} H_b) = 0$ in \eqref{eq:finalreducedsystem} is trivially satisfied for any choice of smooth functions $K(r)$ and $\cF(r)$. Substituting this ansatz into the remaining equations of reduced system \eqref{eq:finalreducedsystem}, and noting that for the aforementioned choice of three-form $H_b$ on a four-dimensional manifold we have the algebraic identity $H_b \circ_{\mathfrak{h}} H_b = |H_b|^2_{\mathfrak{h}} \mathfrak{h}$ on the transverse space $X$, the Einstein and dilaton equations reduce to the following set of coupled ordinary differential equations:
\begin{equation*}
\cF'' + K' \cF' = \mathfrak{e}^2 e^{2(\cF - K)} \, , \qquad 3 K'' + 2 (\cF')^2  =    2\mathfrak{e}^2 e^{2(\cF - K)} \, , \quad   K'' + (K')^2 = 2\lambda\, ,
\end{equation*}
where the prime denotes differentiation with respect to $r$. This system can be recast into the following coupled system of ordinary differential equations for $\cF$ and $K$:
\begin{equation} 
\label{eq:radialODE}
\cF'' + K' \cF' = \mathfrak{e}^2 e^{2(\cF - K)} \, ,   \qquad   K'' + (K')^2 = 2\lambda
\end{equation}
together with the \emph{Hamiltonian constraint}:
\begin{equation}
\label{eq:radialhamiltonian}
C := 3(K')^2 - 2(\mathcal{F}')^2 + 2\mathfrak{e}^2 e^{2(\mathcal{F}-K)} - 6\lambda = 0\, .
\end{equation}

\noindent
Assuming that $\cF$ and $K$ satisfy the \emph{evolution equations} \eqref{eq:radialODE}, it follows that the derivative of the Hamiltonian constraint \eqref{eq:radialhamiltonian} satisfies:
\begin{equation*}
C^{\prime} = -2 K^{\prime} C\, .
\end{equation*}

\noindent
Hence, if $C \vert_{r=0} = 0$ at $r=0$, it then follows that $C= 0$ everywhere $\cF$ and $K$ are defined, and thus we have a well-defined initial data problem. The evolution equation for $K$ in \eqref{eq:radialODE} clearly decouples. Defining $Z = e^{K}$, the equation $K'' + (K')^2 = 2\lambda$ reduces to the following linear second-order ordinary differential equation:
\begin{equation}
\label{eq:linearwarp}
Z''  = 2\lambda Z\, .
\end{equation}

\noindent
Plugging a solution to this equation into the remaining equation in \eqref{eq:radialODE}, gives an explicit ordinary differential equation for $\cF$, which can be studied by specific methods depending on the precise solution chosen for $Z$, which depends on the sign of $\lambda$. Finally, the harmonic function $\bar{\cH}$ on $(N, \mathfrak{h})$ required by Theorem \ref{thm:finalreducedsystem} is determined by the equation $\bar{\cH}'' + K' \bar{\cH}' = 0$, which yields:
\begin{equation*}
\bar{\cH}(r) = c_1 \int e^{-K(r)} \dd r + c_2
\end{equation*}
for integration constants $c_1, c_2 \in \mathbb{R}$. This provides a systematic construction of examples of $\fru$-independent non-twisting Kundt solutions to minimal six-dimensional supergravity.\medskip
 
We proceed to apply the previous formalism to the $\lambda < 0$ case, which we solve explicitly. As explained above, the first task is to solve Equation \eqref{eq:linearwarp}. Defining the characteristic frequency $k := \sqrt{2|\lambda|} > 0$, the linear equation \eqref{eq:linearwarp} for the warp factor $Z = e^K$ becomes $Z'' = -k^2 Z$, whose general solution with zero phase is given by:
\begin{equation}
e^{K(r)} = c \cos(kr) \, ,
\end{equation}
for a real non-zero constant $c\in \mathbb{R}^{\ast}$. To ensure a well-defined metric $\mathfrak{h}$, we restrict the radial coordinate to the interval $r \in (-\frac{\pi}{2k}, \frac{\pi}{2k})$ and hence we take $c > 0$ to be positive. Hence:
\begin{equation*}
K(r) = \ln(c) + \ln(\cos(kr))
\end{equation*}
solves the second equation in \eqref{eq:radialODE}, whereas the first equation in \eqref{eq:radialODE} becomes:
\begin{equation*}
 \cF'' -k \tan(kr) \cF' = \frac{\mathfrak{e}^2 e^{2\cF}}{c^2 \cos^2(kr)}\, ,
\end{equation*}
where we have used that $K' = -k \tan(kr)$. In order to proceed further, it is convenient to introduce the coordinate $\rho$, defined via the differential equation $\dd\rho = e^{-K} \dd r$, which is integrated by:
\begin{equation*}
\rho(r) = \frac{1}{c k} \ln \left| \sec(kr) + \tan(kr) \right| \, .
\end{equation*}

\noindent
Note that as $r \to \pm \frac{\pi}{2k}$, the coordinate $\rho \to \pm \infty$, and hence $\rho$ spans the whole real line $\mathbb{R}$. In this coordinate, the previous equation for $\cF$ reduces to the following one-dimensional Liouville equation:
\begin{equation*}
\partial_{\rho}^2 \mathcal{F}= \mathfrak{e}^2 e^{2\mathcal{F}}\, .
\end{equation*}

\noindent
On the other hand, substituting the specific profiles for $K' = -k \tan(kr)$ and $\mathcal{F}' = e^{-K} \partial_{\rho}\mathcal{F}$ into the Hamiltonian constraint \eqref{eq:radialhamiltonian}, we obtain:
\begin{equation*}
3k^2 \tan^2(kr) - 2 e^{-2K} \left[ (\partial_\rho\mathcal{F})^2 - \mathfrak{e}^2 e^{2\mathcal{F}} \right] + 3k^2 = 0 \implies E = \frac{3}{2} c^2 k^2 = 3 c^2 |\lambda| > 0 \, ,
\end{equation*}
where $E := (\partial_\rho\mathcal{F})^2 - \mathfrak{e}^2 e^{2\mathcal{F}}$ is the constant, necessarily positive, Liouville energy. Since $E > 0$, it follows that $\cF$ belongs to the hyperbolic branch of the Liouville solution, that is, we have:
\begin{equation*}
e^{-\mathcal{F}(\rho)} = \frac{\mathfrak{e}}{\sqrt{E}} \sinh(\sqrt{E}(\rho - \rho_*))
\end{equation*}
for a real integration constant $\rho_*\in \mathbb{R}$. Reconstructing the components for the six-dimensional Kundt metric, we find:
\begin{equation*}
\bar{\cH}(\rho) = m_1 \rho + m_2 \, , \qquad \cH(\rho) = \frac{c\sqrt{3|\lambda|}\left( m_1 \rho + m_2 \right)}{\mathfrak{e} \sinh \big( c \sqrt{3|\lambda|} (\rho - \rho_*)\big)}
\end{equation*}
for real constants $m_1 , m_2 \in \mathbb{R}$. The resulting solution describes a wave-front that topologically pinches off as $\rho$ approaches the limits $\rho \to \pm \infty$.\medskip

The cases $\lambda > 0$ and $\lambda = 0$ are solved analogously. Hence, we obtain the following.

\begin{cor}
Every Einstein three-manifold admits a conformal embedding into a $\fru$-independent non-twisting solution of minimal six-dimensional supergravity.
\end{cor}

\begin{remark}
Non-supersymmetric solutions of minimal six-dimensional supergravity have been considered previously in \cite{Murcia:2019cck} using a construction based on a general notion of metric contact structure in three dimensions that allows the Reeb vector field to be isotropic. It would be interesting to compare the solutions obtained in this section with the formalism and solutions obtained in \emph{opere citato}.
\end{remark}


\subsection{On quasi-supersymmetric wave-fronts}


In this subsection, we relate the six-dimensional minimal supergravity solutions that are constructed via solutions of the reduced system presented in Theorem \ref{thm:finalreducedsystem} with the notion of \emph{quasi-supersymmetry}. That is, by applying Lemma \ref{lemma:skew_torsion_Hermitian}, we immediately obtain a key necessary condition for such a solution to be quasi-supersymmetric. This necessary solution is strong enough that we can conclude that solutions of the reduced system \eqref{eq:finalreducedsystem}, and in particular the class of solutions obtained in Subsection \ref{eq:familiessolutions}, will typically lead to non-quasi-supersymmetric solutions of minimal six-dimensional supergravity.\medskip

Hence, we take $M=\mathbb{R}^2 \times N$ for an orientable four-dimensional manifold $N$, and we consider the six-dimensional minimal supergravity determined on $M$ by a reducible bundle gerbe with connective structure. As explained in Theorem \ref{thm:finalreducedsystem}, given such a triple $(\mathfrak{h},b,\cF,\bar{\cH})$, the corresponding six-dimensional solution $(g,\bar{b})$ has metric $g$ given by:
\begin{equation*}
g = e^{\cF} (\bar{\cH} \dd\fru \otimes \dd \fru + \dd\fru\odot \dd\frv) + e^{-\cF}\mathfrak{h}\, ,
\end{equation*}

\noindent
whereas the curvature of the curving $\bar{b}$ reads:
\begin{equation}
\label{eq:reducedH}
H_{\bar{b}} =  H_b + \mu e^{2\cF} \dd\fru \wedge \dd \frv  \wedge \ast_{\mathfrak{h}} H_b\, .
\end{equation}

\noindent
Such a class of six-dimensional solutions may admit various types of skew-torsion parallel spinors with torsion $H_{\bar{b}}$. Given a tuple $(\mathfrak{h},b,\cF,\bar{\cH})$, we say that an irreducible complex spinor $\eta\in \Gamma(S^{\mu})$ on $(M,g)$ is \emph{adapted} if its Hermitian square $\widehat{\alpha}=u+iu\wedge\omega-\mu*_{g}u$ satisfies $u = e^{\cF} \dd\fru$. If such a solution is quasi-supersymmetric with respect to an adapted spinor, then we will say that it is \emph{adapted quasi-supersymmetric}.

\begin{prop}
\label{prop:conditiondualH}
A solution $(\mathfrak{h},b,\cF,\bar{\cH})$ of the reduced system \eqref{eq:finalreducedsystem} defines an adapted quasi-super\-symmetric solution of minimal six-dimensional supergravity only if:
\begin{equation*}
\ast_{\mathfrak{h}}H_b = -\mu e^{-\cF} \dd\cF\, .
\end{equation*}
In particular, $\mathfrak{h}$ is Ricci-flat.
\end{prop}

\begin{proof}
Applying Lemma \ref{lemma:skew_torsion_Hermitian}, it follows that a tuple $(h,b,\cF)$ is an adapted quasi-supersymmetric configuration only if:  
\begin{equation*} 
\nabla^g_w (e^{\cF} \dd \fru) = \frac{1}{2}H_{\bar{b}}(w,\partial_{\frv}),\qquad  w\in\Gamma(TM)\, ,
\end{equation*}

\noindent
where $H_{\bar{b}}$ is given as in Equation \eqref{eq:reducedH}. The symmetric projection of this equation is automatically satisfied. On the other hand, its skew-symmetric projection gives:
\begin{equation*}
e^{\cF} \dd\cF \wedge \dd\fru  + H_{\bar{b}}(\partial_{\frv}) = 0\, ,
\end{equation*}
which, upon use of Equation \eqref{eq:reducedH}, reduces to:
\begin{equation*}
\ast_{\mathfrak{h}}H_b = -\mu e^{-\cF} \dd\cF\, .
\end{equation*}

\noindent
Assuming this equation, we compute:
\begin{align*}
(H_b\circ_{\mathfrak{h}} H_b)(w_1 , w_2) &= e^{-2\cF} \langle \ast_{\mathfrak{h}}(\dd\cF \wedge w_1^{\flat_{\mathfrak{h}}}) , \ast_{\mathfrak{h}}(\dd\cF \wedge w_2^{\flat_{\mathfrak{h}}}) \rangle_{\mathfrak{h}} \\
&= e^{-2\cF} (\vert \dd\cF\vert_{\mathfrak{h}}^2 \mathfrak{h} - \dd\cF \otimes \dd\cF)(w_1,w_2)\, ,
\end{align*}
which plugged into the first equation in \eqref{eq:finalreducedsystem} gives $\mathrm{Ric}^{\mathfrak{h}} = 0$.
\end{proof}

Elaborating on Proposition \ref{prop:conditiondualH}, we obtain the following strict characterization of tuples $(\mathfrak{h},b,\cF,\bar{\cH})$ satisfying $\ast_{\mathfrak{h}}H_b = -\mu e^{-\cF} \dd\cF$ and defining $\fru$-independent, non-twisting, solutions of six-dimensional supergravity as explained in Theorem \ref{thm:finalreducedsystem}.

\begin{prop}
A tuple $(\mathfrak{h},b,\cF,\bar{\cH})$ satisfying $\ast_{\mathfrak{h}}H_b = -\mu e^{-\cF} \dd\cF$ defines a $\fru$-independent, non-twisting, solution of minimal six-dimensional supergravity if and only if $\mathfrak{h}$ is Ricci-flat and both $\smash{\bar{\cH}}$ and $e^{-\cF}$ are harmonic functions on $(N,\mathfrak{h})$.   
\end{prop}

\begin{proof}
By Proposition \ref{prop:conditiondualH}, the Einstein equation of the reduced system \eqref{eq:finalreducedsystem} evaluated on a tuple $(\mathfrak{h},b,\cF,\bar{\cH})$ such that $\ast_{\mathfrak{h}}H_b = -\mu e^{-\cF} \dd\cF$ simplifies to the condition $\mathrm{Ric}^{\mathfrak{h}} = 0$. The Maxwell equation in \eqref{eq:finalreducedsystem} is clearly satisfied automatically, since:
\begin{equation*}
\dd (e^{2\cF} \ast_{\mathfrak{h}} H_b) = - \mu\,\dd( e^{\cF} \dd\cF) = 0\, .
\end{equation*}

\noindent
For the remaining equation, we observe:
\begin{equation*}
0 = \nabla^{\mathfrak{h}\ast}\dd \cF  + e^{2\cF} \vert H_b \vert^2_{\mathfrak{h}} = \nabla^{\mathfrak{h}\ast}\dd \cF  +  \vert \dd\cF \vert^2_{\mathfrak{h}} = - e^{\cF} \nabla^{\mathfrak{h}\ast}\dd e^{-\cF}
\end{equation*}
and thus we conclude.
\end{proof}

Hence, every tuple $(\mathfrak{h},b,\cF,\bar{\cH})$ determining a quasi-supersymmetric solution of minimal six-dimensional supergravity via Theorem \ref{thm:finalreducedsystem} has conformally Ricci-flat wave-front $(N,h)$, which implies the following.

\begin{cor}
\label{cor:nonquasisusy}
The class of solutions constructed in Subsection \ref{eq:familiessolutions} with Einstein constant $\lambda\neq 0$ is not quasi-supersymmetric.
\end{cor}


\appendix



\section{The Kähler-Atiyah model for the Clifford algebra}
\label{app:KA}


Let $V$ be an oriented $d$-dimensional real vector space equipped with a non-degenerate metric $h$ of signature $(p,q)$ and let $(V^*,h^*)$ be the quadratic space dual to $(V,h)$, where $h^*$ denotes the metric dual to $h$. Let $\mathrm{Cl}(V^*,h^*)$ be the real Clifford algebra of the quadratic vector space $(V^*,h^*)$, viewed as a $\Z_2$-graded associative algebra with decomposition: $$\mathrm{Cl}(V^*,h^*)=\mathrm{Cl}^{\mathrm{ev}}(V^*,h^*)\oplus\mathrm{Cl}^{\mathrm{odd}}(V^*,h^*).$$

In our conventions the Clifford algebra satisfies: $$\theta^2=h^*(\theta,\theta),\qquad\theta\in V^*.$$

We identify the real Clifford algebra $\mathrm{Cl}(V^*,h^*)$ with the \emph{Kähler-Atiyah algebra} of $(V^*,h^*)$, which we denote by $(\wedge V^*,\diamond)$ (see \cite{Chev54,Chev55}). The map $\diamond\colon\wedge V^*\times\wedge V^*\to\wedge V^*$ denotes the \emph{geometric product} determined by $h$. This is given by the linear and associative extension of the following expression: \begin{equation}\label{eq:def_geometric_product}
    \theta\diamond\alpha=\theta\wedge\alpha+\iota_{\theta^\sharp}\alpha,\qquad\theta\in V^*,\,\alpha\in\wedge V^*,
\end{equation} where $\theta^\sharp\in V$ denotes the $h$-dual vector of the one-form $\theta$.\medskip

In order to do computations with the geometric product, it is convenient to introduce the \emph{generalized products} of $(V^*,h^*)$. These are the bilinear operators:
\begin{equation}
\label{eq:bilinearDelta}
\triangle_k\colon\wedge^i V^*\times\wedge^j V^*\to\wedge^{i + j - 2k}V^*,
\end{equation}

\noindent
where $k=0,\ldots,d$, defined through the expansion: 
\begin{equation*}
\alpha\diamond\beta=\sum_{k=0}^d(-1)^{\binom{k+1}{2}+jk}\alpha\triangle_k\beta,\qquad\alpha\in\wedge^j V^*,\,\beta\in\wedge V^*.
\end{equation*}

\noindent
Choosing a basis $\{e_1,\ldots,e_d\}$ of $V$ we can express the generalized products as: 
\begin{equation*}
    \alpha\triangle_k\beta=\tfrac{1}{k!}h^{i_1j_1}\cdots h^{i_kj_k}(\iota_{e_{i_1}}\ldots\iota_{e_{i_k}}\alpha)\wedge(\iota_{e_{j_1}}\ldots\iota_{e_{j_k}}\beta).
\end{equation*}

\noindent
Below, we collect some useful properties of the generalized products that we will use in the computations.

\begin{prop}[{\cite{LB13,LBC13,LBC16}}]\label{prop:appendix_properties}
    Let $\alpha\in\wedge^aV^*$ and $\beta\in\wedge^bV^*$. Then: \begin{itemize}
        \item $\alpha\triangle_k\beta=0$ if $k>a$ or $k>b$.
        \item $\alpha\triangle_k\beta=(-1)^{(a-k)(b-k)}\beta\triangle_k\alpha$. In particular $\alpha\triangle_k\alpha=0$ if $a-k$ is odd.
        \item $\alpha\triangle_0\beta=\alpha\wedge\beta$ and $\alpha\triangle_a\beta=\escal{\alpha,\beta}$ if $b=a$.
        \item $\alpha\triangle_a(*\beta)=*(\beta\wedge\alpha)$ if $a+b\leq d$.
    \end{itemize}
\end{prop}

\begin{proof}
    The first three properties are clear. To prove the last property we use $\iota_v(*\beta)=*(\beta\wedge v^\flat)$ for $v\in V$ and $\beta\in\wedge V^*$, and $(e_j)^\flat=h_{jk}e^k$: \begin{align*}
        \alpha\triangle_a(*\beta)&=\tfrac{1}{a!}h^{i_1j_1}\cdots h^{i_aj_a}(\iota_{e_{i_1}}\ldots\iota_{e_{i_a}}\alpha)\wedge(\iota_{e_{j_1}}\ldots\iota_{e_{j_a}}(*\beta))\\
        &=\tfrac{1}{a!}h^{i_1j_1}\cdots h^{i_aj_a}\alpha_{{i_a}\cdots{i_1}}*(\beta\wedge(e_{j_a})^\flat\wedge\cdots\wedge(e_{j_1})^\flat)\\
        &=\tfrac{1}{a!}h^{i_1j_1}\cdots h^{i_aj_a}h_{j_ak_a}\cdots h_{j_1k_1}*(\beta\wedge\alpha_{{i_a}\cdots{i_1}}e^{k_a}\wedge\cdots\wedge e^{k_1})\\
        &=*(\beta\wedge(\tfrac{1}{a!}\alpha_{{i_1}\cdots{i_a}}e^{i_1}\wedge\cdots\wedge e^{i_a}))\\
        &=*(\beta\wedge\alpha).
    \end{align*}
\end{proof}

As a unital and associative algebra, the Kähler-Atiyah algebra $(\wedge V^*,\diamond)$ is isomorphic to the Clifford algebra $\mathrm{Cl}(V^*,h^*)$ through the $h$-dependent \emph{Chevalley-Riesz isomorphism} (see \cite{CLS21,LBC13,LBC16}), which we denote by: \begin{equation}\label{eq:ChevRiesz_iso}
    \Psi\colon(\wedge V^*,\diamond)\to\mathrm{Cl}(V^*,h^*).
\end{equation}

We denote by $\pi$ the \emph{parity automorphism} of the Kähler-Atiyah algebra, which is defined as the unique unital algebra automorphism which acts as minus the identity on $V^*\subset\wedge V^*$, and by $\tau$ the \emph{reversion anti-automorphism}, defined as the unique unital algebra anti-automorphism which acts as the identity on $V^*$. For $\alpha\in\wedge^aV^*$ we have: \begin{equation*}
    \pi(\alpha)=(-1)^a\alpha,\qquad\tau(\alpha)=(-1)^{\binom{a}{2}}\alpha.
\end{equation*}

Note that $\tau\circ\pi=\pi\circ\tau$ and $(\pi\circ\tau)(\alpha)=(-1)^{\binom{a+1}{2}}\alpha$.\medskip

Let $\nu\in\wedge^dV^*$ be the pseudo-Riemannian volume form on $(V,h)$. Then we have: $$\nu\diamond\nu=\begin{cases}
    (-1)^{\frac{p-q}{2}}&\mbox{ if $d$ is even},\\
    (-1)^{\frac{p-q-1}{2}}&\mbox{ if $d$ is odd},
\end{cases}$$ and: $$\nu\diamond\alpha=\pi^{d-1}(\alpha)\diamond\nu,\qquad \alpha\in\wedge V^*.$$

Moreover, the volume form $\nu$ satisfies the following properties:

\begin{prop}\label{prop:product_volume_form}
Let $(V,h)$ be a quadratic vector space of signature $(p,q)$ and dimension $d$. Then the following identities hold for all $\alpha\in\wedge V^*$:
\begin{equation}\label{eq:diamond_volume}
    \alpha\diamond\nu=*\tau(\alpha),\qquad \nu\diamond\alpha=*(\pi^{d-1}\circ\tau)(\alpha).
\end{equation}
\end{prop}

\begin{proof}
    Since multiplication by the volume form $\nu$ is $\R$-linear, it suffices to consider homogeneous elements $\alpha=e^{i_1}\wedge\cdots\wedge e^{i_k}$ with $1\leq i_1<\cdots<i_k\leq d$, where $\{e^1,\ldots,e^d\}$ is an orthonormal basis of $(V^*,h^*)$. Set $\varepsilon_i:=h^*(e^i,e^i)\in\Z_2$. Let us start with $\alpha=e^i$. On one hand we have: $$e^i\diamond\nu=(-1)^{i-1}\varepsilon_ie^1\diamond\cdots\diamond e^{i-1}\diamond e^{i+1}\diamond\cdots\diamond e^d.$$
    
    On the other hand we have $e^i\diamond(*e^i)=e^i\wedge(*e^i)=h^*(e^i,e^i)\nu=\varepsilon_i\nu$, thus: $$*e^i=\varepsilon_i(e^i)^{-1}\diamond\nu=(-1)^{i-1}\varepsilon_ie^1\diamond\cdots\diamond e^{i-1}\diamond e^{i+1}\diamond\cdots\diamond e^d$$ and $e^i\diamond\nu=*e^i=*\tau(e^i)$. Arguing analogously for $\alpha=e^{i_1}\wedge\cdots\wedge e^{i_k}$ we get on one hand: $$\alpha\diamond\nu=(-1)^{i_1+\cdots+i_k-k}\varepsilon_{i_1}\cdots\varepsilon_{i_k}e^1\diamond\cdots\diamond e^{i_1-1}\diamond e^{i_1+1}\diamond\cdots\diamond e^{i_k-1}\diamond e^{i_k+1}\diamond\cdots\diamond e^d,$$ and on the other hand: $$(-1)^{\binom{k}{2}}*(e^{i_1}\wedge\cdots\wedge e^{i_k})=\varepsilon_{i_1}\cdots\varepsilon_{i_k}(e^{i_1})^{-1}\diamond\cdots\diamond(e^{i_k})^{-1}\diamond\nu=\alpha\diamond\nu.$$
    
    Hence $\alpha\diamond\nu=(-1)^{\binom{k}{2}}*\alpha=*\tau(\alpha)$. Using $\nu\diamond\alpha=(-1)^{k(d-1)}\alpha\diamond\nu=\pi^{d-1}(\alpha)\diamond\nu$ we conclude.
\end{proof}


\bibliographystyle{myamsplain}
\bibliography{biblio}

\end{document}

%% file: Preamble.tex
\usepackage[a4paper,
left=1in,right=1in,top=1.2in,bottom=1.2in,%
footskip=.25in]{geometry}

\usepackage{bm}
\usepackage{latexsym, amsmath, amstext, amssymb, amsfonts, amscd, bm, array, multirow, amsbsy, mathrsfs, stmaryrd}
\usepackage{amsthm}
\usepackage{t1enc}
\usepackage[mathscr]{eucal}
\usepackage{indentfirst}
\usepackage{pb-diagram}
\usepackage{graphicx}
\usepackage{fancyhdr}
\usepackage{fancybox}
\usepackage{color}
\usepackage{tikz-cd}
\usetikzlibrary{arrows}
\usepackage[all]{xy}
\usepackage{hyperref}
\usepackage{tikz}
\usepackage{xparse}
\hypersetup{colorlinks=false,pdfborderstyle={/S/U/W 0}}
\usetikzlibrary{matrix}
\usepackage{upgreek}
\usepackage[utf8]{inputenc}
\usepackage[T1]{fontenc}
\usepackage{bbm}
\usepackage{enumitem}
\setlist[enumerate]{leftmargin=*,noitemsep}
\setlist[itemize]{leftmargin=*,noitemsep}
\usepackage{float}
\usepackage{soul}

\newcommand{\doi}[1]{\url{https://doi.org/#1}}
\usepackage[noadjust]{cite}

\makeatletter
\def\@defaultbiblabelstyle#1{[#1]}
\makeatother

\usepackage{url}
\theoremstyle{plain}
\newtheorem{thm}{Theorem}[section]

\newtheorem{prop}[thm]{Proposition}

\newtheorem{lemma}[thm]{Lemma}
\newtheorem{cor}[thm]{Corollary}

\theoremstyle{definition}
\newtheorem{definition}[thm]{Definition}

\theoremstyle{remark}
\newtheorem{remark}[thm]{Remark}
\newtheorem{example}[thm]{Example}

\newtheorem*{ack}{Acknowledgements}

\newcommand{\Sym}{\mathrm{Sym}}

\newcommand{\End}{\mathrm{End}}

\newcommand{\Mat}{\mathrm{Mat}}

\DeclareFontFamily{U}{rsf}{}
\DeclareFontShape{U}{rsf}{m}{n}{<5> <6> rsfs5 <7> <8> <9> rsfs7 <10-> rsfs10}{}
\DeclareMathAlphabet\Scr{U}{rsf}{m}{n}

\def\N{\mathbb{N}}
\def\Z{\mathbb{Z}}
\def\C{\mathbb{C}}
\def\R{\mathbb{R}}

\def\rk{{\rm rk}}

\def\dd{\mathrm{d}}

\def\v{\mathbf{\mathrm{v}}}

\def\Ad{\mathrm{Ad}}

\def\cQ{\mathcal{Q}}

\def\Id{\mathrm{Id}}

\def\fra{\mathfrak{a}}
\def\frq{\mathfrak{q}}
\def\frf{\mathfrak{f}}

\newcommand{\be}{\begin{equation*}}
\newcommand{\ee}{\end{equation*}}
\newcommand{\ben}{\begin{equation}}
\newcommand{\een}{\end{equation}}
\newcommand{\beqa}{\begin{eqnarray*}}
\newcommand{\eeqa}{\end{eqnarray*}}
\newcommand{\beqan}{\begin{eqnarray}}
\newcommand{\eeqan}{\end{eqnarray}}

\newcommand{\Tr}{\mathrm{Tr}}

\def\cR{{\mathcal R}}

\def\scB{\Scr B}
\def\scS{\Scr H}

\def\cH{\mathcal{H}}

\def\Spin{\mathrm{Spin}}

\def\Spin{\mathrm{Spin}}

\def\SO{\mathrm{SO}}

\def\U{\mathrm{U}}

\def\cD{\mathcal{D}}

\def\cA{\mathcal{A}}
\def\cE{\mathcal{E}}
\def\hcE{\smash{\widehat{\mathcal{E}}}}
\def\cI{\mathcal{I}}
\def\cP{\mathcal{P}}

\def\cP{\mathcal{P}}
\def\cF{\mathcal{F}}

\def\G_2{\mathrm{G_2}}

\def\P{\mathbb{P}}

\newcommand{\Lor}{{\rm Lor}}

\def\v{\mathbf{v}}

\def\Im{\mathrm{Im}}

\def\G{\mathrm{G}}

\def\R{\mathbb{R}}

\def\mCl{\mathbb{C}\mathrm{l}}

\def\cL{\mathcal{L}}
\def\cH{\mathcal{H}}
\def\dd{\mathrm{d}}

\newcommand{\escal}[1]{\langle#1\rangle}

\def\fru{\mathfrak{u}}
\def\frv{\mathfrak{v}}

\def\frc{\mathfrak{F}}
\def\frc{\mathfrak{c}}

\def\Ric{\mathrm{Ric}}

\def\Hess{\mathrm{Hess}}

\renewcommand{\ker}{\mathrm{Ker}}

\newcolumntype{P}[1]{>{\centering\arraybackslash}p{#1}}

\usepackage[math]{anttor}
\usepackage{orcidlink}

\allowdisplaybreaks

%% file: biblio.bib
@MISC{Algebraic_Complex_Square_2025,
    author       = "A. Gil-García and C. S. Shahbazi",
    title        = "The algebraic square of an irreducible complex spinor",
    howpublished = "\url{https://arxiv.org/abs/2510.13801}",
    year         = "2025"
}

@ARTICLE{CLS21,
    author  = "V. Cortés and C. Lazaroiu and C. S. Shahbazi",
    title   = "Spinors of real type as polyforms and the generalized {K}illing equation",
    journal = "Math. Z.",
    year    = "2021",
    volume  = "299",
    number  = "3-4",
    pages   = "1351-1419",
    note    = "\doi{10.1007/s00209-021-02726-6}"
}

@MISC{Sha24,
    author       = "C. S. Shahbazi",
    title        = "Differential spinors and {K}undt three-manifolds with skew-torsion",
    howpublished = "\url{https://arxiv.org/abs/2405.03756}",
    year         = "2024"
}

@ARTICLE{LBC13,
    author  = "C. I. Lazaroiu and E. M. Babalic and I. A. Coman",
    title   = "The geometric algebra of {F}ierz identities in arbitrary dimensions and signatures",
    journal = "J. High Energy Phys.",
    year    = "2013",
    volume  = "2013",
    number  = "9",
    pages   = "156",
    note    = "\doi{10.1007/JHEP09(2013)156}"
}

@ARTICLE{LBC16,
    author  = "C. I. Lazaroiu and E. M. Babalic and I. A. Coman",
    title   = "Geometric algebra techniques in flux compactifications",
    journal = "Adv. High Energy Phys.",
    year    = "2016",
    volume  = "2016",
    number  = "1",
    pages   = "7292534 (42 pp)",
    note    = "\doi{10.1155/2016/7292534}"
}

@ARTICLE{AC97,
    author  = "D. V. Alekseevsky and V. Cortés",
    title   = "Classification of ${N}$-(super)-extended {P}oincaré algebras and bilinear invariants of the spinor representation of $\mathrm{Spin}(p,q)$",
    journal = "Commun. Math. Phys.",
    year    = "1997",
    volume  = "183",
    number  = "3",
    pages   = "477-510",
    note    = "\doi{10.1007/s002200050039}"
}

@ARTICLE{ACDVP05,
    author  = "D. V. Alekseevsky and V. Cortés and C. Devchand and A. Van Proeyen",
    title   = "Polyvector super-{P}oincaré algebras",
    journal = "Commun. Math. Phys.",
    year    = "2005",
    volume  = "253",
    number  = "2",
    pages   = "385-422",
    note    = "\doi{10.1007/s00220-004-1155-y}"
}

@ARTICLE{LB13,
    author  = "C. I. Lazaroiu and E. M. Babalic",
    title   = "Geometric algebra techniques in flux compactifications {II}",
    journal = "J. High Energy Phys.",
    year    = "2013",
    volume  = "2013",
    number  = "6",
    pages   = "054",
    note    = "\doi{10.1007/JHEP06(2013)054}"
}

@ARTICLE{LS18,
    author  = "C. I. Lazaroiu and C. S. Shahbazi",
    title   = "Complex {L}ipschitz structures and bundles of complex {C}lifford modules",
    journal = "Differential Geom. Appl.",
    year    = "2018",
    volume  = "61",
    pages   = "147-169",
    note    = "\doi{10.1016/j.difgeo.2018.08.006}"
}

@ARTICLE{LS19,
    author  = "C. I. Lazaroiu and C. S. Shahbazi",
    title   = "Real pinor bundles and real {L}ipschitz structures",
    journal = "Asian J. Math.",
    year    = "2019",
    volume  = "23",
    number  = "5",
    pages   = "749-836",
    note    = "\doi{10.4310/AJM.2019.v23.n5.a3}"
}

@ARTICLE{LS22_complex_type,
    author  = "C. Lazaroiu and C. S. Shahbazi",
    title   = "Dirac operators on real spinor bundles of complex type",
    journal = "Differential Geom. Appl.",
    year    = "2022",
    volume  = "80",
    pages   = "101849 (53 pp)",
    note    = "\doi{10.1016/j.difgeo.2022.101849}"
}

@ARTICLE{BGM05,
    author  = "C. Bär and P. Gauduchon and A. Moroianu",
    title   = "Generalized cylinders in semi-{R}iemannian and spin geometry",
    journal = "Math. Z.",
    year    = "2005",
    volume  = "249",
    number  = "3",
    pages   = "545-580",
    note    = "\doi{10.1007/s00209-004-0718-0}"
}

@ARTICLE{Ikemakhen06,
    author  = "A. Ikemakhen",
    title   = "Parallel spinors on pseudo-{R}iemannian spin$^c$ manifolds",
    journal = "J. Geom. Phys.",
    year    = "2006",
    volume  = "56",
    number  = "9",
    pages   = "1473-1483",
    note    = "\doi{10.1016/j.geomphys.2005.07.005}"
}

@ARTICLE{Ikemakhen07,
    author  = "A. Ikemakhen",
    title   = "Parallel spinors on {L}orentzian spin$^c$ manifolds",
    journal = "Differential Geom. Appl.",
    year    = "2007",
    volume  = "25",
    number  = "3",
    pages   = "299-308",
    note    = "\doi{10.1016/j.difgeo.2006.11.008}"
}

@ARTICLE{Murcia:2020zig,
    author  = "Á. Murcia and C. S. Shahbazi",
    title   = "Parallel spinors on globally hyperbolic {L}orentzian four-manifolds",
    journal = "Ann. Glob. Anal. Geom.",
    year    = "2022",
    volume  = "61",
    number  = "2",
    pages   = "253-292",
    note    = "\doi{10.1007/s10455-021-09808-y}"
}

@ARTICLE{Murcia:2021dur,
    author  = "Á. Murcia and C. S. Shahbazi",
    title   = "Parallel spinor flows on three-dimensional {C}auchy hypersurfaces",
    journal = "J. Phys. A: Math. Theor.",
    year    = "2023",
    volume  = "56",
    number  = "20",
    pages   = "205204 (28 pp)",
    note    = "\doi{10.1088/1751-8121/accd2f}"
}

@ARTICLE{Murray1996,
    author  = "M. K. Murray",
    title   = "Bundle gerbes",
    journal = "J. London Math. Soc.",
    year    = "1996",
    volume  = "54",
    number  = "2",
    pages   = "403-416",
    note    = "\doi{10.1112/jlms/54.2.403}"
}

@ARTICLE{Figueroa_Lorentzian_6d_2018,
    author  = "P. de Medeiros and J. Figueroa-O'Farrill and A. Santi",
    title   = "Killing superalgebras for {L}orentzian six-manifolds",
    journal = "J. Geom. Phys.",
    year    = "2018",
    volume  = "132",
    pages   = "13-44",
    note    = "\doi{10.1016/j.geomphys.2018.05.019}"
}

@ARTICLE{Krasnov2024,
    author  = "K. Krasnov",
    title   = "Lorentzian {C}ayley form",
    journal = "J. Geom. Phys.",
    year    = "2024",
    volume  = "201",
    pages   = "105211 (25 pp)",
    note    = "\doi{10.1016/j.geomphys.2024.105211}"
}

@ARTICLE{Moroianu1997,
    author  = "A. Moroianu",
    title   = "Parallel and {K}illing spinors on spin$^c$ manifolds",
    journal = "Commun. Math. Phys.",
    year    = "1997",
    volume  = "187",
    number  = "2",
    pages   = "417-427",
    note    = "\doi{10.1007/s002200050142}"
}

@ARTICLE{Herzlich_Moroianu_1999,
    author  = "M. Herzlich and A. Moroianu",
    title   = "Generalized {K}illing spinors and conformal eigenvalue estimates for spin$^c$ manifolds",
    journal = "Ann. Glob. Anal. Geom.",
    year    = "1999",
    volume  = "17",
    number  = "4",
    pages   = "341-370",
    note    = "\doi{10.1023/A:1006546915261}"
}

@ARTICLE{Grosse_Nakad_2015,
    author  = "N. Große and R. Nakad",
    title   = "Complex generalized {K}illing spinors on {R}iemannian spin$^c$ manifolds",
    journal = "Results Math.",
    year    = "2015",
    volume  = "67",
    number  = "1-2",
    pages   = "177-195",
    note    = "\doi{10.1007/s00025-014-0401-7}"
}

@MISC{Lockman_2025,
    author       = "S. Lockman",
    title        = "Semi-{R}iemannian spin$^c$ manifolds carrying generalized {K}illing spinors and the classification of {R}iemannian spin$^c$ manifolds admitting a type {I} imaginary generalized {K}illing spinor",
    howpublished = "\url{https://arxiv.org/abs/2509.08477}",
    year         = "2025"
}

@ARTICLE{Freedman_1977,
    author  = "D. Z. Freedman",
    title   = "Supergravity with axial-gauge invariance",
    journal = "Phys. Rev. D",
    year    = "1977",
    volume  = "15",
    number  = "4",
    pages   = "1173-1174",
    note    = "\doi{10.1103/PhysRevD.15.1173}"
}

@ARTICLE{FriedrichTrautman,
    author  = "T. Friedrich and A. Trautman",
    title   = "Spin spaces, {L}ipschitz groups, and spinor bundles",
    journal = "Ann. Glob. Anal. Geom.",
    year    = "2000",
    volume  = "18",
    number  = "3-4",
    pages   = "221-240",
    note    = "\doi{10.1023/A:1006713405277}"
}

@BOOK{Chev54,
    author    = "C. Chevalley",
    title     = "The algebraic theory of spinors",
    publisher = "Columbia University Press",
    year      = "1954",
    note      = "\doi{10.7312/chev93056}"
}

@BOOK{Chev55,
    author    = "C. Chevalley",
    title     = "The construction and study of certain important algebras",
    publisher = "The Mathematical Society of Japan",
    year      = "1955"
}

@BOOK{Friedrich2000,
    author    = "T. Friedrich",
    title     = "Dirac operators in {R}iemannian geometry",
    publisher = "American Mathematical Society",
    year      = "2000"
}

@BOOK{Mei13,
    author    = "E. Meinrenken",
    title     = "Clifford algebras and {L}ie theory",
    publisher = "Springer Berlin, Heidelberg",
    year      = "2013",
    note      = "\doi{10.1007/978-3-642-36216-3}"
}

@BOOK{HarveyBook,
    author    = "F. R. Harvey",
    title     = "Spinors and calibrations",
    publisher = "Academic Press",
    year      = "1990",
    volume    = "9",
    series    = "Perspectives in mathematics"
}

@BOOK{Petersen2016,
    author    = "P. Petersen",
    title     = "Riemannian geometry",
    publisher = "Springer Cham",
    year      = "2016",
    note      = "\doi{10.1007/978-3-319-26654-1}"
}

@ARTICLE{AtiyahSeminar,
    author  = "M. Atiyah",
    title   = "What is a spinor?",
    journal = "Conference at the Institut des Hautes Etudes Scientifiques",
    year    = "2013",
    pages   = "Paris",
    note    = "\url{https://www.youtube.com/watch?v=SBdW978Ii_E}"
}

@PHDTHESIS{ShahbaziThesis,
    author = "C. S. Shahbazi",
    title  = "Torsion parallel spinors on {L}orentzian four-manifolds and supersymmetric evolution flows on bundle gerbes",
    school = "Universität Hamburg",
    year   = "2025",
    note   = "\url{https://ediss.sub.uni-hamburg.de/handle/ediss/11620}"
}

@ARTICLE{Gutowski:2003rg,
    author  = "J. B. Gutowski and D. Martelli and H. S. Reall",
    title   = "All supersymmetric solutions of minimal supergravity in six dimensions",
    journal = "Class. Quantum Grav.",
    year    = "2003",
    volume  = "20",
    number  = "23",
    pages   = "5049-5078",
    note    = "\doi{10.1088/0264-9381/20/23/008}"
}

@ARTICLE{Tod:1983pm,
    author  = "K. P. Tod",
    title   = "All metrics admitting super-covariantly constant spinors",
    journal = "Phys. Lett. B",
    year    = "1983",
    volume  = "121",
    number  = "4",
    pages   = "241-244",
    note    = "\doi{10.1016/0370-2693(83)90797-9}"
}

@ARTICLE{Frolov:2005in,
    author  = "V. P. Frolov and D. V. Fursaev",
    title   = "Gravitational field of a spinning radiation beam pulse in higher dimensions",
    journal = "Phys. Rev. D",
    year    = "2005",
    volume  = "71",
    pages   = "104034 (16 pp)",
    note    = "\doi{10.1103/PhysRevD.71.104034}"
}

@ARTICLE{brinkmann1925einstein,
    author  = "H. W. Brinkmann",
    title   = "Einstein spaces which are mapped conformally on each other",
    journal = "Math. Ann.",
    year    = "1925",
    volume  = "94",
    number  = "1",
    pages   = "119-145",
    note    = "\doi{10.1007/BF01208647}"
}

@ARTICLE{Meessen:2006tu,
    author  = "P. Meessen and T. Ortín",
    title   = "The supersymmetric configurations of ${N}=2$, $d=4$ supergravity coupled to vector supermultiplets",
    journal = "Nucl. Phys. B",
    year    = "2006",
    volume  = "749",
    number  = "1-3",
    pages   = "291-324",
    note    = "\doi{10.1016/j.nuclphysb.2006.05.025}"
}

@ARTICLE{Meessen:2010ph,
    author  = "P. Meessen and T. Ortín",
    title   = "Ultracold spherical horizons in gauged ${N}=1$, $d=4$ supergravity",
    journal = "Phys. Lett. B",
    year    = "2010",
    volume  = "693",
    number  = "3",
    pages   = "358-361",
    note    = "\doi{10.1016/j.physletb.2010.08.050}"
}

@ARTICLE{Samann:2016svf,
    author  = "C. Sämann and R. Steinbauer and R. Svarc",
    title   = "Completeness of general pp-wave spacetimes and their impulsive limit",
    journal = "Class. Quantum Grav.",
    year    = "2016",
    volume  = "33",
    number  = "21",
    pages   = "215006 (27 pp)",
    note    = "\doi{10.1088/0264-9381/33/21/215006}"
}

@ARTICLE{Candela:2002rr,
    author  = "A. M. Candela and J. L. Flores and M. Sánchez",
    title   = "On general plane fronted waves. {G}eodesics",
    journal = "Gen. Relativ. Gravit.",
    year    = "2003",
    volume  = "35",
    number  = "4",
    pages   = "631-649",
    note    = "\doi{10.1023/A:1022962017685}"
}

@ARTICLE{Flores:2002fx,
    author  = "J. L. Flores and M. Sánchez",
    title   = "Causality and conjugate points in general plane waves",
    journal = "Class. Quantum Grav.",
    year    = "2003",
    volume  = "20",
    number  = "11",
    pages   = "2275-2291",
    note    = "\doi{10.1088/0264-9381/20/11/322}"
}

@ARTICLE{Duff_Lu_1994,
    author  = "M. J. Duff and J. X. Lu",
    title   = "Black and super $p$-branes in diverse dimensions",
    journal = "Nucl. Phys. B",
    year    = "1994",
    volume  = "416",
    number  = "1",
    pages   = "301-334",
    note    = "\doi{10.1016/0550-3213(94)90586-X}"
}

@ARTICLE{Gauntlett:2002nw,
    author  = "J. P. Gauntlett and J. B. Gutowski and C. M. Hull and S. Pakis and H. S. Reall",
    title   = "All supersymmetric solutions of minimal supergravity in five dimensions",
    journal = "Class. Quantum Grav.",
    year    = "2003",
    volume  = "20",
    number  = "21",
    pages   = "4587-4634",
    note    = "\doi{10.1088/0264-9381/20/21/005}"
}

@MISC{Chamseddine:2003yy,
    author       = "A. Chamseddine and J. Figueroa-O'Farrill and W. Sabra",
    title        = "Supergravity vacua and {L}orentzian {L}ie groups",
    howpublished = "\url{https://arxiv.org/abs/hep-th/0306278}",
    year         = "2003"
}

@MISC{Carmona_2026,
    author       = "J. L. {Carmona Jiménez}",
    title        = "On generalized imaginary {S}pin$^c$-{K}illing spinors",
    howpublished = "\url{https://arxiv.org/abs/2605.07813}",
    year         = "2026"
}

@ARTICLE{Legramandi:2019ulq,
    author  = "A. Legramandi and A. Tomasiello",
    title   = "Breaking supersymmetry with pure spinors",
    journal = "J. High Energy Phys.",
    year    = "2020",
    volume  = "2020",
    number  = "11",
    pages   = "098",
    note    = "\doi{10.1007/JHEP11(2020)098}"
}

@BOOK{Tomasiello:2022dwe,
    author    = "A. Tomasiello",
    title     = "Geometry of string theory compactifications",
    publisher = "Cambridge University Press",
    year      = "2022",
    note      = "\doi{10.1017/9781108635745}"
}

@BOOK{Ortin:2015hya,
    author    = "T. Ortín",
    title     = "Gravity and strings",
    publisher = "Cambridge University Press",
    year      = "2015",
    note      = "\doi{10.1017/CBO9781139019750}"
}

@BOOK{MarioStreets,
    author    = "M. Garcia-Fernandez and J. Streets",
    title     = "Generalized {R}icci flow",
    publisher = "American Mathematical Society",
    year      = "2021",
    note      = "\doi{10.1090/ulect/076}"
}

@MISC{Garcia-Fernandez:2015lsa,
    author       = "M. Garcia-Fernandez and C. S. Shahbazi",
    title        = "Self-dual generalized metrics for pure $\mathcal{N}=1$ six-dimensional supergravity",
    howpublished = "\url{https://arxiv.org/abs/1505.03088}",
    year         = "2015"
}

@ARTICLE{StreetsUstinovskiy,
    author  = "J. Streets and Y. Ustinovskiy",
    title   = "Classification of generalized {K}ähler-{R}icci solitons on complex surfaces",
    journal = "Comm. Pure Appl. Math.",
    year    = "2021",
    volume  = "74",
    number  = "9",
    pages   = "1896-1914",
    note    = "\doi{10.1002/cpa.21947}"
}

@ARTICLE{StreetsUstinovskiyII,
    author  = "J. Streets and Y. Ustinovskiy",
    title   = "The {G}ibbons-{H}awking ansatz in generalized {K}ähler geometry",
    journal = "Commun. Math. Phys.",
    year    = "2022",
    volume  = "391",
    number  = "2",
    pages   = "707-778",
    note    = "\doi{10.1007/s00220-022-04329-6}"
}

@MISC{Bunk:2025glt,
    author       = "S. Bunk and M. {Pino Carmona} and C. S. Shahbazi",
    title        = "The {C}auchy problem for gradient generalized {R}icci solitons on a bundle gerbe",
    howpublished = "\url{https://arxiv.org/abs/2510.25888}",
    year         = "2025"
}

@ARTICLE{StreetsI,
    author  = "J. Streets",
    title   = "Classification of solitons for pluriclosed flow on complex surfaces",
    journal = "Math. Ann.",
    year    = "2019",
    volume  = "375",
    number  = "3-4",
    pages   = "1555-1595",
    note    = "\doi{10.1007/s00208-019-01887-4}"
}

@ARTICLE{Murcia:2019cck,
    author  = "Á. Murcia and C. S. Shahbazi",
    title   = "Contact metric three manifolds and {L}orentzian geometry with torsion in six-dimensional supergravity",
    journal = "J. Geom. Phys.",
    year    = "2020",
    volume  = "158",
    pages   = "103868 (37 pp)",
    note    = "\doi{10.1016/j.geomphys.2020.103868}"
}
